\documentclass{article}
\usepackage[a4paper,margin=0.75in]{geometry}
\usepackage{graphicx} 
\usepackage{amsmath,amssymb,amsthm,xcolor, array,mathdots,comment} 
\usepackage{dsfont}
\usepackage[font=small]{caption}
\usepackage{subcaption}
\usepackage{hyperref}
\hypersetup{
    colorlinks,
    linkcolor={red!80!black},
    citecolor={blue!80!black},
    urlcolor={blue!80!black}
}
\usepackage{cleveref}
\usepackage{enumitem}
\usepackage{thmtools}
\usepackage{thm-restate}
\newtheorem{definition}{Definition}[section]
\newtheorem{lemma}[definition]{Lemma}
\newtheorem{theorem}[definition]{Theorem}
\newtheorem{cor}[definition]{Corollary}

\newtheorem{prop}[definition]{Proposition}
\theoremstyle{remark}
\newtheorem{remark}[definition]{Remark}
\newtheorem{example}[definition]{Example}
\newtheorem*{remark*}{Remark}
\usepackage{tikz,float}
\usepackage{tikz-cd}
\usetikzlibrary{patterns}
\usetikzlibrary{arrows.meta}
\usetikzlibrary{decorations.markings} 
\usetikzlibrary{shapes.geometric}
\usetikzlibrary{decorations.pathmorphing}

\tikzset{snake it/.style={decorate, decoration=snake}}

\usepackage[all]{xy}
\usepackage{mathdots}
\usepackage{comment}

\usepackage{xcolor}
\definecolor{pea}{RGB}{68, 168, 50}
\definecolor{plum}{RGB}{125, 50, 168}
\definecolor{Dgreen}{RGB}{2,100,64}

\newcommand{\Orb}{\mathcal{O}}
\newcommand{\canakci}{{\c{C}}anak{\c{c}}{\i} }
\newcommand{\calP}{\mathcal{P}}
\newcommand{\calG}{\mathcal{G}}
\newcommand{\gb}{\mathbf{g}}

\newcommand{\GGen}{\mathcal{G}^{\text{gen}}}
\newcommand{\Match}{\mathrm{Match}}
\newcommand{\cross}{\mathrm{cross}}

\newcommand{\close}{\circlearrowright\! }

\newcommand{\rank}{\mathcal{W}}
\newcommand{\rmm}{\mathcal{M}}
\newcommand{\drm}{\mathcal{M}}
\newcommand{\lloop}{\mathrm{loop}}
\newcommand{\rloop}{\mathrm{loop}_{t}}
\newcommand{\mup}{\mathrm{U}}
\newcommand{\tf}{T_\updownarrow}

\newcommand{\mdo}{\mathrm{D}}
\newcommand{\mupm}[1]{\begin{bmatrix}#1 &1\\0&1 \end{bmatrix}}
\newcommand{\mdom}[1]{\begin{bmatrix} 1+ #1 & -#1\\1&0 \end{bmatrix}}

\newcommand{\boldp}{\mathbf{p}}

\newdimen\R
\usepackage{biblatex}
\newcommand{\highlight}[1]{\colorbox{gray!15}{\small{#1}}}

\title{Cluster Expansions from Punctured Orbifolds}
\author{Esther Banaian, Wonwoo Kang,\\
Elizabeth Kelley, Ezgi Kantarcı Oğuz, Emine Yıldırım}
\date{}

\begin{document}
\maketitle

\abstract{
    We provide multiple combinatorial expansion formulas - in terms of snake graphs, labelled posets, matrices, and $T$-walks - for elements in generalized cluster algebras associated to arcs on punctured orbifolds and illustrate their equivalence. This work generalizes and unifies existing work on combinatorial expansion formulas from surfaces and unpunctured orbifolds.}
\tableofcontents

\section{Introduction}

Fomin and Zelevinsky introduced \emph{cluster algebras} in 2002 in their study of dual canonical bases~\cite{FZ2002}. 
Cluster algebras have subsequently attracted significant interest, in part due to their deep connections to many branches of mathematics, including total positivity, quiver representations, Teichm\"{u}ller theory, tropical geometry, Lie theory, and Poisson geometry. 

An important class of cluster algebras is that of \emph{surface-type}. This subtype was formalized by Fomin, Shapiro, and Thurston \cite{FST-I}, building on previous work by Fock and Goncharov \cite{fock2006moduli, FG09} and Gekhtman, Shapiro, and Vainshtein \cite{GSV03}. Cluster algebras of surface-type tend to be well-behaved; for example, most cluster algebras of finite mutation-type arise from surfaces \cite{felikson2012-b}. Through the machinery of quivers with potential, these cluster algebras also have interesting connections to representation theory~\cite{geiss2020generic,geiss2023bangle}.

In follow-up work to \cite{FST-I}, Fomin and Thurston showed that surface-type cluster algebras have a ``geometric realization'' in terms of the decorated Teichm{\"u}ller space of the associated surface \cite{fomin2018cluster}. Motivated by this, Chekhov and Shapiro later introduced \emph{generalized cluster algebras}, a special class of which are associated to decorated Teichm{\"u}ller spaces of orbifolds \cite{Chekhov-Shapiro}. In a generalized cluster algebra, the binomial exchange relation for cluster variables is replaced by an arbitrary polynomial relation. Note that these orbifold-type generalized cluster algebras are similar to, but distinct from, the ordinary cluster algebras from orbifolds studied in \cite{Ciliberti2,Ciliberti1,felikson2012-a}. Generalized cluster algebras share many properties with ordinary cluster algebras, including the Laurent phenomenon \cite{Chekhov-Shapiro} and  positivity~\cite{BLM25}. They also have an associated upper generalized cluster algebra, which enjoys similar properties to the ordinary case \cite{BCDX20}. These algebras have applications in representation theory and number theory \cite{BanaianSen,banaian2025snake, GSV18,GyodaMatsushita,labardini2019family}. 

In their initial paper, Fomin and Zelevinsky conjectured that all cluster variables can be written as Laurent polynomials with positive coefficients in an arbitrary cluster \cite{FZ2002}. This conjecture turned out to be difficult, inspiring many partial results for over a decade before being finally settled in the skew-symmetrizable case by Lee and Schiffler \cite{LeeSchiffler} and in general by Gross, Hacking, Keel, and Kontsevich \cite{GHKK18}. 

Prior to either of these works, Musiker, Schiffler, and Williams gave a combinatorial proof of the positivity of cluster algebras of surface type \cite{musiker2011positivity}, building on results by the first two authors for unpunctured surfaces \cite{musiker2010cluster}. Specifically, this combinatorial proof shows that every cluster variable can be expressed as a dimer partition function for a family of graphs called \emph{snake graphs}. Positivity follows as an immediate corollary. Giving an explicit formula for cluster variables is a fortuitous by-product of this proof; these formulas can be useful, for instance, in studying bases and structure coefficients \cite{musiker2013bases}.

Two of the authors of this current article extended the Musiker-Schiffler-Williams snake graph construction to unpunctured orbifolds in \cite{banaian2020snake} by introducing hexagonal tiles. The goal of this paper is to further generalize the setting to include generalized cluster algebras from punctured orbifolds. In the presence of punctures, modelling the dynamics of the associated generalized cluster algebra requires one to work with tagged arcs and tagged triangulations, as in \cite{FST-I, fomin2018cluster}. To construct graphs that produce expansion formulas for the generalized cluster variables associated with tagged arcs (following Wilson's \emph{loop graphs} \cite{wilson2020surface}), we must recast the hexagonal tile construction and introduce the notion of \emph{auxiliary tiles}. The overall idea is to construct a local lift of an arc on a triangulated orbifold, producing a surface with partial triangulation, and then completing the triangulation. The auxiliary tiles correspond to the added arcs. This construction is based on a choice of ``triangulation completion''; see Section \ref{sec:conclusion} for more details, including a discussion as to why the final output does not depend on this choice. Auxiliary tiles allow us to leverage the well-understood combinatorics of surface snake graphs.

\begin{theorem}\label{thm:Correctness}
Let $\mathcal{O}$ be an orbifold with a tagged triangulation $T$. Let $\gamma$ be an arc on $\mathcal{O}$ and let $x_\gamma^T$ be the associated cluster variable, expressed in terms of the cluster indexed by $T$. If $\gamma$ is doubly-notched, then suppose $\mathcal{O}$ is not a twice-punctured closed orbifold. Let $\mathcal{G}_{\gamma,T}$ be the loop with auxiliary tiles associated to $(\gamma,T)$. Then, we have \[
x_\gamma^T = \frac{1}{\mathrm{cross}(\gamma,T)} \sum_M x(M)y(M)
\]
where we sum over all good matchings of $\mathcal{G}_{\gamma,T}$.
\end{theorem}

We highlight that the proof of Theorem \ref{thm:Correctness} appears in Section \ref{sec:main}. 
A second goal of this project is to connect snake and loop graphs from orbifolds to other combinatorial gadgets used for cluster algebra computations, extending the work of two other authors \cite{ezgieminecluster2024} who did this for punctured surfaces. One such framework is through the lens of posets. Propp observed that the set of perfect matchings of a snake graph can be endowed with a partial order which turns out to be a distributive lattice \cite{Propp25}. By the Fundamental Theorem of Finite Distributive Lattices, this means the data of this lattice is packaged into its poset of join-irreducibles. These lattices were described explicitly for snake graphs by Musiker, Schiffler, and Williams \cite{musiker2013bases}, and for loop graphs by \cite{wilson2020surface}. In \cite{ezgieminecluster2024} and independently in \cite{HuangLattice,pilaud2023posets}, a method of computing a cluster variable from such a poset was provided. 

Translating this expansion formula into the language of posets has several advantages. The posets that appear in \cite{ezgieminecluster2024, pilaud2023posets} are either fence posets or result from adding one extra relation to a fence poset, meaning their structure is quite simple. Indeed, using the language of posets allowed three of the present authors to extend \canakci and Schiffler's snake graph calculus \cite{CS13,CS15} to tagged arcs \cite{banaian2024skein}. Posets also closely resemble diagrams associated with \emph{strings}, words that parametrize a set of indecomposable representations in a \emph{string algebra}. This perspective has proven useful, as in \cite{geiss2022schemes, geiss2023bangle}. Fence posets have appeared in~\cite{weng2023f} in the study of certain Donaldson--Thomas $F$-polynomials.

Fence posets also recently appeared in the construction of \emph{$q$-rationals} by Morier-Genoud and Ovsienko \cite{MGO20}. These $q$-analogues of rational numbers can be seen as coming from a rank-generating function for the lattice of order ideals of a fence poset. In the same work, Morier-Genoud and Ovsienko conjectured that this rank-generating function is always unimodal. There has been much work on the topic, showing the conjecture in special cases and also exploring interesting dynamics of fence posets; see, for example, \cite{AL25, elizalde2023partial,elizalde2023rowmotion, mcconville2025hyperbinary,mcconville2021rank}.  Morier-Genoud–Ovsienko's conjecture was shown in full generality by one of the present authors in joint work with Ravichandran \cite{ouguz2023rank}. Another author, in joint work with Lee and Lim, further established the unimodality property for a class of closely related posets known as \emph{loop fence posets}~\cite{KLL25}.  

The language of fence posets brings in the machinery of Kantarcı Oğuz's matrices introduced in~\cite{ouguz2025oriented}, which were later used in~\cite{ezgieminecluster2024} to easily compute elements of surface-type cluster algebras. Such matrices are called rank matrices. Matrix formulas are nothing new in the study of cluster algebras. Indeed, a critical step to the proof of several bases for surface-type cluster algebras in \cite{musiker2013bases} was the development of matrix formulas for both arcs on a surface and snake graphs in \cite{musiker2013}, building on work by Fock and Goncharov \cite{fock2006moduli, FG09, FG07}. Chekhov and Shapiro give a family of matrices in the spirit of Fock-Goncharov in their work on generalized cluster algebras \cite{Chekhov-Shapiro} that were further studied in the unpunctured orbifold case in \cite{banaian2020snake}. It would be interesting to precisely understand the relationship between those matrix formulas and the matrix formulas that appear in this work. One interesting feature of the matrices in this paper, first pointed out in \cite{ouguz2025oriented}, is their resemblance to Cohn matrices, used in the study of Markov numbers. This idea was further explored in \cite{banaian2025cluster,evans2025q}.

Along with posets and matrices, we also present a reformulation of our expansion formula in terms of \emph{$T$-walks}. Schiffler used it to prove positivity for type $A$ cluster algebras \cite{SchifflerTypeA}. The construction was then extended to all unpunctured surfaces in \cite{ST09} and to include coefficients, still on unpunctured surfaces, in \cite{S10}. Gunawan and Musiker further extended the definition for polygons with a single puncture, which provides a geometric model for type $D$ cluster algebras, in \cite{GM15}. Generalizations of $T$-paths have also appeared in other settings \cite{BCKZ22,ccanakcci2021friezes}.

In the unpunctured case, there is a bijection between $T$-paths associated to a pair of an arc and a triangulation and perfect matchings of the associated snake graph; this is, for example, explained in depth in Claussen's thesis \cite{claussen2020expansion}. In~\cite{ezgieminecluster2024}, two of the present authors generalized the concept of $T$-paths to \emph{$T$-walks} associated with generalized arcs on arbitrary surfaces, possibly with punctures, and showed how these relate to loop graphs. We take this one step further, now linking $T$-walks to orbifolds. While $T$-paths (specifically those in \cite{ccanakcci2021friezes}) were an inspiration in the previous work for unpunctured orbifolds (see \cite[Section 11.1]{banaian2020snake}), a precise construction was not provided previously even in this case.

Our second goal is summarized as follows: the notation will be described in detail in Section \ref{sec:Expansion}. In brief, for a generalized arc or a closed curve $\gamma$, we consider the following objects:
\begin{itemize}
    \item $\chi(\mathcal{G}_{\gamma,T})$ is a weighted sum of good matchings of the graph $\mathcal{G}_{\gamma,T}$;
    \item $x_{\min}^{\gamma}$ is a minimal term associated with $\gamma$;
    \item $\mathcal{W}(\calP_\gamma^T,w)$ is a weighted sum of order ideals of a poset $\calP_\gamma$ equipped with a specific weight function $w$;
    \item $\rmm_w(\calP_\gamma^T\!\searrow)$ is a rank matrix of a poset;
    \item $\mathfrak{W}(\gamma,T)$ is a weighted sum of $T$-walks.
\end{itemize}
Table \ref{table:whole} provides a comparative overview of all the combinatorial languages under consideration.

\begin{restatable}{theorem}{Expan}\label{thm:MainExpansion}
    Let $\mathcal{O}$ be an orbifold, possibly with punctures. Let $\gamma$ be a (possibly generalized) arc or closed curve on $\mathcal{O}$, and let $T$ be a triangulation of $\mathcal{O}$. Then $\mathcal{W}(\calP_\gamma^T,w)$ is equal to the top-left entry of the rank matrix $\rmm_w(\calP_\gamma^T\!\searrow)$ when $\gamma$ is an arc, and is equal to the trace of $\rmm_w(\calP_\gamma\!\searrow)$ when $\gamma$ is a closed curve. Moreover,
\[
 \chi(\mathcal{G}_{\gamma,T}) = x_{\min}^{\gamma} \mathcal{W}(\calP_\gamma^T, w) = \mathfrak{W}(\gamma,T).
\]
\end{restatable}

\begin{table}[H]
\begin{center}
\begin{tabular}{|c|c|c|}
\hline
\begin{tikzpicture}[scale=0.5, transform shape]

    \draw[thick] (0,0) circle (3cm);

    \node[below, scale=1.5] at (0,-3) {$v_1$}; 
    \node[above, scale=1.5] at (0,3) {$v_2$};  
    \node[left, scale=1.5] at (-3,0) {$v_3$};  

    \draw[thick] (-3,0) .. controls (3,4) and (2,-2.5) .. (0,-3);

    \draw[thick] (0,-3) .. controls (-2.5,2) and (2.5,2) .. (0,-3);

    \draw[in=-90,out=-20,looseness=1, orange, line width = 1.2pt] (-3,0) to (0.5,0);
    \draw[in=90,out=90,looseness=1.5, orange, line width = 1.2pt] (0.5,0) to (-0.5,0);
    \draw[in=-100,out=-90,looseness=1, orange, line width = 1.2pt,  ->] (-0.5,0) to (2,0);
    \draw[in=-40,out=80,looseness=1, orange, line width = 1.2pt] (2,0) to (0,3);

    \filldraw (0,-3) circle (3pt); 
    \filldraw (0,3) circle (3pt);  
    \filldraw (-3,0) circle (3pt); 

    \node[scale=2, thick] at (0,0) {$\times$};
    \node[scale=1.5] at (-1.4,1.2) {$a$};
    \node[scale=1.5] at (1,0) {$\rho$};
    \node[scale=1.5] at (-2.3,-2.6) {$b$};
    \node[scale=1.5] at (-2.3,2.6) {$c$};
    \node[scale=1.5] at (3.2,0) {$d$};

   \node[red, scale=1.5] at (1.3,1.4) {$\gamma$};

\end{tikzpicture}

&

\begin{tikzpicture}[scale = 1.25]
    \draw[gray,thick,dashed] (-0.7,1.7) to node[midway,below,xshift=-2pt,yshift=2pt]{$\rho$} (1,1);
    \draw[gray,thick,dashed] (0,2.4) to node[midway,below,xshift=-2pt,yshift=2pt]{$\rho$} (1.7,1.7);
    \draw[gray,thick,dashed] (1,2.4) to node[midway,below]{$a$} (2.4,2.4);
    \draw[thick] (-0.7,1.7) to node[midway,above,xshift=-20pt, yshift=-1pt]{$U_2\rho$} (1.7,1.7);
    \draw[ultra thick,white] (0,2.4) to (1,1);
    \draw[thick] (0,2.4) to node[midway,right, xshift=6pt,yshift=-10pt]{$U_0\rho$} (1,1);
    \draw[thick] (1,1) to node[midway,below,xshift=3pt]{$b$} (0,1);
    \draw[thick] (0,1) to node[midway,below,,xshift=-1pt]{$a$} (-0.7,1.7) to node[midway,above,xshift=-4pt]{$U_1\rho$} (0,2.4) to node[midway,above,yshift=-1pt]{$a$} (1,2.4) to node[midway,right,xshift=-2pt,yshift=2pt]{$b$}(1.7,1.7) to node[midway,below,xshift=7pt]{$U_1\rho$} (1,1);
    \draw[thick] (1.7,1.7) to node[midway,below,xshift=2pt]{$\rho$} (2.4,2.4) to node[midway,above,xshift=2pt]{$d$}(1.7,3.1) to node[midway,left,xshift=-2pt]{$c$}(1,2.4);
\end{tikzpicture}

&
\begin{tikzpicture}[scale = 1.25]
\draw[thick] (0,0) to node[below]{$b$} (1,0) to node[below]{$\rho$} (2,0) to node[below, scale = 0.7]{$U_1\rho$} (3,0) to node[right]{$a$} (3,1) to node[above]{$b$} (2,1) to node[above]{$\rho$} (1,1) to node[above]{$U_1\rho$} (0,1) to node[left]{$a$} (0,0);
\draw[thick] (1,0) to node[right, scale = 0.7, yshift = -15pt]{$U_{0}\rho$} (1,1);
\draw[thick] (2,0) to node[right, scale = 0.7,yshift = -15pt]{$U_{2}\rho$} (2,1);
\draw[gray,dashed] (0,1) to node[right]{$\rho$} (1,0);
\draw[gray,dashed] (1,1) to node[right,  scale = 0.6]{$U_1\rho$} (2,0);
\draw[gray,dashed] (2,1) to node[right]{$\rho$} (3,0);
\draw[gray,dashed] (2,2) to node[right]{$a$} (3,1);
\draw[thick] (2,1) to (3,1) to node[right]{$c$}  (3,2) to node[above]{$d$} (2,2) to node[left]{$\rho$} (2,1);
\end{tikzpicture}

\\
(I) & (II) & (III)\\
    \hline
\begin{tikzpicture}[scale=0.5, transform shape]
    \draw[thick] (5,0) to (7,2);
    \draw[thick] (7,2) to (9,0);
    \draw[thick] (9,0) to (11,-2);

    \filldraw[] (5,0) circle (3pt);
    \filldraw[] (7,2) circle (3pt);
    \filldraw[] (9,0) circle (3pt);
    \filldraw[] (11,-2) circle (3pt);

    \node[scale=2] at (5,-0.5) {$w_1$};
    \node[scale=2] at (7,2.5) {$w_2$};
    \node[scale=2] at (8.5,-0.5) {$w_3$};
    \node[scale=2] at (11,-2.5) {$w_4$};

\end{tikzpicture}

&

\begin{tikzpicture}[scale=0.5, transform shape]
    \draw[thick] (5,0) to (7,2);
    \draw[thick] (7,2) to (9,0);
    \draw[thick] (9,0) to (11,-2);
    \draw[->, blue, dashed,thick] (11,-2)--(12,-3);

    \filldraw[] (7,2) circle (3pt);
    \filldraw[] (9,0) circle (3pt);
    \filldraw[blue] (11,-2) circle (4pt);

    \fill[white] (5,0) circle(.2) ;
    \fill[red] (5,0) circle(.1) ;
    \draw[red] (5,0)circle(.2);

    \node[scale=2] at (5,-0.8) {$U(w_1)$};
    \node[scale=2] at (7,2.5) {$D(w_2)$};
    \node[scale=2] at (8.3,-0.7) {$D(w_3)$};
    \node[scale=2] at (11.5,-1.2) {$D(w_4)$};

\end{tikzpicture}

&
\begin{tikzpicture}[scale=0.5, transform shape]

    \draw[thick] (0,0) circle (3cm);

    \node[below, scale=2] at (0,-3.2) {$v_1$}; 
    \node[above, scale=2] at (0,3.2) {$v_2$};  
    \node[left, scale=2] at (-3.2,0) {$v_3$};  


    \draw[very thick, postaction={decorate, decoration={markings, mark=at position 0.15 with {\arrow[scale=1.5, black!30!green]{<}}}}] (-3,0) .. controls (3,4) and (2,-2.5) .. (0,-3);

    \draw[very thick, postaction={decorate, decoration={markings, mark=at position 0.15 with {\arrow[scale=1.5, black!30!green]{>}}, mark=at position 0.88 with {\arrow[scale=1.5, black!30!green]{<}}}}] (0,-3) .. controls (-2.5,2) and (2.5,2) .. (0,-3);

    \draw[in=-90,out=-20,looseness=1, orange, line width = 1.2pt] (-3,0) to (0.5,0);
    \draw[in=90,out=90,looseness=1.5, orange, line width = 1.2pt] (0.5,0) to (-0.5,0);
    \draw[in=-100,out=-90,looseness=1, orange, line width = 1.2pt,  ->] (-0.5,0) to (2,0);
    \draw[in=-40,out=80,looseness=1, orange, line width = 1.2pt] (2,0) to (0,3);

    \filldraw (0,-3) circle (3pt); 
    \filldraw (0,3) circle (3pt);  
    \filldraw (-3,0) circle (3pt); 

    \node[scale=2, thick] at (0,0) {$\times$};
    \node[scale=1.5] at (-1.4,1.2) {$a$};
    \node[scale=1.5] at (1,0) {$\rho$};
    \node[scale=1.5] at (-2.3,-2.6) {$b$};
    \node[scale=1.5] at (-2.3,2.6) {$c$};
    \node[scale=1.5] at (3.2,0) {$d$};

   \node[red, scale=1.5] at (1.3,1.4) {$\gamma$};

    \node[scale=1.5] at (0,-4.5) {$(c,b,U_1b,b,b,a,d)$};
\end{tikzpicture}
\\
(IV) & (V) & (VI)\\
\hline

\end{tabular}
\end{center}
\caption{An example of each of the different combinatorial gadgets for representing a generalized cluster algebra element: (I) surface triangulation, (II) snake graph of $\gamma$ with a hexagonal tile, (III) snake graph of $\gamma$ with an auxiliary tile, (IV) poset of $\gamma$, (V) poset of $\gamma$ for the matrix computation, (VI) one possible $T$-walk for $\gamma$.}
\label{table:whole}
\end{table}

We emphasize that Theorem \ref{thm:MainExpansion} is more than a combinatorial exercise; each formulation has a benefit. The snake graph formula is our most prominent link to other works; crucially, our proof of Theorem \ref{thm:Correctness} relies on first showing how, in the unpunctured case, hexagonal tiles can be replaced with sequences of square tiles, including auxiliary tiles (and hence, these results will be proven in the opposite order as they appeared in the introduction). Connecting to posets allows us to use the poset skein relations given in \cite{banaian2024skein,Tsironis25}. Meanwhile, the matrix formulas are computationally the easiest and also most readily given to Sage~\cite{sage} or other computer programs.

The connection to $T$-walks is perhaps the most interesting. Here, we do not rely on auxiliary tiles. This in particular shows that the choices present in our other constructions do not affect the output (see Proposition \ref{prop:Invariance}).

We prove Theorem \ref{thm:MainExpansion} by exhibiting a series of combinatorial maps in Section \ref{sec:Equivalence}, as are summarized below.

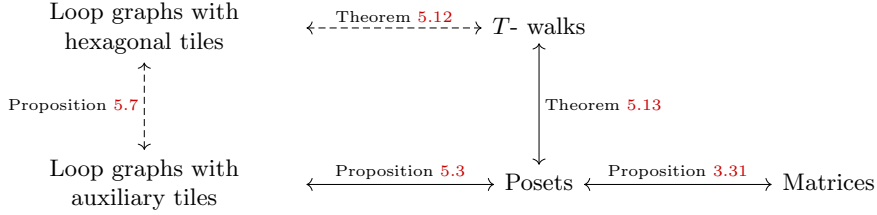
\begin{center}
\begin{tikzcd}
  \parbox{40mm}{\centering Loop graphs with \\hexagonal tiles} 
  \arrow[leftrightarrow, swap, dashed]{dd}{\text{\Cref{prop:CanChooseAuxiliaryTiles}}}
  \arrow[leftrightarrow, dashed]{rr}{\text{\Cref{thm:SnakeTWalks}}}   &\hspace{0.5cm}&T\text{- walks}\arrow[leftrightarrow]{dd}{\text{\Cref{prop:TWalkPosetNotchedCorner}}}&& \\
  \\
  \parbox{40mm}{\centering Loop graphs with \\auxiliary tiles}\arrow[leftrightarrow]{rr}[pos=0.5]{\text{\Cref{prop:AuxAndPoset}}}  & &\text{Posets}\arrow[leftrightarrow]{rr}[pos=0.5]{\text{\Cref{prop:FormulasAgree}}} &\hspace{0.5cm}& \text{Matrices}
\end{tikzcd}
\end{center}

The structure of the paper is as follows. \Cref{sec:background} provides a concise overview of cluster algebras from surfaces and generalized cluster algebras from orbifolds. In this work, we are combining the tagged arc construction from \cite{FST-I, fomin2018cluster} with the construction of generalized cluster algebras from orbifolds from \cite{Chekhov-Shapiro}. We also collect a few of these structural results.
In~\Cref{sec:Expansion} and~\Cref{sec:Twalk}, we establish our four expansion formulas, with many examples throughout. Then in \Cref{sec:Equivalence} we prove the equivalences among these approaches. In~\Cref{sec:main}, we show that these expansion formulas yield the correct Laurent expressions for cluster variables on punctured orbifolds. Finally, \Cref{sec:conclusion} discusses our insights regarding the projection to the orbifold $\mathcal{O}$, demonstrating that it is independent of the choice of lift. We also outline future directions, specifically on the \emph{skein relations} in this generalized setting.

\section*{Acknowledgements}
W.K and E.Y. were supported by the Bulgarian Science Fund with grant no KP-06- N92/5, the Ministry of Education, and Science of the Republic of Bulgaria, grant DO1-239/10.12.2024 and the Simons Foundation, grant SFI-MPS-T-Institutes-00007697. E.K.O. was supported by TÜBİTAK 1001 Grant 123F121. E.B. was supported by the German Research Foundation SFB-TRR 358/1 2023 – 491392403. The authors also thank Emily Gunawan for asking inspiring questions and Jon Wilson for providing valuable feedback on an earlier draft. 

\section{Background}\label{sec:background}

\subsection{Generalized Cluster Algebras}

Many of the basic definitions for generalized cluster algebras follow the same structure as the corresponding definitions for ordinary cluster algebras, with the differences arising from the fact that the exchange relations are now allowed to be polynomial expressions rather than strictly required to be binomial expressions. We will briefly record the most useful notation here.

In the following, we let $\mathbb{P} = \mathrm{Trop}(u_1,\ldots,u_m)$ denote a \emph{tropical semifield}, i.e., a multiplicative abelian group freely generated by the $u_i$ and equipped with the tropical addition $\oplus$ defined by $\left(\prod_{j}u_j^{a_j}\right) \oplus \left(\prod_{j}u_j^{b_j}\right) := \prod_j u_j^{\min(a_j,b_j)}.$  Let $\mathbb{Z}\mathbb{P}$ be the group ring of $\mathbb{P}$, and let $\mathbb{Q}\mathbb{P}$ denote its field of fractions. We define the ambient field $\mathcal{F}$ as the field of rational functions in $n$ algebraically independent variables $\{x_1, \dots, x_n\}$ with coefficients in $\mathbb{Q}\mathbb{P}$, i.e., $\mathcal{F} = \mathbb{Q}\mathbb{P}(x_1, \dots, x_n)$.  By defining $\mathbb{P}$ to be a tropical semifield instead of a general semifield, we are restricting to the common setting of a \emph{geometric generalized cluster algebra}.

\begin{definition}\label{def:gca}
    A \emph{seed} for a generalized cluster algebra is a quadruple $( B,\mathbf{x}, \mathbf{y}, \mathbf{Z})$ where 
    \begin{itemize}
     \item $B=(b_{ij})_{1\leq i \leq n+m,1\leq j\leq n}$ is an $n \times n$ skew-symmetrizable exchange matrix with integer entries;
        \item $\mathbf{x} = \{ x_1, \dots, x_n \}$ is  free generating set of $\mathcal{F}$;
        \item $\mathbf{y} = \{ y_1, \dots, y_n \}$ is an $n$-tuple of elements from $\mathbb{P}$ satisfying $y_k = \prod_{i=n+1}^m u_i^{b_{ik}}$; 
         and
        \item $\mathbf{Z} = \{ z_1, \dots, z_n \}$ is a collection of monic, palindromic  polynomials.
    \end{itemize}
    The set $\mathbf{x}$ is called a cluster, and any $x_i$ appearing in a cluster is called a generalized cluster variable.
\end{definition}

If for all $i$, $z_i = 1 + u$, then the above recovers the definition of a seed for an ordinary cluster algebra, and similarly for the following definitions.

Let $[a]_{+} := \textrm{max}(a,0)$. Given a seed $(B,\mathbf{x}, \mathbf{y},  \mathbf{Z})$, let $\hat{y}_i = y_i\prod_{j=1}^{n} x_j^{b_{ji}}$ and let $d_i$ denote the degree of the $i$-th exchange polynomial $z_i$. In particular, we will write $z_i(u) = \sum_{s=0}^{d_i} z_{i,s}u^s$ where $z_{i,0} = z_{i,d_i}=1$ and $z_{i,s} = z_{i,d_i-s}$. With this helpful notation, we next define the notion of \emph{mutation}, which produces a new seed from an old one.

\begin{definition}
   Let $\Sigma = ( B, \mathbf{x}, \mathbf{y},\mathbf{Z})$ be a seed in a generalized cluster algebra. The \emph{generalized seed mutation} $\mu_k$ in direction $k \in \{1, \dots, n\}$ transforms the seed into $\Sigma' = (B',\mathbf{x}', \mathbf{y}', \mathbf{Z})$, defined by:

    \begin{itemize}
        \item The exchange matrix entries $b_{ij}'$ are given by
        \[ b_{ij}' = \begin{cases}  
            b_{ij} + d_k \left( [-b_{ik}]_+ b_{kj} + b_{ik} [b_{kj}]_+ \right) & \text{if } i,j \neq k
            \\ 
            -b_{ij} & \text{if } i=k \text{ or } j=k . 
        \end{cases} \]
        \item The cluster variables $x_i'$ are given by
        \[ x_i' = \begin{cases} 
            x_i & \text{if } i \neq k \\ 
            \displaystyle x_k^{-1} \left( \prod_{j=1}^n x_j^{[-b_{jk}]_+} \right)^{d_k} \frac{\sum_{s=0}^{d_k} z_{k,s} \hat{y}_k^s}{\bigoplus_{s=0}^{d_k} z_{k,s} y_k^s} & \text{if } i = k .
        \end{cases} \]
        
        \item The coefficient variables $y_i'$ are given by
        \[ y_i' = \begin{cases} 
            y_i \left( y_k^{[b_{ki}]_+} \right)^{d_k} \left( \bigoplus_{s=0}^{d_k} z_{k,s} y_k^s \right)^{-b_{ki}} & \text{if } i \neq k \\ 
            y_k^{-1} & \text{if } i = k .
        \end{cases} \]
    \end{itemize}
\end{definition}

Now, we define the generalized cluster algebra associated to a given seed $\Sigma$. This is given by considering all seeds which are related to $\Sigma$ via mutation.

\begin{definition}
\label{def:generalized_cluster_algebra}
Let $\mathbb{T}_n$ be an $n$-regular tree whose edges are labelled by the integers \(1, \ldots, n\), such that each edge incident to a given vertex has a distinct label. A \emph{generalized cluster pattern} is an assignment of labelled seeds \(\Sigma_t = (B_t,\mathbf{x}_t, \mathbf{y}_t, \mathbf{Z})\) to each vertex \(t \in \mathbb{T}_n\) such that the seeds assigned to adjacent endpoints $t \stackrel{k}{\text{---}} t'$ are related by mutation in direction $k$, i.e., such that $\mu_k(\Sigma_t) = \Sigma_{t'}$.

For a given generalized cluster pattern, let  
\[
X = \bigcup_{t \in \mathbb{T}_n} \mathbf{x}_t
\]  
be the union of all clusters across the seeds in the pattern. The \emph{generalized cluster algebra} \(\mathcal{A}\) associated with this generalized cluster pattern is the \(\mathbb{ZP}\)-subalgebra of the ambient field \(\mathcal{F}\) generated by all cluster variables, \(\mathcal{A} = \mathbb{ZP}[X]\).

\end{definition}

We say that a generalized cluster algebra \emph{has principal coefficients} if $B$ is a $2n \times n$ matrix where the bottom $n \times n$ portion is the identity matrix. In particular, this implies $\mathbb{P} = \mathrm{Trop}(y_1,\ldots,y_n)$.

In their paper introducing these algebras, Chekhov and Shapiro showed that generalized cluster algebras retain the Laurent Phenomenon.

\begin{theorem}[\cite{Chekhov-Shapiro}]
    Let $\mathcal{A}$ be a generalized cluster algebra. If $\mathbf{x}$ is a cluster in $\mathcal{A}$ and let $x'$ is a generalized cluster variable, then $x'$ can be expressed as a Laurent polynomial in the cluster $\mathbf{x}$ with coefficients in $\mathbb{Z}\mathbb{P}[z_{i,s}]$
\end{theorem}

For more detailed exposition about the structure of generalized cluster algebras, the interested reader may consult \cite{Chekhov-Shapiro,MR4563311, Nakanishi15}.

\subsection{Triangulated Orbifolds}

In this section, we introduce \emph{triangulated orbifolds} and present some definitions and features that will be important in our work. An orbifold is a generalization of a manifold where the local structure is given by quotients of open subsets of $\mathbb{R}^n$ under finite group actions. For the purposes of this paper, we will think about orbifolds in the following very concrete way.

\begin{definition}
    An \emph{orbifold} $\mathcal{O}$ is a triple $(S,M,Q)$ where $S$ is a bordered surface, $M$ is a finite set of \emph{marked points}, and $Q$ is a finite set of \emph{orbifold points}, such that $M \cap Q = \emptyset$ (i.e., a point cannot be both a marked point and an orbifold point); orbifold points cannot appear on the boundary of $S$; and each boundary of $S$ contains at least one marked point.
\end{definition}

Note that when $Q = \emptyset$, i.e., when there are no orbifold points, this definition reduces to the definition of a marked surface, $\mathcal{S} = (S,M)$. Marked points that appear in the interior of $S$ are referred to as \emph{punctures}. If $M$ does not contain any punctures, then we refer to $\mathcal{O}$ (or $\mathcal{S}$) as \emph{unpunctured}.

As is standard~\cite{FST-I,fomin2018cluster}, we will omit a few small cases of orbifolds. Let $p$ be the number of punctures and $o = \vert Q \vert$. Throughout the article, we will assume $\mathcal{O}$ is not a sphere with $p + o < 4$, a monogon with $p+o < 2$, or a bigon or triangle with $p+o = 0$.

Each orbifold point has an associated \emph{order}, $\mathbf{p} \in \mathbb{Z}_{\geq 2}$, and constant $\lambda_{\mathbf{p}} = 2\cos(\pi/\mathbf{p})$. For clarity of exposition, we will require in this paper that all orbifold points have order $\geq 3$ as orbifold points of order 2 require special treatment. However, one could combine the constructions here with the snake graphs from \cite{ccanakcci2019bases} to extend our results to orbifold points with at least one orbifold point of order 2.

\begin{remark}
    Our definition of the order of an orbifold point differs from the one used by Felikson, Shapiro, and Tumarkin in their work on ordinary cluster algebras from orbifolds \cite{felikson2012-a, felikson2012-b, FT17}. In their language, orbifold points have weight either $2$ or $\frac{1}{2}$. Our orbifold points of order 2 coincide with their weight $\frac{1}{2}$ points, and we do not work with orbifold points analogous to those with weight 2.
\end{remark}

\begin{definition}
    An \emph{arc} on an orbifold $\mathcal{O} = (S, M, Q)$ is a curve in $S$ with endpoints in $M$ with no self-crossing. The interior of the arc must be disjoint from $M$, the set of orbifold points $Q$, and the boundary $\partial S$. Furthermore, the arc must not be isotopic (relative to its endpoints) to a segment of $\partial S$. Arcs are considered up to isotopy relative to $M$.
\end{definition}

We refer to arcs that cut out an unpunctured monogon containing an orbifold point as \emph{pending} arcs and to all other arcs as \emph{standard} arcs. We will often use the variable $\rho$ to represent pending arcs and $\tau$ to represent standard arcs. Both pending and standard arcs are referred to as \emph{ordinary} arcs. A curve that contains a self-intersection while satisfying the other conditions of this definition is called a \emph{generalized} arc. We will see later that an important subtype of generalized arcs is those that wind non-trivially around an orbifold point before their first intersection with an arc in the triangulation (or, by reverse orientation, after all such intersections). We refer to such arcs as \emph{corner arcs} and an example appears in the last row of \Cref{table:PuzzlePiecesWinding}.

Orbifold points are the singular points of the local finite group actions that define the orbifold. This means there are many equivalent ways to represent an arc $\gamma$ that winds around an orbifold point. For the following definition, at each pending arc $\rho$, we draw a dashed line from the unique marked point incident to $\rho$ to the orbifold point it encloses, denoted $\overline{\rho}$. Any generalized arc or closed curve $\gamma$ which intersects $\rho$ nontrivially will also intersect $\overline{\rho}$.  If the interior of $\gamma$ intersects $\overline{\rho}$ consecutively $\kappa$ times , keeping the orbifold point to the left (right), we say $\gamma$ has ``signed intersection number'' $\kappa$ ($-\kappa$). Moreover, if $\gamma$ shares an endpoint with $\overline{\rho}$ and the orbifold point lies to the left (right) of the direction of travel, then we add (subtract) $\frac12$ to the intersection number. The definition of an orbifold point of order $\mathbf{p}$ implies that an arc with signed intersection number  $\kappa$  is isotopic to one with signed intersection number $\mathbf{p}-\kappa$. Having an intersection number $\kappa$ satisfying $\vert \kappa \vert < 1$ should be interpreted as being homotopic to an arc which does not intersect $\overline{\rho}$.

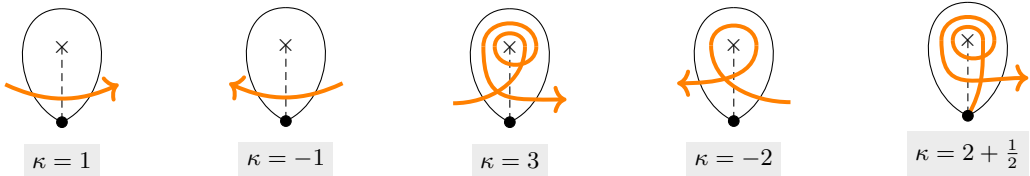
\begin{center}
    \begin{tikzpicture}
        \draw[out=25,in=0,looseness=1.25] (0,0) to (0,1.5);
        \draw[out=155,in=180,looseness=1.25] (0,0) to (0,1.5);
        \filldraw (0,0) circle[radius=2pt];
        \draw[densely dashed] (0,1) to (0,0);
        \node at (0,1) {$\times$};
        \draw[line width = 1.5pt, orange,->,out=-25,in=205] (-0.75,0.5) to (0.75,0.5);

        \node at (0,-0.5) {\highlight{$\kappa=1$}};
    \end{tikzpicture}
    \qquad\qquad
    \begin{tikzpicture}
        \draw[out=25,in=0,looseness=1.25] (0,0) to (0,1.5);
        \draw[out=155,in=180,looseness=1.25] (0,0) to (0,1.5);
        \filldraw (0,0) circle[radius=2pt];
        \draw[densely dashed] (0,1) to (0,0);
        \node at (0,1) {$\times$};
        \draw[line width = 1.5pt, orange,->,out=205,in=-25] (0.75,0.5) to (-0.75,0.5);

        \node at (0,-0.5) {\highlight{$\kappa=-1$}};
    \end{tikzpicture}
    \qquad\qquad
    \begin{tikzpicture}
        \draw[out=25,in=0,looseness=1.25] (0,0) to (0,1.5);
        \draw[out=155,in=180,looseness=1.25] (0,0) to (0,1.5);
        \filldraw (0,0) circle[radius=2pt];
        \draw[densely dashed] (0,1) to (0,0);
        \node at (0,1) {$\times$};
        \draw[line width = 1.5pt, orange,out=0,in=-90] (-0.75,0.25) to (0.2,1);
        \draw[line width = 1.5pt, orange, out = 90, in = 90,looseness=1.5] (0.2,1) to (-0.2,1);
        \draw[line width = 1.5pt, orange, out = -90, in = -90,looseness=1.5] (-0.2,1) to (0.35,1);
        \draw[line width = 1.5pt, orange, out = 90, in = 90, looseness=1.5] (0.35,1) to (-0.35,1);
        \draw[line width = 1.5pt, orange, out = -90, in = 180,->,looseness=1.5] (-0.35,1) to (0.75,0.3);

        \node at (0,-0.5) {\highlight{$\kappa = 3$}};
    \end{tikzpicture}
    \qquad\qquad
    \begin{tikzpicture}
        \draw[out=25,in=0,looseness=1.25] (0,0) to (0,1.5);
        \draw[out=155,in=180,looseness=1.25] (0,0) to (0,1.5);
        \filldraw (0,0) circle[radius=2pt];
        \draw[densely dashed] (0,1) to (0,0);
        \node at (0,1) {$\times$};
        \draw[line width = 1.5pt, orange, out = 180, in = -90] (0.75,0.25) to (-0.3,1);
        \draw[line width = 1.5pt, orange, out = 90, in = 90, looseness = 1.5] (-0.3,1) to (0.3,1);
        \draw[line width = 1.5pt, orange, out = -90, in = 0,->] (0.3,1) to (-0.75,0.5);

        \node at (0,-0.5) {\highlight{$\kappa = -2$}};
    \end{tikzpicture}
    \qquad\qquad
        \begin{tikzpicture}
        \draw[out=25,in=0,looseness=1.25] (0,0) to (0,1.5);
        \draw[out=155,in=180,looseness=1.25] (0,0) to (0,1.5);
        \draw[densely dashed] (0,1) to (0,0);
        \node at (0,1) {$\times$};
        \draw[line width = 1.5pt, orange,out=60,in=-90] (0,0) to (0.2,1);
        \draw[line width = 1.5pt, orange, out = 90, in = 90,looseness=1.5] (0.2,1) to (-0.2,1);
        \draw[line width = 1.5pt, orange, out = -90, in = -90,looseness=1.5] (-0.2,1) to (0.35,1);
        \draw[line width = 1.5pt, orange, out = 90, in = 90, looseness=1.5] (0.35,1) to (-0.35,1);
        \draw[line width = 1.5pt, orange, out = -90, in = 180,->,looseness=1.5] (-0.35,1) to (0.8,0.5);
        \node at (0,-0.5) {\highlight{$\kappa = 2+\frac12$}};
        \filldraw (0,0) circle[radius=2pt];
    \end{tikzpicture}
\end{center}

Our convention will be to choose a representative of an arc which has a signed intersection number strictly between 0 and $\mathbf{p}-1$. This, in particular, implies that arcs will always wind around orbifold points in a counterclockwise direction. If the signed intersection number $\kappa$ of such an arc is an integer, we will say its \emph{winding number} is $k:= \kappa - 1$. If $\kappa$ is fractional (i.e., if $\gamma$ is a corner arc), we set $k := \kappa - \frac12$.

This modular winding behaviour makes the following definition useful when we associate generalized cluster algebra elements to such arcs.  Let $U_k(x)$ denote the $k$-th normalized Chebyshev polynomial of the second kind, defined by the recurrence:
\[
U_{-1}(x) = 0, \quad U_0(x) = 1, \quad U_k(x) = xU_{k-1}(x) - U_{k-2}(x) \quad \text{for } k \geq 1.
\]

With our indexing, $\lambda_{\mathbf{p}}$ is a root of $U_{\mathbf{p}-1}(x)$.
Throughout this work, we often use $U_k$ as a shorthand for $U_k(\lambda_{\mathbf{p}})$, when $\mathbf{p}$, the order of the corresponding orbifold point, is understood. This will primarily occur in figures. The normalized Chebyshev polynomials satisfy the following Casini-type identity, which will be essential in subsequent sections.

\begin{lemma}\label{lem:ChebyshevLogConcave}
For all $i \geq 0$,  $U_{i-1}(x)U_{i+1}(x) + 1 = U_i^2(x)$.
\end{lemma}

The connection between these polynomials and our combinatorial construction arises from the geometry of a particularly useful covering space, which we refer to as the $\mathbf{p}$-fold covering space. To construct this cover, we choose an orbifold point of order $\mathbf{p}$, take an equilateral $\mathbf{p}$-gon, and attach to each edge a copy of the triangulated orbifold where the pending arc around the chosen orbifold point has been identified with the edge of the $\mathbf{p}$-gon. Because our constructions only require us to consider the local geometry, it suffices to consider the $\mathbf{p}$-fold covering space for a single orbifold point at a time.

The following diagram contains an example of the 6-fold covering space for a triangulated orbifold containing a single orbifold point.
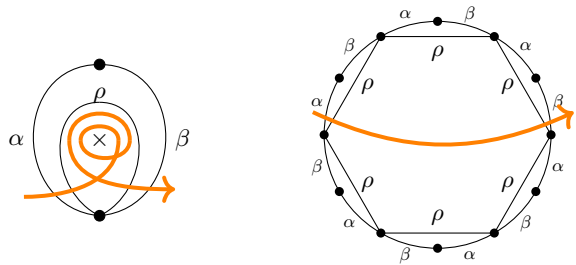
\begin{center}
    \begin{tikzpicture}
        \draw[out=25,in=0,looseness=1.25] (0,0) to (0,1.5);
        \draw[out=155,in=180,looseness=1.25] (0,0) to (0,1.5);
        \draw[out=175,in=180,looseness=1.5] (0,0) to (0,2);
        \draw[out=5,in=0,looseness=1.5] (0,0) to (0,2);
        \filldraw (0,0) circle[radius=2pt];
        \node at (0,1) {$\times$};
        \filldraw (0,2) circle[radius=2pt];
        \draw[line width = 1.5pt,orange,out=0,in=-90] (-1,0.25) to (0.25,1);
        \draw[line width = 1.5pt, orange, out=90,in=90,looseness=1.25] (0.25,1) to (-0.25,1);
        \draw[line width = 1.5pt, orange, out=-90,in=-90,looseness=1.25] (-0.25,1) to (0.4,1);
        \draw[line width = 1.5pt, orange, out=90,in=90,looseness=1.5] (0.4,1) to (-0.4,1);
        \draw[line width = 1.5pt, orange,out=-90,in=180,->] (-0.4,1) to (1,0.35);
        \node at (-1.1,1) {$\alpha$};
        \node at (1.1,1) {$\beta$};
        \node at (0,1.6) {$\rho$};
        \node[white] at (0,-0.5) {3};
    \end{tikzpicture}
    \qquad\qquad
    \begin{tikzpicture}[scale=1.5]
        \def\r{1}

        \draw (0,0) circle (\r);

        \foreach \i in {0,...,5} {
            \coordinate (V\i) at ({60*\i}:\r);
        }

        \draw (V0) -- node[midway,left]{$\rho$} (V1) -- node[midway,below]{$\rho$} (V2) -- node[midway,right]{$\rho$} (V3) -- node[midway,right]{$\rho$} (V4) -- node[midway,above]{$\rho$} (V5) -- node[midway,left]{$\rho$} (V0);

        \foreach \i in {0,...,5} {
            \fill (V\i) circle (0.04);
        }

        \foreach \i in {0,...,5} {
            \coordinate (M\i) at ({60*\i + 30}:\r);
            \fill (M\i) circle (0.04);
        }

        \foreach \i in {0,...,5} {
            \node[scale=0.75] at ({60*\i + 15}:1.1) {$\beta$};
        }

        \foreach \i in {0,...,5} {
            \node[scale=0.75] at ({60*\i + 45}:1.1) {$\alpha$};
        }

        \draw[line width = 1.5pt, orange,out=-25,in=-155,looseness=1,->] (-1.1,0.2) to (1.2,0.2);

    \end{tikzpicture}
\end{center}

To obtain a genuine triangulation of this $\mathbf{p}$-fold cover, we must add edges within the central $\mathbf{p}$-gon. We refer to any diagonal that separates $k$ vertices from the remaining vertices as a \emph{$k$-diagonal}. Under this convention, an edge of the polygon is a $0$-diagonal, and a $k$-diagonal in a $\mathbf{p}$-gon is equivalent to a $(\mathbf{p}-k-2)$-diagonal. If we regard the edge labels as lengths, then the geometry of regular polygons provides a canonical way to assign labels to these additional edges.

\begin{lemma}\label{lem:ChebyshevFact}
The length of a $k$-diagonal in a regular $\mathbf{p}$-gon with side length $s$ is given by $U_k(\lambda_{\mathbf{p}})s$.
\end{lemma}

This relationship between diagonal lengths and Chebyshev polynomials allows us to assign weights to the variables in the expansion of generalized arcs on orbifolds.

Next, we turn to a discussion of triangulations, which will be our geometric model of a cluster.

\begin{definition}
    For any two arcs $\gamma$ and $\gamma'$ on $S$, their \emph{crossing number} $e(\gamma,\gamma')$ is defined as the minimal number of intersections between representatives $\alpha$ and $\alpha'$ of the isotopy classes of $\gamma$ and $\gamma'$, where $\alpha$ and $\alpha'$ are allowed to range over the entire isotopy class. If there exist representatives $\alpha$ and $\alpha'$ with no intersections, so $e(\gamma,\gamma') = 0$, we say that $\gamma$ and $\gamma'$ are \emph{compatible}.
\end{definition}

\begin{definition}
    An \emph{ideal triangulation} is a maximal collection of pairwise compatible arcs on $S$, together with the boundary arcs.
\end{definition}

As shown below, there are three types of triangles that may appear in an ideal triangulation of an orbifold: those with three distinct sides, those that are cut out by a single pending arc, and those with only two distinct sides, which are referred to as \emph{self-folded} triangles. Self-folded triangles always consist of an arc $\ell$ that encircles a single puncture and a radius $r$ from the basepoint of $\ell$ to the encircled puncture.

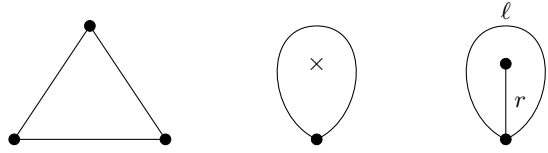
\begin{center}
    \begin{tikzpicture}
        \draw (0,0) to (2,0) to (1,1.5) to (0,0);
        \filldraw (0,0) circle[radius=2pt];
        \filldraw (2,0) circle[radius=2pt];
        \filldraw (1,1.5) circle[radius=2pt];

        \draw[out=25,in=0,looseness=1.25] (4,0) to (4,1.5);
        \draw[out=155,in=180,looseness=1.25] (4,0) to (4,1.5);
        \filldraw (4,0) circle[radius=2pt];
        \node at (4,1) {$\times$};

        \draw[out=25,in=0,looseness=1.25] (6.5,0) to (6.5,1.5);
        \draw[out=155,in=180,looseness=1.25] (6.5,0) to (6.5,1.5);
        \draw (6.5,0) to node[right]{$r$} (6.5,1);
        \node at (6.5,1.7) {$\ell$};
        \filldraw (6.5,0) circle[radius=2pt];
        \filldraw (6.5,1) circle[radius=2pt];
    \end{tikzpicture}
\end{center}

\begin{remark}
    Although the second type of face - i.e., the one cut out by a single pending arc - does not actually have three sides, it is useful to still refer to this as a ``triangle''. Indeed, since all of our orbifold points have order at least three, pending arcs will lift in a cover to the boundary of a regular polygon. With this intuition in mind, we will henceforth refer to this type of monogon as a \emph{pending triangle}.
\end{remark}

Any ideal triangulation of an orbifold can be formed by gluing together copies of these three types of triangles, where we forbid gluing a self-folded triangle on its folded edge (``$r$'' above). Ideal triangulations are related via sequences of \emph{flips}. Each flip removes a single arc $\gamma$ from a triangulation $T$ and replaces $\gamma$ with the unique distinct arc $\gamma'$ such that $(T \cup \{ \gamma' \})\backslash \{ \gamma \}$ forms a new ideal triangulation of the surface. Flips correspond to mutation relations in the associated generalized cluster algebra. One example of an ideal triangulation and the flip of a pending arc is shown in \Cref{fig:Orbi}. One can readily see that the flip of a standard arc is always a standard arc and likewise that a flip of a pending arc is always a pending arc.

\begin{figure}[H]
    \centering
    \begin{tikzpicture}[transform shape, scale=2]
\tikzset{->-/.style={decoration={
  markings,
  mark=at position #1 with {\arrow{>}}},postaction={decorate}}}


\draw[out=30,in=-30,looseness=1.5] (0,0.3) to (0,1.5);
\draw[out=150,in=-150,looseness=1.5] (0,0.3) to (0,1.5);
\draw[out=-10,in=0,looseness=1.8] (0,1.5) to (0,0);
\draw[out=-170,in=180,looseness=1.8] (0,1.5) to (0,0);

\draw[blue,in=0,out=45,looseness =1] (0,0.3) to (0,1.2);
\draw[blue,in=180,out=135,looseness=1] (0,0.3) to (0,1.2);

\draw[fill=black] (0,0.3) circle [radius=1pt];
\draw[fill=black] (0,1.5) circle [radius=1pt];
\draw[thick] (0,1) node {$\mathbf{\times}$};

\node[scale=0.5] at (-0.55,1) {$\alpha$};
\node[scale=0.5] at (0.55,1) {$\beta$};
\node[scale=0.5, blue] at (-0.17,1.2) {$\rho$};
\node[scale=0.5] at (0,0.1) {$\epsilon$};

    \end{tikzpicture}
    \qquad
    \begin{tikzpicture}[transform shape, scale=2]
\tikzset{->-/.style={decoration={
  markings,
  mark=at position #1 with {\arrow{>}}},postaction={decorate}}}


\draw[out=30,in=-30,looseness=1.5] (0,0.3) to (0,1.5);
\draw[out=150,in=-150,looseness=1.5] (0,0.3) to (0,1.5);
\draw[out=-10,in=0,looseness=1.8] (0,1.5) to (0,0);
\draw[out=-170,in=180,looseness=1.8] (0,1.5) to (0,0);

\draw[red,in=0,out=-45,looseness =1] (0,1.5) to (0,0.8);
\draw[red,in=-180,out=-135,looseness=1] (0,1.5) to (0,0.8);

\draw[fill=black] (0,0.3) circle [radius=1pt];
\draw[fill=black] (0,1.5) circle [radius=1pt];
\draw[thick] (0,1) node {$\mathbf{\times}$};

\node[scale=0.5] at (-0.55,1) {$\alpha$};
\node[scale=0.5] at (0.55,1) {$\beta$};
\node[scale=0.5, red] at (-0.25,0.9) {$\rho'$};
\node[scale=0.5] at (0,0.1) {$\epsilon$};

    \end{tikzpicture}

    \caption{An example of an ideal triangulation of an orbifold and the unique flip $\rho'$ of its pending arc $\rho$.}
    \label{fig:Orbi}
\end{figure}
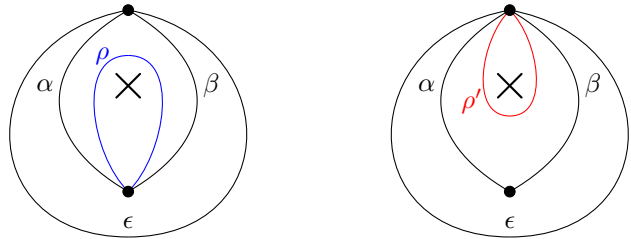

Notice that it is not possible, however, to flip the radius in a self-folded triangle - there is no distinct arc that can be added to create a new ideal triangulation. In the context of punctured surfaces, this issue motivated Fomin, Shapiro, and Thurston to introduce  the slightly more general notion of \emph{tagged arcs} \cite{FST-I}. This notion is also useful in the orbifold setting.

\begin{definition}
    A \emph{tagged arc} is obtained from any ordinary arc $\gamma$ that does not cut out a once-punctured monogon by ``tagging'' each end of $\gamma$ either \emph{plain} or \emph{notched} (denoted by decorating that end with $\bowtie$), such that any endpoints of $\gamma$ on $\partial \mathcal{O}$ are tagged plain and both ends of $\gamma$ have the same tagging if they are incident to the same marked point.
\end{definition}
For a tagged arc $\gamma$, let $\gamma^{0}$ denote the underlying plain arc obtained by forgetting the decorations at both endpoints. An arc with at least one end notched is sometimes called a notched arc. If $\gamma$ has a notched endpoint at puncture $p$, we may stress this by writing $\gamma^{(p)}$.

\begin{definition}
    Two tagged arcs $\alpha$ and $\beta$ are \emph{compatible} if and only if: the arcs $\alpha^0$ and $\beta^0$ are compatible; if $\alpha^0 = \beta^0$, then at least one end of $\alpha$ has the same tagging as the corresponding end of $\beta$; and if $\alpha^0 \neq \beta^0$ but $\alpha$ and $\beta$ share an endpoint, then $\alpha$ and $\beta$ have the same tagging at their shared endpoint.
\end{definition}

\begin{definition}
    A \emph{tagged triangulation} is a maximal collection of pairwise compatible tagged arcs.
\end{definition}

One can also define \emph{flips} for tagged triangulations. In this setting, every arc can be uniquely flipped.
In \cite[Theorem 4.2]{felikson2012-a}, Felikson, Shapiro, and Tumarkin show that, outside of a few extreme cases, flips act transitively on the set of tagged triangulations of an orbifold.  

Ideal triangulations can be translated into tagged triangulations by representing each ordinary arc $\gamma$ in the ideal triangulation as a tagged arc $\iota(\gamma)$. If $\gamma$ does not cut out a once-punctured monogon, then $\iota(\gamma)$ is simply $\gamma$ with both ends tagged plain. If $\gamma$ is the encircling arc in a self-folded triangle with radius $r$, then $\iota(\gamma)$ is obtained by tagging $r$ notched at the encircled puncture and plain at its other endpoint.

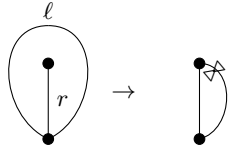
\begin{center}
    \begin{tikzpicture}
        \draw[out=25,in=0,looseness=1.25] (6.5,0) to (6.5,1.5);
        \draw[out=155,in=180,looseness=1.25] (6.5,0) to (6.5,1.5);
        \draw (6.5,0) to node[right]{$r$} (6.5,1);
        \node at (6.5,1.7) {$\ell$};
        \filldraw (6.5,0) circle[radius=2pt];
        \filldraw (6.5,1) circle[radius=2pt];

       \node[] at (7.5,0.6){$\to$};
       \draw[out=15,in=0,looseness=1.25] (8.5,0) to (8.5,1);
        \draw (8.5,0) to (8.5,1);
        \filldraw (8.5,0) circle[radius=2pt];
        \filldraw (8.5,1) circle[radius=2pt];
        \node[rotate=40] at (8.7,.9) {$\bowtie$};

    \end{tikzpicture}
\end{center}

It is occasionally useful to consider a \emph{tag-changing transformation} on a tagged triangulation. Each puncture has an associated tag-changing transformation that switches the decorations at that endpoint of every arc incident to the puncture. Notice that this operation preserves the compatibility of the set of arcs. By applying a sequence of tag-changing transformations, one can convert any triangulation into an entirely plain triangulation.

The coefficient variables of $\mathcal{A}$ correspond to laminations on $\mathcal{O}$. We use the definition of lamination given by Felikson, Shapiro, and Tumarkin in \cite{felikson2012-a}; although they work in a slightly different setting with an object called an \emph{associated orbifold}, the definition transfers into our setting.

\begin{definition}[\cite{felikson2012-a}, Definition 6.1]\label{def:lamination}
    Let $\mathcal{O} = (S,M,Q)$ be an orbifold. An \emph{integral unbounded lamination} (which we will henceforth refer to as just a \emph{lamination}) on $\mathcal{O}$ is a finite collection of non-self-intersecting and pairwise non-intersecting curves on $S$ such that:
    \begin{itemize}
        \item each curve is either a closed curve, a non-closed curve where each endpoints lies on an unmarked point on the boundary of $\mathcal{O}$, a spiral (either clockwise or counterclockwise) around a puncture contained in $M$, or a curve winding around an orbifold point in $Q$;
        \item no curve bounds an unpunctured disk or a disk containing a unique point of $M \cup Q$;
        \item no curve with both endpoints on the boundary of $\mathcal{O}$ is isotopic to a portion of the boundary containing either no or one marked point(s); and
        \item  no two curves begin at the same orbifold point.
    \end{itemize}
Finally, a \emph{multi-lamination} $\mathbf{L} = (L_1,\ldots,L_m)$ is a finite multi-set of laminations. 
\end{definition}

For surface-type ordinary cluster algebras, the shear coordinates introduced by W. Thurston provide a coordinate system for these laminations. Fomin and Thurston updated the definition to account for tagged arcs \cite{fomin2018cluster}, and in \cite{banaian2020snake}, two of the authors extended this coordinate system to accommodate the introduction of pending arcs.

\begin{definition}[\cite{banaian2020snake}, Definition 3.8]\label{def:shearCoordinates}
    Let $\mathcal{O}$ be an orbifold with triangulation $T$ and $L$ be a lamination on $\mathcal{O}$. For each arc $\gamma \in T$, the \emph{shear coordinate} of $L$ with respect to $\tau$ is $b_{\gamma}(T, L)$, which is the result of summing up all contributions from individual curves in $L$. When all of the arcs in $T$ are plain, the shear coordinates $b_{\gamma}(T,L_i)$ are defined as:

    
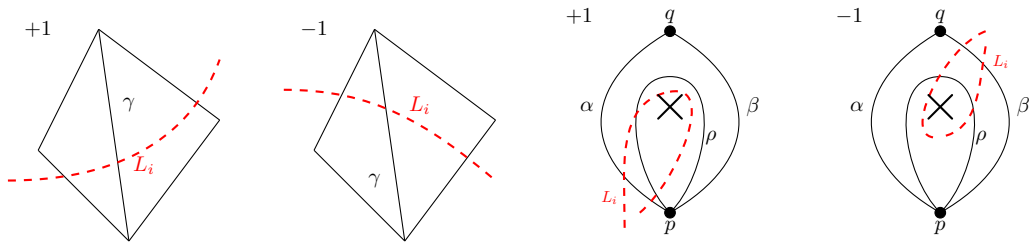
\begin{figure}[H]
    \centering
    \begin{tikzpicture}[transform shape, scale=0.4]

    \draw (-3,-1) to (0,-4);
    \draw (0,-4) to (3,0);
    \draw (0,-4) to (-1,3);
    \draw (3,0) to (-1,3);
    \draw (-3,-1) to (-1,3);
    
    \draw[out=0,in=-110,looseness=1,thick,dashed, red] (-4,-2) to (3,2);
    
    \node[scale=2] at (0,0.5) {$\gamma$};
    \node[scale=2,red] at (0.5,-1.5) {$L_i$};
    \node[scale=2] at (-3,3) {$+1$};
    
    \end{tikzpicture}
    \qquad
    \begin{tikzpicture}[transform shape, scale=0.4]
    
    \draw (-3,-1) to (0,-4);
    \draw (0,-4) to (3,0);
    \draw (0,-4) to (-1,3);
    \draw (3,0) to (-1,3);
    \draw (-3,-1) to (-1,3);
    \draw[out=0,in=140,looseness=1,thick, dashed,red] (-4,1) to (3,-2);
    
    \node[scale=2] at (-1,-2) {$\gamma$};
    \node[scale=2, red] at (0.5,0.5) {$L_i$};
    \node[scale=2] at (-3,3) {$-1$};
    
    \end{tikzpicture}
    \qquad
    \begin{tikzpicture}[transform shape, scale=2]
    
    \tikzset{->-/.style={decoration={
      markings,
      mark=at position #1 with {\arrow{>}}},postaction={decorate}}}
    
    
    \draw[out=30,in=-30,looseness=1.5] (0,0.3) to (0,1.5);
    \draw[out=150,in=-150,looseness=1.5] (0,0.3) to (0,1.5);
    
    \draw[in=0,out=45,looseness =1] (0,0.3) to (0,1.2);
    \draw[in=180,out=135,looseness=1] (0,0.3) to (0,1.2);
    
    \draw[in=10,out=40,looseness =1, dashed, red, thick] (-0.2,0.3) to (0,1.1);
    \draw[in=-170,out=90,looseness=1, dashed, red,thick] (-0.3,0.2) to (0,1.1);
            
    \draw[fill=black] (0,0.3) circle [radius=1pt];
    \draw[fill=black] (0,1.5) circle [radius=1pt];
    \draw (0,1) node {$\times$};
    
    \node[scale=0.4] at (0,0.2) {$p$};
    \node[scale=0.4] at (0,1.6) {$q$};
    \node[scale=0.4] at (-0.55,1) {$\alpha$};
    \node[scale=0.4] at (0.55,1) {$\beta$};
    \node[scale=0.4] at (0.27,0.8) {$\rho$};
    \node[scale=0.4] at (-0.6,1.6) {$+1$};
    \node[scale=0.3,red] at (-0.4,0.4) {$L_i$};

    \end{tikzpicture}
    \qquad
    \begin{tikzpicture}[transform shape, scale=2]
    
    \tikzset{->-/.style={decoration={
      markings,
      mark=at position #1 with {\arrow{>}}},postaction={decorate}}}
    
    
    \draw[out=30,in=-30,looseness=1.5] (0,0.3) to (0,1.5);
    \draw[out=150,in=-150,looseness=1.5] (0,0.3) to (0,1.5);
    
    \draw[in=0,out=45,looseness =1] (0,0.3) to (0,1.2);
    \draw[in=180,out=135,looseness=1] (0,0.3) to (0,1.2);
    
    \draw[in=10,out=-100,looseness =1, dashed, red, thick] (0.3,1.4) to (0,0.8);
    \draw[in=-170,out=-160,looseness=1, dashed, red,thick] (0.3,1.5) to (0,0.8);
            
    \draw[fill=black] (0,0.3) circle [radius=1pt];
    \draw[fill=black] (0,1.5) circle [radius=1pt];
    \draw (0,1) node {$\times$};
    
    \node[scale=0.4] at (0,0.2) {$p$};
    \node[scale=0.4] at (0,1.6) {$q$};
    \node[scale=0.4] at (-0.55,1) {$\alpha$};
    \node[scale=0.4] at (0.55,1) {$\beta$};
    \node[scale=0.4] at (0.27,0.8) {$\rho$};
    \node[scale=0.4] at (-0.6,1.6) {$-1$};
    \node[scale=0.3,red] at (0.4,1.3) {$L_i$};

    \end{tikzpicture}
   
    \end{figure}

When some of the arcs in $T$ are notched, one can use tag-changing transformations and the rules in \cite[Definition 13.1]{fomin2018cluster}.
\end{definition}

Unless stated otherwise, i.e., multi-lamination, we will be using \emph{principal coefficients}. For a given ideal triangulation $T$, this corresponds to a multi-lamination consisting of the elementary laminations associated with each arc of $T$.

\begin{definition}
    The \emph{elementary lamination} \(L_\gamma\) associated with an arc \(\gamma\) is a lamination defined by winding counterclockwise (resp. clockwise) into each plain (resp. notched) endpoint of $\gamma$ when the endpoint is a puncture, and by perturbing each endpoint of $\gamma$ on $\partial S$ in the clockwise direction. Given a triangulation $T = \{\tau_1, \ldots, \tau_n\}$, the associated \emph{elementary multi-lamination} is denoted as $\mathbf{L}_T := \{L_{\tau_1}, \ldots, L_{\tau_n}\}$.
\end{definition}

By construction, given $\mathbf{L}_T:= \{L_{\tau_1},\ldots,L_{\tau_n}\}$,  $b_{\tau_i}(T,L_{\tau_j})$ is 1 if $i = j$ and is 0 otherwise. Sometimes we will abbreviate $L_i = L_{\tau_i}$.

\subsection{Orbifold-Type Generalized Cluster Algebras}

It is well known that a subset of ordinary cluster algebras has a geometric realization in terms of triangulated surfaces. There is an analogous subset of the \emph{generalized cluster algebras} introduced by Chekhov and Shapiro in \cite{Chekhov-Shapiro} that can be realized geometrically by triangulated orbifolds. We will refer to this subset, on which our work focuses, as \emph{orbifold-type generalized cluster algebras}.

In the unpunctured setting, orbifold-type generalized cluster algebras have been explored in recent studies, including \cite{banaian2020snake,banaian2025snake, labardini2019family}. Felikson, Shapiro, and Tumarkin investigated ordinary, \emph{skew-symmetrizable} cluster algebras from orbifolds with different types of orbifold points in \cite{felikson2012-a,felikson2012-b}.  However, there are few results concerning orbifolds with orbifold points of order $\mathbf{p} \geq 3$ which also include punctures. Here, we address this larger set of generalized cluster algebras and show how punctured orbifolds provide a model for their structure.

Given an orbifold $\mathcal{O}$ equipped with a tagged triangulation $T = \{\tau_1,\ldots,\tau_n\}$ and a multi-lamination $\mathbf{L} = \{L_1,\ldots, L_m\}$, we build a seed as follows. Let $\mathbf{x} = \{x_1,\ldots,x_n\}$ and $\mathbf{y} = \{y_1,\ldots,y_n\}$. The matrix $B_{T,\mathbf{L}}$ be the $(n+m) \times n$ integer matrix where, for $1 \leq i \leq  n$, entry $b_{i,j}$ is the number of triangles in which $\tau_j$ follows $\tau_i$ in clockwise order minus the number of triangles in which $\tau_i$ follows $\tau_j$ in clockwise order and for $1 \leq i \leq m$, $b_{n+i,j} = b_{\tau_j}(T,L_i)$. Finally, we set $z_i = 1+\lambda_{\mathbf{p}} u + u^2$ is $\tau_i$ is a pending arc enclosing an orbifold point of order $\mathbf{p}$, and otherwise, if $\tau_i$ is standard, we set $z_i = 1 + u$. Let $\mathcal{A}_{\mathcal{O},\mathbf{L}}$ be the generalized cluster algebra associated to this data. If $\mathbf{L}$ is the elementary multi-lamination, our usual setting, we will just denote this as $\mathcal{A}_{\mathcal{O}}$. Notice that, in this case, the cluster algebra has principal coefficients. 

We will discuss here how, in the context of tagged triangulations, the dictionary between triangulations and seeds as for cluster algebras from surfaces discussed in~\cite{FST-I, fomin2018cluster} holds in the punctured orbifold setting. In particular, this justifies the lack of reference to $T$ in the notation $\mathcal{A}_{\mathcal{O},\mathbf{L}}$. The most important step in this verification is comparing the $B$-matrix mutation with flipping arcs and the effect on shear coordinates.

\begin{lemma}\label{lem:FlipAndMutate}
Let $T = \{\tau_1,\ldots,\tau_n\}$ be a triangulation of an orbifold $\mathcal{O}$ and let $\mathbf{L} = \{L_{1},\ldots,L_{m}\}$ be a multi-lamination. Let extended exchange matrix $\widetilde{B}_{T,\mathbf{L}}$ be the $(n+m) \times n$ matrix as defined above. Let $T'$ be the triangulation resulting from flipping $\tau_k$ and let $\widetilde{B}_{T'\mathbf{L}}$ be its extended exchange matrix, again with respect to $\mathbf{L}$. Then, $\mu_k(\widetilde{B}_T) = \widetilde{B}_{T',\mathbf{L}}$.  
\end{lemma}

\begin{proof}
Let $\mu_k(\widetilde{B}_{T,\mathbf{L}}) = \{b_{ij}'\}_{\substack{1 \leq i \leq n+m\\ 1 \leq j \leq n}}$. Let $T' = \{\tau_1',\ldots,\tau_n'\}$. From the definition of flipping an arc, we have $\tau_i' = \tau_i$ for $i \neq k$. 
We want to show that for $i \leq n$, $b'_{ij}$ records adjacencies in $T'$ and for $i > n$, $b'_{ij} = b_{\tau_j'}(T',L_i)$. When $i = k$ or $j = k$, the statements follow directly from definitions. If all the arcs involved are standard, then this follows from \cite[Theorem 13.5]{fomin2018cluster}. In general, the statement for $i \leq n$ is implicit from the definitions in \cite{Chekhov-Shapiro}. While the interactions of pending arcs and notched arcs are not explicitly discussed in the latter article, verification of the statement is straightforward.

Now we consider entries $b_{ij}'$ where $i > n$. Laminations and coefficient systems were not discussed in \cite{Chekhov-Shapiro}. However, if none of the arcs involved are in a self-folded triangle, the desired statement follows from \cite[Proposition 9.1]{banaian2020snake}, possibly after changing tagging at one or several punctures as in \cite[Definition 13.1.i]{fomin2018cluster}. Thus, in our present setting of punctured orbifolds, it suffices to check several specific situations involving both pending arcs and pairs of plain and notched arcs inside a bigon. One such scenario is drawn below, where $T$ is drawn on the left, $T'$ is drawn on the right, and $L_{i}$ is a lamination which contains $\ell_1$ with multiplicity $m_1$ and $\ell_2$ with multiplicity $m_2$.

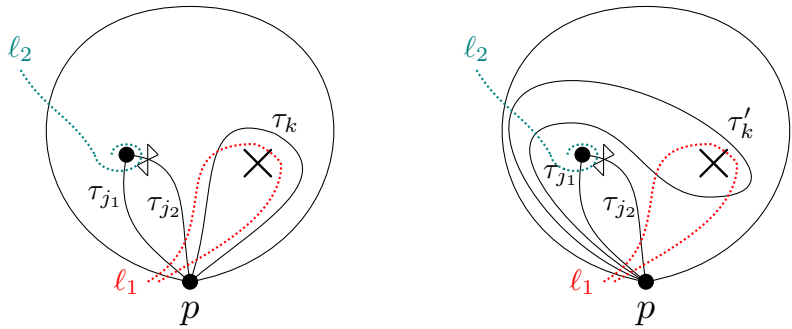
\begin{figure}[H]
\centering
\begin{tikzpicture}[transform shape, scale=2.8]
\tikzset{->-/.style={decoration={
  markings,
  mark=at position #1 with {\arrow{>}}},postaction={decorate}}}

\draw[out=10,in=0,looseness=1.8] (0,0) to (0,1.3);
\draw[out=170,in=180,looseness=1.8] (0,0) to (0,1.3);

\draw[out=100,in=-5,looseness=1.2] (0,0) to (-0.3,0.6);
\draw[out=140,in=-100,looseness=1.2] (0,0) to (-0.3,0.6);
\draw (-0.2,0.5) to (-0.2,0.65);
\draw (-0.25,0.55) to (-0.15,0.6);
\draw (-0.2,0.5) to (-0.25,0.55);
\draw (-0.15,0.6) to (-0.2,0.65);

\draw[out=70,in=160,looseness=1.3] (0,0) to (0.4,0.7);
\draw[out=40,in=-30,looseness=1.3] (0,0) to (0.4,0.7);

\draw[fill=black] (0,0) circle [radius=1pt];
\draw[fill=black] (-0.3,0.6) circle [radius=1pt];
\draw[scale=0.8,thick] (0.4,0.7) node {$\mathbf{\times}$};

\node[scale=0.5] at (0,-0.15) {$p$};
\node[scale=0.4] at (-0.4,0.4) {$\tau_{j_1}$};
\node[scale=0.4] at (-0.12,0.35) {$\tau_{j_2}$};
\node[scale=0.4] at (0.45,0.75) {$\tau_{k}$};

\draw[out=-60,in=120,densely dotted, thick, teal, looseness=1] (-0.8,1) to (-0.45,0.6);
\draw[out=-80,in=-90,densely dotted, thick, teal, looseness=1.2] (-0.45,0.6) to (-0.23,0.6);
\draw[out=90,in=90,densely dotted, thick, teal, looseness=1.2] (-0.23,0.6) to (-0.37,0.6);
\node[scale=0.4, teal] at (-0.8,1.1) {$\ell_2$};

\draw[out=30,in=-90,densely dotted, thick, red, looseness=0.8] (-0.15,0) to (0.43,0.55);
\draw[out=30,in=180,densely dotted, thick, red, looseness=1.4] (-0.2,0) to (0.3,0.65);
\draw[out=0,in=90,densely dotted, thick, red, looseness=0.5] (0.3,0.65) to (0.43,0.55);
\node[scale=0.4, red] at (-0.3,0) {$\ell_1$};

\end{tikzpicture}
    \qquad
\begin{tikzpicture}[transform shape, scale=2.8]
\tikzset{->-/.style={decoration={
  markings,
  mark=at position #1 with {\arrow{>}}},postaction={decorate}}}

\draw[out=10,in=0,looseness=1.8] (0,0) to (0,1.3);
\draw[out=170,in=180,looseness=1.8] (0,0) to (0,1.3);

\draw[out=100,in=-5,looseness=1.2] (0,0) to (-0.3,0.6);
\draw[out=140,in=-100,looseness=1.2] (0,0) to (-0.3,0.6);
\draw (-0.2,0.5) to (-0.2,0.65);
\draw (-0.25,0.55) to (-0.15,0.6);
\draw (-0.2,0.5) to (-0.25,0.55);
\draw (-0.15,0.6) to (-0.2,0.65);

\draw[fill=black] (0,0) circle [radius=1pt];
\draw[fill=black] (-0.3,0.6) circle [radius=1pt];
\draw[scale=0.8,thick] (0.4,0.7) node {$\mathbf{\times}$};

\draw[out=150,in=180,looseness=1.4] (0,0) to (-0.35,0.75);
\draw[out=160,in=180,looseness=1.5] (0,0) to (-0.35,0.95);
\draw[out=0,in=0,looseness=1.6] (-0.35,0.95) to (0.3,0.4);
\draw[out=0,in=180,looseness=1] (-0.35,0.75) to (0.3,0.4);

\node[scale=0.5] at (0,-0.15) {$p$};
\node[scale=0.4] at (-0.4,0.5) {$\tau_{j_1}$};
\node[scale=0.4] at (-0.12,0.35) {$\tau_{j_2}$};
\node[scale=0.4] at (0.45,0.75) {$\tau_{k}'$};

\draw[out=-60,in=120,densely dotted, thick, teal, looseness=1] (-0.8,1) to (-0.45,0.6);
\draw[out=-80,in=-90,densely dotted, thick, teal, looseness=1.2] (-0.45,0.6) to (-0.23,0.6);
\draw[out=90,in=90,densely dotted, thick, teal, looseness=1.2] (-0.23,0.6) to (-0.37,0.6);
\node[scale=0.4, teal] at (-0.8,1.1) {$\ell_2$};

\draw[out=30,in=-90,densely dotted, thick, red, looseness=0.8] (-0.15,0) to (0.43,0.55);
\draw[out=30,in=180,densely dotted, thick, red, looseness=1.4] (-0.2,0) to (0.3,0.65);
\draw[out=0,in=90,densely dotted, thick, red, looseness=0.5] (0.3,0.65) to (0.43,0.55);
\node[scale=0.4, red] at (-0.3,0) {$\ell_1$};

\end{tikzpicture}
\caption{Flip of an arc $\tau_k$ when $\tau_k$ is a pending arc. Note that $\ell_1$ is an elementary lamination of $\tau_k$ and $\ell_2$ is an elementary lamination of $\tau_{j_1}$.}
\end{figure}

In this case, we have $b_{kj_1} = b_{kj_2} = 1$. From Definition \ref{def:shearCoordinates}, we have $b_{i+n,k} = m_1$. Following the definition of shear coordinates with respect to tagged triangulations in \cite[Chapter 13]{fomin2018cluster}, with respect to $T$, we have $b_{\tau_{j_1}}(T,L_i) = m_2$ and $b_{\tau_{j_1}}(T,L_i) = 0$. Then, with respect to $T'$, we calculate \[
b_{\tau_{j_1}}(T',L_i) = 2m_1 + m_2 = b_{\tau_{j_1}}(T,L_i) + 2b_{kj_1} b_{\tau_k}(T,L_i) 
\]
and
 \[
b_{\tau_{j_2}}(T',L_i) = 2m_1 = b_{\tau_{j_2}}(T,L_i) + 2 b_{kj_2} b_{\tau_k}(T,L_i).
\]

Note that the factor of 2 in the shear coordinate update corresponds to the degree $d_k = 2$ of the exchange polynomial $z_k$ associated with the pending arc $\tau_k$.
All other cases can be checked similarly.

\end{proof}

With the previous lemma in hand, it is straightforward to check that the seeds of $\mathcal{A}_{\mathcal{O},\mathbf{L}}$ are in correspondence with tagged triangulations. We direct a reader to \cite{fomin2018cluster} for a description of the laminated, decorated Teichm{\"u}ller space of  surface and \cite{Chekhov-Shapiro,felikson2012-a} for the adaptation to an orbifold.  

\begin{theorem}\label{thm:GeometryOfGenCA}
Let $\mathcal{O}$ be an orbifold with a tagged triangulation $T = (\tau_1,\ldots,\tau_n)$ and a multi-lamination $\mathbf{L} = (L_1,\ldots, L_m)$. Let $\mathcal{A}_{\mathcal{O},\mathbf{L}}$ be the associated generalized cluster algebra. If $\mathcal{O}$ is not a once-punctured closed surface, then $\mathcal{A}_{\mathcal{O},\mathbf{L}}$ satisfies the following.
\begin{enumerate}
    \item The cluster variables of $\mathcal{A}_{\mathcal{O},\mathbf{L}}$  are indexed by tagged arcs in $\mathcal{O}$. 
    \item The seeds of $\mathcal{A}_{\mathcal{O}}$ are indexed by tagged triangulations of $\mathcal{O}$ in such a way that the seed indexed by $T$ contains the cluster variables $\{x_\gamma\}_{\gamma \in T}$.
\end{enumerate}

If $\mathcal{O}$ is a once-punctured closed surface, then the cluster variables (seeds) of $\mathcal{A}_{\mathcal{O},\mathbf{L}}$ correspond to either plain arcs (triangulations) or notched arcs (triangulations), depending on the initial data.

Moreover, the \emph{laminated, decorated Teichm{\"u}ller space} of $\mathcal{O}$ with respect to the elementary lamination gives a positive realization (in the sense of \cite{fomin2018cluster}) of the exchange pattern of $\mathcal{A}_{\mathcal{O}}$.
\end{theorem}

\begin{proof}
The case for a once-punctured closed surface follows from Proposition 7.10 of~\cite{FST-I}. Even when allowing for orbifold points, flip of tagged arcs will never change the notch of an endpoint.  Hence, the proof of Proposition 7.10 of~\cite{FST-I} carries over to the orbifold setting.

The other items follow from Lemma \ref{lem:FlipAndMutate}. The last item follows from verifying that mutation relations in $\mathcal{A}_{\mathcal{O}}$ agree with relations on lambda lengths. This can be done by following the proof of \cite[Theorem 15.6]{fomin2018cluster} but, if necessary, replacing two-term relations with three term relations, as in \cite[Lemma 3.1]{Chekhov-Shapiro}.
\end{proof}

For the remainder of the article, we will only be concerned with the elementary multi-lamination. We will associate an initial seed of the cluster algebra $\mathcal{A}_{\mathcal{O}}$ to a tagged triangulation $T = \{\tau_1,\ldots,\tau_n\}$ with associated cluster $\{x_i\}_{1 \leq i \leq n}$ where $x_i$ is indexed by $\tau_i$. Since we are working with principal coefficients, we also index the $y$-variables $y_1,\ldots,y_n$ associated to $\tau_1,\ldots,\tau_n$ repsectively. Given any ordinary arc $\gamma$, let $x_\gamma^T$ denote the cluster variable indexed by the arc $\gamma$ written with respect to this cluster indexed by $T$. Later, we will also define such terms for generalized arcs and closed curves. We expect these will be important expressions to consider in future work concerning bases of generalized cluster algebras. 

\section{Combinatorial Expansion Formulas}\label{sec:Expansion}
\setcounter{subsection}{-1}

\subsection{Setting}\label{sec:Setting}

We describe here the setting that will apply to future sections. Fix an orbifold $\mathcal{O}$ and a tagged triangulation $T$. Unless stated otherwise, we denote a pending arc in $T$ by $\rho$. We assume throughout that the only notched arcs in $T$ are singly-notched and that these appear in pairs with a plain arc, such that the underlying undecorated arcs coincide. In such a setting, we replace $T$ with the ideal triangulation $T^\circ$, i.e., the result of replacing each notched arc $\gamma^{(p)}$ with a loop cutting out a singly-punctured monogon containing $p$. 

We present here several combinatorial formulas which produce a Laurent polynomial for a (possibly generalized) arc or closed curve $\gamma$ with respect to $T$.  The Laurent polynomials will lie in the ambient field of rational functions in the cluster indexed by $T$ of the generalized cluster algebra $\mathcal{A}_\mathcal{O}$.

We arbitrarily choose an orientation of $\gamma$. We choose a representative of $\gamma$ that minimizes intersections with $T$ and has no self-intersections lying on the arcs of $T$. We also choose a representative that has winding number $k \in \{0,1,\ldots,\mathbf{p}-2\}$ at every orbifold point. Recall this means that $\gamma$ winds in a counterclockwise direction around the orbifold point. Ordinary arcs will always have winding number 0 or $\mathbf{p}-2$; in the latter case, the arc would be isotopic to one without self-intersection.

With our choice of orientation of an arc $\gamma$, denote $s(\gamma)$ as the starting point of $\gamma$ and $t(\gamma)$ as the terminal point of $\gamma$. We index the intersection points of $\gamma$ and $T$ by $c_1,\ldots,c_d$. Our last stipulation guarantees that the $c_j$'s are distinct. Let $\tau_{i_j} \in T$ be such that $c_j$ lies on $\tau_{i_j}$. While the $c_j$'s are distinct, it can be that $\tau_{i_j} = \tau_{i_{j'}}$ for $j \neq j'$. Finally, let the faces that $\gamma$ passes through be $\Delta_0,\Delta_1,\ldots,\Delta_d$. These can be ordinary triangles, pending triangles, or self-folded triangles.

If $\gamma$ is a closed curve, choose a point $\mathrm{s}$ on $\gamma$ which lies in an ordinary triangle and let $\gamma'$ be the curve which starts and ends at $\mathrm{s}$. We will use notation as above with respect to $\gamma'$.

Finally, if $\gamma$ is singly-notched, orient $\gamma$ so that the notch occurs at its start point. Let $s(\gamma)=p$ be the puncture at which $\gamma$ is notched. The symbol $\widetilde{\gamma}$ will represent the  result of completing nearly a whole clockwise or counterclockwise revolution around $p$, starting and ending in $\Delta_0$, then following $\gamma$. Following \cite{wilson2020surface}, we call these near-revolutions clockwise or counterclockwise hooks. We specify the direction only when it affects the construction. If $\gamma$ is doubly-notched, then we introduce a hook at each endpoint. Although we are currently not dictating the direction of the hook(s), we will see later that for defining $T$-walks, some terms need fixed orientations. This will be discussed in~\Cref{sec:Twalk}.

Refer to arcs in $T$ with at least one endpoint on a puncture as \emph{spokes}. Suppose the spokes in $T$ incident to a fixed puncture $p$ are $\sigma_1,\ldots,\sigma_m$, indexed counterclockwise so that $\Delta_0$ is bordered by $\tau_{i_1}$, $\sigma_1$, and $\sigma_m$. If an arc of $T$ has both endpoints at $p$ (e.g., a pending arc), then it receives two labels in this notation. Notice that, if $\gamma^{(p)}$ is notched at $p$, then  $\widetilde{\gamma}$ crosses each arc $\sigma_i$ once for each endpoint of $\tau$ lying at $p$. See Figure \ref{fig:resolutions}.

\begin{figure}[h]
\begin{center}
    \begin{tikzpicture}[scale=1.35]
    \draw(-5,0) to node[left]{$\sigma_1$}(-4.5,1);
    \draw(-5,0) to node[left]{$\sigma_m$}(-4.5,-1);
    \draw(-4.5,1) to node[right, yshift = 15pt]{$\tau_{i_1}$} (-4.5,-1);
    \draw(-5,0) to (-5.5,0.5);
    \draw(-5,0) to (-5.5,-0.5);
    \draw (-2,0) to node[right]{$\eta_h$} (-2.5,1);
    \draw (-2,0) to   node[right]{$\eta_1$}(-2.5,-1);
    \draw (-2,0) to (-1.5,0.5);
    \draw (-2,0) to (-1.5,-0.5);
    \draw (-2.5,-1) to node[left, yshift = 15pt]{$\tau_{i_d}$} (-2.5,1);
    \draw[orange,line width = 1.2pt] (-5,0) to node[midway,below]{$\gamma^{(p,q)}$} (-2,0);
    \draw[orange,thick] (2.9-5,0.1) to (2.8-5,-0.1);
    \draw[orange,thick] (2.8-5,0.1) to (2.9-5,-0.1);
    \draw[orange,thick] (2.9-5,0.1) to (2.8-5,0.1);
    \draw[orange,thick] (2.9-5,-0.1) to (2.8-5,-0.1);
    \draw[orange,thick] (3.1-8,0.1) to (3.2-8,-0.1);
    \draw[orange,thick] (3.2-8,0.1) to (3.1-8,-0.1);
    \draw[orange,thick] (3.1-8,0.1) to (3.2-8,0.1);
    \draw[orange,thick] (3.1-8,-0.1) to (3.2-8,-0.1);
    \draw[fill=black] (-2,0) circle [radius=1pt];
    \node at (-1.95,0.2) {$q$};
    \draw[fill=black] (-5,0) circle [radius=1pt];
    \node at (-5.05,0.2) {$p$};
\draw[fill=black] (0,0) circle [radius=1pt];
    \node at (-0.05,0.2) {$p$};
    \draw(0,0) to node[left]{$\sigma_1$}(0.5,1);
    \draw(0,0) to node[left]{$\sigma_m$}(0.5,-1);
    \draw(0.5,1) to node[right, yshift = 15pt]{$\tau_{i_1}$} (0.5,-1);
    \draw(0,0) to (-0.5,0.5);
    \draw(0,0) to (-0.5,-0.5);
    \draw (3,0) to node[right]{$\eta_h$} (2.5,1);
    \draw (3,0) to   node[right]{$\eta_1$}(2.5,-1);
    \draw (3,0) to (3.5,0.5);
    \draw (3,0) to (3.5,-0.5);
    \draw (2.5,-1) to node[left, yshift = 15pt]{$\tau_{i_d}$} (2.5,1);
    \draw[fill=black] (3,0) circle [radius=1pt];
    \node at (3.05,0.2) {$q$};
    \draw[orange,line width = 1.2pt] (0.2,0) to node[midway,below]{$\widetilde{\gamma}$} (2.8,0);
    \draw[orange, line width = 1.2pt, out=90,in=90,looseness=2.25] (2.8,0) to (3.3,0);
    \draw[orange, line width = 1.2pt, out=-90,in=-70,looseness=2] (3.3,0) to (2.8,-0.15);
    \draw[orange, line width = 1.2pt, out=-90,in=-90,looseness=2] (0.2,0) to (-0.3,0);
    \draw[orange, line width = 1.2pt, out=90,in=110,looseness=2] (-0.3,0) to (0.2,0.15);
    \end{tikzpicture}
    \end{center}
    \caption{Resolutions of notches at the punctures.}
    \label{fig:resolutions}
\end{figure}
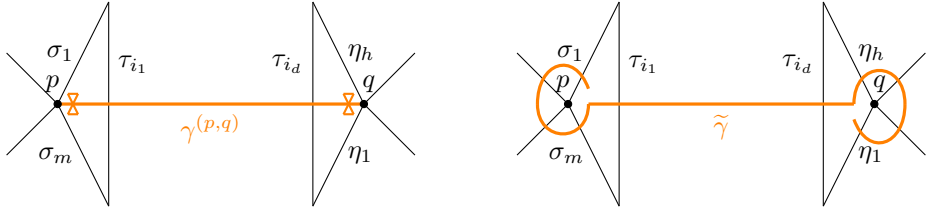

Finally, we remark that we will ignore arcs and closed curves which wind around an orbifold point and do not cross any arcs from $T$. Each of these should be associated to a scalar given by evaluating a Chebyshev polynomial of the first or second kind at the value $\lambda_{\mathbf{p}}$ associated to the orbifold point.

\subsection{Snake Graphs: Unpunctured Orbifolds and Hexagonal Tiles}\label{sec:HexagonalTiles}

In \cite{banaian2020snake}, two of the authors gave a construction for snake graphs from unpunctured orbifolds. Here we briefly recall this definition. In the next section, we will give an alternate combinatorial model which will be more readily extended to the punctured case. For this reason, we will use the notation $\GGen$ in this section to distinguish from the graphs in the next section, denoted simply $\calG$. We refer to a graph $\GGen$ as a ``snake graph with hexagonal tiles.''

Let $\mathcal{O}$ be an unpunctured orbifold with ideal triangulation $T = \{\tau_1,\ldots,\tau_n\}$.  Consider first a (possibly generalized) arc $\gamma$ on $\mathcal{O}$. If $\gamma \in T$, then we define $\GGen_\gamma$ to be the graph consisting of one edge which is labelled $x_\tau$. Otherwise, $\gamma$ crosses at least one arc in $T$. To each standard arc which $\gamma$ crosses, we associate two square tiles, one with orientation matching that of the surface and one with opposite orientation.

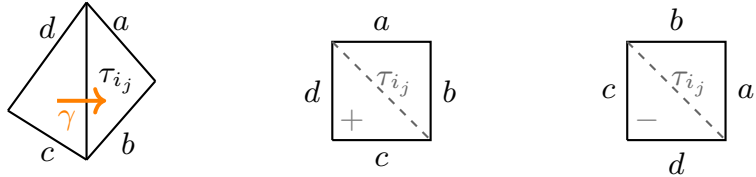
\begin{center}
    \begin{tikzpicture}[scale=1.3, transform shape]

        \draw[thick] (-3.3,0.3) to node[midway,above, yshift=2pt]{$d$} (-2.5,1.4);
        \draw[thick] (-3.3,0.3) to node[midway,below]{$c$}(-2.5,-0.2);
        \draw[thick] (-1.8,0.6) to node[midway,above]{$a$}(-2.5,1.4);
        \draw[thick] (-1.8,0.6) to node[midway,below,xshift=2pt]{$b$}(-2.5,-0.2);
        \draw[thick] (-2.5,-0.2) to node[midway,right]{$\tau_{i_j}$} (-2.5,1.4);
        \draw[line width = 1.5,orange, ->] (-2.8,0.4) to (-2.3,0.4);
        \node[orange] at (-2.7,0.2){$\gamma$};
        
        \draw[thick] (0,0) to node[below,midway]{$c$} (1,0) to node[right,midway]{$b$} (1,1) to node[midway,above]{$a$} (0,1) to node[left,midway]{$d$} (0,0);
        \draw[thick, dashed, black!60!white] (0,1) to node[above,midway,xshift=4,yshift=-5]{$\tau_{i_j}$} (1,0);
        \node[gray] at (0.2,0.2){$+$};
        
        \draw[thick] (3,0) to node[below,midway]{$d$} (4,0) to node[right,midway]{$a$} (4,1) to node[midway,above]{$b$} (3,1) to node[midway,left]{$c$} (3,0);
        \draw[thick, dashed, black!60!white] (3,1) to node[above,midway,xshift=4,yshift=-5]{$\tau_{i_j}$} (4,0);
        \node[gray] at (3.2,0.2){$-$};

    \end{tikzpicture}
\end{center}

If in our sequence of intersections we have $\tau_{i_j} = \tau_{i_{j+1}}$, then $\tau_{i_j}$ is necessarily a pending arc. This pair of intersections contributes a single hexagonal tile, as drawn in the first row of~\Cref{table:PuzzlePiecesWinding}. In the figures, $k$ is the winding number of $\gamma$ around the pending arc. As hexagonal tiles are composites of two squares, we describe their orientation with a tuple. The orientation of the hexagonal tile on the left below is $(-,+)$ and the orientation of the hexagonal tile on the right below is $(+,-)$. When $k=0$, the edge weight $U_{-1}(\lambda_\mathbf{p})$ vanishes, and the hexagonal tile collapses into the square-tile configuration shown in~\Cref{table:PuzzlePiecesWinding}.
 Similarly, if $k = \mathbf{p}-2$, then $\gamma$ is isotopic to one with 0 winding. Since $U_{\mathbf{p}-1}(\lambda_{\mathbf{p}}) = 0$, we again can remove one edge and recover two square tiles.

The other case to consider is when $\gamma$ is a corner arc, i.e., when it winds non-trivially around an orbifold point before its first intersection with arcs of $T$ (or winds non-trivially around an orbifold point after its last intersection with arcs of $T$). This type of intersection corresponds to a single square tile. Following our convention for winding, the case for counterclockwise winding appears in the third row of~\Cref{table:PuzzlePiecesWinding}. As before, the positive tile orientation is shown on the left and the negative tile orientation on the right. 

With all tiles constructed, we now form the snake graph $\mathcal{G}_{\gamma,T}$ by fixing an orientation of $\gamma$ and then gluing together subsequent tiles in the order of the corresponding intersections of $\gamma$ with $T$. By convention, we use the positive orientation (or orientation $(+,-)$ if it is a hexagonal tile) of the tile associated to the first crossing. Then, we proceed by gluing tiles one-by-one, using alternating tile orientations. Suppose we have accounted for crossings $\tau_{i_1},\ldots,\tau_{i_j}$, and that the tile associated to $\tau_{i_j}$ has orientation $\pm$, or if it is hexagonal, has orientation $(\mp, \pm)$. Then, given  $\tau_{i_j} \neq \tau_{i_{j+1}}$, we know that these two arcs form a triangle with a third arc, which we denote $\tau_{[j]}$. From the construction, we can see that the tiles associated to each intersection will have an edge labelled $\tau_{[j]}$.  We glue the tile associated to $\tau_{i_{j+1}}$ with orientation $\mp$ or $(\mp,\pm)$ to our partially constructed snake graph along the edge labelled $\tau_{[j]}$. The result after accounting for all intersections is the snake graph $\mathcal{G}_{\gamma,T}$. A result of this construction is that the orientations of the tiles of $\mathcal{G}_{\gamma,T}$ will alternate as we look across the graph. For example, see Table \ref{table:whole}.

Our construction of $\mathcal{G}_{\gamma,T}$ also invites a natural indexing of the tiles, $(G_1,\ldots,G_d)$, where $G_j$ is associated to the intersection with $\tau_{i_j}$. The hexagonal tiles are indexed as two tiles, $G_j$ and $G_{j+1}$, simultaneously. In fact, these can be seen as two square tiles glued together with an extra edge between the two tiles added.

The main theorem of \cite{banaian2020snake} is that these snake graphs encode the Laurent expansion of cluster variables. To write this precisely, we recall some notation. First, given the sequence $\tau_{i_1},\ldots,\tau_{i_d}$ of arcs from $T$ crossed by $\gamma$, let $\mathrm{cross}(\gamma,T) = x_{{i_1}} \cdots x_{{i_d}}$.

\begin{table}[H]
\centering
\begin{tabular}{|c|c|}
\hline
\begin{tikzpicture}[scale=1.7]
\tikzset{->-/.style={decoration={
  markings,
  mark=at position #1 with {\arrow{>}}},postaction={decorate}}}
\draw[out=15,in=-15,looseness=1.5] (0,0.3) to (0,1.65);
\draw[out=165,in=-165,looseness=1.5] (0,0.3) to (0,1.65);
\draw[in=0,out=45,looseness =1.25] (0,0.3) to (0,1.3);
\draw[in=180,out=135,looseness=1.25] (0,0.3) to (0,1.3);
\draw[orange,line width = 1.2pt,out=0,in=210] (-0.4,0.7) to (0,0.75);
\draw[orange,line width = 1.2pt,out=30,in=0,looseness=2,->] (0,0.75) to (0,1.2);
\draw[orange,line width = 1.2pt,out=180,in=150,looseness=2] (0,1.2) to (0,0.75);
\draw[orange,line width = 1.2pt,out=-30,in=180] (0,0.75) to (0.4,0.7);

\draw[fill=black] (0,0.3) circle [radius=2pt];
\draw[fill=black] (0,1.65) circle [radius=2pt];
\draw[thick] (0,1) node {$\mathbf{\times}$};

\node[scale=1] at (-0.56,1.4) {$\alpha$};
\node[scale=1] at (0.56,1.41) {$\beta$};
\node[scale=1] at (0,1.4) {$\rho$};

\end{tikzpicture}
&
\begin{tikzpicture}[scale = 1.5]
\draw[gray,thick,dashed] (-0.7,1.7) to node[midway,below,scale=0.75,xshift=-2pt,yshift=2pt]{$\rho$} (1,1);
\draw[gray,thick,dashed] (0,2.4) to node[midway,below,scale=0.75,xshift=-2pt,yshift=2pt]{$\rho$} (1.7,1.7);
\draw[thick] (-0.7,1.7) to node[midway,above,xshift=-20pt, yshift=-1pt,scale=0.75]{$U_{k-1}\rho$} (1.7,1.7);
\draw[ultra thick,white] (0,2.4) to (1,1);
\draw[thick] (0,2.4) to node[midway,xshift=19pt,yshift=-10pt,scale=0.75]{$U_{k+1}\rho$} (1,1);
\draw[thick] (1,1) to node[midway,below,xshift=3pt,yshift=2pt,scale=0.75]{$\beta$} (0,1);
\draw[thick] (0,1) to node[midway,below,,xshift=-1pt,scale=0.75]{$\alpha$} (-0.7,1.7) to node[midway,above,xshift=-2pt,scale=0.75]{$U_k \rho$} (0,2.4) to node[midway,above,scale=0.75,yshift=-2pt]{$\alpha$\textcolor{white}{$\beta$}} (1,2.4) to node[midway,above,xshift=-1pt,yshift=2pt,scale=0.75]{$\beta$}(1.7,1.7) to node[midway,below,scale=0.75,xshift=7pt]{$U_k \rho$} (1,1);
\end{tikzpicture}
\begin{tikzpicture}[scale = 1.5]
\draw[gray,thick,dashed] (-0.7,1.7) to node[midway,below,scale=0.75,xshift=-2pt,yshift=2pt]{$\rho$} (1,1);
\draw[gray,thick,dashed] (0,2.4) to node[midway,below,scale=0.75,xshift=-2pt,yshift=2pt]{$\rho$} (1.7,1.7);
\draw[thick] (-0.7,1.7) to node[midway,above,xshift=-20pt, yshift=-1pt,scale=0.75]{$U_{k+1}\rho$} (1.7,1.7);
\draw[ultra thick,white] (0,2.4) to (1,1);
\draw[thick] (0,2.4) to node[midway,xshift=19pt,yshift=-10pt,scale=0.75]{$U_{k-1}\rho$} (1,1);
\draw[thick] (1,1) to node[midway,below,xshift=3pt,yshift=2pt,scale=0.75]{$\alpha$\textcolor{white}{$\beta$}} (0,1);
\draw[thick] (0,1) to node[midway,below,,xshift=-1pt,scale=0.75]{$\beta$} (-0.7,1.7) to node[midway,above,xshift=-2pt,scale=0.75]{$U_k \rho$} (0,2.4) to node[midway,above,scale=0.75,yshift=-2pt]{$\beta$} (1,2.4) to node[midway,above,xshift=-1pt,yshift=2pt,scale=0.75]{$\alpha$}(1.7,1.7) to node[midway,below,scale=0.75,xshift=7pt]{$U_k \rho$} (1,1);
\end{tikzpicture}\\ \hline

\begin{tikzpicture}[scale=1.7]
\tikzset{->-/.style={decoration={
  markings,
  mark=at position #1 with {\arrow{>}}},postaction={decorate}}}
\draw[out=15,in=-15,looseness=1.5] (0,0.3) to (0,1.65);
\draw[out=165,in=-165,looseness=1.5] (0,0.3) to (0,1.65);
\draw[in=0,out=45,looseness =1.25] (0,0.3) to (0,1.3);
\draw[in=180,out=135,looseness=1.25] (0,0.3) to (0,1.3);
\draw[orange,line width =1.2pt,out=-10,in=190, looseness=1, ->] (-0.4,0.7) to (0.4,0.7);

\draw[fill=black] (0,0.3) circle [radius=2pt];
\draw[fill=black] (0,1.65) circle [radius=2pt];
\draw[thick] (0,1) node {$\mathbf{\times}$};

\node[scale=1] at (-0.56,1.4) {$\alpha$};
\node[scale=1] at (0.56,1.41) {$\beta$};
\node[scale=1] at (0,1.4) {$\rho$};

\end{tikzpicture}
&
\begin{tikzpicture}[scale = 1.25]
\draw[gray,thick,dashed] (0,2) to node[midway,right,yshift=2pt,xshift=-2pt]{$\rho$} (1,1);
\draw[gray,thick,dashed] (1,2) to node[midway,right,yshift=2pt,xshift=-2pt]{$\rho$} (2,1);

\draw[thick] (0,1) to node[midway,left,scale=0.75,xshift=1pt]{$\alpha$} (0,2) to node[midway,above,scale=0.75,xshift=-2pt]{$\rho$}  (1,2) to node[midway,right,scale=0.75,xshift=-2pt, yshift = -8pt]{$U_1\rho$}  (1,1) to node[midway,below,scale=0.75,xshift=1pt]{$\beta$} (0,1);
\draw[thick] (1,1) to (1,2) to node[midway,above,scale=0.75,xshift=-2pt]{$\alpha$}  (2,2) to node[midway,right,scale=0.75,xshift=-2pt, yshift = -8pt]{$\beta$}  (2,1) to node[midway,below,scale=0.75,xshift=1pt]{$\rho$} (1,1);

\draw[thick] (3,0.5) to node[below,scale=0.75,yshift = 2pt]{$\alpha$} (4,0.5) to node[right,scale=0.75]{$\rho$} (4,1.5) to node[right,scale=0.75]{$\alpha$} (4,2.5) to node[above,scale=0.75,yshift = -2pt]{$\beta$} (3,2.5) to node[left,scale=0.75]{$\rho$} (3,1.5) to node[left,scale=0.75]{$\beta$} (3,0.5);
\draw[thick] (3,1.5) to node[above,scale=0.75]{$U_1\rho$}(4,1.5);
\draw[gray, thick, dashed] (3,1.5) to node[midway,right,yshift=2pt,xshift=-2pt]{$\rho$} (4,0.5);
\draw[gray, thick, dashed] (3,2.5) to node[midway,right,yshift=2pt,xshift=-2pt]{$\rho$} (4,1.5);

\end{tikzpicture}

\\
\hline
\begin{tikzpicture}[scale=1.7]
\tikzset{->-/.style={decoration={
  markings,
  mark=at position #1 with {\arrow{>}}},postaction={decorate}}}
\draw[out=15,in=-15,looseness=1.5] (0,0.3) to (0,1.65);
\draw[out=165,in=-165,looseness=1.5] (0,0.3) to (0,1.65);
\draw[in=0,out=45,looseness =1.25] (0,0.3) to (0,1.3);
\draw[in=180,out=135,looseness=1.25] (0,0.3) to (0,1.3);

\draw[orange, line width = 1.2pt,out=80,in=-10,looseness = 1,->] (0,0.3) to (0,1.25);
\draw[orange, line width = 1.2pt,out=180,in=180,looseness = 1.5] (0,1.25) to (0,0.78);
\draw[orange, line width = 1.2pt,out=0,in=0,looseness = 1.3] (0,0.78) to (0,1.15);
\draw[orange, line width = 1.2pt,out=180,in=180,looseness = 1.7] (0,1.15) to (0.05,0.85);
\draw[orange, line width = 1.2pt,out=0,in=270,looseness = 1] (0.05,0.85) to (0.25,1.45); 

\draw[fill=black] (0,0.3) circle [radius=2pt];
\draw[fill=black] (0,1.65) circle [radius=2pt];
\draw[thick] (0,1) node {$\mathbf{\times}$};

\node[scale=1] at (-0.56,1.4) {$\alpha$};
\node[scale=1] at (0.56,1.41) {$\beta$};
\node[scale=1] at (0,1.4) {$\rho$};

\end{tikzpicture}&
\begin{tikzpicture}[scale = 1.25]

\draw[gray,thick,dashed] (2,2) to node[midway,right,yshift=2pt,xshift=-2pt]{ $\rho$} (3,1);
\draw[thick] (2,1) to node[midway,left,scale=0.75,yshift=8pt]{$U_{k} \rho$}  (2,2) to node[midway,above,scale=0.75,xshift=-1pt]{$\alpha$}  (3,2) to node[midway,right,scale=0.75,xshift=-1pt]{$\beta$} (3,1)  to node[midway,below,scale=0.75,yshift=-1pt]{$U_{k-1} \rho$}  (2,1);
\end{tikzpicture}

\begin{tikzpicture}[scale = 1.25]
\draw[gray,thick,dashed] (0,2) to node[midway,right,yshift=2pt,xshift=-2pt]{ $\rho$} (1,1);
\draw[thick] (0,1) to node[midway,left,scale=0.75,xshift=1pt]{$U_{k-1} \rho$} (0,2) to node[midway,above,scale=0.75,xshift=-2pt]{$\beta$}  (1,2) to node[midway,right,scale=0.75,xshift=-2pt, yshift = -8pt]{$\alpha$}  (1,1) to node[midway,below,scale=0.75,xshift=1pt]{$U_{k} \rho$} (0,1);
\end{tikzpicture}
\\ \hline
\end{tabular}
\caption{Two tiles with different orientations, each coming from intersections with a pending arc. The winding number of the arcs in the first and third row is $k$.}\label{table:PuzzlePiecesWinding}
\end{table}

Next, recall that a perfect matching (or a 1-dimer cover) of a graph $G = (V,E)$ is a subset of the edge set $M \subseteq E$ such that for all $v \in V$, there is exactly one $e \in M$ such that $v$ is incident to $e$. Let the set of all matchings of a graph be $\Match(G)$. Since our graphs have weighted edges, we can naturally assign a weight to a perfect matching, i.e., $x(M) := \prod_{e \in M} \mathrm{wt}(e)$. In this construction, if an edge $e$ corresponds to a boundary component of the orbifold $\mathcal{O}$, we set $\mathrm{wt}(e) = 1$. If an edge $e$ is labelled as $U_k \rho$, its weight is given by $\mathrm{wt}(e) = U_k x_\rho$.

To encode the coefficients, one can associate another statistic $y(M)$ to a perfect matching $M$. We say an edge of a snake graph is a \emph{boundary edge} if it only borders one tile, and it is not one of the crossing edges in a hexagonal tile. We remark that the following definition interacts with our convention of having tile $G_1$ have positive orientation. 

\begin{definition}\label{def:MinimalMatching1}
Given a snake graph $\GGen_{\gamma,T} = (G_1,\ldots,G_d)$, let $M_-$ denote the perfect matching of $\mathcal{G}_{\gamma,T}$ that (i) consists entirely of boundary edges and (ii) does not connect the southern vertices of $G_1$ to $G_2$.
\end{definition}

In this setting, one could more simply say that the minimal matching uses the south edge of the first tile. However, the definition as written more readily adapts to the generality of loop graphs, which will be introduced in the next section. 

Given any perfect matching $M$ of $\mathcal{G}_{\gamma,T}$, recall $M \ominus M_- = (M \cup M_-) \backslash (M \cap M_-)$. The set $M \ominus M_-$ forms a set of disjoint cycles in $\mathcal{G}_{\gamma,T}$. We define $y(M)$ to be the product of the $y$-variables associated to the diagonals enclosed in the cycles in $M \ominus M_-$. We are now ready to precisely state how a snake graph encodes the cluster expansion of a cluster variable $x_\gamma$.

\begin{theorem}[Theorem 1.1 \cite{banaian2020snake}]\label{thm:BKOld}
Let $\mathcal{O}$ be an unpunctured orbifold with triangulation $T$. 
Then we have \[
x_\gamma^T = \frac{1}{\mathrm{cross}(\gamma,T)} \sum_{M \in \Match(\mathcal{G}_{\gamma,T})} x(M)y(M).
\]
\end{theorem}

\begin{example}\label{ex:1poset} The top middle entry of Table \ref{table:whole} is the snake graph associated to the arc on the top left of the same table. Following Theorem \ref{thm:BKOld}, the cluster expansion of $x_\gamma^T$ is given by 
    \[\frac{U_1(\lambda_{\mathbf{p}})x_bx_c}{x_a}+\frac{U_1(\lambda_{\mathbf{p}})x_b^2x_d}{x_ax_\rho}y_a+U_0(\lambda_{\mathbf{p}})x_cy_\rho+\frac{U_0(\lambda_{\mathbf{p}})x_bx_d}{x_\rho}y_ay_\rho+\frac{U_2(\lambda_{\mathbf{p}})x_bx_d}{x_\rho}y_ay_\rho+\frac{U_1(\lambda_{\mathbf{p}})x_ax_d}{x_\rho}y_ay_\rho^2.\]
\end{example}

If $\gamma$ is a closed curve, we replace $\gamma$ with $\gamma'$ as described in Section \ref{sec:Setting}. We then construct the snake graph $\widetilde{\mathcal{G}}_{\gamma',T}$. In this configuration, the first and last arcs intersected by $\gamma'$, denoted $\tau_{i_1}$ and $\tau_{i_d}$, necessarily border a common triangle, implying that their corresponding tiles share a common edge. We then obtain the \emph{band graph with hexagonal tiles} $\mathcal{G}_{\gamma,T}$ for the closed curve by gluing these terminal tiles together along this shared edge.

\subsection{Loop Graphs: Punctured Orbifolds and Auxiliary Tiles}\label{subsec:aux}

In this section, we will give a new type of snake graph for arcs on an orbifold that only consists of square tiles. This is a completely new construction to accommodate loop arcs. Some of the tiles in these graphs will not correspond to any intersection of our arc with an arc in the triangulation. These tiles will be called \emph{auxiliary tiles}, and the overall graph will be called a \emph{graph with auxiliary tiles}. Since snake or loop graphs with auxiliary tiles will be our main object of study, we may at times simply call them snake or loop graphs.

We begin by recalling how to construct a loop graph from a tagged arc, closely following \cite{wilson2020surface} with necessary adaptations for our setting.
For each $\tau_{i_j}$ which is both a standard arc and not in a pending arc, we form a tile $G_j$ as before. Similarly, if $\gamma$ crosses a pending arc as in the second row of Table \ref{table:PuzzlePiecesWinding}, we use the same tile.

Now, suppose $\gamma$ crosses a pending arc with winding number $0 < k < \mathbf{p}-2$ so that $\GGen$ has a hexagonal tile with nonzero edge weights. We replace the hexagonal tile with three square tiles as in Table \ref{table:PuzzlePiecesWindingAuxiliary}. The puzzle piece on the left corresponds to an orientation pattern $+,-,+$, whereas the one on the right corresponds to the orientation pattern $-,+,-$. The middle tile in each case is an auxiliary tile.

\begin{remark}\label{rmk:WhatCCWGivesUs}
The assignment of weights $U_{k \pm 1}(\lambda_{\mathbf{p}})\rho$ to specific edges is a consequence of our convention that arcs wind counterclockwise around orbifold points. 
\end{remark}

\begin{table}[H]
\centering
\begin{tabular}{|c|cc|}
\hline
\raisebox{.25\height}{%
\begin{tikzpicture}[scale=2]
\tikzset{->-/.style={decoration={
  markings,
  mark=at position #1 with {\arrow{>}}},postaction={decorate}}}
\draw[out=15,in=-15,looseness=1.5] (0,0.3) to (0,1.65);
\draw[out=165,in=-165,looseness=1.5] (0,0.3) to (0,1.65);
\draw[in=0,out=45,looseness =1.25] (0,0.3) to (0,1.3);
\draw[in=180,out=135,looseness=1.25] (0,0.3) to (0,1.3);
\draw[orange,line width = 1.2pt,out=0,in=210] (-0.4,0.7) to (0,0.75);
\draw[orange,line width = 1.2pt,out=30,in=0,looseness=2,->] (0,0.75) to (0,1.2);
\draw[orange,line width = 1.2pt,out=180,in=150,looseness=2] (0,1.2) to (0,0.75);
\draw[orange,line width = 1.2pt,out=-30,in=180] (0,0.75) to (0.4,0.7);

\draw[fill=black] (0,0.3) circle [radius=2pt];
\draw[fill=black] (0,1.65) circle [radius=2pt];
\draw[thick] (0,1) node {$\mathbf{\times}$};

\node[scale=1] at (-0.56,1.4) {$\alpha$};
\node[scale=1] at (0.56,1.41) {$\beta$};
\node[scale=1] at (0,1.4) {$\rho$};

\end{tikzpicture}
}
 &
\raisebox{.55\height}{%
\begin{tikzpicture}[scale = 1.25]
\draw[thick] (0,0) to node[below]{$\alpha$} (1,0) to node[below]{$\rho$} (2,0) to node[below]{$U_k \rho$} (3,0) to node[right]{$\beta$} (3,1) to node[above]{$\alpha$} (2,1) to node[above]{$\rho$} (1,1) to node[above]{$U_k \rho$} (0,1) to node[left]{$\beta$} (0,0);
\draw[thick] (1,0) to node[right, scale = 0.6, yshift = -15pt]{$U_{k-1}\rho$} (1,1);
\draw[thick] (2,0) to node[right, scale = 0.6,yshift = -15pt]{$U_{k+1}\rho$} (2,1);
\draw[gray,dashed] (0,1) to node[right]{$\rho$} (1,0);
\draw[gray,dashed] (1,1) to node[right,  scale = 0.6]{$U_k \rho$} (2,0);
\draw[gray,dashed] (2,1) to node[right]{$\rho$} (3,0);
\end{tikzpicture}
}
&
\begin{tikzpicture}[scale = 1.25]
\draw[thick] (0,0) to node[below]{$\beta$} (1,0) to node[right]{$U_k \rho$} (1,1) to node[right]{$\rho$} (1,2) to node[right]{$\alpha$} (1,3) to node[above]{$\beta$} (0,3) to node[left]{$U_k \rho$} (0,2) to node[left]{$\rho$} (0,1) to node[left]{$\alpha$} (0,0);
\draw[thick] (0,1) to node[above, scale = 0.6]{$U_{k-1}\rho$} (1,1);
\draw[thick] (1,2) to node[above, scale = 0.6]{$U_{k+1}\rho$} (0,2);
\draw[gray,dashed] (0,1) to node[above]{$\rho$} (1,0);
\draw[gray,dashed] (1,1) to node[above,  scale = 0.6]{$U_k \rho$} (0,2);
\draw[gray,dashed] (1,2) to node[above]{$\rho$} (0,3);
\end{tikzpicture}
\\ \hline
\end{tabular}
\caption{A combination of three tiles we associate to two consecutive intersections of an arc with a pending arc in the triangulation where the arc winds nontrivially around the orbifold point. The two puzzle pieces correspond to having positive and negative orientation on the southwest-most tile respectively. The winding number of the arc is $k$.}\label{table:PuzzlePiecesWindingAuxiliary}
\end{table}

\begin{remark}\label{rem:k=0AuxTile}
    We only introduce auxiliary tiles in the presence of  nontrivial winding. However, one could uniformly use hexagonal tiles even in this case and replace these with three square tiles. When $k = 0$, the puzzle piece in Table \ref{table:PuzzlePiecesWindingAuxiliary} has three nonzero perfect matchings. Their weights are 
    \[
    x_\alpha^2 x_\rho^2, \quad U_1 x_\alpha x_\rho^2 x_\beta y_\rho, \quad \text{and} \quad U_1 x_\beta^2 x_\rho^2 y_\rho^2.
    \]
    The weights of the two other perfect matchings contain $U_{-1}(\lambda_{\mathbf{p}})\rho$, so they vanish since $U_{-1}(x) = 0$. After dividing by the crossings term $x_\rho^3$, the resulting expression aligns precisely with the weights of the three perfect matchings of the associated hexagonal tile from Table \ref{table:PuzzlePiecesWinding}. A similar statement holds when $k = \mathbf{p}-2$.
\end{remark}

Next, we address the puzzle pieces we use for intersections between $\gamma$ and a self-folded triangle. One way this can occur is as below, in which case we associate the following three tile puzzle pieces, where the left corresponds to orientations $+,-,+$ and the right corresponds to orientations $-,+,-$.  If $\gamma$ winds nontrivially around the puncture, then the corresponding puzzle piece is a larger zig-zag, with one tile labelled $r$ for each intersection. It is also possible that $\gamma$ begins at the puncture enclosed in a self-folded, in which case we use the $\ell$ tile as below. More specific figures are given in \cite[Figure 3]{musiker2013}.

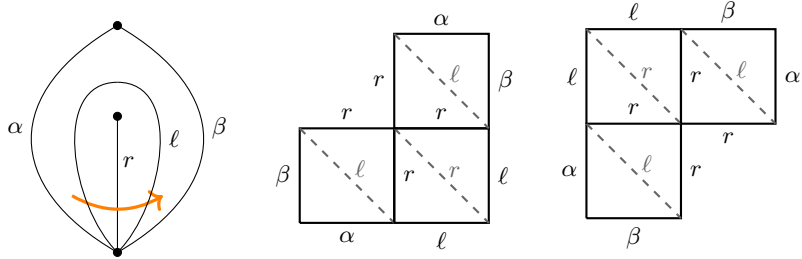
\begin{center}
\begin{tabular}{ccc}
\begin{tikzpicture}[scale=1.5]
\tikzset{->-/.style={decoration={
  markings,
  mark=at position #1 with {\arrow{>}}},postaction={decorate}}}

\draw (0,0.2) to (0,1.4);

\draw[out=30,in=-30,looseness=1.5] (0,0.2) to (0,2.2);
\draw[out=150,in=-150,looseness=1.5] (0,0.2) to (0,2.2);

\draw[line width = 1.2pt, out=-30, in=-150, looseness=1, ->, orange] (-0.4, 0.7) to (0.4, 0.7);

\draw[in=0,out=45,looseness =1] (0,0.2) to (0,1.7);
\draw[in=180,out=135,looseness=1] (0,0.2) to (0,1.7);

\draw[fill=black] (0,0.2) circle [radius=1pt];
\draw[fill=black] (0,2.2) circle [radius=1pt];
\draw[fill=black] (0,1.4) circle [radius=1pt];

\node at (-0.9,1.3) {$\alpha$};
\node at (0.9,1.3) {$\beta$};
\node at (0.5,1.2) {$\ell$};
\node at (0.1,1) {$r$};
\end{tikzpicture}&
\begin{tikzpicture}[scale = 1.25]
\draw[thick] (2,0) to node[below,midway]{$\alpha$} (3,0) to node[right,midway]{$r$} (3,1) to node[midway,above]{$r$} (2,1) to node[left,midway]{$\beta$} (2,0);
\draw[thick, dashed, black!60!white] (2,1) to node[above,midway,xshift=5,yshift=-4]{$\ell$} (3,0);
\draw[thick] (3,0) to node[below,midway]{$\ell$} (4,0) to node[right,midway]{$\ell$} (4,1) to node[midway,above]{$r$} (3,1) to (3,0);
\draw[thick, dashed, black!60!white] (3,1) to node[above,midway,xshift=5,yshift=-4]{$r$} (4,0);
\draw[thick] (3,1) to (4,1) to node[right,midway]{$\beta$} (4,2) to node[midway,above]{$\alpha$} (3,2) to node[midway,left]{$r$} (3,1);
\draw[thick, dashed, black!60!white] (3,2) to node[above,midway,xshift=5,yshift=-4]{$\ell$} (4,1);
\end{tikzpicture}&
\begin{tikzpicture}[scale = 1.25]
\draw[thick] (6,0) to node[below,midway]{$\beta$} (7,0) to node[right,midway]{$r$} (7,1) to node[midway,above]{$r$} (6,1) to node[left,midway]{$\alpha$} (6,0);
\draw[thick, dashed, black!60!white] (6,1) to node[above,midway,xshift=5,yshift=-4]{$\ell$} (7,0);
\draw[thick] (6,1) to node[left,midway]{$\ell$} (6,2) to node[above,midway]{$\ell$} (7,2) to node[midway,right]{$r$} (7,1);
\draw[thick, dashed, black!60!white] (7,1) to node[above,midway,xshift=5,yshift=-4]{$r$} (6,2);
\draw[thick] (7,1) to node[midway, below]{$r$} (8,1) to node[right,midway]{$\alpha$} (8,2) to node[midway,above,yshift=-1]{$\beta$} (7,2) to (7,1);
\draw[thick, dashed, black!60!white] (7,2) to node[above,midway,xshift=5,yshift=-4]{$\ell$} (8,1);
    \end{tikzpicture}
\end{tabular}
\end{center}

At this point, we have discussed all possible intersections of a plain arc $\gamma$ with $T^\circ$. As in the previous section, to construct $\mathcal{G}_{\gamma,T}$, we glue the tiles and puzzle pieces together in order given by an orientation of $\gamma$. We also preserve our convention of requiring that the first tile has a positive orientation.

Now, consider a singly-notched arc $\gamma^{(p)}$ where $p = s(\gamma)$, as in Section \ref{sec:Setting}. Recall that the spokes incident to $p$ are labelled $\sigma_1,\ldots,\sigma_m$ and $\widetilde{\gamma}$ is obtained by prepending nearly a full revolution around $p$ in the clockwise or counterclockwise direction to $\gamma$. In this section, we choose the clockwise direction as a convention without loss of generality, as using a counterclockwise hook results in an isomorphic loop graph. Even though $\widetilde{\gamma}$ is technically not an arc since it has an endpoint which does not lie in $M$, we can still construct its loop graph $\mathcal{G}_{\widetilde{\gamma},T}$ as described in this section. Notice that even if a spoke incident to $p$ is a pending arc, $\widetilde{\gamma}$ will not wind nontrivially around the enclosed orbifold point, so we use puzzle pieces as in the middle of Table \ref{table:PuzzlePiecesWinding}.

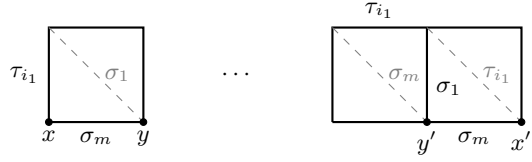
\begin{center}
\begin{tikzpicture}[scale = 1.25]
\draw[thick] (0,0) to node[below,midway]{$\sigma_m$} (1,0) to (1,1) to (0,1) to node[left,midway]{$\tau_{i_1}$} (0,0); 
\draw[gray,dashed] (0,1) to node[right]{$\sigma_1$} (1,0);
\node[] at (2,0.5){$\cdots$};
\draw[thick] (3,0) to (4,0) to node[below]{$\sigma_m$} (5,0) to (5,1) to (4,1) to node[above]{$\tau_{i_1}$} (3,1) to (3,0);
\draw[thick] (4,0) to node[right,yshift = -5pt]{$\sigma_1$} (4,1);
\draw[gray,dashed] (3,1) to node[right]{$\sigma_m$} (4,0);
\draw[gray,dashed] (4,1) to node[right]{$\tau_{i_1}$} (5,0);
\draw[fill=black] (0,0) circle [radius=1pt];
\node[below] at (0,0){$x$};
\draw[fill=black] (1,0) circle [radius=1pt];
\node[below] at (1,0){$y$};
\draw[fill=black] (4,0) circle [radius=1pt];
\node[below] at (4,0){$y'$};
\draw[fill=black] (5,0) circle [radius=1pt];
\node[below] at (5,0){$x'$};
\end{tikzpicture}
\end{center}

From the construction of tiles, the tiles corresponding to $\tau_{i_1}$ and to $\sigma_1$ each contain an edge weighted by $\sigma_m$. We obtain the loop graph $\mathcal{G}_{\gamma,T}$ by gluing these two edges in $\mathcal{G}_{\widetilde{\gamma},T}$ so that $x$ is identified with $x'$ and $y$ with $y'$ as described above. If $\gamma$ is notched at both endpoints, we simply perform this procedure at each end. We remark that using a counterclockwise hook in place of a clockwise one produces the same loop graph.

Now consider the case where $\gamma$ is a notched arc such that its underlying arc $\gamma^0$ belongs to the triangulation $T$. 
For a singly-notched arc $\gamma^{(p)}$, let $\ell$ be the arc such that $\gamma$ and $\ell$ form a self-folded triangle. 
We define $\mathcal{G}_{\gamma^{(p)}}$ to be the result of gluing the edges labelled $x_{\gamma^0}$ on the extreme tiles of $\mathcal{G}_{\ell}$.  

For a doubly-notched arc $\gamma^{(p,q)}$ satisfying $\tau^0 \in T$, let the spokes in $T$ incident to $p$ be $\sigma_1, \dots, \sigma_m$, indexed as below and let the spokes incident to $q$ similarly be  $\eta_1, \dots, \eta_h$. We construct a loop graph by labelling the tiles $\sigma_1, \dots, \sigma_m, \gamma^0, \eta_h, \dots, \eta_1$. To complete the construction, we append a second copy of the tile for $\gamma^0$ after $\eta_1$. We then identify edge $\sigma_m$ with its counterpart in the $\eta_1$-tile, and edge $\eta_h$ with its counterpart in the $\sigma_1$-tile.
This resulting structure constitutes the loop graph for $\gamma^{(p,q)}$. Note that this construction for a doubly-notched arc where $\gamma^0 \in T$ remains valid in the case where $p=q$, such as a doubly-notched pending arc. 

\begin{remark}
    At the time of writing the article, we were aware of a construction in the surface case of a loop graph for a doubly-notched arc $\gamma$ where $\gamma^0 \in T$  from talk slides by Jon Wilson. However, we were unable to find a formal construction in the available literature. This construction will appear in the near future in \cite{wilson2020surface}. Selisko also recently conjectured an alternate approach to such arcs which combines Wilson's construction with  \emph{$\gamma$-compatible matchings} \cite{Selisko}.
\end{remark}

\begin{example}
The following diagram shows an example of a loop graph corresponding to a doubly-notched arc $\gamma$ whose underlying plain arc $\gamma^{0}$ is in the triangulation $T$. A hooked arc $\widetilde{\gamma}$ associated to $\gamma$ is shown on the left. In the loop graph, there are three pairs of edges, all labelled $\gamma^0$, that need to be identified: the blue zigzag and dashed edges, as well as the red wiggly edges. These coloured markings indicate the specific pairs of edges that are identified to complete the construction of the graph.

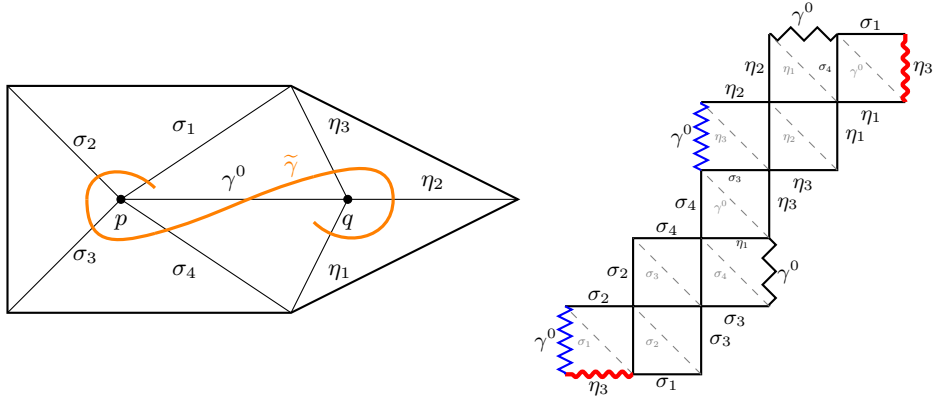
\begin{center}
\begin{tikzpicture}[scale=1.5, baseline=(current bounding box.center)]
    \draw[thick] (-2, 1) -- (0.5, 1) -- (2.5, 0) -- (0.5, -1) -- (-2, -1) -- cycle;

    \draw[fill=black] (-1, 0) circle [radius=1pt];
    \node at (-1, -0.2) {\small $p$};
    
    \draw (-1,0) -- (0.5,1)     node[pos=0.5, above left]{\small $\sigma_1$};
    \draw (-1,0) -- (-2,1)    node[pos=0.5, right]{\small $\sigma_2$};
    \draw (-1,0) -- (-2,-1)   node[pos=0.5, right]{\small $\sigma_3$};
    \draw (-1,0) -- (0.5,-1)  node[pos=0.5, below left]{\small $\sigma_4$};

    \draw[fill=black] (1, 0) circle [radius=1pt];
    \node at (1, -0.2) {\small $q$};
    
    \draw (1,0) -- (0.5,1)    node[pos=0.5, above right]{\small $\eta_3$};
    \draw (1,0) -- (2.5,0)    node[pos=0.5, above]{\small $\eta_2$};
    \draw (1,0) -- (0.5,-1)   node[pos=0.5, below right]{\small $\eta_1$};
    
    \draw (-1,0) -- (1,0) node[pos=0.5, above]{\small $\gamma^0$};

    \draw[orange,line width = 1.2pt,out=135,in=90,looseness=1.25] (-0.7,0.1) to (-1.3,0);
    \draw[orange,line width = 1.2pt, out=-90,in=-165,looseness=1.5] (-1.3,0) to (-0.7,-0.3);
    \draw[orange,line width = 1.2pt, out=15,in=-165] (-0.7,-0.3) to (0.7,0.2);
    \draw[orange, line width = 1.2pt, out=15,in=90,looseness=1.25] (0.7,0.2) to (1.4,0);
    \draw[orange, line width = 1.2pt, out=-90,in=-45,looseness=1.25] (1.4,0) to (0.7,-0.2);

    \node[orange] at (0.5, 0.3) {\small $\widetilde{\gamma}$};
\end{tikzpicture}
\begin{tikzpicture}[scale=0.9, transform shape, baseline=(current bounding box.center)] 
    \draw[thick] (5,6) to node[above, yshift=-2]{$\sigma_1$} (4,6);
    \draw[thick, decorate, decoration={zigzag}] (4,6) to node[above]{$\gamma^0$} (3,6);
    \draw[thick] (3,6) to node[left, xshift=2]{$\eta_2$} (3,5);
    \draw[draw=none] (5,6) to node[right]{$\eta_3$} (5,5);
    \draw[thick] (5,5) to node[below]{$\eta_1$} (4,5);
    \draw[gray,dashed] (5,5) to node[left, scale=0.6, xshift=-2, yshift=-2]{$\gamma^0$} (4,6);
    \draw[thick] (4,6) to node[left, scale=0.6]{$\sigma_4$} (4,5);
    \draw[gray,dashed] (4,5) to node[left, scale=0.6, xshift=-2, yshift=-2]{$\eta_1$} (3,6);
    \draw[thick] (4,5) to (3,5) to node[above, yshift=-2]{$\eta_2$} (2,5);
    \draw[draw=none] (2,5) to node[left]{$\gamma^0$} (2,4);
    \draw[thick] (4,5) to node[right]{$\eta_1$} (4,4) to node[below]{$\eta_3$} (3,4);
    \draw[gray,dashed] (4,4) to node[left, scale=0.6, xshift=-2, yshift=-2]{$\eta_2$} (3,5);
    \draw[thick] (3,5) to (3,4);
    \draw[gray,dashed] (3,4) to node[left, scale=0.6, xshift=-2, yshift=-2]{$\eta_3$} (2,5);
    \draw[thick] (3,4) to node[below,scale=0.6]{$\sigma_3$} (2,4) to node[left, xshift=2]{$\sigma_4$} (2,3) to node[below, yshift=1, scale=0.6, xshift=5]{$\eta_1$} (3,3) to node[right]{$\eta_3$} (3,4);     
    \draw[gray,dashed] (3,3) to node[left, scale=0.6, xshift=-2, yshift=-2]{$\gamma^0$} (2,4);
    
    \draw[thick, decorate, decoration={zigzag}] (3,3) to node[right]{$\gamma^0$} (3,2);
    \draw[thick] (3,2) to node[below]{$\sigma_3$} (2,2) to (1,2) to node[left, xshift=2]{$\sigma_2$} (1,3) to node[above, yshift=-2]{$\sigma_4$} (2,3);
    \draw[thick] (2,3) to (2,2);
    \draw[gray,dashed] (3,2) to node[left, scale=0.6, xshift=-2, yshift=-2]{$\sigma_4$} (2,3);
    \draw[gray,dashed] (2,2) to node[left, scale=0.6, xshift=-2, yshift=-2]{$\sigma_3$} (1,3);
    
    \draw[thick] (2,2) to node[right]{$\sigma_3$} (2,1) 
                to node[below, yshift=2]{$\sigma_1$} (1,1);
    \draw[draw=none] (1,1) to node[below]{$\eta_3$} (0,1) to node[left]{$\gamma^0$} (0,2);
    \draw[thick] (0,2) to node[above, yshift=-2]{$\sigma_2$} (1,2);
    \draw[thick] (1,2) to (1,1);
    \draw[gray,dashed] (2,1) to node[left, scale=0.6, xshift=-2, yshift=-2]{$\sigma_2$} (1,2);
    \draw[gray,dashed] (1,1) to node[left, scale=0.6, xshift=-2, yshift=-2]{$\sigma_1$} (0,2);
    
    \draw[ultra thick, red, decorate, decoration={snake, amplitude=1pt, segment length=5pt}] (1,1) to (0,1);
    \draw[ultra thick, red, decorate, decoration={snake, amplitude=1pt, segment length=5pt}] (5,5) to (5,6);
    \draw[thick, blue, decorate, decoration={zigzag, segment length=5pt}] (0,2) to (0,1);
    \draw[thick, blue, decorate, decoration={zigzag, segment length=5pt}] (2,4) to (2,5);
\end{tikzpicture}
\end{center}
\end{example}

Now assume that $\gamma$ is a corner arc for which the first arc it crosses, $\tau_{i_1}$, is a pending arc $\rho$, and that $\Delta_0$ is a pending triangle, as in the bottom row of~\Cref{table:PuzzlePiecesWinding}. Since representatives are always chosen to minimize intersections, this implies that the winding number of $\gamma$ around the orbifold point enclosed by $\tau_{i_1}$, $k$, is necessarily positive.  Suppose further that $\gamma$ is notched at $s(\gamma)$.  In this case, we will construct a loop which contains two auxiliary tiles, one with label $U_k(\lambda_{\mathbf{p}}) \rho$ and one with label $U_{k-1}(\lambda_{\mathbf{p}}) \rho$. The general rule for such arcs can be gleaned from the following example. One can verify that these loop graphs with auxiliary tiles cannot be represented using hexagonal tiles. As we will demonstrate, the notched corner arc case generally exhibits unique behaviour requiring special treatment.

\begin{example}\label{ex:NotchCorner}

    The following diagram shows an example of a notched corner arc with winding number $k=2$ and its corresponding loop graph. The zigzag blue edge decoration indicates that those edges are identified. 
    
    \begin{center}
    \raisebox{0.015\textheight}{
    \begin{tikzpicture}[scale=0.65, transform shape]
    \draw[thick] (0,0) circle (3cm);

    \draw[thick] (-3,0) .. controls (3,4) and (2,-2.5) .. (0,-3);
    \draw[thick] (0,-3) .. controls (-2.5,2) and (2.5,2) .. (0,-3);
    \draw[out=0, in=90, thick, looseness=1] (0,3) to (4,0);
    \draw[out=-90, in=0, thick, looseness=1] (4,0) to (0,-4);
    \draw[out=180, in=-90, thick, looseness=1] (0,-4) to (-3,0);

    \draw[in=0,out=-40,looseness=1, orange, line width = 1.2pt] (0,3) to (0,-.8);
    \draw[in=180,out=180,looseness=1.5, orange, line width = 1.2pt] (0,-.8) to (0,.5);
    \draw[in=0,out=0,looseness=1.5, orange, line width = 1.2pt] (0,-.5) to (0,.5);
    \draw[in=180,out=180,looseness=1.5, orange, line width = 1.2pt] (0,-.5) to (0,.3);
    \draw[in=80,out=0,looseness=0.6, orange, line width = 1.2pt] (0,.3) to (0,-3);
    \node[scale=1.5, rotate=-10, orange] at (0.17,-2.1) {$\bowtie$};

    \filldraw (0,-3) circle (3pt); 
    \filldraw (0,3) circle (3pt);  
    \filldraw (-3,0) circle (3pt); 

    \node[scale=2, thick] at (0,0) {$\times$};
    \node[scale=1.5] at (-1.4,1.2) {$a$};
    \node[scale=1.5] at (0.9,-0.9) {$\rho$};
    \node[scale=1.5] at (-1.8,-2) {$b$};
    \node[scale=1.5] at (-2.3,2.6) {$c$};
    \node[scale=1.5] at (2.7,0) {$d$};
    \node[scale=1.5] at (0,-3.7) {$e$};

    \end{tikzpicture}}\qquad
    \begin{tikzpicture}
    \draw[thick] (6,4) to node[above,yshift=-2]{$c$} (5,4) to node[above,yshift=-2]{$a$} (4,4) to node[above,yshift=-2]{$\rho$} (3,4) to node[above, scale = 0.9,yshift=-2]{$U_2\rho$} (2,4) to node[left,xshift=2]{$b$} (2,3) to node[below,yshift=1,scale=0.6,xshift=5]{$a$} (3,3) to node[below,yshift=2]{$\rho$}(4,3);
    \draw[draw=none] (4,3) to node[below, scale=0.9,yshift=-3]{$U_2\rho$} (5,3);
    \draw[thick] (5,3) to node[below,yshift=2]{$\rho$} (6,3) to node[right]{$d$} (6,4);
    
    \draw[thick] (5,4) to node[left, scale=0.6,yshift=10]{$b$} (5,3);
    \draw[thick] (4,4) to node[left, scale=0.6, yshift=10]{$U_{1}\rho$} (4,3); 
    \draw[thick] (3,4) to node[left, scale=0.6, yshift=10]{$U_{3}\rho$} (3,3); 
    
    \draw[gray,dashed] (6,3) to node[left,scale=0.6, xshift=-2,yshift=-2]{$a$} (5,4);
    \draw[gray,dashed] (5,3) to node[left,scale=0.6, xshift=-2,yshift=-2]{$\rho$} (4,4);
    \draw[gray,dashed] (4,3) to node[left,scale=0.6, xshift=-2,yshift=-2]{$U_2\rho$} (3,4);
    \draw[gray,dashed] (3,3) to node[left,scale=0.6, xshift=-2,yshift=-2]{$\rho$} (2,4);
    
    \draw[thick] (3,3) to node[right,xshift=-2]{$\rho$} (3,2) to node[below]{$d$} (2,2) to node[below, scale=0.6, yshift=2pt,xshift=5]{$c$} (1,2) to node[left,xshift=2]{$a$} (1,3) to node[above,yshift=-2]{$b$} (2,3);
    \draw[thick] (2,3) to node[left, scale=0.6,yshift=10]{$e$} (2,2);
    \draw[gray,dashed] (3,2) to node[left,scale=0.6, xshift=-2,yshift=-2]{$b$} (2,3);
    \draw[gray,dashed] (2,2) to node[left,scale=0.6, xshift=-2,yshift=-2]{$d$} (1,3);
    
    \draw[thick] (2,2) to node[right]{$d$} (2,1) to node[below,yshift=2]{$\rho$} (1,1) to node[below, scale=0.6, xshift=5]{$U_0\rho$} (0,1) to node[left]{$U_1 \rho$} (0,2) to node[above,yshift=-2]{$a$} (1,2);
    \draw[thick] (1,2) to node[left, scale=0.6,yshift=10]{$b$} (1,1);
    \draw[gray,dashed] (2,1) to node[left,scale=0.6, xshift=-2,yshift=-2]{$a$} (1,2);
    \draw[gray,dashed] (1,1) to node[left,scale=0.6, xshift=-2,yshift=-2]{$\rho$} (0,2);
    
    \draw[thick] (1,1) to node[right,xshift=-2]{$\rho$} (1,0);
    \draw[draw=none] (1,0) to node[below, scale=0.9,yshift=-3]{$U_2\rho$} (0,0);
    \draw[thick] (0,0) to node[left]{$\rho$} (0,1);
    \draw[gray,dashed] (1,0) to node[left, below, scale=0.6, xshift=-3,yshift=-2]{$U_1\rho$} (0,1);
    
    \draw[ultra thick, blue,decorate, decoration={zigzag, segment length=6pt}] (1,0) to (0,0);
    \draw[ultra thick, blue,decorate, decoration={zigzag, segment length=6pt}] (5,3) to (4,3);
    \end{tikzpicture}
    \end{center}
\end{example}

\begin{remark}
Example \ref{ex:NotchCorner} concerns an arc with winding number 2, and it includes an edge with weight $U_0(\lambda_{\mathbf{p}})  \rho$. Therefore, if we considered an arc with winding number 1 instead, then our graph would have an edge with weight $U_{-1}(\lambda_{\mathbf{p}})  \rho = 0$. One can redraw such a graph to not include the tile labelled $U_0(\lambda_{\mathbf{p}})  \rho$ and to not include the edge with weight 0. The same sort of statement is true if the winding number of $\mathbf{p}-2$; recall $U_{\mathbf{p}-1}(\lambda_{\mathbf{p}}) = 0$.
\end{remark}

For a closed curve $\gamma$ in an orbifold, the construction of its associated graph follows the procedure described in the previous subsection. That is, we pass to a  curve $\gamma'$ with $s(\gamma') = t(\gamma')$ and define $\mathcal{G}_{\gamma,T}$ to be the result of gluing two edges of $\mathcal{G}_{\gamma'}^T$. The resulting graphs are referred to as \emph{band graphs (with auxiliary tiles)}.

Loop and band graphs  are obtained by identifying multiple pairs of edges in a snake graph. We now define the specific set of dimer covers to be considered for such graphs.

\begin{definition}\label{def:GoodMatching}
Let $\mathcal{G} = (V,E)$ be a loop graph or a band graph. Let $\widetilde{\mathcal{G}}$ be the graph obtained by ungluing all edges glued in the construction of $\mathcal{G}$. A \emph{good matching} of $\mathcal{G}$ is a dimer cover $M \subset E$ such that every $v \in V$ is incident to exactly one $e \in M$ and such that $M$ can be extended to a perfect matching of $\widetilde{\mathcal{G}}$.
\end{definition}

We use the same description of $M_-$ as given in Definition \ref{def:MinimalMatching1} to describe the minimal matching of any  snake, loop, and band graph constructed here. Note that the minimal matching of a loop or band graph $\mathcal{G}$ is exactly the good matching whose extension to a perfect matching of the unglued snake graph $\widetilde{\mathcal{G}}$ is the minimal matching. 

In the punctured setting, a loop graph $\mathcal{G}_{\gamma,T^\circ}$ consists of tiles and edges indexed by the ideal triangulation $T^\circ$, but our cluster variable must be written with respect to the associated tagged triangulation $T$. We will define specialization maps to this effect. 

First, for a good matching $M$, we define the \emph{height monomial} as $h(M) = \prod_{i=1}^n h_{\tau_i}^{m_i}$, where $m_i$ denotes the number of tiles labelled $\tau_i$ enclosed by the cycles in the symmetric difference $M \ominus M_-$. Note that tiles labelled $U_k \tau_i$ do not contribute to this monomial as they are not elements of the triangulation. Following \cite{musiker2011positivity}, we define a specialization map $\Phi$ on the weight and height monomials derived from the perfect (or good) matchings of $\mathcal{G}_{\gamma,T^\circ}$.

\begin{definition}\label{def:Phi}
Let $T^\circ$ be an ideal triangulation. Given a monomial $\mathbf{x} = \prod_{\tau \in T^\circ} x_\tau^{a_\tau}$, let $\Phi(\mathbf{x})$ be the monomial obtained by replacing every factor $x_\ell$ for a loop $\ell$ cutting out a once-punctured monogon with $x_r x_{r^{(p)}}$ where $r$ and $\ell$ form a self-folded triangle. 

Given a good matching $M$ of $\mathcal{G}_{\gamma,T}$, recall that $x(M) = \Phi\left(\prod_{e \in M} \mathrm{wt}(e)\right)$.

If $\gamma$ is not a doubly-notched arc satisfying $\gamma^0 \in T$, define $\mathrm{cross}(\mathcal{G}_{\gamma,T})$ to be the result of applying $\Phi$ to the product of (scalar multiples of) cluster variables associated labels of tiles in $\mathcal{G}_{\gamma,T}$. In particular, a tile $U_k(\lambda_{\mathbf{p}}) \rho$ is associated to $U_k(\lambda_{\mathbf{p}}) x_\rho$. If $\gamma$ is a doubly-notched arc satisfying $\gamma^0 \in T$, then we define $\mathrm{cross}(\mathcal{G}_{\gamma,T})$ to be the monomial obtained by dividing the monomial constructed with this process by $x_{\gamma^0}$.

Define $\Phi$ on height monomials to be the multiplicative function satisfying 
\begin{equation}
\Phi(h_\tau) = \begin{cases} 
y_\tau & \text{if } \tau \text{ is not a side of a self-folded triangle;} \\
\frac{y_r}{y_{r^{(p)}}} & \text{if } \tau \text{ is a radius } r \text{ to a puncture } p \text{ in a self-folded triangle;} \\
y_{r^{(p)}} & \text{if } \tau \text{ is a loop } \ell \text{ in a self-folded triangle with radius } r.
\end{cases} \label{eq:DefOfPhiSpecializationMapY}
\end{equation} 
For a given good matching $M$, we denote $y(M) := \Phi(h(M))$.

\end{definition}

\begin{definition}\label{def:LaurentExpansionFromSnakeGraph}
Let $\gamma$ be an arc or closed curve on an orbifold $\mathcal{O}$ with triangulation $T$, and let $\mathcal{G}_{\gamma,T}$ be the associated loop or band graph with auxiliary tiles. We define the combinatorial expansion:
\[
\chi(\mathcal{G}_{\gamma,T}) = \frac{1}{\mathrm{cross}(\mathcal{G}_{\gamma,T})} \sum_{M \in \mathrm{Good}(\mathcal{G}_{\gamma,T})} x(M) y(M)
\]
where the sum is taken over all good matchings $M$ of $\mathcal{G}_{\gamma,T}$.
\end{definition}

\begin{example}
The snake graph with auxiliary tiles corresponding to the arc and triangulation shown in the top left of~\Cref{table:whole} is displayed on the top right. Applying~\Cref{def:LaurentExpansionFromSnakeGraph} and simplifying the result using identities of Chebyshev polynomials yields the computation in~\Cref{ex:1poset}.

\end{example}

\begin{remark}\label{rmk:ExtendHexagonalConstruction}
In Section \ref{sec:HexagonalTiles}, we only discussed the construction of snake and band graphs with hexagonal tiles for arcs on unpunctured orbifolds. However, it is straightforward to combine the constructions discussed here for plain arcs which intersect self-folded triangles and notched arcs other than notched corner arcs with the hexagonal tile construction. Therefore, in future sections, for any arc which is not a notched corner arc, we will reference both a loop or band graph with hexagonal tiles and a loop or band graph (with auxiliary tiles). Recall our convention is to consider auxiliary tiles, hence we may drop the phrase ``with axuiliary tiles.'' This will be an important point for proving the consistency of the formulas in Section \ref{sec:Equivalence}.
\end{remark}

\subsection{Labelled Posets}\label{subsec:Posets}

For a composition $\alpha=(\alpha_1,\alpha_2,\ldots,\alpha_k)$ of $n$, the corresponding \emph{fence poset} is a poset on $n+1$ elements, whose Hasse diagram is given by up and down steps with $\alpha_i$ up steps, followed by $\alpha_{i+1}$ down steps, etc.

We now introduce notation for constructing a poset  $\calP^T_\gamma$ (abbreviated as $\mathcal{P}_{\gamma}$ or simply $\mathcal{P}$ when clear from context) from an  arc or closed curve $\gamma$ on $\mathcal{O}$. This construction generalizes the framework in~\cite{ezgieminecluster2024}. These posets will either be fence posets or the result of adding a few relations to a fence poset. 

Recall the intersection points between $\gamma$ and $T$ are $c_1, c_2,\ldots, c_d$ and the triangle between crossings $c_i$ and $c_{i+1}$ is $\Delta_i$, with $\Delta_0$ and $\Delta_d$ the initial and final triangles. 

First, let $\gamma$ be an arc or closed curve such that $\gamma^0\notin T$ and not a notched corner arc. The poset $\mathcal{P}_{\gamma}^T$  for a (possibly generalized) arc $\gamma$ (plain or notched) consists of elements $v_1,v_2,\ldots,v_d$ corresponding to the crossings and possible extra elements $v_{(\rho,i)}$ that correspond to each time $\gamma$ wraps around an orbifold point between points $c_i$ and $c_{i+1}$. The partial order $\succeq$ defined on these elements is the transitive closure of the relations described below. See Figure \ref{fig:closed_curve_fence_poset} for a visualization.

\begin{itemize}
    \item[(i)] Consider a pair of consecutive crossings $c_i$ and $c_{i+1}$. If $\gamma$ does not wind around an orbifold point between these two crossings, then portions of the crossed arcs and $\gamma$ form the boundary of a contractible disc. This disc has a marked point, $t$, which is a common endpoint of the two arcs crossed. We set $v_{i} \preceq v_{i+1}$ if $t$ lies to the left of $\gamma$ (with respect to its orientation) and otherwise set $v_{i} \succeq v_{i+1}$. If $\gamma$ winds 0 ($\mathbf{p}-2$)  times  around an orbifold point of order $\mathbf{p}$ between $c_i$ and $c_{i+1}$, then we set $v_{i} \succeq v_{i+1}$ ($v_{i} \preceq v_{i+1}$).
    
    If $\gamma$ winds $0 < k < \mathbf{p}-2$ times around an orbifold point encapsulated by a pending arc $\rho$, then instead of relating $v_{i}$ and $v_{i+1}$, we add an extra element $v_{(\rho, i)}$ and relations  $ v_i\preceq v_{(\rho,i)}$ and $v_{(\rho,i)}\succeq v_{i+1}$. We will at times abuse notation and write $v_{\rho_L},v_{\rho^*},$ and $v_{\rho_R}$  instead of $v_i, v_{(\rho,i)},v_{i+1}$ respectively. 
    
    \item[(ii)] If the arc $\gamma$ is a closed curve, we follow the steps in (i) for the arc $\gamma'$ and also perform the same procedure to introduce a relation between $v_{d}$ and $v_1$  by considering the portion of $\gamma$ between $c_d$ and $c_1$ in the triangle $\Delta_0 = \Delta_d$.
    
    \item[(iii)] If $\gamma$ is a notched arc, then orient $\gamma$ so that $s(\gamma)$ is notched and replace $\gamma$ with a hook $\widetilde{\gamma}$. First, we construct the poset $\calP_{\widetilde{\gamma}}^T$. Then,  we include a relation between the element associated to the first spoke crossed by $\widetilde{\gamma}$ and the first arc crossed by $\gamma^0$ using the same rules as in (i). If $\gamma$ is also notched at $t(\gamma)$, we do the same but with respect to the last spoke crossed by $\widetilde{\gamma}$ and the last arc crossed by $\gamma^0$.

\end{itemize}

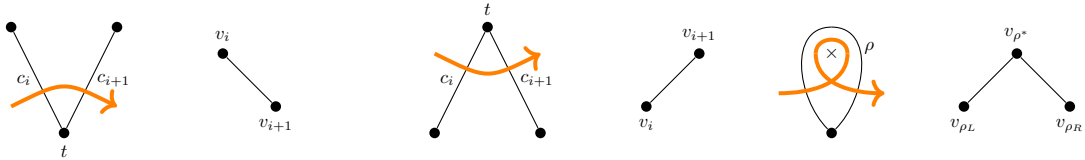
\begin{figure}[H]
\centering
\begin{tikzpicture}[scale=0.7, transform shape]
    \coordinate (P1) at (0, 3);
    \coordinate (P2) at (1, 1);
    \coordinate (P3) at (2, 3);
    \coordinate (P4) at (4, 2.5);
    \coordinate (P5) at (5, 1.5);

    \node[circle, fill=black, inner sep=2pt] (P1) at (P1) {};
    \node[circle, fill=black, inner sep=2pt, label=below:$t$] (P2) at (P2) {};
    \node[circle, fill=black, inner sep=2pt] (P3) at (P3) {};
    \node[circle, fill=black, inner sep=2pt, label=above:$v_i$] (P4) at (P4) {};
    \node[circle, fill=black, inner sep=2pt, label=below:$v_{i+1}$] (P5) at (P5) {};
    
    \draw[-] (P1) -- (P2) node[pos=0.5, left] {$c_i$};
    \draw[-] (P2) -- (P3) node[pos=0.5, right] {$c_{i+1}$};
    \draw[-] (P4) -- (P5);

    \coordinate (Q1) at (0, 1.5);
    \coordinate (Q2) at (1, 2);
    \coordinate (Q3) at (2, 1.5);

    \draw[smooth, ->, line width=1.5pt, color=orange] (Q1) .. controls (Q2) .. (Q3);

    \coordinate (P6) at (8, 1);
    \coordinate (P7) at (9, 3);
    \coordinate (P8) at (10, 1);
    \coordinate (P9) at (12, 1.5);
    \coordinate (P10) at (13, 2.5);

    \node[circle, fill=black, inner sep=2pt] (P6) at (P6) {};
    \node[circle, fill=black, inner sep=2pt, label=above:$t$] (P7) at (P7) {};
    \node[circle, fill=black, inner sep=2pt] (P8) at (P8) {};
    \node[circle, fill=black, inner sep=2pt, label=below:$v_i$] (P9) at (P9) {};
    \node[circle, fill=black, inner sep=2pt, label=above:$v_{i+1}$] (P10) at (P10) {};
    \draw[-] (P6) -- (P7) node[pos=0.5, left] {$c_i$};
    \draw[-] (P7) -- (P8) node[pos=0.5, right] {$c_{i+1}$};
    \draw[-] (P9) -- (P10);

    \coordinate (Q4) at (8, 2.5);
    \coordinate (Q5) at (9, 2);
    \coordinate (Q6) at (10, 2.5);

    \draw[smooth, ->, line width=1.5pt, color=orange] (Q4) .. controls (Q5) .. (Q6);

    \coordinate (P12) at (15.5, 1);
    \coordinate (P12') at (15.5, 3);
    \coordinate (P13) at (15.5, 2.5);

    \node[circle, fill=black, inner sep=2pt] (P12) at (P12) {};

    \draw[in=0,out=45,looseness =1.1] (P12) to node[midway,right,yshift=10,xshift=-2]{$\rho$} (P12');
    \draw[in=180,out=135,looseness=1.1] (P12) to (P12');
    \node at (P13) {$\times$};

    \coordinate (P14) at (18, 1.5);
    \coordinate (P15) at (19, 2.5);
    \coordinate (P16) at (20, 1.5);

    \node[circle, fill=black, inner sep=2pt, label=below:$v_{\rho_{L}}$] (P14) at (P14) {};
    \node[circle, fill=black, inner sep=2pt, label=above:$v_{\rho^*}$] (P15) at (P15) {};
    \node[circle, fill=black, inner sep=2pt, label=below:$v_{\rho_{R}}$] (P16) at (P16) {};

    \draw (P14) -- (P15);
    \draw (P15) -- (P16);

    \draw[line width = 1.5pt, color=orange,out=0, in=-90] (14.5, 1.75) to (15.8,2.5);
    \draw[line width = 1.5pt, color=orange, out=90,in=90,looseness=1.5] (15.8, 2.5) to (15.2, 2.5);
    \draw[line width = 1.5pt, color=orange, out=-90,in=-180,->] (15.2,2.5) to (16.5,1.75);

\end{tikzpicture}
\caption{Building a fence poset for an arc.}
\label{fig:closed_curve_fence_poset}
\end{figure}

\begin{example} \label{ex:notchedposet} Below are two examples of posets associated to doubly-notched arcs with the same endpoints. The elements are labelled according to the rule which we just discussed. 
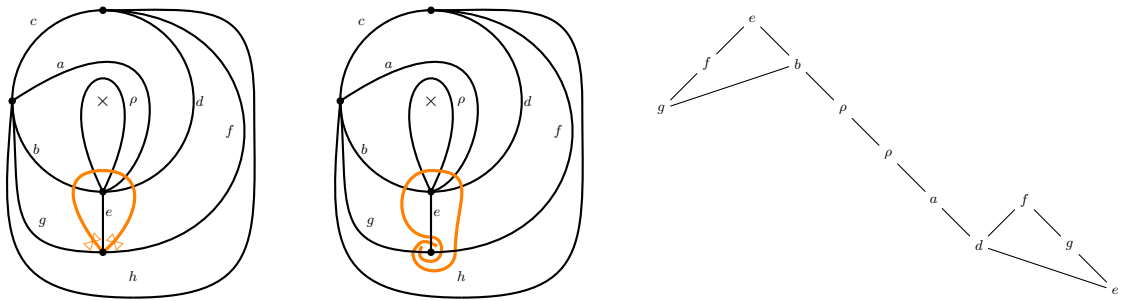
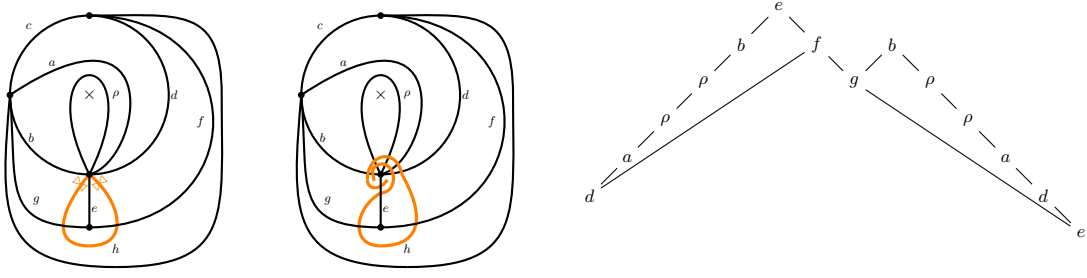
\begin{center}
\begin{tabular}{ccc} 

    \begin{tikzpicture}[scale=0.4, transform shape]
    \draw[thick] (0,0) circle (3cm);

    \draw[thick] (-3,0) .. controls (3,4) and (2,-2.5) .. (0,-3);
    \draw[thick] (0,-3) .. controls (-2.5,2) and (2.5,2) .. (0,-3);

    \draw[in=0,out=40,looseness=1.5, orange, line width = 1.2pt] (0,-5) to (0,-2.3);
    \draw[in=180,out=130,looseness=1.5, orange, line width = 1.2pt] (0,-5) to (0,-2.3);

    \draw[thick] (0,-3) to (0,-5);
    \draw[in=0,out=0,looseness=2, thick] (0,-5) to (0,3);
    \draw[in=-85,out=180,looseness=1.5, thick] (0,-5) to (-3,0);

    \draw[in=-95,out=180,looseness=1.5, thick] (1,-6.5) to (-3,0);
    \draw[in=-90,out=0,looseness=1.5, thick] (1,-6.5) to (5,0);
    \draw[in=90,out=0,looseness=1.5, thick] (0,3) to (5,0);

    \filldraw (0,-3) circle (3pt); 
    \filldraw (0,3) circle (3pt);  
    \filldraw (-3,0) circle (3pt); 
    \filldraw (0,-5) circle (3pt); 

    \node[scale=2, thick] at (0,0) {$\times$};
    \node[scale=1.5] at (-1.4,1.2) {$a$};
    \node[scale=1.5] at (1,0) {$\rho$};
    \node[scale=1.5] at (-2.2,-1.6) {$b$};
    \node[scale=1.5] at (-2.3,2.6) {$c$};
    \node[scale=1.5] at (3.2,0) {$d$};
    \node[scale=1.5] at (4.2,-1) {$f$};
    \node[scale=1.5] at (-2,-4) {$g$};
    \node[scale=1.5] at (0.2,-3.7) {$e$};
    \node[scale=1.5] at (1,-5.8) {$h$};

    \node[scale=2, rotate=135, orange] at (0.4,-4.65) {$\bowtie$};
    \node[scale=2, rotate=45, orange] at (-0.34,-4.65) {$\bowtie$};

    \end{tikzpicture}

     &  
    \begin{tikzpicture}[scale=0.4, transform shape]
    \draw[thick] (0,0) circle (3cm);

    \draw[thick] (-3,0) .. controls (3,4) and (2,-2.5) .. (0,-3);
    \draw[thick] (0,-3) .. controls (-2.5,2) and (2.5,2) .. (0,-3);

    \draw[in=0,out=90,looseness=1.5, orange, line width = 1.2pt] (0.8,-5) to (0,-2.3);
    \draw[in=-90,out=-90,looseness=1.5, orange, line width = 1.2pt] (0.8,-5) to (-0.6,-5);
    \draw[in=60,out=150,looseness=1.5, orange, line width = 1.2pt] (.2,-4.8) to (-0.6,-5);

    \draw[in=180,out=180,looseness=1.5, orange, line width = 1.2pt] (0,-4.5) to (0,-2.3);
    \draw[in=0,out=0,looseness=1.5, orange, line width = 1.2pt] (0,-4.5) to (0,-5.3);
    \draw[in=180,out=-120,looseness=1.5, orange, line width = 1.2pt] (-0.3,-4.8) to (0,-5.3);

    \draw[thick] (0,-3) to (0,-5);
    \draw[in=0,out=0,looseness=2, thick] (0,-5) to (0,3);
    \draw[in=-85,out=180,looseness=1.5, thick] (0,-5) to (-3,0);

    \draw[in=-95,out=180,looseness=1.5, thick] (1,-6.5) to (-3,0);
    \draw[in=-90,out=0,looseness=1.5, thick] (1,-6.5) to (5,0);
    \draw[in=90,out=0,looseness=1.5, thick] (0,3) to (5,0);

    \filldraw (0,-3) circle (3pt); 
    \filldraw (0,3) circle (3pt);  
    \filldraw (-3,0) circle (3pt); 
    \filldraw (0,-5) circle (3pt); 

    \node[scale=2, thick] at (0,0) {$\times$};
    \node[scale=1.5] at (-1.4,1.2) {$a$};
    \node[scale=1.5] at (1,0) {$\rho$};
    \node[scale=1.5] at (-2.2,-1.6) {$b$};
    \node[scale=1.5] at (-2.3,2.6) {$c$};
    \node[scale=1.5] at (3.2,0) {$d$};
    \node[scale=1.5] at (4.2,-1) {$f$};
    \node[scale=1.5] at (-2,-4) {$g$};
    \node[scale=1.5] at (0.2,-3.7) {$e$};
    \node[scale=1.5] at (1,-5.8) {$h$};

    \end{tikzpicture}
    
    &
    
     \begin{tikzpicture}[scale=0.6, transform shape]
        \node (G1) at (0,1){$g$};
        \node (F1) at (1,2){$f$};
        \node (H1) at (2,3) {$e$};
        \node (C) at (3,2) {$b$};
        \node (B1) at (4,1) {$\rho$};
        \node (B2) at (5,0) {$\rho$};
        \node (A) at (6,-1) {$a$};
        \node (E) at (7,-2) {$d$};
        \node (F2) at (8,-1) {$f$};
        \node (G2) at (9,-2) {$g$};
        \node (H2) at (10,-3) {$e$};

        \draw (G1) -- (F1);
        \draw (F1) -- (H1);
        \draw (H1) -- (C);
        \draw (G1) -- (C);
        \draw (C) -- (B1);
        \draw (B1) -- (B2);
        \draw (B2) -- (A);
        \draw (A) -- (E);
        \draw (E) -- (F2);
        \draw (F2) -- (G2);
        \draw (G2) -- (H2);
        \draw (E) -- (H2);
    \end{tikzpicture}

\end{tabular}
\end{center}

\begin{center}
\begin{tabular}{ccc}

 \begin{tikzpicture}[scale=0.35, transform shape]
    \draw[thick] (0,0) circle (3cm);

    \draw[thick] (-3,0) .. controls (3,4) and (2,-2.5) .. (0,-3);
    \draw[thick] (0,-3) .. controls (-2.5,2) and (2.5,2) .. (0,-3);

    \draw[in=0,out=-40,looseness=1.5, orange, line width = 1.2pt] (0,-3) to (0,-5.7);
    \draw[in=180,out=-130,looseness=1.5, orange, line width = 1.2pt] (0,-3) to (0,-5.7);

    \node[scale=2, rotate=45, orange] at (0.4,-3.35) {$\bowtie$};
    \node[scale=2, rotate=135, orange] at (-0.34,-3.35) {$\bowtie$};

    \draw[thick] (0,-3) to (0,-5);
    \draw[in=0,out=0,looseness=2, thick] (0,-5) to (0,3);
    \draw[in=-85,out=180,looseness=1.5, thick] (0,-5) to (-3,0);

    \draw[in=-95,out=180,looseness=1.5, thick] (1,-6.5) to (-3,0);
    \draw[in=-90,out=0,looseness=1.5, thick] (1,-6.5) to (5,0);
    \draw[in=90,out=0,looseness=1.5, thick] (0,3) to (5,0);

    \filldraw (0,-3) circle (3pt); 
    \filldraw (0,3) circle (3pt);  
    \filldraw (-3,0) circle (3pt); 
    \filldraw (0,-5) circle (3pt); 

    \node[scale=2, thick] at (0,0) {$\times$};
    \node[scale=1.5] at (-1.4,1.2) {$a$};
    \node[scale=1.5] at (1,0) {$\rho$};
    \node[scale=1.5] at (-2.2,-1.6) {$b$};
    \node[scale=1.5] at (-2.3,2.6) {$c$};
    \node[scale=1.5] at (3.2,0) {$d$};
    \node[scale=1.5] at (4.2,-1) {$f$};
    \node[scale=1.5] at (-2,-4) {$g$};
    \node[scale=1.5] at (0.2,-4.3) {$e$};
    \node[scale=1.5] at (1,-5.8) {$h$};

    \end{tikzpicture}

     &  
    \begin{tikzpicture}[scale=0.35, transform shape]
    \draw[thick] (0,0) circle (3cm);

    \draw[thick] (-3,0) .. controls (3,4) and (2,-2.5) .. (0,-3);
    \draw[thick] (0,-3) .. controls (-2.5,2) and (2.5,2) .. (0,-3);

    \draw[in=0,out=-60,looseness=1.5, orange, line width = 1.2pt] (0.8,-3) to (0,-5.7);
    \draw[in=90,out=100,looseness=2, orange, line width = 1.2pt] (0.8,-3) to (-0.5,-3.2);
    \draw[in=-90,out=-120,looseness=1.5, orange, line width = 1.2pt] (0.25,-3.2) to (-0.5,-3.2);

    \draw[in=180,out=-150,looseness=1.5, orange, line width = 1.2pt] (0,-3.7) to (0,-5.7);
    \draw[in=0,out=0,looseness=1.5, orange, line width = 1.2pt] (0,-3.7) to (0,-2.6);
    \draw[in=200,out=100,looseness=1.5, orange, line width = 1.2pt](-0.25,-3.25) to (0,-2.6);

    \draw[thick] (0,-3) to (0,-5);
    \draw[in=0,out=0,looseness=2, thick] (0,-5) to (0,3);
    \draw[in=-85,out=180,looseness=1.5, thick] (0,-5) to (-3,0);

    \draw[in=-95,out=180,looseness=1.5, thick] (1,-6.5) to (-3,0);
    \draw[in=-90,out=0,looseness=1.5, thick] (1,-6.5) to (5,0);
    \draw[in=90,out=0,looseness=1.5, thick] (0,3) to (5,0);

    \filldraw (0,-3) circle (3pt); 
    \filldraw (0,3) circle (3pt);  
    \filldraw (-3,0) circle (3pt); 
    \filldraw (0,-5) circle (3pt); 

    \node[scale=2, thick] at (0,0) {$\times$};
    \node[scale=1.5] at (-1.4,1.2) {$a$};
    \node[scale=1.5] at (1,0) {$\rho$};
    \node[scale=1.5] at (-2.2,-1.6) {$b$};
    \node[scale=1.5] at (-2.3,2.6) {$c$};
    \node[scale=1.5] at (3.2,0) {$d$};
    \node[scale=1.5] at (4.2,-1) {$f$};
    \node[scale=1.5] at (-2,-4) {$g$};
    \node[scale=1.5] at (0.2,-4.3) {$e$};
    \node[scale=1.5] at (1,-5.8) {$h$};

    \end{tikzpicture}
    &
    \raisebox{0.1\height}{%
    \begin{tikzpicture}[scale=0.5, transform shape]
        \node[scale=1.5] (E1) at (0,1){$d$};
        \node[scale=1.5] (A1) at (1,2){$a$};
        \node[scale=1.5] (B1) at (2,3) {$\rho$};
        \node[scale=1.5] (B2) at (3,4) {$\rho$};
        \node[scale=1.5] (C1) at (4,5) {$b$};
        \node[scale=1.5] (H1) at (5,6) {$e$};
        \node[scale=1.5] (F) at (6,5) {$f$};
        \node[scale=1.5] (G) at (7,4) {$g$};
        \node[scale=1.5] (C2) at (8,5) {$b$};
        \node[scale=1.5] (B3) at (9,4) {$\rho$};
        \node[scale=1.5] (B4) at (10,3) {$\rho$};
        \node[scale=1.5] (A2) at (11,2) {$a$};
        \node[scale=1.5] (E2) at (12,1) {$d$};
        \node[scale=1.5] (H2) at (13,0) {$e$};

        \draw (E1) -- (A1);
        \draw (A1) -- (B1);
        \draw (B1) -- (B2);
        \draw (B2) -- (C1);
        \draw (C1) -- (H1);
        \draw (H1) -- (F);
        \draw (E1) -- (F);
        \draw (F) -- (G);
        \draw (G) -- (C2);
        \draw (G) -- (H2);
        \draw (C2) -- (B3);
        \draw (B3) -- (B4);
        \draw (B4) -- (A2);
        \draw (A2) -- (E2);
        \draw (E2) -- (H2);

    \end{tikzpicture}
}
\end{tabular}
\end{center}
\end{example}

If $\gamma$ is a plain arc, then $\calP_\gamma^T$ is a fence poset. If $\gamma$ is a closed curve, we call the resulting poset  a \emph{circular fence} poset (as in~\cite{ouguz2023rank}), and if $\gamma$ is a notched arc, then we call the resulting poset a \emph{loop} fence poset (see~\cite{KLL25, ezgieminecluster2024}). 

Before addressing some special cases, we will discuss some extra data we associate with this construction. Our posets $\mathcal{P}_\gamma$ come equipped with two functions, a label function and a weight function. Most of the elements of the poset correspond to an intersection point $c_j$. We label such an element with the arc $\tau_{i_j}$. The other elements are those of the from $v_{(\rho,i)}$, which we label with the formal product $U_k(\lambda_\mathbf{p}) \cdot \rho$ where $k$ is the winding number of $\gamma$ between the two intersection points $c_i$ and $c_{i+1}$.

We next discuss the weights, which largely follow the conventions established in \cite{banaian2024skein, ezgieminecluster2024}.  For each arc $\tau \in T$, let $x_{\mathrm{CCW}}(\tau) = x_j x_k$, where $\tau_j, \tau_k \in T$ are the arcs immediately counterclockwise from $\tau$ within the two triangles sharing $\tau$. Boundary arcs are assigned a weight of $1$ and thus do not contribute to this product. The monomial $x_{\mathrm{CW}}(\tau)$ is defined analogously using the clockwise neighbours. For each $\tau \in T$, we define:
\[ 
\hat{y}_\tau :=  \Phi\bigg(\frac{x_{\mathrm{CCW}}(\tau)}{x_{\mathrm{CW}}(\tau)}h_\tau\bigg),
\]
using the specialization map $\Phi$ on height monomials from \Cref{def:Phi}.

If an element $v_i \in \mathcal{P}_\gamma$ is labelled by a standard arc $\tau_i$, its weight is $w_{i} = \hat{y}_{\tau_i}$. For example, the elements labelled $b$ in Example \ref{ex:notchedposet} have weight $\frac{x_\rho x_g}{x_ax_c}y_b$. It is important to note that this definition naturally extends to the case of self-folded triangles. For instance, the weight for a radius $r$ of a self-folded triangle is computed directly as \(\frac{y_r}{y_{r^{(p)}}}\).

Next, we define the weights for the poset elements $v_{(\rho,i)}$ and $v_i$ associated with a pending arc $\rho$. Note that when $v_i \in \mathcal{P}_\gamma$ is labelled by a pending arc $\rho$, we get $\hat{y}_\rho = \frac{x_\alpha}{x_\beta} y_\rho$, where $\alpha$ and $\beta$ are the arcs in $T$ that form the enclosing bigon for $\rho$, with $\alpha$ the counterclockwise neighbor of $\rho$. 

Suppose $\gamma$ has two consecutive intersections with $\rho$, denoted by $\tau_{i_j} = \tau_{i_{j+1}} = \rho$, with a winding number $0 < k < \mathbf{p}-2$. Recall we write $v_{\rho_L} = v_j$ and $v_{\rho_R} = v_{j+1}$ which are both covered by $v_{\rho^*} = v_{(\rho,j)}$. In this situation,  the weights are defined as:
\[ 
    w_{\rho_L} = \frac{U_{k-1}(\lambda_{\mathbf{p}})}{U_{k}(\lambda_{\mathbf{p}})}\hat{y}_\rho, \quad
    w_{\rho^*} = \frac{1}{U_{k-1}(\lambda_{\mathbf{p}}) U_{k+1}(\lambda_{\mathbf{p}})}, \quad
    w_{\rho_R} = \frac{U_{k+1}(\lambda_{\mathbf{p}})}{U_{k}(\lambda_{\mathbf{p}})}\hat{y}_\rho.
\]
We stress that the weight of $w_{\rho^*}$ is just a scalar. As discussed in Remark \ref{rmk:WhatCCWGivesUs}, the weights here hinge on the fact that $\gamma$ winds counterclockwise.

\begin{example}
\label{ex:standardwindingarc_poset}
The following diagram contains an arc $\gamma$ that winds about an orbifold point with $k=1$ and its corresponding poset, where the $\boxed{\text{label}}$ and weight of each element are indicated.
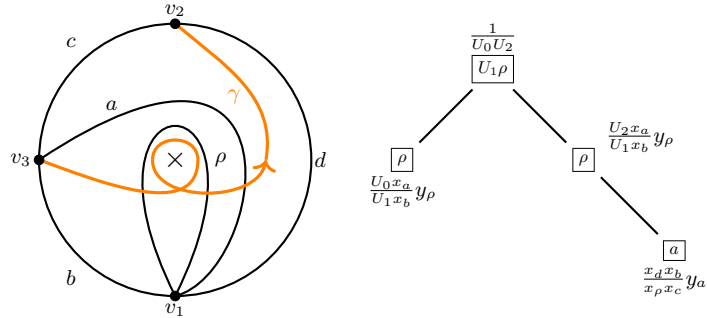
\begin{center}
\begin{tikzpicture}[scale=0.6, transform shape]

    \draw[thick] (0,0) circle (3cm);

    \node[below, scale=1.5] at (0,-3) {$v_1$}; 
    \node[above, scale=1.5] at (0,3) {$v_2$};  
    \node[left, scale=1.5] at (-3,0) {$v_3$};  

    \draw[thick] (-3,0) .. controls (3,4) and (2,-2.5) .. (0,-3);

    \draw[thick] (0,-3) .. controls (-2.5,2) and (2.5,2) .. (0,-3);

    \draw[in=-90,out=-20,looseness=1, orange, line width = 1.2pt] (-3,0) to (0.5,0);
    \draw[in=90,out=90,looseness=1.5, orange, line width = 1.2pt] (0.5,0) to (-0.5,0);
    \draw[in=-100,out=-90,looseness=1, orange, line width = 1.2pt,  ->] (-0.5,0) to (2,0);
    \draw[in=-40,out=80,looseness=1, orange, line width = 1.2pt] (2,0) to (0,3);

    \filldraw (0,-3) circle (3pt); 
    \filldraw (0,3) circle (3pt);  
    \filldraw (-3,0) circle (3pt); 

    \node[scale=2, thick] at (0,0) {$\times$};
    \node[scale=1.5] at (-1.4,1.2) {$a$};
    \node[scale=1.5] at (1,0) {$\rho$};
    \node[scale=1.5] at (-2.3,-2.6) {$b$};
    \node[scale=1.5] at (-2.3,2.6) {$c$};
    \node[scale=1.5] at (3.2,0) {$d$};

    \node[orange, scale=1.5] at (1.3,1.4) {$\gamma$};

    \node[scale=1.2](1) at (5,0){$\boxed{\rho}$};
    \node[scale=1.2](2) at (7,2){$\boxed{U_1\rho}$};
    \node[scale=1.2](3) at (9,0){$\boxed{\rho}$};
    \node[scale=1.2](4) at (11,-2){$\boxed{a}$};
    \draw[thick] (1) -- (2);
     \draw[thick] (3) -- (2);
     \draw[thick] (3) -- (4);

    \node[scale=1.5] at (5,-0.7) {$\frac{U_0 x_a}{U_1x_b}y_\rho$};
    \node[scale=1.5] at (7,2.7) {$\frac{1}{U_0U_2}$};
    \node[scale=1.5] at (10.25,0.5) {$\frac{U_2x_a}{U_1x_b}y_\rho$};
    \node[scale=1.5] at (11,-2.7) {$\frac{x_dx_b}{x_\rho x_c}y_a$};

\end{tikzpicture}
\end{center}
\end{example}

If the winding number satisfies $k \in \{0,\mathbf{p}-2\}$, we do not have an element $v_{\rho^*}$. If $k = 0$, then we have $\rho_L \succeq \rho_R$ and we set weights \[ 
    w_{\rho_L} = \frac{1}{U_{1}(\lambda_{\mathbf{p}})}\hat{y}_\rho, \quad w_{\rho_R} = U_1(\lambda_{\mathbf{p}})\hat{y}_\rho.
\] 
Otherwise, $k = \mathbf{p}-2$, and we have $\rho_L \preceq \rho_R$ and we reverse the weights of $\rho_L$ and $\rho_R$. 

\begin{remark}\label{rmk:TreatAsFormalVariable}
    One can optionally still include the $v_{\rho^*}$ element when $k \in \{0,\mathbf{p}-2\}$ if we treat the ideal weights as products of formal variables. If $k = 0$, then $U_{-1} = 0$ and so $w_{\rho_L} = 0$. There are only three ideals of the subposet formed by $v_{\rho_L}, v_{\rho^*}$, and $v_{\rho_R}$ with nonzero weight: $\emptyset$, $\{v_{\rho_R}\}$, and $\{v_{\rho_L}, v_{\rho^*}, v_{\rho_R}\}$. We observe that the weight of the last ideal is given by 
    \[ w_{\rho_L}w_{\rho^*}w_{\rho_R} = \left(\frac{U_{-1}(\lambda_{\mathbf{p}})}{U_0(\lambda_{\mathbf{p}})}\hat{y_\rho}\right)\left(\frac{1}{U_{-1}(\lambda_{\mathbf{p}})U_1(\lambda_{\mathbf{p}})}\right)\left(\frac{U_1(\lambda_{\mathbf{p}})}{U_0(\lambda_{\mathbf{p}})}\hat{y_\rho}\right) = \frac{1}{U_0^2(\lambda_{\mathbf{p}})}\left(\hat{y_\rho} \right)^2 = \left(\hat{y_\rho} \right)^2 \]
    and more generally, that summing over the ideal weights of this subposet recovers the exchange relation for pending arcs,
        \[w(\emptyset) + w(\{\rho_R \}) + w(\{ \rho_L, \rho^*, \rho_R \}) =  1 + \frac{U_1(\lambda_{\mathbf{p}})}{U_0(\lambda_{\mathbf{p}})}\hat{y_{\rho}} + \left(\hat{y}_\rho \right)^2 = 1 + \lambda_{\rho}\hat{y_{\rho}} + \left(\hat{y_{\rho}} \right)^2.\]
\end{remark}

We now turn to special cases which we have not yet addressed. If $\gamma = \gamma^{(p)}$ is a notched corner arc, then as for other notched arcs, the poset $\calP_\gamma$ is the result  of including a ``loop'' at one or two ends of $\calP_{\gamma^0}$. If $\sigma_1,\ldots,\sigma_m,\rho$ are the spokes incident to $p$ where $\rho$ is the pending arc enclosing the pending triangle $\Delta_0$, then we attach a loop with elements 
\[
    v_1 \succ v_{\rho^-} \prec v_{U_{k-1}\rho} \prec v_{\sigma_1} \prec \cdots \prec v_{\sigma_m} \prec v_{U_k\rho} \prec v_{\rho^+} \succ v_1
\]
The notation suggests the weights of each element; notice that $v_1$ would also be labelled with $\rho$.
The weights associated with these elements are defined as follows:
\begin{align*}
    w_{1}& =\frac{U_{k}(\lambda_{\mathbf{p}})}{U_{k-1}(\lambda_{\mathbf{p}})}\,\hat{y}_{\rho},&  w_{\rho^+} &= \frac{U_{k-2}(\lambda_{\mathbf{p}})}{U_{k-1}(\lambda_{\mathbf{p}})} \hat{y}_{\rho}, & w_{\rho^-} = \frac{U_{k}(\lambda_{\mathbf{p}})}{U_{k+1}(\lambda_{\mathbf{p}})} \hat{y}_{\rho}, \\
    w_{U_{k-1}\rho} &= \frac{1}{U_{k-2}(\lambda_{\mathbf{p}})U_{k}(\lambda_{\mathbf{p}})}, & w_{U_k\rho} &= U_{k-1}(\lambda_{\mathbf{p}})U_{k+1}(\lambda_{\mathbf{p}}).&
\end{align*}

\begin{example}\label{ex:enteringnotched} We give an explicit example of a notched corner arc with winding number and its associated poset. The weights in this case are given by the following.

\[w_{\rho} = \frac{U_{2}(\lambda_{\mathbf{p}})}{U_{1}(\lambda_{\mathbf{p}})} \, \hat{y}_{\rho}, \quad w_{\rho^+} = \frac{U_{0}(\lambda_{\mathbf{p}})}{U_{1}(\lambda_{\mathbf{p}})} \, \hat{y}_{\rho}, \quad w_{\rho^-} = \frac{U_{2}(\lambda_{\mathbf{p}})}{U_{3}(\lambda_{\mathbf{p}})} \, \hat{y}_{\rho},
\]
\[w_{U_1\rho} = \frac{1}{U_{0}(\lambda_{\mathbf{p}})U_{2}(\lambda_{\mathbf{p}})}, \quad , w_{U_2\rho} = U_{1}(\lambda_{\mathbf{p}})U_{3}(\lambda_{\mathbf{p}}),
\]
In these weights and the poset below, we use the labels $\rho^{+}$ and $\rho^{-}$ to distinguish between different elements labelled with $\rho$.

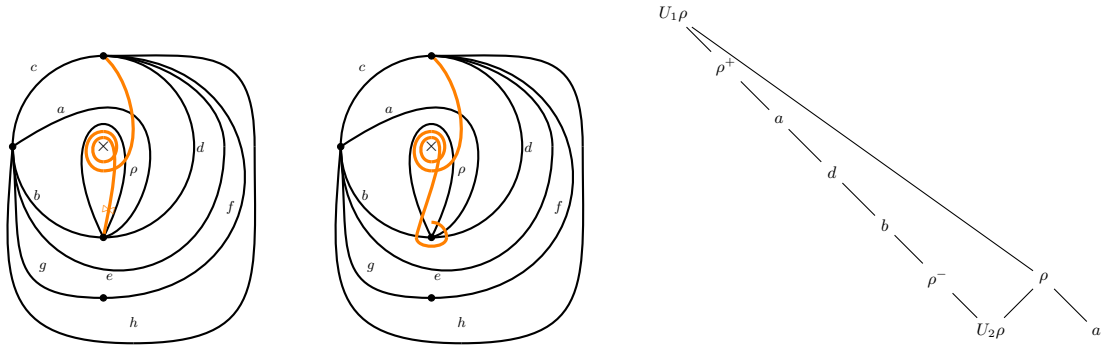
\begin{center}
\begin{tabular}{ccc}
    \begin{tikzpicture}[scale=0.4, transform shape]
    \draw[thick] (0,0) circle (3cm);

    \draw[thick] (-3,0) .. controls (3,4) and (2,-2.5) .. (0,-3);
    \draw[thick] (0,-3) .. controls (-2.5,2) and (2.5,2) .. (0,-3);

    \draw[in=0,out=-40,looseness=1, orange, line width = 1.2pt] (0,3) to (0,-.8);
    \draw[in=180,out=180,looseness=1.5, orange, line width = 1.2pt] (0,-.8) to (0,.5);
    \draw[in=0,out=0,looseness=1.5, orange, line width = 1.2pt] (0,-.5) to (0,.5);
    \draw[in=180,out=180,looseness=1.5, orange, line width = 1.2pt] (0,-.5) to (0,.3);
    \draw[in=80,out=0,looseness=0.6, orange, line width = 1.2pt] (0,.3) to (0,-3);
    \node[scale=1.5, rotate=-10, orange, line width = 1.2pt] at (0.17,-2.1) {$\bowtie$};

    \draw[in=-90, out=-90, looseness=2, thick] (4,0) to (-3,0);
    \draw[in=0, out=90, looseness=0.9, thick] (4,0) to (0,3);
    \draw[in=0,out=0,looseness=2, thick] (0,-5) to (0,3);
    \draw[in=-85,out=180,looseness=1.5, thick] (0,-5) to (-3,0);

    \draw[in=-95,out=180,looseness=1.5, thick] (1,-6.5) to (-3,0);
    \draw[in=-90,out=0,looseness=1.5, thick] (1,-6.5) to (5,0);
    \draw[in=90,out=0,looseness=1.5, thick] (0,3) to (5,0);

    \filldraw (0,-3) circle (3pt); 
    \filldraw (0,3) circle (3pt);  
    \filldraw (-3,0) circle (3pt); 
    \filldraw (0,-5) circle (3pt); 

    \node[scale=2, thick] at (0,0) {$\times$};
    \node[scale=1.5] at (-1.4,1.2) {$a$};
    \node[scale=1.5] at (1,-0.8) {$\rho$};
    \node[scale=1.5] at (-2.2,-1.6) {$b$};
    \node[scale=1.5] at (-2.3,2.6) {$c$};
    \node[scale=1.5] at (3.2,0) {$d$};
    \node[scale=1.5] at (4.2,-2) {$f$};
    \node[scale=1.5] at (-2,-4) {$g$};
    \node[scale=1.5] at (0.2,-4.3) {$e$};
    \node[scale=1.5] at (1,-5.8) {$h$};

    \end{tikzpicture}

     &  
    \begin{tikzpicture}[scale=0.4, transform shape]
    \draw[thick] (0,0) circle (3cm);

    \draw[thick] (-3,0) .. controls (3,4) and (2,-2.5) .. (0,-3);
    \draw[thick] (0,-3) .. controls (-2.5,2) and (2.5,2) .. (0,-3);

    \draw[in=0,out=-40,looseness=1, orange, line width = 1.2pt] (0,3) to (0,-.8);
    \draw[in=180,out=180,looseness=1.5, orange, line width = 1.2pt] (0,-.8) to (0,.5);
    \draw[in=0,out=0,looseness=1.5, orange, line width = 1.2pt] (0,-.5) to (0,.5);
    \draw[in=180,out=180,looseness=1.5, orange, line width = 1.2pt] (0,-.5) to (0,.3);
    \draw[in=80,out=0,looseness=0.6, orange, line width = 1.2pt] (0,.3) to (-.5,-3);
    \draw[in=-90,out=-90,looseness=1, orange, line width = 1.2pt] (-0.5,-3) to (.5,-3);
    \draw[in=0,out=90,looseness=1, orange, line width = 1.2pt] (0.5,-3) to (0,-2.5);

    \draw[in=-90, out=-90, looseness=2, thick] (4,0) to (-3,0);
    \draw[in=0, out=90, looseness=0.9, thick] (4,0) to (0,3);

    \draw[in=0,out=0,looseness=2, thick] (0,-5) to (0,3);
    \draw[in=-85,out=180,looseness=1.5, thick] (0,-5) to (-3,0);

    \draw[in=-95,out=180,looseness=1.5, thick] (1,-6.5) to (-3,0);
    \draw[in=-90,out=0,looseness=1.5, thick] (1,-6.5) to (5,0);
    \draw[in=90,out=0,looseness=1.5, thick] (0,3) to (5,0);

    \filldraw (0,-3) circle (3pt); 
    \filldraw (0,3) circle (3pt);  
    \filldraw (-3,0) circle (3pt); 
    \filldraw (0,-5) circle (3pt); 

    \node[scale=2, thick] at (0,0) {$\times$};
    \node[scale=1.5] at (-1.4,1.2) {$a$};
    \node[scale=1.5] at (1,-0.8) {$\rho$};
    \node[scale=1.5] at (-2.2,-1.6) {$b$};
    \node[scale=1.5] at (-2.3,2.6) {$c$};
    \node[scale=1.5] at (3.2,0) {$d$};
    \node[scale=1.5] at (4.2,-2) {$f$};
    \node[scale=1.5] at (-2,-4) {$g$};
    \node[scale=1.5] at (0.2,-4.3) {$e$};
    \node[scale=1.5] at (1,-5.8) {$h$};
    \end{tikzpicture}
    
    &
    \raisebox{0\height}{%
    \begin{tikzpicture}[scale=0.7, transform shape]
        \node (A) at (0,1){$a$};
        \node (L) at (-1,2){$\rho$};
        \node (B1) at (-2,1) {$U_2\rho$};
        \node (R) at (-3,2) {$\rho^-$};
        \node (C) at (-4,3) {$b$};
        \node (E) at (-5,4) {$d$};
        \node (A2) at (-6,5) {$a$};
        \node (B2) at (-7,6) {$\rho^+$};
        \node (B3) at (-8,7) {$U_{1}\rho$};

        \draw (A) -- (L);
        \draw (L) -- (B1);
        \draw (B1) -- (R);
        \draw (R) -- (C);
        \draw (C) -- (E);
        \draw (E) -- (A2);
        \draw (A2) -- (B2);
        \draw (B2) -- (B3);
        \draw (B3) to (L);

    \end{tikzpicture}
    }
\end{tabular}
\end{center}

\end{example}

It remains to consider notched arcs whose underlying plain version lies in the triangulation $T$. These were previously treated in~\cite{pilaud2023posets,weng2023f} and we follow the same convention if an arc is a standard arc. If $\gamma = \gamma^{(p)}$ is a singly-notched standard arc, then we define $\calP_\gamma^T = \calP_{\widetilde{\gamma}}^T$, where $\widetilde{\gamma}$ is a hooked version of $\gamma$.

\begin{example}\label{ex:plainversionintriangulation} We give an example of a poset $\calP_\gamma^T$ associated to a singly-notched arc $\gamma = \gamma^{(p)}$ such that $\gamma^0 \in T$. It is interesting to compare this to $\calP_\gamma^{T'}$ where $T'$ is a similar triangulation which does not include $\gamma^0$. We use the main construction in  the latter setting.
    
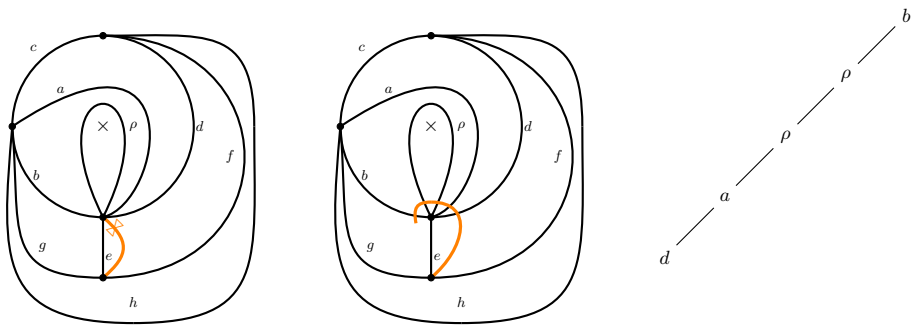
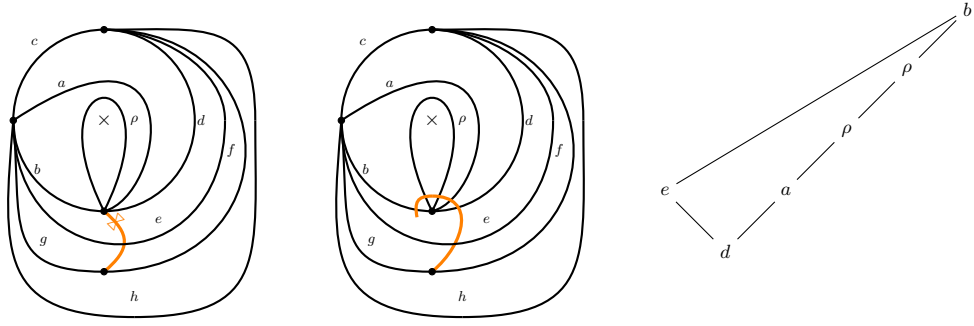
\begin{center}
\begin{tabular}{ccc}
    \begin{tikzpicture}[scale=0.4, transform shape]
    \draw[thick] (0,0) circle (3cm);

    \draw[thick] (-3,0) .. controls (3,4) and (2,-2.5) .. (0,-3);
    \draw[thick] (0,-3) .. controls (-2.5,2) and (2.5,2) .. (0,-3);

    \draw[in=-40,out=40,looseness=1.5, orange, line width = 1.2pt] (0,-5) to (0,-3);
    \node[scale=2, rotate=45, orange] at (0.4,-3.35) {$\bowtie$};

    \draw[thick] (0,-3) to (0,-5);
    \draw[in=0,out=0,looseness=2, thick] (0,-5) to (0,3);
    \draw[in=-85,out=180,looseness=1.5, thick] (0,-5) to (-3,0);

    \draw[in=-95,out=180,looseness=1.5, thick] (1,-6.5) to (-3,0);
    \draw[in=-90,out=0,looseness=1.5, thick] (1,-6.5) to (5,0);
    \draw[in=90,out=0,looseness=1.5, thick] (0,3) to (5,0);

    \filldraw (0,-3) circle (3pt); 
    \filldraw (0,3) circle (3pt);  
    \filldraw (-3,0) circle (3pt); 
    \filldraw (0,-5) circle (3pt); 

    \node[scale=2, thick] at (0,0) {$\times$};
    \node[scale=1.5] at (-1.4,1.2) {$a$};
    \node[scale=1.5] at (1,0) {$\rho$};
    \node[scale=1.5] at (-2.2,-1.6) {$b$};
    \node[scale=1.5] at (-2.3,2.6) {$c$};
    \node[scale=1.5] at (3.2,0) {$d$};
    \node[scale=1.5] at (4.2,-1) {$f$};
    \node[scale=1.5] at (-2,-4) {$g$};
    \node[scale=1.5] at (0.2,-4.3) {$e$};
    \node[scale=1.5] at (1,-5.8) {$h$};

    \end{tikzpicture}

     &  
    \begin{tikzpicture}[scale=0.4, transform shape]
    \draw[thick] (0,0) circle (3cm);

    \draw[thick] (-3,0) .. controls (3,4) and (2,-2.5) .. (0,-3);
    \draw[thick] (0,-3) .. controls (-2.5,2) and (2.5,2) .. (0,-3);

    \draw[in=0,out=40,looseness=1.5, orange, line width = 1.2pt] (0,-5) to (0,-2.5);
    \draw[in=180,out=100,looseness=1.5, orange, line width = 1.2pt] (-0.5,-3.2) to (0,-2.5);

    \draw[thick] (0,-3) to (0,-5);
    \draw[in=0,out=0,looseness=2, thick] (0,-5) to (0,3);
    \draw[in=-85,out=180,looseness=1.5, thick] (0,-5) to (-3,0);

    \draw[in=-95,out=180,looseness=1.5, thick] (1,-6.5) to (-3,0);
    \draw[in=-90,out=0,looseness=1.5, thick] (1,-6.5) to (5,0);
    \draw[in=90,out=0,looseness=1.5, thick] (0,3) to (5,0);

    \filldraw (0,-3) circle (3pt); 
    \filldraw (0,3) circle (3pt);  
    \filldraw (-3,0) circle (3pt); 
    \filldraw (0,-5) circle (3pt); 

    \node[scale=2, thick] at (0,0) {$\times$};
    \node[scale=1.5] at (-1.4,1.2) {$a$};
    \node[scale=1.5] at (1,0) {$\rho$};
    \node[scale=1.5] at (-2.2,-1.6) {$b$};
    \node[scale=1.5] at (-2.3,2.6) {$c$};
    \node[scale=1.5] at (3.2,0) {$d$};
    \node[scale=1.5] at (4.2,-1) {$f$};
    \node[scale=1.5] at (-2,-4) {$g$};
    \node[scale=1.5] at (0.2,-4.3) {$e$};
    \node[scale=1.5] at (1,-5.8) {$h$};
    \end{tikzpicture}
    
    &
    \raisebox{0.2\height}{%
    \begin{tikzpicture}[scale=0.8, transform shape]
        \node (E) at (0,1){$d$};
        \node (A) at (1,2){$a$};
        \node (B1) at (2,3) {$\rho$};
        \node (B2) at (3,4) {$\rho$};
        \node (C) at (4,5) {$b$};

        \draw (E) -- (A);
        \draw (A) -- (B1);
        \draw (B1) -- (B2);
        \draw (B2) -- (C);
    \end{tikzpicture}
    }
\end{tabular}
\end{center}

\begin{center}
\begin{tabular}{ccc}
    \begin{tikzpicture}[scale=0.4, transform shape]
    \draw[thick] (0,0) circle (3cm);

    \draw[thick] (-3,0) .. controls (3,4) and (2,-2.5) .. (0,-3);
    \draw[thick] (0,-3) .. controls (-2.5,2) and (2.5,2) .. (0,-3);

    \draw[in=-40,out=40,looseness=1.5, orange, line width = 1.2pt] (0,-5) to (0,-3);
    \node[scale=2, rotate=45, orange] at (0.4,-3.35) {$\bowtie$};

    \draw[in=-90, out=-90, looseness=2, thick] (4,0) to (-3,0);
    \draw[in=0, out=90, looseness=0.9, thick] (4,0) to (0,3);

    \draw[in=0,out=0,looseness=2, thick] (0,-5) to (0,3);
    \draw[in=-85,out=180,looseness=1.5, thick] (0,-5) to (-3,0);

    \draw[in=-95,out=180,looseness=1.5, thick] (1,-6.5) to (-3,0);
    \draw[in=-90,out=0,looseness=1.5, thick] (1,-6.5) to (5,0);
    \draw[in=90,out=0,looseness=1.5, thick] (0,3) to (5,0);

    \filldraw (0,-3) circle (3pt); 
    \filldraw (0,3) circle (3pt);  
    \filldraw (-3,0) circle (3pt); 
    \filldraw (0,-5) circle (3pt); 

    \node[scale=2, thick] at (0,0) {$\times$};
    \node[scale=1.5] at (-1.4,1.2) {$a$};
    \node[scale=1.5] at (1,0) {$\rho$};
    \node[scale=1.5] at (-2.2,-1.6) {$b$};
    \node[scale=1.5] at (-2.3,2.6) {$c$};
    \node[scale=1.5] at (3.2,0) {$d$};
    \node[scale=1.5] at (4.2,-1) {$f$};
    \node[scale=1.5] at (-2,-4) {$g$};
    \node[scale=1.5] at (1.8,-3.3) {$e$};
    \node[scale=1.5] at (1,-5.8) {$h$};

    \end{tikzpicture}

     &  
     
    \begin{tikzpicture}[scale=0.4, transform shape]
    \draw[thick] (0,0) circle (3cm);

    \draw[thick] (-3,0) .. controls (3,4) and (2,-2.5) .. (0,-3);
    \draw[thick] (0,-3) .. controls (-2.5,2) and (2.5,2) .. (0,-3);

    \draw[in=0,out=40,looseness=1.5, orange, line width = 1.2pt] (0,-5) to (0,-2.5);
    \draw[in=180,out=100,looseness=1.5, orange, line width = 1.2pt] (-0.5,-3.2) to (0,-2.5);

    \draw[in=-90, out=-90, looseness=2, thick] (4,0) to (-3,0);
    \draw[in=0, out=90, looseness=0.9, thick] (4,0) to (0,3);

    \draw[in=0,out=0,looseness=2, thick] (0,-5) to (0,3);
    \draw[in=-85,out=180,looseness=1.5, thick] (0,-5) to (-3,0);

    \draw[in=-95,out=180,looseness=1.5, thick] (1,-6.5) to (-3,0);
    \draw[in=-90,out=0,looseness=1.5, thick] (1,-6.5) to (5,0);
    \draw[in=90,out=0,looseness=1.5, thick] (0,3) to (5,0);

    \filldraw (0,-3) circle (3pt); 
    \filldraw (0,3) circle (3pt);  
    \filldraw (-3,0) circle (3pt); 
    \filldraw (0,-5) circle (3pt); 

    \node[scale=2, thick] at (0,0) {$\times$};
    \node[scale=1.5] at (-1.4,1.2) {$a$};
    \node[scale=1.5] at (1,0) {$\rho$};
    \node[scale=1.5] at (-2.2,-1.6) {$b$};
    \node[scale=1.5] at (-2.3,2.6) {$c$};
    \node[scale=1.5] at (3.2,0) {$d$};
    \node[scale=1.5] at (4.2,-1) {$f$};
    \node[scale=1.5] at (-2,-4) {$g$};
    \node[scale=1.5] at (1.8,-3.3) {$e$};
    \node[scale=1.5] at (1,-5.8) {$h$};

    \end{tikzpicture}
    &
    \raisebox{0.2\height}{%
    \begin{tikzpicture}[scale=0.8, transform shape]
        \node (H) at (-1,2){$e$};
        \node (E) at (0,1){$d$};
        \node (A) at (1,2){$a$};
        \node (B1) at (2,3) {$\rho$};
        \node (B2) at (3,4) {$\rho$};
        \node (C) at (4,5) {$b$};

        \draw (H) -- (E);
        \draw (E) -- (A);
        \draw (A) -- (B1);
        \draw (B1) -- (B2);
        \draw (B2) -- (C);
        \draw (H) to (C);
    \end{tikzpicture}
    }
\end{tabular}
\end{center}

\end{example}

We will discuss the case of singly-notched pending arcs later. Note that such an arc cannot appear in any triangulation, and thus these do not correspond to cluster variables.

We now turn to doubly-notched arcs $\gamma^{(p,q)}$ such that $\gamma^0 \in T$. We refer the reader to \cite{pilaud2023posets,weng2023f} for the case when $\gamma$ is an ordinary arc. Now, consider a doubly-notched pending arc $\gamma^{(q,q)}$. 

Label the set of all spokes of $T$ at $p$ as $\sigma_1,\ldots,\sigma_m$ in counterclockwise order, with $\sigma_m=\gamma$. We define the total orders $\mathcal{P}_{l_q}$ and $\mathcal{P}_{r_q}$ on labelled elements as follows:
\[
\mathcal{P}_{l_q} : \ \sigma_1^l \prec \sigma_2^l \prec \cdots \prec \sigma_{m-1}^l \prec \rho^l \prec \lambda_\mathbf{p}\rho^l,
\]
\[
\mathcal{P}_{r_q} : \ \rho^r \prec \sigma_1^r \prec \sigma_2^r \prec \cdots \prec \sigma_{m-1}^r.
\]
We then define $\mathcal{P}_{\rho^{(q,q)}}$ to be the poset on 
\[
\{\sigma_1^l,\ldots,\sigma_{m-1}^l,\rho^l,\lambda_\mathbf{p}\rho^l,\rho^r,\sigma_1^r,\ldots,\sigma_{m-1}^r,\rho^-,\rho^+\}
\] 
with the relations inherited from $\mathcal{P}_{l_q}$ and $\mathcal{P}_{r_q}$, together with the additional relations
\[
\rho^- \prec \sigma_1^l \prec \sigma_{m-1}^r \prec \rho^+,
\quad
\rho^- \prec \rho^r \prec \lambda_\mathbf{p}\rho^l \prec \rho^+,
\]
so that $\rho^-$ is the minimal element and $\rho^+$ the maximal element. The elements $\rho^\pm$ of $\mathcal{P}_{\rho^{(q,q)}}$ represent the underlying plain pending arc $\rho^0 \in T$. 

The weights $w_{\sigma_i}$ are defined as usual, i.e., $w_{\sigma_i} = \hat{y}_{\sigma_i}$. The weights of the elements labelled by $\rho$ are defined as follows:

\[w_{\rho^-} = U_1(\lambda_{\mathbf{p}})\hat{y}_\rho, \quad w_{\rho^+} = U_1(\lambda_{\mathbf{p}})\hat{y}_\rho, \quad w_{\rho^r} = \frac{1}{U_1(\lambda_{\mathbf{p}})}\hat{y}_\rho,\]
\[w_{\rho^l} = \frac{U_2(\lambda_{\mathbf{p}})}{U_1(\lambda_{\mathbf{p}})}\hat{y}_\rho,\quad w_{\lambda_\mathbf{p}\rho^l} = \frac{1}{U_2(\lambda_{\mathbf{p}})}.\]

\begin{example}
\label{ex:doubly notched pending arc}As an example, consider a doubly-notched pending arc in the case where the corresponding plain pending arc lies in the triangulation. The hook convention illustrates how the associated poset is constructed. We remark that similar lattices also appeared in \cite{HuangLattice}. 

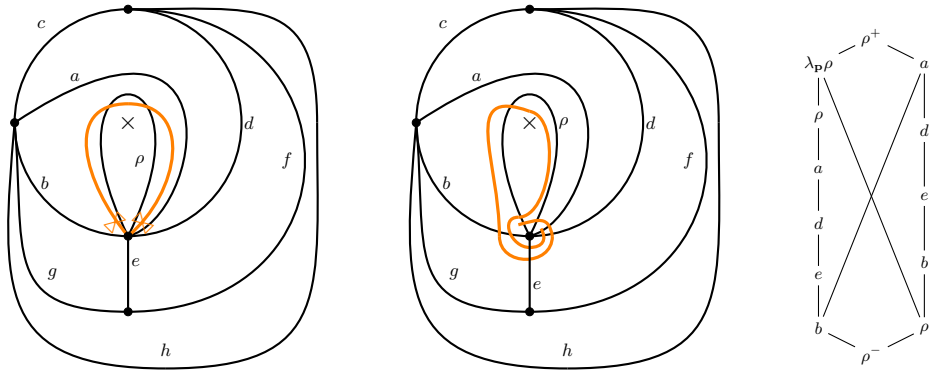
\begin{center}
\begin{tabular}{ccc}

    \begin{tikzpicture}[scale=0.5, transform shape]
    \draw[thick] (0,0) circle (3cm);

    \draw[thick] (-3,0) .. controls (3,4) and (2,-2.5) .. (0,-3);
    \draw[thick] (0,-3) .. controls (-2.5,2) and (2.5,2) .. (0,-3);

    \draw[in=0,out=40,looseness=1.3, orange, line width = 1.2pt] (0,-3) to (0,0.5);
    \draw[in=180,out=130,looseness=1.3, orange, line width = 1.2pt] (0,-3) to (0,0.5);

    \draw[thick] (0,-3) to (0,-5);
    \draw[in=0,out=0,looseness=2, thick] (0,-5) to (0,3);
    \draw[in=-85,out=180,looseness=1.5, thick] (0,-5) to (-3,0);

    \draw[in=-95,out=180,looseness=1.5, thick] (1,-6.5) to (-3,0);
    \draw[in=-90,out=0,looseness=1.5, thick] (1,-6.5) to (5,0);
    \draw[in=90,out=0,looseness=1.5, thick] (0,3) to (5,0);

    \filldraw (0,-3) circle (3pt); 
    \filldraw (0,3) circle (3pt);  
    \filldraw (-3,0) circle (3pt); 
    \filldraw (0,-5) circle (3pt); 

    \node[scale=2, thick] at (0,0) {$\times$};
    \node[scale=1.5] at (-1.4,1.2) {$a$};
    \node[scale=1.5] at (0.3,-1) {$\rho$};
    \node[scale=1.5] at (-2.2,-1.6) {$b$};
    \node[scale=1.5] at (-2.3,2.6) {$c$};
    \node[scale=1.5] at (3.2,0) {$d$};
    \node[scale=1.5] at (4.2,-1) {$f$};
    \node[scale=1.5] at (-2,-4) {$g$};
    \node[scale=1.5] at (0.2,-3.7) {$e$};
    \node[scale=1.5] at (1,-6) {$h$};

    \node[scale=2, rotate=135, orange] at (0.4,-2.65) {$\bowtie$};
    \node[scale=2, rotate=45, orange] at (-0.34,-2.65) {$\bowtie$};

    \end{tikzpicture}
    &
     
    \begin{tikzpicture}[scale=0.5, transform shape]
    \draw[thick] (0,0) circle (3cm);

    \draw[thick] (-3,0) .. controls (3,4) and (2,-2.5) .. (0,-3);
    \draw[thick] (0,-3) .. controls (-2.5,2) and (2.5,2) .. (0,-3);

    \draw[in=160,out=90,looseness=1.5, orange, line width = 1.2pt] (-0.8,-3) to (0,.3);
    \draw[in=-90,out=-90,looseness=1.5, orange, line width = 1.2pt] (-0.8,-3) to (0.6,-3);
    \draw[in=100,out=15,looseness=1.5, orange, line width = 1.2pt] (-0.3,-2.7) to (0.6,-3);

    \draw[in=-10,out=0,looseness=0.8, orange, line width = 1.2pt] (-0.3,-2.5) to (0,.3);
    \draw[in=180,out=180,looseness=1.5, orange, line width = 1.2pt] (-0.3,-2.5) to (0,-3.3);
    \draw[in=0,out=-60,looseness=1.5, orange, line width = 1.2pt] (0.3,-2.8) to (0,-3.3);

    \draw[thick] (0,-3) to (0,-5);
    \draw[in=0,out=0,looseness=2, thick] (0,-5) to (0,3);
    \draw[in=-85,out=180,looseness=1.5, thick] (0,-5) to (-3,0);

    \draw[in=-95,out=180,looseness=1.5, thick] (1,-6.5) to (-3,0);
    \draw[in=-90,out=0,looseness=1.5, thick] (1,-6.5) to (5,0);
    \draw[in=90,out=0,looseness=1.5, thick] (0,3) to (5,0);

    \filldraw (0,-3) circle (3pt); 
    \filldraw (0,3) circle (3pt);  
    \filldraw (-3,0) circle (3pt); 
    \filldraw (0,-5) circle (3pt); 

    \node[scale=2, thick] at (0,0) {$\times$};
    \node[scale=1.5] at (-1.4,1.2) {$a$};
    \node[scale=1.5] at (0.9,0) {$\rho$};
    \node[scale=1.5] at (-2.2,-1.6) {$b$};
    \node[scale=1.5] at (-2.3,2.6) {$c$};
    \node[scale=1.5] at (3.2,0) {$d$};
    \node[scale=1.5] at (4.2,-1) {$f$};
    \node[scale=1.5] at (-2,-4) {$g$};
    \node[scale=1.5] at (0.2,-4.3) {$e$};
    \node[scale=1.5] at (1,-6) {$h$};

    \end{tikzpicture}
    
    &
     \begin{tikzpicture}[scale=0.7, transform shape]
    \node (k1-) at (0,0.5){$\rho^-$};
    \node (k) at (-1,1){$b$};
    \node (dotsl) at (-1,2){$e$};
    \node (k2) at (-1,3){$d$};
    \node (k3) at (-1,4){$a$};
    \node (k4) at (-1,5){$\rho$};
    \node (k5) at (-1,6){$\lambda_\mathbf{p} \rho$};
    \node (as) at (1,1){$\rho$};
    \node (s1) at (1,2.25){$b$};
    \node (dotsr) at (1,3.5){$e$};
    \node (a1) at (1,4.75){$d$};
    \node (a2) at (1,6){$a$};
    \node (k1+) at (0,6.5){$\rho^+$};
    \draw (k1+) -- (a2);
    \draw(k1+) -- (k5);
    \draw(k4) -- (k5);
    \draw(k3) -- (k4);
    \draw(k2) -- (k3);
    \draw(k1-) -- (k);
    \draw(k1-) -- (as);
    \draw (as) --(s1);
    \draw (dotsr) --(s1);
    \draw (dotsr) -- (a1);
    \draw (a1) -- (a2);
    \draw(k) -- (dotsl);
    \draw(dotsl) -- (k2);
    \draw(k5) -- (as);
    \draw(a2) -- (k);
    \end{tikzpicture}

\end{tabular}
\end{center}
\end{example}

We have now defined a poset $\calP_\gamma^T$ for every generalized arc or closed curve $\gamma$. We will next use this poset to associate another Laurent polynomial to the pair $(\gamma,T)$. 

An \emph{order ideal} $I$ of a poset $\mathcal{P}$ is a down-closed subset of vertices of $\mathcal{P}$, i.e., a subset where $x\in I$ and $y\preceq x$ imply $y\in I$.  The order ideals of a poset ordered by inclusion have the structure of a distributive lattice, called the \emph{order ideal lattice} $J(\mathcal{P})$ of $\mathcal{P}$.

Given a weight function on our poset $\calP_\gamma$, the weight \(w\) of an order ideal \(I\) is defined by
\[
w(I)\;:=\;\prod_{v\in I}w_v,
\]
where the product is taken with multiplicity.
We adopt the convention \(w(\emptyset)=1\).

Given a poset $\mathcal{P}_\gamma$ for any arc $\gamma$ and a weight function $w$, we finally define the weight polynomial of the entire poset as
\[
\mathcal{W}(\calP_\gamma,w) = \sum_{I \in J(\calP_\gamma)} w(I).
\]

When the weight function and triangulation are understood, we will just write $\mathcal{W}(\mathcal{P}_\gamma)$.

\begin{example} \label{ex:fenceposet}The fence poset $\mathcal{P}$ for the composition $\alpha=(1,2)$ is depicted below.

  \begin{center}
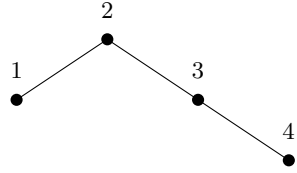
\begin{tabular}{c c}
	\begin{tikzpicture}[scale=.8]
\draw (0,0)--(1.5,1)--(4.5,-1);
\fill (0,0) circle(.1) node[above,yshift=.17cm] {$1$} ;
\fill (1.5,1) circle(.1) node[above,yshift=.17cm] {$2$} ;
\fill (3,0) circle(.1) node[above,yshift=.17cm] {$3$} ;
\fill (4.5,-1) circle(.1) node[above,yshift=.17cm] {$4$} ;
\end{tikzpicture}& \raisebox{12mm}{\begin{tabular}{c} The poset has the  ideals:
\\ $\varnothing$, $\{1\}$, $\{4\}$,\\ $\{1,4\}$, $\{3,4\}$\\ $\{1,3,4\}$, $\{1,2,3,4\}$. 
\end{tabular}} \end{tabular}\end{center}

It has the weight polynomial: $\rank(\mathcal{P},w)=1+w_1+w_4+w_1w_4+w_3w_4+w_1w_3w_4+w_1w_2w_3w_4.$
\end{example}

Our final ingredient will be associating a minimal term to each poset. To compute this, we define an extended poset $\widehat{\mathcal{P}}_\gamma$ as follows. If $\gamma$ is a closed curve or $\gamma$ is a notched arc such that $\gamma^0 \in T$, then $\widehat{\mathcal{P}}_\gamma=\mathcal{P}_\gamma$. If an endpoint of $\gamma$ is plain, we adjoin one additional element adjacent to it. Recall that the first and last triangles crossed by $\gamma$ are $\Delta_0$ and $\Delta_d$. If $s(\gamma)$ is plain and there is a non-boundary arc $\tau_s \in T$ which follows $\tau_{i_1}$ in counterclockwise in $\Delta_0$, then we add a new element $v_0 \succeq v_1$ and label $v_0$ with $\tau_s$. Similarly, if $t(\gamma)$ is plain, we add $v_{d} \preceq v_{d+1}$ if there is a non-boundary arc $\tau_t$ which  follows $\tau_d$ in counterclockwise order in $\Delta_d$. If $v_1$ or $v_d$ is a pending arc, i.e., $\gamma$ is a plain corner arc, we instead adjoin an element labelled $U_k\rho$ as $v_0$ or $v_{d+1}$, respectively, where $k$ is the winding number in $\Delta_0$ or $\Delta_d$, respectively. We will not need to assign a weight to these elements.

Let $\min(\gamma)$ denote the set of minimal elements of the poset $\widehat{\mathcal{P}}_{\gamma}$.
We also define $\text{max}(\gamma)$ to be all elements of $\widehat{\mathcal{P}}_{\gamma}$ which cover two elements and which do not lie in loops (i.e., chains at one or both sides corresponding to notches). It is important here to note that when $\gamma$ is notched and $\gamma^0 \in T$, then we regard the entire poset $\calP_\gamma^T$ to be a loop.

Let $\mathcal{L}$ be defined as follows: if a vertex  $v \in \widehat{\mathcal{P}}_\gamma$ is labelled by $\tau$, we set $\mathcal{L}(v_j) = x_{\tau}$ and if $v$ is labelled by $U_k(\lambda_{\mathbf{p}}) \rho$, then we set $\mathcal{L}(v) = U_k x_{\rho}$.

\begin{definition}\label{def:PosetMin}
    Let $\gamma \notin T$ be a generalized arc in $\Orb$. If $\gamma$ is not a singly-notched arc where $\gamma^0 \notin T$, the \emph{minimal term} \(x_{\min}^\gamma\) is then defined as
    \[
    x_{\min}^\gamma = \frac{\prod_{v \text{ in max}(\gamma)} \mathcal{L}(v)}{\prod_{v \text{ in min} (\gamma)} \mathcal{L}(v)}.
    \]
    If $\gamma^0 \in T$ and $\gamma$ is a singly-notched arc, \(x_{\min}^\gamma\) is defined as
    \[x_{\min}^\gamma=\frac{\prod_{v \text{ in max}(\gamma)} \mathcal{L}(v)}{x_{\gamma^0}\cdot\prod_{v \text{ in min} (\gamma)} \mathcal{L}(v)}\]
\end{definition}

Simply put, the vector $x_{\min}^\gamma$ is read by looking at the \emph{peaks} and \emph{valleys} of the poset $\widehat{\calP}_\gamma$, taking some care at the ends.

\begin{example}
We can apply \Cref{def:PosetMin} to compute the minimal term for each of the preceding examples in this section. For both examples in \Cref{ex:plainversionintriangulation}, the minimal term is $\frac{x_f}{x_d}$. In \Cref{ex:enteringnotched}, the minimal term is $\frac{x_c}{U_2x_a}$. In \Cref{ex:notchedposet}, the minimal terms for the first and second cases are, respectively, $\frac{x_b}{x_gx_e}$ and $\frac{x_f}{x_dx_e}$. Finally, the minimal term in \Cref{ex:doubly notched pending arc} is $\frac{1}{x_\rho}.$
\end{example}

The degree vector of $x_{\min}^\gamma$ plays a special role in cluster theory. The following is straightforward.

\begin{lemma}\label{lem:g}
If $\gamma^0 \notin T$ or $\gamma$ is not singly-notched and $\gb_\gamma$ is the degree vector of $x_{\min}^\gamma$ i.e.,  $x_{\min}^\gamma = A \mathbf{x}^{\mathbf{g}_\gamma}$, then  $\gb_\gamma = \mathbf{a}_\gamma + \mathbf{b}_\gamma$ where $\mathbf{a}_\gamma$ and $\mathbf{b}_\gamma$ are given as follows.

\begin{enumerate}
    \item $\mathbf{a}_\gamma=(a_\tau)_{\tau \in T}$ where $a_\tau$ is the number minimal elements of $\widehat{\calP}_{\gamma}$ labelled by a scalar multiple of $x_\tau$. 
    \item $\mathbf{b}_\gamma=(b_\tau)_{\tau \in T}$ where $b_\tau$ is the number of maximal elements of $\widehat{\calP}_\gamma$  which do not lie in a hook and which are labelled by a scalar multiple of $x_\tau$.
\end{enumerate}

If $\gamma^0 \in T$ and $\gamma$ is singly-notched, then the same is true where we add $\mathbf{e}_{\gamma^0}$ to $\mathbf{a}_\gamma$.
Moreover, if $\gamma$ is an ordinary arc, $x^\gamma_{\min} = \mathbf{x}^{\mathbf{g}_{\gamma}}$.
\end{lemma}

\begin{remark}
The notation for $\gb_{\gamma}$ is inspired by the notation for the $\gb$-vector of a string module, as in \cite{palu2021non}. When $\gamma$ is a plain arc or a closed curve, $\gb_\gamma$ coincides with the $\gb$-vector of the corresponding \emph{arc module} over the Jacobian algebra associated with $T$, as described in~\cite{labardini2009quivers, labardini2009quivers2} whereas if $\gamma$ is notched, the vector $\gb_\gamma$ appears to agree with the $\gb$-vector of the arc module studied in~\cite{dominguez2017arc}. To our knowledge, there is no proof of the fact that the combinatorial recipe in~\Cref{lem:g} coincides with the description of the $\gb$-vector of such module. In the case of orbifolds of order~3, related discussions can be found in~\cite{banaian2025snake}. Note that, in our present notation, the vector $\mathbf{b}$ represents $\mathbf{b} + \mathbf{r}$ in \cite{banaian2025snake,palu2021non}.
\end{remark}

We are now prepared to define a family of Laurent polynomials associated to posets from arcs on $\mathcal{O}$.

\begin{definition}\label{def:LaurentExpansionFromFencePoset}
Let $\gamma$ be an (possibly generalized) arc or closed curve on an orbifold $\mathcal{O}$ with triangulation $T$ and let $\calP_\gamma^T$ be the poset associated to $\gamma$. We define  \[
\chi(\calP_\gamma^T) := x^\gamma_{\min}\mathcal{W}(\calP_\gamma^T,w).
\]
\end{definition}

\subsection{From Posets to Matrices}

Using the theory of oriented posets (see~\cite{ouguz2025oriented,ezgieminecluster2024}), we can calculate the weight polynomial of any fence poset using the \emph{down-step} matrix, $D(w_i)$, and \emph{up-step} matrix, $U(w_i)$, defined as follows:
\begin{align}
\mdo(w_i)=\mdom{w_i},\quad \qquad \mup(w_i):=\mupm{w_i}.
\end{align}~\label{down-up}

Here, an up-step can be visualized as a node with a potential upwards connection and a down-step can be visualised as a node with a potential down-step connection. A fence poset of $n+1$ nodes can be created by combining $n$ such steps, where the final potential step is taken to be downwards as a convention. See~\Cref{fig:steps} for an example.

\begin{prop}[\cite{ouguz2025oriented}, Proposition 5.9]\label{prop:rank_mat}

Let $(\alpha_1, \alpha_2,\ldots,\alpha_k)$ be a composition of $n$. The weight polynomial of the fence poset $\mathcal{P}$ of $\alpha$ is given by the top left entry of the \emph{rank matrix} of $\mathcal{P}$ given by:
\[\left(U(w_1)\cdots U(w_{\alpha_1})\right)\cdot (D(w_{\alpha_1+1})\cdots D(w_{\alpha_1+\alpha_2})) \cdot (U(w_{\alpha_1+\alpha_2+1})\cdots U(w_{\alpha_1+\alpha_2+\alpha_3}))\cdots D(w_{n}) \]

where $w_i$ is a weight of $i$. In other words, $\mathcal{W}(\mathcal{P}_\gamma,w)$ is the top left entry of the matrix product above.
\end{prop}

\begin{example}\label{ex:orientedbuilding} We combine an up-step with three down-steps to get the fence poset corresponding to $\alpha=(1,2)$ from~\Cref{ex:fenceposet}.
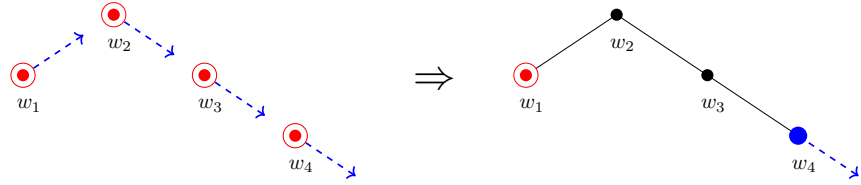
\begin{figure}[h]
\centering
    \scalebox{.8}{
\begin{tikzpicture}
\draw[->, blue, dashed, thick] (1.5,-1)--(2.5,-1/3);
\draw[->, blue, dashed, thick] (3,0)--(4,-2/3);
\draw[->, blue, dashed, thick] (4.5,-1)--(5.5,-5/3);
\draw[->, blue, dashed, thick] (6,-2)--(7,-8/3);

\fill (1.5,-1) circle(.2) node[white] {$1$} ;
\draw (1.6,-1.5) node{$\displaystyle w_1$} ;
\draw (3.1,-.5) node{$\displaystyle w_2$} ;
\fill (4.5,-1) circle(.2) node[white] {$3$} ;
\draw (4.6,-1.5) node{$\displaystyle w_3$} ;
\fill (6,-2) circle(.2) node[white] {$4$} ;
\draw (6.1,-2.5) node{$\displaystyle w_4$} ;
\fill[white] (1.5,-1) circle(.2) ;
\fill[red] (1.5,-1) circle(.1) ;
\draw[red] (1.5,-1) circle(.2);
\fill[white] (3,0) circle(.2) ;
\fill[red] (3,0) circle(.1) ;
\draw[red] (3,0) circle(.2);
\fill[white] (4.5,-1) circle(.2) ;
\fill[red] (4.5,-1) circle(.1) ;
\draw[red] (4.5,-1) circle(.2);
\fill[white] (6,-2) circle(.2) ;
\fill[red] (6,-2) circle(.1) ;
\draw[red] (6,-2) circle(.2);
\end{tikzpicture} \raisebox{15mm}{$\qquad$ \scalebox{2}{$\Rightarrow$} $\qquad$}
\begin{tikzpicture}
\draw (1.5,-1)-- (3,0)--(6,-2);
\draw[->, blue, dashed,thick] (6,-2)--(7,-8/3);
\fill (1.5,-1) circle(.1);
\draw (1.6,-1.5) node{$\displaystyle w_1$} ;
\fill (3,0) circle(.1) ;
\draw (3.1,-.5) node{$\displaystyle w_2$} ;
\fill (4.5,-1) circle(.1) ;
\draw (4.6,-1.5) node{$\displaystyle w_3$} ;
\fill[blue] (6,-2) circle(.15);
\draw (6.1,-2.5) node{$\displaystyle w_4$} ;
\fill[white] (1.5,-1) circle(.2) ;
\fill[red] (1.5,-1) circle(.1) ;
\draw[red] (1.5,-1)circle(.2);
\end{tikzpicture}}
\caption{Multiplying up and down posets to get the fence poset for $(1,2)$.}\label{fig:steps}
\end{figure}

 The corresponding weight polynomial is given by the top left entry of the following matrix:
\begin{align*}
&\mup(w_1)\mdo(w_2)\mdo(w_3)\mdo(w_4)=  \mupm{w_1}\mdom{w_2}\mdom{w_3}\mdom{w_4}=\\
 & \begin{bmatrix}
1 \!+\! w_1 \!+\! w_4 \!+\! w_1w_4 \!+\! w_3w_4 \!+\! w_1w_3w_4 \!+\! w_1w_2w_3w_4 & \!-w_4 \!-\! w_1w_4 \!-\! w_3w_4 \!-\! w_1w_3w_4 \!-\! w_1w_2w_3w_4 \\1 \!+\! w_4 \!+\! w_3w_4 &\!-w_4 \!-\! w_3w_4\end{bmatrix}
\end{align*}
\end{example}

These two matrices are sufficient for constructing all fence posets and calculating their weight polynomials, along with this definition: let $x_L$ be a vertex that corresponds to $w_1$ and $x_R$ be a vertex that corresponds to $w_{n+1}$. Now, we can define the matrix $\rmm_w(\calP\!\searrow)$ and its companion $\drm_w(\calP\!\nearrow)$ for any poset $\calP$ as follows:

\begin{itemize}
    \item \emph{Down pointing weight matrix of $\calP$}: $$\displaystyle \rmm_w(\calP\!\searrow):=\begin{bmatrix} \rank(\calP,w) & -\rank(\calP,w)|_{x_R\in I}\\ \rank(\calP,w)|_{x_L \notin I} & -\rank(\calP,w)|_{\substack{x_R\in I\\x_L \notin I}}
		\end{bmatrix}$$
    \item  \emph{Up pointing weight matrix of $\calP$}: 
    $$	\displaystyle \drm_w(\calP\!\nearrow):=\begin{bmatrix} \rank(\calP,w)|_{x_R\in I} & \rank(\calP,w)|_{x_R\notin I}\\ \rank(\calP,w)|_{\substack{x_R\in I\\x_L \notin I}} & \rank(\calP,w)|_{\substack{x_R\notin I\\x_L \notin I}}
			\end{bmatrix}$$
\end{itemize}

$\rmm_w(\calP\!\searrow)$ is the matrix product we described in~\Cref{prop:rank_mat}, which is called the \emph{rank matrix}. The entries are the weight polynomials of the ideals satisfying the given constraints; for instance, $\rank(\calP,w)|_{x_L \notin I}$ means that the weight polynomials of the ideals not containing the left endpoint $x_L$.
We can easily go from one matrix to the other by using the \emph{tail flip} matrix $\tf=\begin{bmatrix}
				1&-1\\1&0
			\end{bmatrix}$:
		$$\drm_w(\calP\!\nearrow)=\rmm_w(\calP\!\searrow) \cdot \tf, \qquad \qquad \rmm_w(\calP\!\searrow)=\drm_w(\calP\!{\nearrow}) \cdot \tf^{-1}.$$

Other types of posets constructed in this work can also be accommodated via some additional operations developed in \cite{ezgieminecluster2024}. We include a brief description below, and accompanying visuals in~\Cref{table:tableofmoves}.

\begin{itemize}
    \item \textbf{Circular fence posets for closed curves:} (\cite{ezgieminecluster2024}, Theorem 5.3) To construct circular fence posets, we use the trace operation. Taking the trace of $\calP\!\searrow$ gives the weight function of $\close(\calP\!\searrow)$  obtained by connecting the left end and right endpoints via $x_L \preceq x_R$. If the alternate connection   $x_L \succeq x_R$ is necessary, we take the trace of $\calP\!\nearrow$ instead.

    \item \textbf{Loop posets for notched curves:}  (\cite{ezgieminecluster2024}, Theorem 5.5, Theorem 5.6) To construct the loops, we use loop functions. For notches at the beginning of the curve, we have the functions $   \lloop_\searrow$ and    $\lloop_\nearrow$ depending on the direction of the inequality. For loops at the end of the curve, we have a single function $\rloop$ and the direction of the curve is determined via whether we apply the operation to $\calP\!\searrow$ or $\calP\!\nearrow$ as in the case of closed curves. The weight functions of notched curves are given by the top left entry of the resulting matrix.

    \begin{align*}
    \lloop_\searrow\left( \begin{bmatrix}
    a&b\\c&d
    \end{bmatrix}\right):&= \begin{bmatrix}
    a+d& b-d\\a+b&0
    \end{bmatrix} \qquad \qquad     &\lloop_\nearrow\left( \begin{bmatrix}
    a&b\\c&d
    \end{bmatrix}\right):=  \begin{bmatrix}
    a+d& -a\\d&0
    \end{bmatrix}\\
    \rloop\left( \begin{bmatrix}
    a&b\\c&d
    \end{bmatrix}\right): &= \begin{bmatrix}
    a+d& c-a\\c+d&0
    \end{bmatrix}&
\end{align*}
\item\textbf{Poset for Doubly-notched pending arc: } For a doubly-notched pending arc $\mathcal{P}_{\gamma^{(p,p)}}$, removing the elements $\gamma^\pm$ gives a circular fence poset $\mathcal{Q}$ whose weight polynomial we already know how to calculate via matrices. As $\gamma^-\preceq \gamma^+$ in the poset, we have three types of ideals. Ones without $\gamma^-$ (just the empty ideal), ones with $\gamma^-$ but not $\gamma^+$ (ideals of $\mathcal{Q}$ with $\gamma^-$ added), and ones with both $\gamma^-$ and $\gamma^+$ (the maximal ideal containing all vertices). This gives us the following formula.
   \begin{align*}\mathcal{W}(\mathcal{P}_{\gamma^{(p,p)}},w)&=1+w_{\gamma^-}\mathcal{W}(\mathcal{Q},w)+\prod_{v\in J(\mathcal{P}_{\gamma^{(p,p)}})}w_v
   \end{align*}
\end{itemize}

\begin{remark}
    One can notice that we can use the idea of deriving a matrix formula of a poset for a doubly-notched pending arc for a doubly-notched arc $\gamma^{(p,q)}$ where $\gamma \in T$. In other words, we can extend the matrix formula from ``nondegenerate" cases in~\cite{ezgieminecluster2024} to ``degenerate" cases in~\cite{pilaud2023posets}.
\end{remark}

\begin{table}[ht]
    \centering
    \newcommand{\mystrut}{\rule[0.5ex]{0pt}{8ex}} 
    
    \resizebox{\textwidth}{!}{ 
    \begin{tabular}{|c|c|c||c|c|c|}
        \hline
        \multicolumn{3}{|c||}{\textbf{Loop}} & \multicolumn{3}{c|}{\textbf{Source Loop}} \\ \hline
        $\close(\calP\!\searrow)$ & $\operatorname{tr}(\rmm_w({\calP\!\searrow}))$ & 
        \begin{tabular}{c} \mystrut
            \begin{tikzpicture}[scale=.45, baseline=(current bounding box.center)]
                \draw (0,0)--(1,1)--(3,-1);
                \fill[white] (0,0) circle(.2) ; \fill[red] (0,0) circle(.1) ; \draw[red] (0,0) circle(.2);
                \fill (1,1) circle(.1) ; \fill (2,0) circle(.1) ; \fill[blue] (3,-1) circle(.1) ;
                \draw[->, blue,dashed, thick] (3,-1)--(3.5,-1.5);
                \node at (1.5,-1.5) {$\scriptstyle{\calP\!\searrow}$};
            \end{tikzpicture} $\rightarrow$ 
            \begin{tikzpicture}[scale=.45, baseline=(current bounding box.center)]
                \draw (0,-1)--(1,1)--(3,-.2)--(0,-1);
                \fill (0,-1) circle(.1) ; \fill (1,1) circle(.1) ; \fill (2,0.4) circle(.1) ; \fill (3,-.2) circle(.1);
                \node at (1.5,-1.5) {$\scriptstyle{\close\,(\calP\!\searrow)}$};
            \end{tikzpicture}
        \end{tabular} &
        $\rhd(\calP\!\searrow)$ & $\lloop_\searrow (\rmm_w({\calP\!\searrow}))$ & 
        \begin{tabular}{c} \mystrut
            \begin{tikzpicture}[scale=.45, baseline=(current bounding box.center)]
                \draw (0,0)--(1,1)--(3,-1);
                \fill[white] (0,0) circle(.2) ; \fill[red] (0,0) circle(.1) ; \draw[red] (0,0) circle(.2);
                \fill (1,1) circle(.1) ; \fill (2,0) circle(.1) ; \fill[blue] (3,-1) circle(.1) ;
                \draw[->, blue,dashed, thick] (3,-1)--(3.5,-1.5);
                \node at (1.5,-1.5) {$\scriptstyle{\calP\!\searrow}$};
            \end{tikzpicture} $\rightarrow$ 
            \begin{tikzpicture}[scale=.45, baseline=(current bounding box.center)]
                \draw (0,-1)--(1,1)--(3,-.2)--(0,-1);
                \fill (0,-1) circle(.1) ; \fill (1,1) circle(.1) ; \fill (2,0.4) circle(.1) ;
                \fill[white] (3,-.2) circle(.2) ; \fill[red] (3,-.2) circle(.1) ; \draw[red] (3,-.2) circle(.2);
                \draw[->, blue,dashed, thick] (3,-.2)--(3.8,-1);
                \node at (1.5,-1.5) {$\scriptstyle{\rhd(\calP\!\searrow)\!\searrow}$};
            \end{tikzpicture}
        \end{tabular} \\ \hline
        $\close(\calP\nearrow)$ & $\operatorname{tr}(\rmm_w({\calP\!\nearrow}))$ & 
        \begin{tabular}{c} \mystrut
            \begin{tikzpicture}[scale=.45, baseline=(current bounding box.center)]
                \draw (0,0)--(1,1)--(3,-1);
                \fill[white] (0,0) circle(.2) ; \fill[red] (0,0) circle(.1) ; \draw[red] (0,0) circle(.2);
                \fill (1,1) circle(.1) ; \fill (2,0) circle(.1) ; \fill[blue] (3,-1) circle(.1) ;
                \draw[->, blue,dashed, thick] (3,-1)--(3.5,-.5);
                \node at (1.5,-1.5) {$\scriptstyle{\calP\!\nearrow}$};
            \end{tikzpicture} $\rightarrow$ 
            \begin{tikzpicture}[scale=.45, baseline=(current bounding box.center)]
                \draw (0,0)--(1,1)--(3,-1)--(0,0);
                \fill (0,0) circle(.1) ; \fill (1,1) circle(.1) ; \fill (2,0) circle(.1) ; \fill (3,-1) circle(.1);
                \node at (1.5,-1.5) {$\scriptstyle{\close\,(\calP\!\nearrow)}$};
            \end{tikzpicture}
        \end{tabular} &
        $\rhd(\calP\!\nearrow)$ & $\lloop_\nearrow(\rmm_w({\calP\!\nearrow}))$ & 
        \begin{tabular}{c} \mystrut
            \begin{tikzpicture}[scale=.45, baseline=(current bounding box.center)]
                \draw (0,0)--(1,1)--(3,-1);
                \fill[white] (0,0) circle(.2) ; \fill[red] (0,0) circle(.1) ; \draw[red] (0,0) circle(.2);
                \fill (1,1) circle(.1) ; \fill (2,0) circle(.1) ; \fill[blue] (3,-1) circle(.1) ;
                \draw[->, blue,dashed, thick] (3,-1)--(3.5,-.5);
                \node at (1.5,-1.5) {$\scriptstyle{\calP\!\nearrow}$};
            \end{tikzpicture} $\rightarrow$ 
            \begin{tikzpicture}[scale=.45, baseline=(current bounding box.center)]
                \draw (0,0)--(1,1)--(3,-1)--(0,0);
                \fill (0,0) circle(.1) ; \fill (1,1) circle(.1) ; \fill (2,0) circle(.1) ;
                \fill[white] (3,-1) circle(.2) ; \fill[red] (3,-1) circle(.1) ; \draw[red] (3,-1) circle(.2);
                \draw[->, blue,dashed, thick] (3,-1)--(3.9,-1.5);
                \node at (1.5,-1.5) {$\scriptstyle{\rhd(\calP\!\nearrow)\!\searrow}$};
            \end{tikzpicture}
        \end{tabular} \\ \hline
        \multicolumn{3}{|c||}{\textbf{Target Loop}} & \multicolumn{3}{c|}{} \\ \cline{1-3}
        $\lhd(\calP\!\searrow)$ & $\rloop(\rmm_w({\calP\!\searrow}))$ & 
        \begin{tabular}{c} \mystrut
            \begin{tikzpicture}[scale=.45, baseline=(current bounding box.center)]
                \draw (0,0)--(1,1)--(3,-1);
                \fill[white] (0,0) circle(.2) ; \fill[red] (0,0) circle(.1) ; \draw[red] (0,0) circle(.2);
                \fill (1,1) circle(.1) ; \fill (2,0) circle(.1) ; \fill[blue] (3,-1) circle(.1) ;
                \draw[->, blue,dashed, thick] (3,-1)--(3.5,-1.5);
                \node at (1.5,-1.5) {$\scriptstyle{\calP\!\searrow}$};
            \end{tikzpicture} $\rightarrow$ 
            \begin{tikzpicture}[scale=.45, baseline=(current bounding box.center)]
                \draw (0,-1)--(1,1)--(3,-.2)--(0,-1);
                \fill (3,-.2) circle(.1) ; \fill (1,1) circle(.1) ; \fill (2,0.4) circle(.1) ;
                \fill[white] (0,-1) circle(.2) ; \fill[red] (0,-1) circle(.1) ; \draw[red] (0,-1) circle(.2);
                \draw[->, blue,dashed, thick] (0,-1)--(-.5,-1.5);
                \node at (1.5,-1.5) {$\scriptstyle{\lhd(\calP\!\searrow)\!\searrow}$};
            \end{tikzpicture}
        \end{tabular} & \multicolumn{3}{c|}{} \\ \cline{1-3}
        $\lhd(\calP\!\nearrow)$ & $\rloop (\rmm_w({\calP\!\nearrow}))$ & 
        \begin{tabular}{c} \mystrut
            \begin{tikzpicture}[scale=.45, baseline=(current bounding box.center)]
                \draw (0,0)--(1,1)--(3,-1);
                \fill[white] (0,0) circle(.2) ; \fill[red] (0,0) circle(.1) ; \draw[red] (0,0) circle(.2);
                \fill (1,1) circle(.1) ; \fill (2,0) circle(.1) ; \fill[blue] (3,-1) circle(.1) ;
                \draw[->, blue,dashed, thick] (3,-1)--(3.5,-.5);
                \node at (1.5,-1.5) {$\scriptstyle{\calP\!\nearrow}$};
            \end{tikzpicture} $\rightarrow$ 
            \begin{tikzpicture}[scale=.45, baseline=(current bounding box.center)]
                \draw (0,0)--(1,1)--(3,-1)--(0,0);
                \fill (0,0) circle(.1) ; \fill (1,1) circle(.1) ; \fill (2,0) circle(.1) ; \fill (3,-1) circle(.1) ;
                \fill[white] (0,0) circle(.2) ; \fill[red] (0,0) circle(.1) ; \draw[red] (0,0) circle(.2);
                \draw[->, blue,dashed, thick] (0,0)--(-.8,-.8);
                \node at (1.5,-1.5) {$\scriptstyle{\lhd(\calP\!\nearrow)\!\searrow}$};
            \end{tikzpicture}
        \end{tabular} & \multicolumn{3}{c|}{} \\ \hline
    \end{tabular}
    }
    \caption{Poset operations, start is highlighted with red.}
    \label{table:tableofmoves}
\end{table}

\begin{example}
  Two loop posets are shown in~\Cref{ex:notchedposet}. Their weight functions are the top left entry of the following matrix formulas:
\begin{align*}
    &\left(\lloop_\searrow ( U(g)U(f)D(e)D(b)\right) \cdot D(\rho) D(\rho)D(a) \cdot\tf \left( \lloop_t(U(d)D(f)D(g)U(e))\right)\\
    &\left(\lloop_\searrow ( U(d)U(a)U(\rho)U(\rho)U(b)D(e)D(f))\right)\cdot D(f) \cdot \tf\left( \lloop_t(U(g)D(b)D(\rho)D(\rho)D(a)D(d)U(e)\right)
\end{align*}
\Cref{ex:doubly notched pending arc} contains a doubly-notched pending arc. As described in the last case, we can calculate it via the following formula, where $w(I_{\max})$ denotes the weight of the maximal ideal containing all elements.
  \begin{align*}
    &1+ w(\rho^-)\cdot\operatorname{trace}\left(U(b)U(e)U(d)U(a)U(\rho)D(\lambda_\mathbf{p}\rho)U(\rho)U(b)U(e)U(d)D(a)\right)+w(I_{\max})
\end{align*}
\Cref{ex:plainversionintriangulation} describes two singly-notched arcs where their notched endpoints are the endpoints of a pending arc in a triangulation. Though this provides a new case algebraically, the posets for the two examples turn out to be quite simple and can be calculated with the following formulas respectively: 
\begin{align*}
    &U(d)U(a)U(\rho)U(\rho)D(b)\\
    &\operatorname{trace}\left(D(e)U(d)U(a)U(\rho)U(\rho)D(b)\right)
\end{align*}
For the first formula, we need to take the top left entry of the matrix to get the weight function. In the second formula, we recover it directly as we are taking the trace.
\end{example}

\begin{prop}\label{prop:FormulasAgree}
For any arc on a (possibly punctured) orbifold, the matrix formula agrees with the expansion formula obtained via labelled posets.
\end{prop}

\begin{proof}
This follows fairly immediately from previous work on matrix formulas in Section 5 of~\cite{ezgieminecluster2024}.
\end{proof}

\begin{example}
    For the arc $\gamma$ from~\Cref{ex:1poset}, the corresponding matrices are constructed as follows: renaming $\frac{U_0x_a}{U_1x_b}y_\rho = w_1$, $\frac{1}{U_0U_2} = w_2$, $\frac{U_2x_a}{U_1x_b}y_\rho = w_3$ and $\frac{x_dx_b}{x_\rho x_c}y_a = w_4$, we obtain $\chi(\calP_\gamma)$ by considering the left corner of the following matrix multiplication multiplied by $\frac{U_1x_bx_c}{x_a}$.
\[\begin{bmatrix}
w_1 & 1 \\
 0  & 1
\end{bmatrix}\begin{bmatrix}
1+w_2 & -w_2 \\
 1    & 0
\end{bmatrix}\begin{bmatrix}
1+w_3 & -w_3 \\
 1    & 0
\end{bmatrix}\begin{bmatrix}
1+w_4 & -w_4 \\
 1    &  0
\end{bmatrix}\]

\end{example}

\section{\texorpdfstring{$T$-walks}{T-walks}}\label{sec:Twalk}

\subsection{\texorpdfstring{$T$-walks for Plain Arcs}{T-walks for Plain Arcs}}
In this section, we explain how the combinatorial models of $T$-paths can be extended to orbifolds. The notion of a $T$-path was first introduced in~\cite{ST09} to provide expansion formulas for cluster variables on unpunctured surfaces. Later, Gunawan and Musiker~\cite{GM15} gave a  combinatorial proof that cluster monomials form the atomic basis of a cluster algebra of type~D by extending the notion of a $T$-path  to arcs on a once-punctured polygon. More recently, a further generalization of the $T$-path construction to generalized arcs on arbitrary punctured surfaces, called a $T$-walk, was introduced in~\cite{ezgieminecluster2024} by a subset of the authors.

Here we will demonstrate how the definition given in ~\cite{ezgieminecluster2024} can be extended to orbifolds. The main new feature will be the method of weighting the paths whenever an arc winds around an orbifold point. 
Before we describe the rule for these coefficients, we will briefly review the notation and definitions from~\cite{ezgieminecluster2024}.

In order to define $T$-walks, we need to assign directions to the arcs crossed by $\gamma$. Suppose that $\gamma$ has crossings $\tau_1, \dots, \tau_d$. For each crossing $\tau_i$, we assign a direction to that crossing by orienting the crossed arc towards one of its endpoints. If an arc is crossed multiple times, it receives an orientation for each time it is crossed. This means that $\gamma$ has a total of $2^d$ possible orientations of its crossings. These orientations are encoded as \emph{orientation vectors}  $\vec{v} \in \{0,1\}^d$ where $v_i=0$ if the arc is oriented counterclockwise in $\Delta_{i-1}$, and $v_i=1$ otherwise. The orientation specified by $\vec{0}$ is called the \emph{minimal direction}.

For the arc $\gamma$ from~\Cref{ex:1poset}, the minimal direction $\vec{0}=(0,0,0)$ is illustrated below in~\Cref{ex:forTwalk}.

\begin{figure}[H]
    \centering
\begin{tikzpicture}[scale=0.6, transform shape]

    \draw[thick] (0,0) circle (3cm);

    \node[below, scale=1.5] at (0,-3) {$v_1$}; 
    \node[above, scale=1.5] at (0,3) {$v_2$};  
    \node[left, scale=1.5] at (-3,0) {$v_3$};  

    \draw[thick,postaction={decorate, decoration={markings,mark=at position 0.9 with {\arrow[scale=1.5, black!30!green]{<}}}}] (-3,0) .. controls (3,4) and (2,-2.5) .. (0,-3);

    \draw[thick, postaction={decorate, decoration={markings, mark=at position 0.12 with {\arrow[scale=1.5, black!30!green]{>}}, mark=at position 0.9 with {\arrow[scale=1.5, black!30!green]{<}}}}] (0,-3) .. controls (-2.5,2) and (2.5,2) .. (0,-3);

    \draw[in=-90,out=-20,looseness=1, orange, line width = 1.2pt] (-3,0) to (0.5,0);
    \draw[in=90,out=90,looseness=1.5, orange, line width = 1.2pt] (0.5,0) to (-0.5,0);
    \draw[in=-100,out=-90,looseness=1, orange, line width = 1.2pt,  ->] (-0.5,0) to (2,0);
    \draw[in=-40,out=80,looseness=1, orange, line width = 1.2pt] (2,0) to (0,3);
    \filldraw (0,-3) circle (3pt); 
    \filldraw (0,3) circle (3pt);  
    \filldraw (-3,0) circle (3pt); 

    \node[scale=2, thick] at (0,0) {$\times$};
    \node[scale=1.5] at (-1.4,1.2) {$a$};
    \node[scale=1.5] at (1,0) {$\rho$};
    \node[scale=1.5] at (-2.3,-2.6) {$b$};
    \node[scale=1.5] at (-2.3,2.6) {$c$};
    \node[scale=1.5] at (3.2,0) {$d$};
   
\end{tikzpicture}
\caption{The minimal direction for the arc $\gamma$.}\label{ex:forTwalk}
\end{figure}
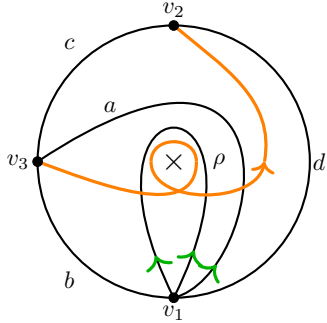

\begin{remark}
In the minimal direction $\vec{0}$, a pair of consecutive arcs $\alpha_{i}$ and $\alpha_{i+1}$ are always both directed either towards or away from their common endpoint in the triangle $\Delta_i$ (\cite[Lemma 3.3]{ezgieminecluster2024}).
\end{remark}

\begin{definition}\label{def:twalk}
Let $\mathcal{O}$ be a (possibly punctured) orbifold with an initial triangulation $T$, and let $\gamma$ be a plain arc or closed curve on $\mathcal{O}$. Suppose that $\gamma$ intersects the arcs of $T$ at a sequence of $d$ crossings, denoted by $\tau_{i_1}, \dots, \tau_{i_d}$. 
Consider a vector  $\vec{v} \in \{0, 1\}^d$, and let $\alpha_j$ be an oriented copy of the arc corresponding to $\tau_{i_j}$ with orientation determined by the component $v_j$. 

\begin{itemize}
    \item Suppose first $\gamma$ is a (generalized) arc. If there exists a sequence of directed paths $\beta_0,\beta_1,\ldots,\beta_d$, each contained within a single triangle of $T$ such that the concatenation  of segments 
    \[\beta_0 \alpha_1 \beta_1 \dots \alpha_d \beta_d\]
    forms a connected walk which is homotopic to $\gamma$, then we call this walk a \emph{$T$-walk} and denote it $T_{\vec{v}}$. 
    \item Suppose next that  $\gamma$ is a closed curve.  If there exists a sequence of directed paths $\beta_1,\ldots,\beta_d$ such that the concatenation  of segments
    \[\alpha_1 \beta_1 \dots \alpha_d \beta_d\]
    forms a connected walk which is homotopic to $\gamma$, then we call this walk a \emph{$T$-walk} and denote it $T_{\vec{v}}$.
\end{itemize}
Let $\mathrm{TW}(\gamma, T)$ (or just $\mathrm{TW}(\gamma)$)  denote the set of all $T$-walks corresponding to $\gamma$ on the initial triangulation $T$. 
\end{definition}

Some analysis of the definition reveals that every orientation vector $\vec{v}$ gives rise to at most one $T$-walk. When a $T$-walk exists for an orientation vector, we may call it a \emph{valid} $T$-walk.
If $\Delta_j$ is not a pending triangle, then $\beta_j$ is either a unique directed path within a triangle $\Delta_j$ of $T$ connecting the terminal point of $\alpha_j$ to the initial point of $\alpha_{j+1}$, if these points are different, or it cannot be constructed, if these points coincide. However, if $\Delta_j$ is a pending triangle, then $\beta_j$ can wind around the orbifold point an arbitrary number of times. This topological ambiguity requires the explicit inclusion of the homotopy equivalence condition in order to ensure the $T$-walk is well-defined. Such a condition was not necessary in the surface case in \cite{ezgieminecluster2024}, although it can be seen as a consequence of the definition. 

Table \ref{tab:HomotopicWinding} contains examples of what $\beta$-steps inside of a pending triangle look like. Each of these $\beta$ steps should have the same winding number at this orbifold point as the arc they are associated to. In this case, it is most convenient to break the base point of the pending arc into two vertices. This is similar to the standard treatment of self-folded triangles \cite{FST-I, ezgieminecluster2024}. With this notation, an $\alpha$ step along a pending arc goes between distinct vertices. 

\begin{table}[t]
\centering
\renewcommand{\arraystretch}{1} 
\begin{tabular}{|c|c|c|c|}
    \hline
    \rule{0pt}{0pt}
    \begin{tikzpicture}[scale=0.7, transform shape]
        \draw[thick] (-0.2,-3) .. controls (-2.5,0) and (2.5,0) .. (0.2,-3);
        \filldraw (-0.2,-3) circle (3pt);
        \filldraw (0.2,-3) circle (3pt);
        \node[scale=2, thick] at (0,-1.5) {$\times$};
        \draw[line width = 1.2pt, color=orange,out=70,in=-90,looseness=1] (0.25,-2.8) to (0.6,-1.5);
        \draw[line width = 1.2pt, color=orange,out=90,in=90,looseness=1.5] (0.6,-1.5) to (-0.6,-1.5);
        \draw[line width = 1.2pt, color=orange,out=-90,in=-90,looseness=1.5] (-0.6,-1.5) to (0.4,-1.5);
        \draw[line width = 1.2pt,color=orange,out=90,in=90,looseness=1.5] (0.4,-1.5) to (-0.4,-1.5);
        \draw[line width = 1.2pt, color=orange,out=-90,in=110,->] (-0.4,-1.5) to (0.1,-2.8);
    \end{tikzpicture}
     &
        \begin{tikzpicture}[scale=0.7, transform shape]
        \draw[thick] (-0.2,-3) .. controls (-2.5,0) and (2.5,0) .. (0.2,-3);
        \filldraw (-0.2,-3) circle (3pt);
        \filldraw (0.2,-3) circle (3pt);
        \node[scale=2, thick] at (0,-1.5) {$\times$};
        \draw[line width = 1.2pt, color=orange,out=70,in=-90,looseness=1] (-0.15,-2.8) to (0.6,-1.5);
        \draw[line width = 1.2pt, color=orange,out=90,in=90,looseness=1.5] (0.6,-1.5) to (-0.6,-1.5);
        \draw[line width = 1.2pt, color=orange,out=-90,in=-90,looseness=1.5] (-0.6,-1.5) to (0.4,-1.5);
        \draw[line width = 1.2pt,color=orange,out=90,in=90,looseness=1.5] (0.4,-1.5) to (-0.4,-1.5);
        \draw[line width = 1.2pt, color=orange,out=-90,in=110,->] (-0.4,-1.5) to (-0.25,-2.8);
    \end{tikzpicture}
    & 
    \begin{tikzpicture}[scale=0.7, transform shape]
        \draw[thick] (-0.2,-3) .. controls (-2.5,0) and (2.5,0) .. (0.2,-3);
        \filldraw (-0.2,-3) circle (3pt);
        \filldraw (0.2,-3) circle (3pt);
        \node[scale=2, thick] at (0,-1.5) {$\times$};
        \draw[line width = 1.2pt, color=orange,out=80,in=-90] (-0.2,-2.8) to (0.4,-1.5);
        \draw[line width = 1.2pt, color=orange,out=90,in=90,looseness=1.5] (0.4,-1.5) to (-0.4,-1.5);
        \draw[line width = 1.2pt, color=orange,out=-90,in=100,->] (-0.4,-1.5) to (0.2,-2.8);
    \end{tikzpicture}
    &
        \begin{tikzpicture}[scale=0.7, transform shape]
        \draw[thick] (-0.2,-3) .. controls (-2.5,0) and (2.5,0) .. (0.2,-3);
        \filldraw (-0.2,-3) circle (3pt);
        \filldraw (0.2,-3) circle (3pt);
        \node[scale=2, thick] at (0,-1.5) {$\times$};
        \draw[line width = 1.2pt, color=orange,out=70,in=-90,looseness=1] (0.25,-2.8) to (0.6,-1.5);
        \draw[line width = 1.2pt, color=orange,out=90,in=90,looseness=1.5] (0.6,-1.5) to (-0.6,-1.5);
        \draw[line width = 1.2pt, color=orange,out=-90,in=-90,looseness=1.5] (-0.6,-1.5) to (0.4,-1.5);
        \draw[line width = 1.2pt,color=orange,out=90,in=90,looseness=1.5] (0.4,-1.5) to (-0.4,-1.5);
        \draw[line width = 1.2pt, color=orange,out=-90,in=110,->] (-0.4,-1.5) to (-0.25,-2.8);
    \end{tikzpicture}
     \\ 
      \highlight{$U_k\rho$} & \highlight{$U_k\rho$}  & \highlight{$U_{k-1}\rho$} &  \highlight{$U_{k+1}\rho$}
      \\ \hline
    \end{tabular}
\caption{Types of $\beta$-steps inside a pending triangle and their associated weights.  }
\label{tab:HomotopicWinding}
\end{table}

If $\tau_{i_j} = \tau_{i_{j+1}}$ is a pending arc and $\gamma$ winds $0 < k < \mathbf{p}-2$ times around the enclosed orbifold point, then we can find valid $\beta$-steps to connect all four orientations of $\tau_{i_j}$ and $\tau_{i_{j+1}}$. We find it convenient to suppress drawing $\beta_j$ and to draw $\alpha_j,\beta_j,\alpha_{j+1}$ as a compound step, as shown on the middle row and bottom left entry of Table \ref{table:Twalkpending}. As shown in the top row, if $k = 0$, then one of the four orientations is not possible, and if  $k = \mathbf{p}-2$, any orientation $\vec{v}$ with $v_{j} = 1$ and $v_{j+1} = 0$ is not possible. If $\gamma$ is a corner arc, we can similarly consider the step $\beta_0$ to wind inside of $\Delta_0$, and we show the two possible compound steps in the bottom row of Table \ref{table:Twalkpending}.

\begin{remark}
Another perspective one can take for the $T$-walks discussed in this section is that they are the amalgamation of the concept of $T$-walks in \cite{ezgieminecluster2024} with weak $T$-paths from \cite{ccanakcci2021friezes} in a covering space of the orbifold. We require the notion of a weak $T$-path since the lift of the triangulation of $T$ may not fully triangulate the covering space.  From this vantage point, a $\beta$-step inside pending triangles is the projection of a $\beta$-step inside a regular $\mathbf{p}$-gon in the cover.
\end{remark}

\begin{table}[H]
\centering
\renewcommand{\arraystretch}{1.5} 
\begin{tabular}{|c|c|c|}
    \hline
    \begin{tikzpicture}[scale=0.65, transform shape]
        \draw[thick, postaction={decorate, decoration={markings, mark=at position 0.12 with {\arrow[scale=1.5, black!30!green]{>}}, mark=at position 0.9 with {\arrow[scale=1.5, black!30!green]{<}}}}] (0,-3) .. controls (-2.5,0) and (2.5,0) .. (0,-3);
        \filldraw (0,-3) circle (3pt);
        \node[scale=2, thick] at (0,-1.5) {$\times$};
        \draw[smooth, tension=1, ->, line width=1.2pt,color=orange] plot coordinates {(-1,-2.5) (0,-2) (1,-2.5)};
    \end{tikzpicture}%
    &
    \begin{tikzpicture}[scale=0.65, transform shape]
        \draw[thick, postaction={decorate, decoration={markings, mark=at position 0.12 with {\arrow[scale=1.5, black!30!green]{<}}, mark=at position 0.92 with {\arrow[scale=1.5, black!30!green]{>}}}}] (0,-3) .. controls (-2.5,0) and (2.5,0) .. (0,-3);
        \filldraw (0,-3) circle (3pt);
        \node[scale=2, thick] at (0,-1.5) {$\times$};
        \draw[smooth, tension=1, ->, line width=1.2pt,color=orange] plot coordinates {(-1,-2.5) (0,-2) (1,-2.5)};
    \end{tikzpicture}%
    &
    \begin{tikzpicture}[scale=0.65, transform shape]
        \draw[thick, postaction={decorate, decoration={markings, mark=at position 0.12 with {\arrow[scale=1.5, black!30!green]{>}}, mark=at position 0.92 with {\arrow[scale=1.5, black!30!green]{>}}}}] (0,-3) .. controls (-2.5,0) and (2.5,0) .. (0,-3);
        \filldraw (0,-3) circle (3pt);
        \node[scale=2, thick] at (0,-1.5) {$\times$};
        \draw[smooth, tension=1, ->, line width=1.2pt, color=orange] plot coordinates {(-1,-2.5) (0,-2) (1,-2.5)};
    \end{tikzpicture} \\
        \highlight{$\rho$} & \highlight{$\rho$} & \highlight{$U_1\rho$}
    \\ \hline
    \begin{tikzpicture}[scale=0.65, transform shape]
        \draw[thick, postaction={decorate, decoration={markings, mark=at position 0.12 with {\arrow[scale=1.5, black!30!green]{>}}, mark=at position 0.9 with {\arrow[scale=1.5, black!30!green]{<}}}}] (0,-3) .. controls (-2.5,0) and (2.5,0) .. (0,-3);
        \filldraw (0,-3) circle (3pt);
        \node[scale=2, thick] at (0,-1.5) {$\times$};
        \draw[line width = 1.2pt,color=orange,out=0,in=-90] (-1,-2) to (0.4,-1.5);
        \draw[line width = 1.2pt,color=orange,out=90,in=90,looseness=1.5] (0.4,-1.5) to (-0.4,-1.5);
        \draw[line width = 1.2pt,color=orange,out=-90,in=180,->] (-0.4,-1.5) to (1.1,-2);
    \end{tikzpicture}%
    &
    \begin{tikzpicture}[scale=0.65, transform shape]
        \draw[thick, postaction={decorate, decoration={markings, mark=at position 0.12 with {\arrow[scale=1.5, black!30!green]{<}}, mark=at position 0.92 with {\arrow[scale=1.5, black!30!green]{>}}}}] (0,-3) .. controls (-2.5,0) and (2.5,0) .. (0,-3);
        \filldraw (0,-3) circle (3pt);
        \node[scale=2, thick] at (0,-1.5) {$\times$};
        \draw[line width = 1.2pt,color=orange,out=0,in=-90] (-1,-2) to (0.4,-1.5);
        \draw[line width = 1.2pt,color=orange,out=90,in=90,looseness=1.5] (0.4,-1.5) to (-0.4,-1.5);
        \draw[line width = 1.2pt,color=orange,out=-90,in=180,->] (-0.4,-1.5) to (1.1,-2);
    \end{tikzpicture}%
    &
    \begin{tikzpicture}[scale=0.65, transform shape]
        \draw[thick, postaction={decorate, decoration={markings, mark=at position 0.12 with {\arrow[scale=1.5, black!30!green]{>}}, mark=at position 0.92 with {\arrow[scale=1.5, black!30!green]{>}}}}] (0,-3) .. controls (-2.5,0) and (2.5,0) .. (0,-3);
        \filldraw (0,-3) circle (3pt);
        \node[scale=2, thick] at (0,-1.5) {$\times$};
        \draw[line width = 1.2pt,color=orange,out=0,in=-90] (-1,-2) to (0.4,-1.5);
        \draw[line width = 1.2pt,color=orange,out=90,in=90,looseness=1.5] (0.4,-1.5) to (-0.4,-1.5);
        \draw[line width = 1.2pt,color=orange,out=-90,in=180,->] (-0.4,-1.5) to (1.1,-2);
    \end{tikzpicture}%
    \\ 
        \highlight{$U_k\rho$} & \highlight{$U_k\rho$} & \highlight{$U_{k+1}\rho$}
    \\ \hline
    \begin{tikzpicture}[scale=0.65, transform shape]
        \draw[thick, postaction={decorate, decoration={markings, mark=at position 0.12 with {\arrow[scale=1.5, black!30!green]{<}}, mark=at position 0.9 with {\arrow[scale=1.5, black!30!green]{<}}}}] (0,-3) .. controls (-2.5,0) and (2.5,0) .. (0,-3);
        \filldraw (0,-3) circle (3pt);
        \node[scale=2, thick] at (0,-1.5) {$\times$};
        \draw[line width = 1.2pt,color=orange,out=0,in=-90] (-1,-2) to (0.4,-1.5);
        \draw[line width = 1.2pt,color=orange,out=90,in=90,looseness=1.5] (0.4,-1.5) to (-0.4,-1.5);
        \draw[line width = 1.2pt,color=orange,out=-90,in=180,->] (-0.4,-1.5) to (1.1,-2);
    \end{tikzpicture}%
    &
    \begin{tikzpicture}[scale=0.65, transform shape]
        \draw[thick, postaction={decorate, decoration={markings, mark=at position 0.85 with {\arrow[scale=1.5, black!30!green]{<}}}}] (0,-3) .. controls (-2.5,0) and (2.5,0) .. (0,-3);
        \node[scale=2, thick] at (0,-1.5) {$\times$};
        \draw[line width = 1.2pt, color=orange,out=75,in=-90] (0,-3) to (0.4,-1.5);
        \draw[line width = 1.2pt, color=orange, out=90, in = 90,looseness=1.5] (0.4,-1.5) to (-0.4,-1.5);
        \draw[line width = 1.2pt, color=orange, out=-90, in = -130] (-0.4, -1.5) to (0.6,-1.6);
        \draw[line width = 1.2pt, color=orange, out=50,in=-90,->] (0.6,-1.6) to (1,-0.5);
        \filldraw (0,-3) circle (3pt);
    \end{tikzpicture}%
    &
    \begin{tikzpicture}[scale=0.65, transform shape]
        \draw[thick, postaction={decorate, decoration={markings, mark=at position 0.85 with {\arrow[scale=1.5, black!30!green]{>}}}}] (0,-3) .. controls (-2.5,0) and (2.5,0) .. (0,-3);
        \node[scale=2, thick] at (0,-1.5) {$\times$};
        \draw[line width = 1.2pt, color=orange,out=75,in=-90] (0,-3) to (0.4,-1.5);
        \draw[line width = 1.2pt, color=orange, out=90, in = 90,looseness=1.5] (0.4,-1.5) to (-0.4,-1.5);
        \draw[line width = 1.2pt, color=orange, out=-90, in = -130] (-0.4, -1.5) to (0.6,-1.6);
        \draw[line width = 1.2pt, color=orange, out=50,in=-90,->] (0.6,-1.6) to (1,-0.5);
        \filldraw (0,-3) circle (3pt);
    \end{tikzpicture}%
    \\ 
        \highlight{$U_{k-1}\rho$} & \highlight{$U_{k-1}\rho$} & \highlight{$U_{k}\rho$}
    \\ \hline
\end{tabular}
\caption{Orientations on pending arcs and corresponding $\beta_i$'s}
\label{table:Twalkpending}
\end{table}

Each $T$-walk $T_{\vec{v}}$ is assigned an $x$-weight and a $y$-weight, defined as:
\[
x(T_{\vec{v}}) := \frac{\prod x_{\beta_i}}{\prod x_{\alpha_i}}, \qquad y(T_{\vec{v}}) := \prod_{v_i=1} \Phi(h_{\alpha_i}),
\]
where we set $x_{U_k(\lambda_\boldp)\rho} := U_k(\lambda_\boldp) x_{\rho}$ and let $\Phi$ be the specialization map from \Cref{def:Phi}. Boundary arcs are assigned a weight of $1$, and thus their contributions to the products are effectively ignored. The \emph{$T$-walk expansion} of an arc $\gamma$ with respect to a triangulation $T$ is then defined as:
\[
\mathfrak{W}(\gamma, T) = \sum_{T_{\vec{v}} \in \mathrm{TW}(\gamma, T)} x(T_{\vec{v}}) \, y(T_{\vec{v}}).
\]
When the triangulation is clear from context, we abbreviate this as $\mathfrak{W}(\gamma)$. In \Cref{sec:Equivalence}, we will show that this combinatorial expansion coincides with the cluster variable expansion $x_{\gamma}$.

\begin{example} The arc $\gamma$ from \Cref{ex:1poset} has three crossings: $b$, $b$, and $a$. Out of the eight possible orientations of these crossings, only six correspond to valid $T$-walks. The minimal direction $(0,0,0)$ is shown in \Cref{ex:forTwalk}. For the six valid orientations, each corresponding $T$-walk has the form $(\beta_0,\alpha_1,\beta_1,\alpha_2,\beta_2,\alpha_3,\beta_3)$. Below, we explicitly list the valid $T$-walks and their corresponding weights.

\begin{enumerate}
    \item[(i)] For $\vec{v}_1=(0,0,0)$, the $T$-walk is $(b,\rho,U_1(\lambda_{\mathbf{p}})\rho,\rho,\rho,a,c)$ with weights $x_{\vec{v}_1}=\frac{U_1(\lambda_{\mathbf{p}})x_bx_c}{x_a}$, $ y_{\vec{v}_1}=1$.
    \item[(ii)] For $\vec{v}_2=(1,0,0)$, the $T$-walk is $(a,\rho,U_0(\lambda_{\mathbf{p}})\rho,\rho,\rho,a,c)$ with weights $x_{\vec{v}_2}={U_0(\lambda_{\mathbf{p}})x_c}$ and $y_{\vec{v}_2}=y_\rho$.
    \item[(iii)] For $\vec{v}_3=(0,0,1)$, the $T$-walk is $(b,\rho,U_1(\lambda_{\mathbf{p}})\rho,\rho,b,a,d)$ with weights $x_{\vec{v}_3}=\frac{U_1(\lambda_{\mathbf{p}})x_b^2x_d}{x_ax_\rho}$ and  $y_{\vec{v}_c}=y_a$.
    \item[(iv)] For $\vec{v}_4=(1,0,1)$, the $T$-walk is $(a,\rho,U_0(\lambda_{\mathbf{p}})\rho,\rho,b,a,d)$ with weights $x_{\vec{v}_4}=\frac{U_0(\lambda_{\mathbf{p}})x_bx_d}{x_\rho}$ and $y_{\vec{v}_4}=y_ay_\rho$.
    \item[(v)] For $\vec{v}_5=(0,1,1)$, the $T$-walk is $(b,\rho,U_2(\lambda_{\mathbf{p}})\rho,\rho,a,a,d)$ with weights $x_{\vec{v}_5}=\frac{U_2(\lambda_{\mathbf{p}})x_bx_d}{x_\rho}$ and $y_{\vec{v}_5}=y_ay_\rho$.
    \item[(vi)] For $\vec{v}_6=(1,1,1)$, the $T$-walk is $(a,\rho,U_1(\lambda_{\mathbf{p}})\rho,\rho,a,a,d)$ with weights $x_{\vec{v}_6}=\frac{U_1(\lambda_{\mathbf{p}})x_ax_d}{x_\rho}$ and $y_{\vec{v}_6}=y_ay_\rho^2$.
\end{enumerate}

\begin{figure}[H]
    \centering
\scalebox{.65}{
\begin{tikzpicture}[scale=0.5, transform shape]

    \draw[thick] (0,0) circle (3cm);

    \node[below, scale=2] at (0,-3.2) {$v_1$}; 
    \node[above, scale=2] at (0,3.2) {$v_2$};  
    \node[left, scale=2] at (-3.2,0) {$v_3$};  

    \draw[very thick, blue, postaction={decorate, decoration={markings, mark=at position 0.15 with {\arrow[scale=1.5, blue]{<}}}}] (-3,0) .. controls (3,4) and (2,-2.5) .. (0,-3);

    \draw[very thick, blue, postaction={decorate, decoration={markings, mark=at position 0.15 with {\arrow[scale=1.5, blue]{>}}, mark=at position 0.88 with {\arrow[scale=1.5, blue]{<}}}}] (0,-3) .. controls (-2.5,2) and (2.5,2) .. (0,-3);

    \draw[in=-90,out=-20,looseness=1, orange, line width = 1.2pt] (-3,0) to (0.5,0);
    \draw[in=90,out=90,looseness=1.5, orange, line width = 1.2pt] (0.5,0) to (-0.5,0);
    \draw[in=-100,out=-90,looseness=1, orange, line width = 1.2pt,  ->] (-0.5,0) to (2,0);
    \draw[in=-40,out=80,looseness=1, orange, line width = 1.2pt] (2,0) to (0,3);

    \filldraw (0,-3) circle (3pt); 
    \filldraw (0,3) circle (3pt);  
    \filldraw (-3,0) circle (3pt); 

    \node[scale=2, thick] at (0,0) {$\times$};
    \node[scale=1.5] at (-1.4,1.2) {$a$};
    \node[scale=1.5] at (1,0) {$\rho$};
    \node[scale=1.5] at (-2.3,-2.6) {$b$};
    \node[scale=1.5] at (-2.3,2.6) {$c$};
    \node[scale=1.5] at (3.2,0) {$d$};

    \node[red, scale=1.5] at (1.3,1.4) {$\gamma$};

    \node[scale=2] at (0,-4.5) {$\vec{v}_1=(0,0,0)$};
    \node[scale=2] at (0,-5.5) {$(b,\rho,U_1\rho,\rho,\rho,a,c)$};
\end{tikzpicture}

\begin{tikzpicture}[scale=0.5, transform shape]

    \draw[thick] (0,0) circle (3cm);

    \node[below, scale=2] at (0,-3.2) {$v_1$}; 
    \node[above, scale=2] at (0,3.2) {$v_2$};  
    \node[left, scale=2] at (-3.2,0) {$v_3$};  


    \draw[very thick, blue, postaction={decorate, decoration={markings, mark=at position 0.15 with {\arrow[scale=1.5, blue]{<}} }}] (-3,0) .. controls (3,4) and (2,-2.5) .. (0,-3);

    \draw[very thick, blue, postaction={decorate, decoration={markings, mark=at position 0.15 with {\arrow[scale=1.5, blue]{<}}, mark=at position 0.88 with {\arrow[scale=1.5, blue]{<}}}}] (0,-3) .. controls (-2.5,2) and (2.5,2) .. (0,-3);

    \draw[in=-90,out=-20,looseness=1, orange, line width = 1.2pt] (-3,0) to (0.5,0);
    \draw[in=90,out=90,looseness=1.5, orange, line width = 1.2pt] (0.5,0) to (-0.5,0);
    \draw[in=-100,out=-90,looseness=1, orange, line width = 1.2pt,  ->] (-0.5,0) to (2,0);
    \draw[in=-40,out=80,looseness=1, orange, line width = 1.2pt] (2,0) to (0,3);

    \filldraw (0,-3) circle (3pt); 
    \filldraw (0,3) circle (3pt);  
    \filldraw (-3,0) circle (3pt); 

    \node[scale=2, thick] at (0,0) {$\times$};
    \node[scale=1.5] at (-1.4,1.2) {$a$};
    \node[scale=1.5] at (1,0) {$\rho$};
    \node[scale=1.5] at (-2.3,-2.6) {$b$};
    \node[scale=1.5] at (-2.3,2.6) {$c$};
    \node[scale=1.5] at (3.2,0) {$d$};

    \node[red, scale=1.5] at (1.3,1.4) {$\gamma$};

    \node[scale=2] at (0,-4.5) {$\vec{v}_2=(1,0,0)$};
    \node[scale=2] at (0,-5.5) {$(a,\rho,U_0\rho,\rho,\rho,a,c)$};
\end{tikzpicture}

\begin{tikzpicture}[scale=0.5, transform shape]

    \draw[thick] (0,0) circle (3cm);

    \node[below, scale=2] at (0,-3.2) {$v_1$}; 
    \node[above, scale=2] at (0,3.2) {$v_2$};  
    \node[left, scale=2] at (-3.2,0) {$v_3$};  


    \draw[very thick, blue, postaction={decorate, decoration={markings, mark=at position 0.85 with {\arrow[scale=1.5, blue]{>}}}}] (-3,0) .. controls (3,4) and (2,-2.5) .. (0,-3);

    \draw[very thick, blue, postaction={decorate, decoration={markings, mark=at position 0.15 with {\arrow[scale=1.5, blue]{>}}, mark=at position 0.88 with {\arrow[scale=1.5, blue]{<}}}}] (0,-3) .. controls (-2.5,2) and (2.5,2) .. (0,-3);

    \draw[in=-90,out=-20,looseness=1, orange, line width = 1.2pt] (-3,0) to (0.5,0);
    \draw[in=90,out=90,looseness=1.5, orange, line width = 1.2pt] (0.5,0) to (-0.5,0);
    \draw[in=-100,out=-90,looseness=1, orange, line width = 1.2pt,  ->] (-0.5,0) to (2,0);
    \draw[in=-40,out=80,looseness=1, orange, line width = 1.2pt] (2,0) to (0,3);

    \filldraw (0,-3) circle (3pt); 
    \filldraw (0,3) circle (3pt);  
    \filldraw (-3,0) circle (3pt); 

    \node[scale=2, thick] at (0,0) {$\times$};
    \node[scale=1.5] at (-1.4,1.2) {$a$};
    \node[scale=1.5] at (1,0) {$\rho$};
    \node[scale=1.5] at (-2.3,-2.6) {$b$};
    \node[scale=1.5] at (-2.3,2.6) {$c$};
    \node[scale=1.5] at (3.2,0) {$d$};

    \node[red, scale=1.5] at (1.3,1.4) {$\gamma$};

    \node[scale=2] at (0,-4.5) {$\vec{v}_3=(0,0,1)$};
    \node[scale=2] at (0,-5.5) {$(b,\rho,U_1\rho,\rho,b,a,d)$};
\end{tikzpicture}

\begin{tikzpicture}[scale=0.5, transform shape]

    \draw[thick] (0,0) circle (3cm);

    \node[below, scale=2] at (0,-3.2) {$v_1$}; 
    \node[above, scale=2] at (0,3.2) {$v_2$};  
    \node[left, scale=2] at (-3.2,0) {$v_3$};  


    \draw[very thick, blue, postaction={decorate, decoration={markings, mark=at position 0.85 with {\arrow[scale=1.5, blue]{>}}}}] (-3,0) .. controls (3,4) and (2,-2.5) .. (0,-3);

    \draw[very thick, blue, postaction={decorate, decoration={markings, mark=at position 0.15 with {\arrow[scale=1.5, blue]{<}}, mark=at position 0.88 with {\arrow[scale=1.5, blue]{<}}}}] (0,-3) .. controls (-2.5,2) and (2.5,2) .. (0,-3);

    \draw[in=-90,out=-20,looseness=1, orange, line width = 1.2pt] (-3,0) to (0.5,0);
    \draw[in=90,out=90,looseness=1.5, orange, line width = 1.2pt] (0.5,0) to (-0.5,0);
    \draw[in=-100,out=-90,looseness=1, orange, line width = 1.2pt,  ->] (-0.5,0) to (2,0);
    \draw[in=-40,out=80,looseness=1, orange, line width = 1.2pt] (2,0) to (0,3);

    \filldraw (0,-3) circle (3pt); 
    \filldraw (0,3) circle (3pt);  
    \filldraw (-3,0) circle (3pt); 

    \node[scale=2, thick] at (0,0) {$\times$};
    \node[scale=1.5] at (-1.4,1.2) {$a$};
    \node[scale=1.5] at (1,0) {$\rho$};
    \node[scale=1.5] at (-2.3,-2.6) {$b$};
    \node[scale=1.5] at (-2.3,2.6) {$c$};
    \node[scale=1.5] at (3.2,0) {$d$};

    \node[red, scale=1.5] at (1.3,1.4) {$\gamma$};

    \node[scale=2] at (0,-4.5) {$\vec{v}_4=(1,0,1)$};
    \node[scale=2] at (0,-5.5) {$(a,\rho,U_0\rho,\rho,b,a,d)$};
\end{tikzpicture}

\begin{tikzpicture}[scale=0.5, transform shape]

    \draw[thick] (0,0) circle (3cm);

    \node[below, scale=2] at (0,-3.2) {$v_1$}; 
    \node[above, scale=2] at (0,3.2) {$v_2$};  
    \node[left, scale=2] at (-3.2,0) {$v_3$};  

    \draw[very thick, blue, postaction={decorate, decoration={markings, mark=at position 0.85 with {\arrow[scale=1.5, blue]{>}}}}] (-3,0) .. controls (3,4) and (2,-2.5) .. (0,-3);

    \draw[very thick, blue, postaction={decorate, decoration={markings, mark=at position 0.15 with {\arrow[scale=1.5, blue]{>}}, mark=at position 0.88 with {\arrow[scale=1.5, blue]{>}}}}] (0,-3) .. controls (-2.5,2) and (2.5,2) .. (0,-3);

    \draw[in=-90,out=-20,looseness=1, orange, line width = 1.2pt] (-3,0) to (0.5,0);
    \draw[in=90,out=90,looseness=1.5, orange, line width = 1.2pt] (0.5,0) to (-0.5,0);
    \draw[in=-100,out=-90,looseness=1, orange, line width = 1.2pt,  ->] (-0.5,0) to (2,0);
    \draw[in=-40,out=80,looseness=1, orange, line width = 1.2pt] (2,0) to (0,3);

    \filldraw (0,-3) circle (3pt); 
    \filldraw (0,3) circle (3pt);  
    \filldraw (-3,0) circle (3pt); 

    \node[scale=2, thick] at (0,0) {$\times$};
    \node[scale=1.5] at (-1.4,1.2) {$a$};
    \node[scale=1.5] at (1,0) {$\rho$};
    \node[scale=1.5] at (-2.3,-2.6) {$b$};
    \node[scale=1.5] at (-2.3,2.6) {$c$};
    \node[scale=1.5] at (3.2,0) {$d$};

    \node[red, scale=1.5] at (1.3,1.4) {$\gamma$};

    \node[scale=2] at (0,-4.5) {$\vec{v}_5=(0,1,1)$};
    \node[scale=2] at (0,-5.5) {$(b,\rho,U_2\rho,\rho,a,a,d)$};
\end{tikzpicture}

\begin{tikzpicture}[scale=0.5, transform shape]

    \draw[thick] (0,0) circle (3cm);

    \node[below, scale=2] at (0,-3.2) {$v_1$}; 
    \node[above, scale=2] at (0,3.2) {$v_2$};  
    \node[left, scale=2] at (-3.2,0) {$v_3$};  


    \draw[very thick, blue, postaction={decorate, decoration={markings, mark=at position 0.85 with {\arrow[scale=1.5, blue]{>}}}}] (-3,0) .. controls (3,4) and (2,-2.5) .. (0,-3);

    \draw[very thick, blue, postaction={decorate, decoration={markings, mark=at position 0.15 with {\arrow[scale=1.5, blue]{<}}, mark=at position 0.88 with {\arrow[scale=1.5, blue]{>}}}}] (0,-3) .. controls (-2.5,2) and (2.5,2) .. (0,-3);

    \draw[in=-90,out=-20,looseness=1, orange, line width = 1.2pt] (-3,0) to (0.5,0);
    \draw[in=90,out=90,looseness=1.5, orange, line width = 1.2pt] (0.5,0) to (-0.5,0);
    \draw[in=-100,out=-90,looseness=1, orange, line width = 1.2pt,  ->] (-0.5,0) to (2,0);
    \draw[in=-40,out=80,looseness=1, orange, line width = 1.2pt] (2,0) to (0,3);

    \filldraw (0,-3) circle (3pt); 
    \filldraw (0,3) circle (3pt);  
    \filldraw (-3,0) circle (3pt); 

    \node[scale=2, thick] at (0,0) {$\times$};
    \node[scale=1.5] at (-1.4,1.2) {$a$};
    \node[scale=1.5] at (1,0) {$\rho$};
    \node[scale=1.5] at (-2.3,-2.6) {$b$};
    \node[scale=1.5] at (-2.3,2.6) {$c$};
    \node[scale=1.5] at (3.2,0) {$d$};

    \node[red, scale=1.5] at (1.3,1.4) {$\gamma$};

    \node[scale=2] at (0,-4.5) {$\vec{v}_6=(1,1,1)$};
    \node[scale=2] at (0,-5.5) {$(a,\rho,U_1\rho,\rho,a,a,d)$};
\end{tikzpicture}
}
    \caption{The complete set of valid $T$-walks for $\gamma$.}
    \label{fig:twalks}
\end{figure}
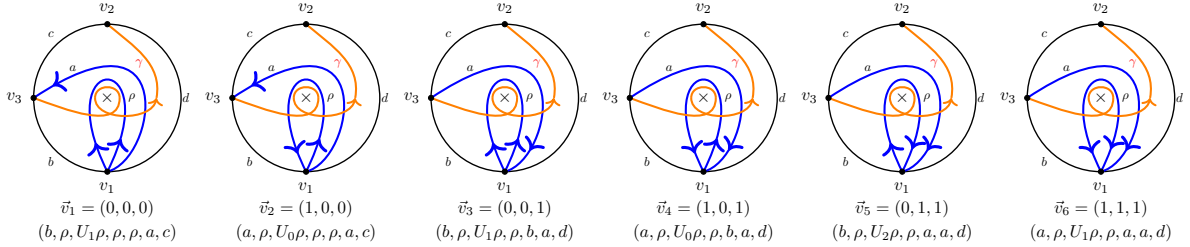

Following our $T$-walk expansion formula, we have that

\[\mathfrak{W}(\gamma, T)=\frac{U_1(\lambda_\boldp)x_bx_c}{x_a}+U_0(\lambda_\boldp)x_cy_\rho+\frac{U_1(\lambda_\boldp)x_b^2x_d}{x_ax_\rho}y_a+\frac{U_0(\lambda_\boldp)x_bx_d}{x_\rho}y_ay_\rho+\frac{U_2(\lambda_\boldp)x_bx_d}{x_\rho}y_ay_\rho+\frac{U_1(\lambda_\boldp)x_ax_d}{x_\rho}y_ay_\rho^2.\]

\end{example}

\subsection{\texorpdfstring{$T$-walks for Notched Arcs}{T-walks for Notched Arcs}}

Next, we construct $T$-walks for notched arcs. Once again, we represent notched arcs with hooks and consequently, we construct $T$-walks for the hook. Examples for surfaces are provided in Section A.3 of \cite{ezgieminecluster2024}.

When working with a singly-notched arc $\gamma^{(p)}$, we must consider a vector of orientations for both the $d$ arcs crossed by the underlying plain arc and the $m$ spokes incident to $p$ that are crossed by the hook. We define $\vec{v}=(\vec{v}_s,\vec{v}_a) \in \{0,1\}^m \times \{0,1\}^d$. The vector $\vec{v}_s$ will record the orientations on the spokes while the vector $\vec{v}_a$ will, as before, record the orientations on arcs crossed by $\gamma^0$. We fix the minimal orientation on the spokes $\sigma_i$ to be pointing outward from the puncture $p$ for crossings in both $\vec{v}_s$ and $\vec{v}_a$.

\begin{definition}\label{def:twalkNOTCHED}
Let $\mathcal{O}$ be a punctured orbifold with an initial triangulation $T$. Let $\gamma^{(p)}$ be an arc notched at the puncture $p = s(\gamma)$. Suppose the underlying arc $\gamma^0$ has $d$ intersections with $T$ at arcs $\tau_{i_1}, \dots, \tau_{i_d}$, and that $p$ is incident to $m$ spokes $\sigma_1, \dots, \sigma_m$ in $T$. 
Consider a vector $\vec{v}=(\vec{v}_s, \vec{v}_a) \in \{0,1\}^m \times \{0,1\}^d$. For each $j$, let $\alpha_j$ and $\alpha_{\sigma_j}$ be oriented copies of $\tau_{i_j}$ and $\sigma_j$, respectively, with orientations determined by the entries $v_j \in \vec{v}_a$ and $v_{\sigma_j} \in \vec{v}_s$. Let $\alpha_{\sigma_j}^-$ denote the arc $\sigma_j$ with the orientation opposite to $\alpha_{\sigma_j}$.

If there exists a sequence of directed paths $\beta_{\sigma_1}, \dots, \beta_{\sigma_m}, \beta_1, \dots, \beta_d$, each contained within a single triangle of $T$, such that either the concatenation
\[
\alpha_{\sigma_m} \beta_{\sigma_m} \alpha_{\sigma_{m-1}} \beta_{\sigma_{m-1}} \cdots \alpha_{\sigma_1} \beta_{\sigma_1} \alpha_1 \beta_1 \alpha_2 \beta_2 \cdots \alpha_d \beta_d
\]
or
\[
\alpha_{\sigma_1}^{-} \beta_{\sigma_1} \alpha_{\sigma_{2}}^- \beta_{\sigma_{2}} \cdots \alpha_{\sigma_m} \beta_{\sigma_m} \alpha_1 \beta_1 \alpha_2 \beta_2 \cdots \alpha_d \beta_d
\]
forms a walk from $s(\gamma)$ to $t(\gamma)$ that is homotopically equivalent to $\gamma$, then we call the resulting walk a \emph{$T$-walk} $T_{\vec{v}}$. For a doubly-notched arc $\gamma^{(p,q)}$, this construction is applied analogously at the endpoint $q$.
\end{definition}

We highlight that if $\gamma$ is notched at its starting point (ending point), then any $T$-walk associated to $\gamma$ starts (ends) on an $\alpha$ step. 

Definition \ref{def:twalkNOTCHED} will yield two $T$-walks for some orientation vectors, but these $T$-walks are related by reversing the initial segment $(\alpha_{\sigma_m},\ldots,\alpha_{\sigma_1})$ and contribute the same weight, so we regard them as the same. Notice that the two possible concatenations correspond to the two possible hooked arcs associated to $\gamma^{(p)}$. In the doubly-notched case, one must consider four possible concatenations associated to the four possible hooked arcs. This flexibility is necessary for the extreme cases. If $\vec{v_s} = \vec{0}$ or  $\vec{v_s} = \vec{1}$, then the choice of hook orientation is strictly constrained. 

\begin{remark}
    Note that in the construction of loop graphs or oriented posets, we are free to choose either a clockwise or counterclockwise hook at any notched endpoint of an arc. The orientation constraints described above are unique to the construction of $T$-walks for notched arcs.
\end{remark}

\begin{example}\label{ex:TWalkFromGoodMatchingOfLoopGraph}
 There is a well-known method to go between perfect matchings of snake graphs and $T$-walks. This is essentially done by walking on the edges in the perfect matching and the diagonal labels of the graph. Here, we explore what applying this process to loop graphs tells us about $T$-walks.
 
 \Cref{fig:minimalnotched} illustrates two valid good matchings, denoted with thick red edges, for the loop graph corresponding to the notched corner arc shown in \Cref{ex:NotchCorner}. The two bullet points denote the two vertices that are associated to $s(\gamma)$. These vertices are also identified in the graph, implying both can be a starting point. 
 
 On the left, we display the minimal matching, associated to the $T$-walk with orientation vector $\vec{0}$. A $T$-walk is constructed from such a matching by alternating between edges in the matching (representing $\beta$ steps) and the dashed diagonal labels (representing $\alpha$ steps). In this configuration, we can construct a walk that begins at the southwest vertex of the initial tile and terminates at the northeast vertex of the final tile.

Compare this situation with the good matching on the right. We can again construct a $T$-walk between the same two vertices. However, we could also begin at the other vertex labelled with a bullet, and travel left, taking diagonals and edges $U_2(\lambda_\boldp)\rho, U_3(\lambda_\boldp)\rho, {\rho}, b,\ldots$. We will eventually reach the lower bullet. Since the two bulleted vertices are identified we can then take the edge $U_1(\lambda_\boldp)\rho$ at the upper bullet and proceed to the northeast vertex of the final tile. The relationship between these two $T$-walks is exactly that described in Definition \ref{def:twalkNOTCHED}.

    \begin{table}[t]
\centering
\renewcommand{\arraystretch}{1.5}
\begin{tabular}{|c|c|c|}
    \hline
    \begin{tikzpicture}[scale=0.6, transform shape]
        \draw[thick, postaction={decorate, decoration={markings, 
            mark=at position 0.3 with {\arrow[scale=1.5, black!30!green]{>}}, 
            mark=at position 0.7 with {\arrow[scale=1.5, black!30!green]{<}}, 
            mark=at position 0.9 with {\arrow[scale=1.5, black!30!green]{<}}}}] (0,-3) .. controls (-2.5,0) and (2.5,0) .. (0,-3);
        \filldraw (0,-3) circle (3pt);
        \node[scale=2, thick] at (0,-1.5) {$\times$};
        \draw[line width = 1.2pt, orange,out=0,in=0,looseness=1.5] (0,-2.6) to (0,-3.3);
        \draw[line width = 1.2pt, orange, out=180,in=-120,looseness=1.5] (0,-3.3) to (-0.2,-2.3);
        \draw[line width = 1.2pt, orange, out=60,in=-90] (-0.2,-2.3) to (0.35,-1.5);
        \draw[line width = 1.2, orange, out=90,in=90,looseness=1.75] (0.35,-1.5) to (-0.35,-1.5);
        \draw[line width = 1.2, orange, out=-90,in=180,->] (-0.35,-1.5) to (1.2,-2);
        \node at (0.8,-2.5) {$1$};
        \node at (1.1,-1.2) {$3$};
        \node at (-1.2,-1.5) {$2$};
    \end{tikzpicture}
    &
    \begin{tikzpicture}[scale=0.6, transform shape]
        \draw[thick, postaction={decorate, decoration={markings, 
            mark=at position 0.3 with {\arrow[scale=1.5, black!30!green]{<}}, 
            mark=at position 0.7 with {\arrow[scale=1.5, black!30!green]{<}}, 
            mark=at position 0.9 with {\arrow[scale=1.5, black!30!green]{<}}}}] (0,-3) .. controls (-2.5,0) and (2.5,0) .. (0,-3);
        \filldraw (0,-3) circle (3pt);
        \node[scale=2, thick] at (0,-1.5) {$\times$};
        \draw[line width = 1.2pt, orange,out=0,in=0,looseness=1.5] (0,-2.6) to (0,-3.3);
        \draw[line width = 1.2pt, orange, out=180,in=-120,looseness=1.5] (0,-3.3) to (-0.2,-2.3);
        \draw[line width = 1.2pt, orange, out=60,in=-90] (-0.2,-2.3) to (0.35,-1.5);
        \draw[line width = 1.2, orange, out=90,in=90,looseness=1.75] (0.35,-1.5) to (-0.35,-1.5);
        \draw[line width = 1.2, orange, out=-90,in=180,->] (-0.35,-1.5) to (1.2,-2);
        \node at (0.8,-2.5) {$1$};
        \node at (1.1,-1.2) {$3$};
        \node at (-1.2,-1.5) {$2$};

    \end{tikzpicture}
    &
     \begin{tikzpicture}[scale=0.6, transform shape]
        \draw[thick, postaction={decorate, decoration={markings, 
            mark=at position 0.3 with {\arrow[scale=1.5, black!30!green]{<}}, 
            mark=at position 0.7 with {\arrow[scale=1.5, black!30!green]{<}}, 
            mark=at position 0.9 with {\arrow[scale=1.5, black!30!green]{>}}}}] (0,-3) .. controls (-2.5,0) and (2.5,0) .. (0,-3);
        \filldraw (0,-3) circle (3pt);
        \node[scale=2, thick] at (0,-1.5) {$\times$};
        \draw[line width = 1.2pt, orange,out=0,in=0,looseness=1.5] (0,-2.6) to (0,-3.3);
        \draw[line width = 1.2pt, orange, out=180,in=-120,looseness=1.5] (0,-3.3) to (-0.2,-2.3);
        \draw[line width = 1.2pt, orange, out=60,in=-90] (-0.2,-2.3) to (0.35,-1.5);
        \draw[line width = 1.2, orange, out=90,in=90,looseness=1.75] (0.35,-1.5) to (-0.35,-1.5);
        \draw[line width = 1.2, orange, out=-90,in=180,->] (-0.35,-1.5) to (1.2,-2);
        \node at (0.9,-2.5) {$1$};
        \node at (1.1,-1.2) {$3$};
        \node at (-1.2,-1.5) {$2$};
    \end{tikzpicture}
    \\
        \highlight{$U_{k}\rho$} & \highlight{$U_{k-1}\rho$} & \highlight{$U_{k-2}\rho$} 
    \\ \hline
    
    \begin{tikzpicture}[scale=0.6, transform shape]
        \draw[thick, postaction={decorate, decoration={markings, 
            mark=at position 0.3 with {\arrow[scale=1.5, black!30!green]{>}}, 
            mark=at position 0.7 with {\arrow[scale=1.5, black!30!green]{>}}, 
            mark=at position 0.9 with {\arrow[scale=1.5, black!30!green]{<}}}}] (0,-3) .. controls (-2.5,0) and (2.5,0) .. (0,-3);
        \filldraw (0,-3) circle (3pt);
        \node[scale=2, thick] at (0,-1.5) {$\times$};
        \draw[line width = 1.2pt, orange,out=0,in=0,looseness=1.5] (0,-2.6) to (0,-3.3);
        \draw[line width = 1.2pt, orange, out=180,in=-120,looseness=1.5] (0,-3.3) to (-0.2,-2.3);
        \draw[line width = 1.2pt, orange, out=60,in=-90] (-0.2,-2.3) to (0.35,-1.5);
        \draw[line width = 1.2, orange, out=90,in=90,looseness=1.75] (0.35,-1.5) to (-0.35,-1.5);
        \draw[line width = 1.2, orange, out=-90,in=180,->] (-0.35,-1.5) to (1.2,-2);
        \node at (0.8,-2.5) {$1$};
        \node at (1.2,-1.2) {$3$};
        \node at (-1.2,-1.5) {$2$};
    \end{tikzpicture}
    &
    \begin{tikzpicture}[scale=0.6, transform shape]
        \draw[thick, postaction={decorate, decoration={markings, 
            mark=at position 0.3 with {\arrow[scale=1.5, black!30!green]{<}}, 
            mark=at position 0.7 with {\arrow[scale=1.5, black!30!green]{>}}, 
            mark=at position 0.9 with {\arrow[scale=1.5, black!30!green]{<}}}}] (0,-3) .. controls (-2.5,0) and (2.5,0) .. (0,-3);
        \filldraw (0,-3) circle (3pt);
        \node[scale=2, thick] at (0,-1.5) {$\times$};
        \draw[line width = 1.2pt, orange,out=0,in=0,looseness=1.5] (0,-2.6) to (0,-3.3);
        \draw[line width = 1.2pt, orange, out=180,in=-120,looseness=1.5] (0,-3.3) to (-0.2,-2.3);
        \draw[line width = 1.2pt, orange, out=60,in=-90] (-0.2,-2.3) to (0.35,-1.5);
        \draw[line width = 1.2, orange, out=90,in=90,looseness=1.75] (0.35,-1.5) to (-0.35,-1.5);
        \draw[line width = 1.2, orange, out=-90,in=180,->] (-0.35,-1.5) to (1.2,-2);
        \node at (0.8,-2.5) {$1$};
        \node at (1.2,-1.2) {$3$};
        \node at (-1.2,-1.5) {$2$};
    \end{tikzpicture}
    &
    \begin{tikzpicture}[scale=0.6, transform shape]
        \draw[thick, postaction={decorate, decoration={markings, 
            mark=at position 0.3 with {\arrow[scale=1.5, black!30!green]{<}}, 
            mark=at position 0.7 with {\arrow[scale=1.5, black!30!green]{>}}, 
            mark=at position 0.9 with {\arrow[scale=1.5, black!30!green]{>}}}}] (0,-3) .. controls (-2.5,0) and (2.5,0) .. (0,-3);
        \filldraw (0,-3) circle (3pt);
        \node[scale=2, thick] at (0,-1.5) {$\times$};
        \draw[line width = 1.2pt, orange,out=0,in=0,looseness=1.5] (0,-2.6) to (0,-3.3);
        \draw[line width = 1.2pt, orange, out=180,in=-120,looseness=1.5] (0,-3.3) to (-0.2,-2.3);
        \draw[line width = 1.2pt, orange, out=60,in=-90] (-0.2,-2.3) to (0.35,-1.5);
        \draw[line width = 1.2, orange, out=90,in=90,looseness=1.75] (0.35,-1.5) to (-0.35,-1.5);
        \draw[line width = 1.2, orange, out=-90,in=180,->] (-0.35,-1.5) to (1.2,-2);
        \node at (0.9,-2.5) {$1$};
        \node at (1.2,-1.2) {$3$};
        \node at (-1.2,-1.5) {$2$};
    \end{tikzpicture} 
    \\
        \highlight{$U_{k+1}\rho$} & \highlight{$U_{k}\rho$} & \highlight{$U_{k-1}\rho$}
    \\ \hline
\end{tabular}
\caption{Orientations on pending arcs when $\gamma$ is a notched corner arc. The arrow numbering indicates the traversal order on this portion of the $T$-walk.}
\label{tab:TwalkLoop}
\end{table}

    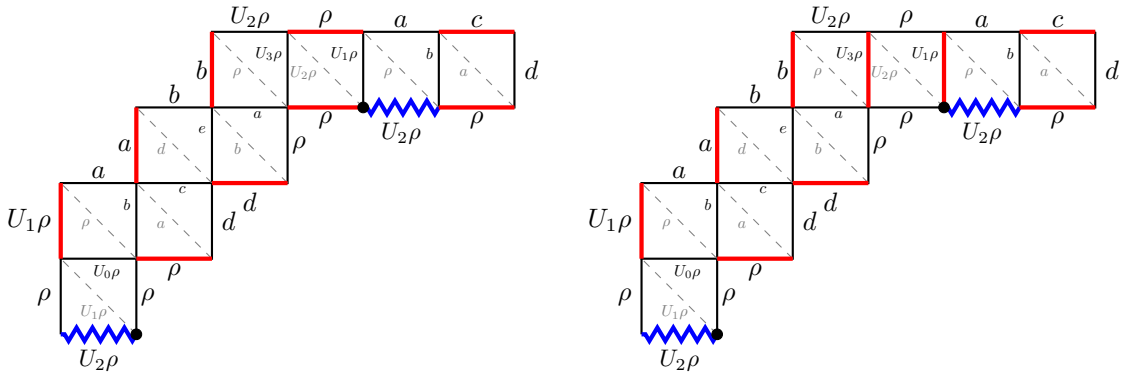
\begin{figure}[H]
        \centering
    \begin{tabular}{cc}
    \begin{tikzpicture}[scale=1]
        \draw[thick] (6,4) to node[above,yshift=-2]{$c$} (5,4) to node[above,yshift=-2]{$a$} (4,4) to node[above,yshift=-2]{$\rho$} (3,4) to node[above, scale = 0.9,yshift=-2]{$U_2\rho$} (2,4) to node[left,xshift=2]{$b$} (2,3) to node[below,yshift=1,scale=0.6,xshift=5]{$a$} (3,3) to node[below,yshift=2]{$\rho$}(4,3);
        \draw[draw=none] (4,3) to node[below, scale=0.9,yshift=-3]{$U_2\rho$} (5,3);
        \draw[thick] (5,3) to node[below,yshift=2]{$\rho$} (6,3) to node[right]{$d$} (6,4);
        
        \draw[thick] (5,4) to node[left, scale=0.6,yshift=10]{$b$} (5,3);
        \draw[thick] (4,4) to node[left, scale=0.6, yshift=10]{$U_{1}\rho$} (4,3); 
        \draw[thick] (3,4) to node[left, scale=0.6, yshift=10]{$U_{3}\rho$} (3,3); 
        
        \draw[gray,dashed] (6,3) to node[left,scale=0.6, xshift=-2,yshift=-2]{$a$} (5,4);
        \draw[gray,dashed] (5,3) to node[left,scale=0.6, xshift=-2,yshift=-2]{$\rho$} (4,4);
        \draw[gray,dashed] (4,3) to node[left,scale=0.6, xshift=-2,yshift=-2]{$U_2\rho$} (3,4);
        \draw[gray,dashed] (3,3) to node[left,scale=0.6, xshift=-2,yshift=-2]{$\rho$} (2,4);
        
        \draw[thick] (3,3) to node[right,xshift=-2]{$\rho$} (3,2) to node[below]{$d$} (2,2) to node[below, scale=0.6, yshift=2pt,xshift=5]{$c$} (1,2) to node[left,xshift=2]{$a$} (1,3) to node[above,yshift=-2]{$b$} (2,3);
        \draw[thick] (2,3) to node[left, scale=0.6,yshift=10]{$e$} (2,2);
        \draw[gray,dashed] (3,2) to node[left,scale=0.6, xshift=-2,yshift=-2]{$b$} (2,3);
        \draw[gray,dashed] (2,2) to node[left,scale=0.6, xshift=-2,yshift=-2]{$d$} (1,3);
        
        \draw[thick] (2,2) to node[right]{$d$} (2,1) to node[below,yshift=2]{$\rho$} (1,1) to node[below, scale=0.6, xshift=5]{$U_0\rho$} (0,1) to node[left]{$U_1 \rho$} (0,2) to node[above,yshift=-2]{$a$} (1,2);
        \draw[thick] (1,2) to node[left, scale=0.6,yshift=10]{$b$} (1,1);
        \draw[gray,dashed] (2,1) to node[left,scale=0.6, xshift=-2,yshift=-2]{$a$} (1,2);
        \draw[gray,dashed] (1,1) to node[left,scale=0.6, xshift=-2,yshift=-2]{$\rho$} (0,2);
        
        \draw[thick] (1,1) to node[right,xshift=-2]{$\rho$} (1,0);
        \draw[draw=none] (1,0) to node[below, scale=0.9,yshift=-3]{$U_2\rho$} (0,0);
        \draw[thick] (0,0) to node[left]{$\rho$} (0,1);
        \draw[gray,dashed] (1,0) to node[left, below, scale=0.6, xshift=-3,yshift=-2]{$U_1\rho$} (0,1);
        
        \draw[ultra thick, blue,decorate, decoration={zigzag, segment length=6pt}] (1,0) to (0,0);
        \draw[ultra thick, blue,decorate, decoration={zigzag, segment length=6pt}] (5,3) to (4,3);
            
        \draw[ultra thick, red] (6,3) to (5,3);
        \draw[ultra thick, red] (6,4) to (5,4);
        \draw[ultra thick, red] (4,4) to (3,4);
        \draw[ultra thick, red] (4,3) to (3,3);
        \draw[ultra thick, red] (2,4) to (2,3);
        \draw[ultra thick, red] (3,2) to (2,2);
        \draw[ultra thick, red] (1,3) to (1,2);
        \draw[ultra thick, red] (2,1) to (1,1);
        \draw[ultra thick, red] (0,2) to (0,1);
        
        \filldraw (4,3) circle (2pt);
        \filldraw (1,0) circle (2pt);
    \end{tikzpicture}
        &
    \begin{tikzpicture}[scale=1]
        \draw[thick] (6,4) to node[above,yshift=-2]{$c$} (5,4) to node[above,yshift=-2]{$a$} (4,4) to node[above,yshift=-2]{$\rho$} (3,4) to node[above, scale = 0.9,yshift=-2]{$U_2\rho$} (2,4) to node[left,xshift=2]{$b$} (2,3) to node[below,yshift=1,scale=0.6,xshift=5]{$a$} (3,3) to node[below,yshift=2]{$\rho$}(4,3);
        \draw[draw=none] (4,3) to node[below, scale=0.9,yshift=-3]{$U_2\rho$} (5,3);
        \draw[thick] (5,3) to node[below,yshift=2]{$\rho$} (6,3) to node[right]{$d$} (6,4);
        
        \draw[thick] (5,4) to node[left, scale=0.6,yshift=10]{$b$} (5,3);
        \draw[thick] (4,4) to node[left, scale=0.6, yshift=10]{$U_{1}\rho$} (4,3); 
        \draw[thick] (3,4) to node[left, scale=0.6, yshift=10]{$U_{3}\rho$} (3,3); 
        
        \draw[gray,dashed] (6,3) to node[left,scale=0.6, xshift=-2,yshift=-2]{$a$} (5,4);
        \draw[gray,dashed] (5,3) to node[left,scale=0.6, xshift=-2,yshift=-2]{$\rho$} (4,4);
        \draw[gray,dashed] (4,3) to node[left,scale=0.6, xshift=-2,yshift=-2]{$U_2\rho$} (3,4);
        \draw[gray,dashed] (3,3) to node[left,scale=0.6, xshift=-2,yshift=-2]{$\rho$} (2,4);
        
        \draw[thick] (3,3) to node[right,xshift=-2]{$\rho$} (3,2) to node[below]{$d$} (2,2) to node[below, scale=0.6, yshift=2pt,xshift=5]{$c$} (1,2) to node[left,xshift=2]{$a$} (1,3) to node[above,yshift=-2]{$b$} (2,3);
        \draw[thick] (2,3) to node[left, scale=0.6,yshift=10]{$e$} (2,2);
        \draw[gray,dashed] (3,2) to node[left,scale=0.6, xshift=-2,yshift=-2]{$b$} (2,3);
        \draw[gray,dashed] (2,2) to node[left,scale=0.6, xshift=-2,yshift=-2]{$d$} (1,3);
        
        \draw[thick] (2,2) to node[right]{$d$} (2,1) to node[below,yshift=2]{$\rho$} (1,1) to node[below, scale=0.6, xshift=5]{$U_0\rho$} (0,1) to node[left]{$U_1 \rho$} (0,2) to node[above,yshift=-2]{$a$} (1,2);
        \draw[thick] (1,2) to node[left, scale=0.6,yshift=10]{$b$} (1,1);
        \draw[gray,dashed] (2,1) to node[left,scale=0.6, xshift=-2,yshift=-2]{$a$} (1,2);
        \draw[gray,dashed] (1,1) to node[left,scale=0.6, xshift=-2,yshift=-2]{$\rho$} (0,2);
        
        \draw[thick] (1,1) to node[right,xshift=-2]{$\rho$} (1,0);
        \draw[draw=none] (1,0) to node[below, scale=0.9,yshift=-3]{$U_2\rho$} (0,0);
        \draw[thick] (0,0) to node[left]{$\rho$} (0,1);
        \draw[gray,dashed] (1,0) to node[left, below, scale=0.6, xshift=-3,yshift=-2]{$U_1\rho$} (0,1);
        
        \draw[ultra thick, blue,decorate, decoration={zigzag, segment length=6pt}] (1,0) to (0,0);
        \draw[ultra thick, blue,decorate, decoration={zigzag, segment length=6pt}] (5,3) to (4,3);
    
    \draw[ultra thick, red] (6,3) to (5,3);
    \draw[ultra thick, red] (6,4) to (5,4);
    \draw[ultra thick, red] (4,4) to (4,3);
    \draw[ultra thick, red] (3,4) to (3,3);
    \draw[ultra thick, red] (2,4) to (2,3);
    \draw[ultra thick, red] (3,2) to (2,2);
    \draw[ultra thick, red] (1,3) to (1,2);
    \draw[ultra thick, red] (2,1) to (1,1);
    \draw[ultra thick, red] (0,2) to (0,1);
    
    \filldraw (4,3) circle (2pt);
    \filldraw (1,0) circle (2pt);
    \end{tikzpicture}
    \end{tabular}
        \caption{Two good matchings of the loop graph associated to the notched corner arc $\gamma$ from~\Cref{ex:NotchCorner}.}
        \label{fig:minimalnotched}
    \end{figure}
\end{example}

Example \ref{ex:TWalkFromGoodMatchingOfLoopGraph} provides a hint for how we should treat $T$-walks for notched corner arcs.  The directed paths $\beta_i$ for the segment $U_k(\lambda_\boldp)\rho$ are constructed following the conventions in \Cref{tab:HomotopicWinding}. When $\gamma$ is a notched corner arc, the compound steps are illustrated in \Cref{tab:TwalkLoop}. To ensure accessibility, the orientations derived from the spoke vector $\vec{v}_s$ are represented by arrows $1$ and $2$. Between the $\alpha$ steps associated to arrows 1 and 2 are all the steps along the other spokes incident to the puncture. Note that the hook depicted in \Cref{tab:HomotopicWinding} corresponds specifically to the minimal $T$-walk.

We now consider the case $\gamma^0\in T$. When $\gamma$ is a standard arc, the same rules apply, with the only modification being the replacement of a notch by a hook. When $\gamma$ is a doubly-notched pending arc or a doubly-notched monogon, as discussed in~\Cref{ex:doubly notched pending arc}, applying the hooks in the appropriate orientation as discussed above yields the valid $T$-walk.

\section{Equivalences of Formulas}\label{sec:Equivalence}

In this section, we prove that given an arc or closed curve $\gamma$, all of our provided combinatorial formulas give equivalent Laurent expansions. Let us restate~\Cref{thm:MainExpansion}.

\Expan*

Recall \(\chi(\calP_\gamma)=x_{\min}^{\gamma}\mathcal{W}(\calP_\gamma,w)\). The equivalence of $\mathcal{W}(\calP_\gamma,w)$ to the top-left entry of the rank matrix $\rmm_w(\calP_\gamma\!\searrow)$ for arcs, and to the trace of $\rmm_w(\calP_\gamma\!\searrow)$ for closed curves has already been established in \Cref{prop:FormulasAgree}. It remains to show that \(\chi(\mathcal{G}_{\gamma,T}) = \chi(\calP_\gamma^T) = \mathfrak{W}(\gamma,T).\)

Our plan for the proof is outlined in the following diagram. The connection between loop and band graphs with auxiliary tiles and posets follows directly from the surface case \cite{ezgieminecluster2024}; we saw already that the posets and matrix formulas also have the same relationship. The dashed edges are incident to the set of loop and band graphs with hexagonal tiles, and therefore these results omit the case of notched corner arcs. However, the connection between graphs with the two types of tiles is necessary for the proof of the correctness of the formulas (Theorem \ref{thm:Correctness}), as this roots us to the earlier work of two authors \cite{banaian2020snake}. We also include the map between loop graphs with hexagonal tiles and $T$-walks as this map is a true bijection, showing how close these two formulations are. Finally, we show the $T$-walk expansion and poset expansion for notched corner arcs to close this gap.

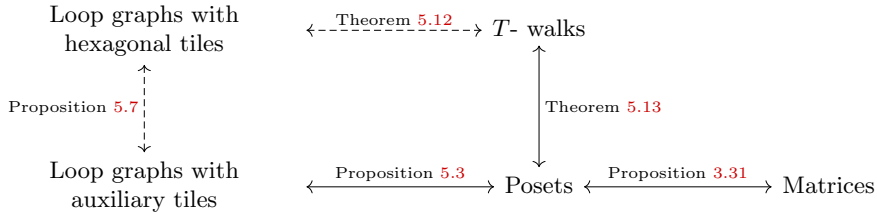
\begin{center}
\begin{tikzcd}
  \parbox{40mm}{\centering Loop graphs with \\hexagonal tiles} 
  \arrow[leftrightarrow, swap, dashed]{dd}{\text{\Cref{prop:CanChooseAuxiliaryTiles}}}
  \arrow[leftrightarrow, dashed]{rr}{\text{\Cref{thm:SnakeTWalks}}}   & \hspace{0.5cm} &T\text{- walks}\arrow[leftrightarrow]{dd}{\text{\Cref{prop:TWalkPosetNotchedCorner}}}&& \\
  \\
  \parbox{40mm}{\centering Loop graphs with \\auxiliary tiles}\arrow[leftrightarrow]{rr}[pos=0.5]{\text{\Cref{prop:AuxAndPoset}}}  & &\text{Posets}\arrow[leftrightarrow]{rr}[pos=0.5]{\text{\Cref{prop:FormulasAgree}}} &\hspace{0.5cm}& \text{Matrices}
\end{tikzcd}
\end{center}

\begin{remark}
    The rows of this table categorize our combinatorial objects into two distinct classes. The top row consists of objects whose definitions do not require auxiliary data, whereas the bottom row includes those that involve such choices. These distinctions are explored further in Section \ref{sec:conclusion}. Notably, because $T$-walks are defined for all arcs—unlike loop graphs with hexagonal tiles—these bijections allow us to conclude that the specific choices of auxiliary tiles or elements do not alter the resulting Laurent polynomials. 
\end{remark}

\subsection{Comparing Loop Graphs with Auxiliary Tiles and Posets}

Our poset construction was designed to create a correspondence with loop graphs with auxiliary tiles. Consequently, the following result is relatively immediate. While we adopt the methodology established in \cite{ezgieminecluster2024}, there is one key difference: in that work, the leading term of the poset expansion (the Laurent polynomial), $x_{\min}^\gamma$, is defined via the loop graph whereas we define $x_{\min}^\gamma$ purely in terms of the poset and the data from the triangulation $T$. Accordingly, we first demonstrate that these two definitions of the minimal term are equivalent.

\begin{lemma}\label{lem:MinTermAgreeAux}
    Let $\gamma$ be an arc or a closed curve on an orbifold $\mathcal{O}$ with triangulation $T$. If $\mathcal{G}_{\gamma,T}$ is the associated loop or band graph with auxiliary tiles with minimal matching $M_-$, then 
    \[
        x_{\min}^\gamma = \frac{x(M_-)}{\mathrm{cross}(\gamma, T)}
    \]
\end{lemma}

\begin{proof}
Following our definition of $x^\gamma_{\min}$, the proof follows directly from that for ordinary arcs and closed curves (see \cite[Proposition~6.2]{musiker2013bases} ) and tagged arcs (see \cite[Lemma 1]{banaian2024skein}). This is due to the fact that the auxiliary tiles, which have labels of the form $U_k(\lambda_\boldp) \rho$, are treated in the same way as non-auxiliary tiles in both definitions, so that they do not require a separate treatment.
\end{proof}

For an illustration of Lemma \ref{lem:MinTermAgreeAux}, consider the following subposet of a typical poset $\calP_\gamma$ and the associated subgraph of $\calG_{\gamma}$, where the edges in the minimal matching are highlighted. 

    \begin{center}
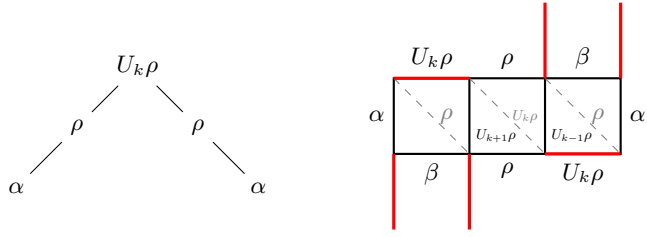
    \begin{tabular}{cc}
    \raisebox{4mm}{
    \begin{tikzpicture}[scale = 0.8]
    \node(1) at (0,0){$\rho$};
    \node(2) at (1,1){$U_k\rho$};
    \node(3) at (2,0){$\rho$};
    \node(a) at (-1,-1){$\alpha$};
    \node(aa) at (3,-1){$\alpha$};
    \draw(1) -- (2);
    \draw(2) -- (3);
    \draw(1) -- (a);
    \draw(3) -- (aa);
    \end{tikzpicture}
    }
    \hspace{5mm}
    &
    \begin{tikzpicture}[scale = 1]
    \draw[thick] (0,0) to node[below]{$\beta$} (1,0) to node[below]{$\rho$} (2,0) to node[below]{$U_k\rho$} (3,0) to node[right]{$\alpha$} (3,1) to node[above]{$\beta$} (2,1) to node[above]{$\rho$} (1,1) to node[above]{$U_k\rho$} (0,1) to node[left]{$\alpha$} (0,0);
    \draw[thick] (1,0) to node[right, below, xshift=10pt, yshift=-3pt, scale = 0.6]{$U_{k+1}\rho$} (1,1);
    \draw[thick] (2,0) to node[right, below, xshift=10pt, yshift=-3pt, scale = 0.6]{$U_{k-1}\rho$} (2,1);
    \draw[gray,dashed] (0,1) to node[right]{$\rho$} (1,0);
    \draw[gray,dashed] (1,1) to node[right, scale=0.6]{$U_k\rho$} (2,0);
    \draw[gray,dashed] (2,1) to node[right]{$\rho$} (3,0);
    \draw[very thick, red] (0,0) to (0,-1);
    \draw[very thick, red] (1,-1) to (1,0);
    \draw[very thick, red] (3,1) to (3,2);
    \draw[very thick, red] (2,2) to (2,1);
    \draw[very thick, red] (0,1) to (1,1);
    \draw[very thick, red] (2,0) to (3,0);
    \end{tikzpicture}
    \\
    \end{tabular}
    \end{center}

In the labelled poset model, the maximal element $U_k(\lambda_\boldp) \rho$ contributes a factor of $U_k(\lambda_\boldp) x_\rho$ to the numerator of $x^\gamma_{\min}$. In comparison, in the loop graph model $x(M_{-})$ contains two factors of $U_k(\lambda_\boldp) x_\rho$ and $\cross(\calG_\gamma)$ contains one factor of $U_k(\lambda_\boldp) x_\rho$. Hence, the minimal terms from both models agree locally.

With our minimal terms aligned, we turn to comparing the entire Laurent polynomial defined from each object. Recall that we assign a partial order to the set of good matchings of a loop graph by computing the sets of tiles in the symmetric differences $M \ominus M_-$ and ordering these via containment.

\begin{prop}\label{prop:AuxAndPoset}
If $\gamma$ is an arc or closed curve on an orbifold $\mathcal{O}$ with triangulation $T$ and $\calG_{\gamma,T}$ and $\calP_\gamma^T$ are its associated loop or band graph and poset respectively, then there is an order-preserving bijection between the good matchings of $\calG_{\gamma,T}$ and order ideals of $\calP_\gamma^T$. Moreover, we have \[
\chi(\calG_{\gamma,T}) = \chi(\calP_{\gamma}^T).
\]
\end{prop}

\begin{proof}
This follows directly from the analogous result for surfaces. For the first statement, we can follow the logic from \cite{musiker2013bases} for snake graphs and band graphs and from \cite{wilson2020surface} for loop graphs. We note that the case of a doubly-notched arc $\gamma$ with $\gamma^0 \in T$ will appear in \cite{wilson2020surface} in the near future. For the second statement, we can closely follow \cite{ezgieminecluster2024} for all types of graphs.
\end{proof}

\subsection{Comparing Hexagonal and Auxiliary Tiles}\label{subsec:hexaux}

In this section, we show that our two types of tiles give equivalent expansion formulas. 
It will be helpful to first develop some intuition for the properties of good matchings of graphs with hexagonal tiles.

\begin{lemma}\label{lem:HexagonMatch}
Let $G$ be a loop or band graph with a hexagonal tile as below. Then, a good matching of $G$ will include exactly one of the edges from the set $\{A,B,C,D\}$.

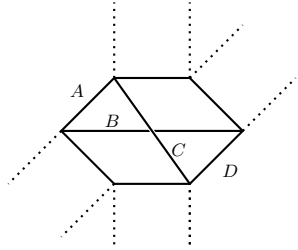
\begin{center}
\begin{tikzpicture}
\draw[thick] (-0.7,1.7) to node[midway,above,xshift=-15pt, yshift=-1pt,scale=0.75]{$B$} (1.7,1.7);
\draw[ultra thick,white] (0,2.4) to (1,1);
\draw[thick] (0,2.4) to node[midway,xshift=10pt,yshift=-7pt,scale=0.75]{$C$} (1,1);
\draw[thick](1,1) to  (0,1) to (-0.7,1.7) to node[midway,above,xshift=-4pt,scale=0.75]{$A$} (0,2.4) to (1,2.4) to(1.7,1.7) to node[midway,below,scale=0.75,xshift=7pt]{$D$} (1,1) ;
\draw[dotted, thick](0,0.2) to (0,1);
\draw[dotted, thick](1,0.2) to (1,1);
\draw[dotted, thick](0,2.4) to (0,3.4);
\draw[dotted, thick](1,2.4) to (1,3.4);
\draw[dotted, thick](0,1) to (-0.7, 0.3);
\draw[dotted, thick](1,0.2) to (1,1);
\draw[dotted, thick] (-0.7,1.7) to (-1.4,1);
\draw[dotted, thick](1,2.4) to (1.7,3.1);
\draw[dotted, thick](1.7,1.7) to (2.4,2.4);
\end{tikzpicture}
\end{center}

\end{lemma}

\begin{proof}
Many of the possible pairs of $\{A,B,C,D\}$ share a vertex, so that it is clear they cannot coexist in a good matching. Using any other pair of these edges, such as $A$ and $D$, will separate the remainder of the graph into two subgraphs that each contain an odd number of vertices, since the graphs connected to the hexagonal tile each have an even number of edges. Similarly, using none of the edges $A,B,C,D$ will either, depending on the specific graph, leave a vertex unmatched or again require a good matching of a subgraph with an odd number of vertices.
\end{proof}

Now, using Lemma \ref{lem:HexagonMatch}, we show how to replace one hexagonal tile in a loop or band graph with three ordinary tiles, one of which is an auxiliary tile.

\begin{prop}\label{prop:NonzeroWindingAuxiliary}
Let the three snake graphs $\calG_1,\calG_2,$ and $\calG_3$, depicted in Figure~\ref{fig:three_graphs}, be isomorphic as labelled graphs apart from the depicted subgraphs. Specifically, suppose that there are snake graphs $H_1$ and $H_2$ such that the portion of each $\calG_i$ southwest of the depicted subgraph is isomorphic to $H_1$ and the portion northeast of the depicted subgraph is isomorphic to $H_2$. Furthermore, suppose the weights satisfy $\mathrm{wt}(A)=\mathrm{wt}(D)=wt(\sigma) =x_\sigma$, $\mathrm{wt}(B)=\mathrm{wt}(B')$, $\mathrm{wt}(C)=\mathrm{wt}(C')$, and the labels satisfy $x_{B'}x_{C'} + x_{\rho}^2 = x_{\sigma}^2$. Set $h_\sigma = 1$; equivalently, let $\calG_2$ and $\calG_3$ be considered as snake graphs with auxiliary tiles labelled by $\sigma$.
 Then, \[
 \chi(\calG_1) = \chi(\calG_2) = \chi(\calG_3).
 \]

\end{prop}

\begin{figure}
\begin{center}
\begin{tabular}{ccc}
  \begin{tikzpicture}[scale = 1.1]
\draw[thick] (-0.7,1.7) to node[midway,above,xshift=-15pt, yshift=-1pt,scale=0.7]{$B$} (1.7,1.7);
\draw[ultra thick,white] (0,2.4) to (1,1);
\draw[thick] (0,2.4) to node[midway,xshift=14pt,yshift=-7pt,scale=0.7]{$C$} (1,1);
\draw[thick](1,1) to node[below]{$b$}  (0,1) to node[left]{$a$} (-0.7,1.7) to node[midway,above,xshift=-4pt,scale=0.7]{$A$} (0,2.4) to node[above]{$a$} (1,2.4) to node[right]{$b$} (1.7,1.7) to node[midway,below,scale=0.7,xshift=7pt]{$D$} (1,1) ;
\draw[dotted, thick] (0,0.2) to (0,1);
\draw[dotted, thick] (1,0.2) to (1,1);
\draw[dotted, thick](0,2.4) to (0,3.4);
\draw[dotted, thick](1,2.4) to (1,3.4);
\draw[dotted, thick] (0,1) to (-0.7, 0.3);
\draw[dotted, thick] (1,0.2) to (1,1);
\draw[dotted, thick] (-0.7,1.7) to (-1.4,1);
\draw[dotted, thick] (1,2.4) to (1.7,3.1);
\draw[dotted, thick] (1.7,1.7) to (2.4,2.4);
\draw[gray,dashed] (-0.7,1.7) to node[right]{$\rho$} (1,1);
\draw[gray,dashed] (0,2.4) to node[right]{$\rho$} (1.7,1.7);
\node at (0.5,-.5) {$\mathcal{G}_1$};
\end{tikzpicture}   
& \qquad 
\raisebox{0.8cm}{
\begin{tikzpicture}[scale = 0.8]
\draw[thick] (0,0) to node[below]{$b$} (1,0) to node[below]{$\rho$} (2,0) to node[below]{$\sigma$} (3,0) to node[right]{$a$} (3,1) to node[above]{$b$} (2,1) to node[above]{$\rho$} (1,1) to node[above]{$\sigma$} (0,1) to node[left]{$a$} (0,0);
\draw[thick] (1,0) to node[right, scale = 0.6]{$C'$} (1,1);
\draw[thick] (2,0) to node[right, scale = 0.6]{$B'$} (2,1);
\draw[gray,dashed] (0,1) to node[right]{$\rho$} (1,0);
\draw[gray,dashed] (1,1) to node[right]{$\sigma$} (2,0);
\draw[gray,dashed] (2,1) to node[right]{$\rho$} (3,0);
\draw[dotted, thick] (-1,1) to (0,1);
\draw[dotted, thick] (-1,0) to (0,0) to (0,-1);
\draw[dotted, thick] (1,-1) to (1,0);
\draw[dotted, thick] (3,0) to (4,0);
\draw[dotted, thick] (4,1) to (3,1) to (3,2);
\draw[dotted, thick] (2,2) to (2,1);
\node at (1.5,-1.5) {$\mathcal{G}_2$};
\end{tikzpicture}}
& \qquad
\begin{tikzpicture}[scale = 0.8]
\draw[thick] (0,0) to node[below]{$b$} (1,0) to node[right]{$\sigma$} (1,1) to node[right]{$\rho$} (1,2) to node[right]{$a$} (1,3) to node[above]{$b$} (0,3) to node[left]{$\sigma$} (0,2) to node[left]{$\rho$} (0,1) to node[left]{$a$}(0,0);
\draw[thick] (0,1) to node[above, scale =0.6]{$B'$} (1,1);
\draw[thick] (0,2) to node[above, scale = 0.6]{$C'$} (1,2);
\draw[gray,dashed] (0,1) to node[right]{$\rho$} (1,0);
\draw[gray,dashed] (0,2) to node[right]{$\sigma$}(1,1);
\draw[gray,dashed] (0,3) to node[right]{$\rho$} (1,2);
\draw[dotted, thick] (1,-1) -- (1,1);
\draw[dotted, thick] (0,-1) -- (0,0) -- (-1,0);
\draw[dotted, thick] (0,1) -- (-1,1);
\draw[dotted, thick] (0,4) -- (0,3);
\draw[dotted, thick] (1,4) -- (1,3) -- (2,3);
\draw[dotted, thick] (1,2) -- (2,2);
\node at (.5,-1.5) {$\mathcal{G}_3$};
\end{tikzpicture}
\\
\end{tabular}
\end{center}
    \caption{The graphs $\mathcal{G}_1$, $\mathcal{G}_2$, and $\mathcal{G}_3$ discussed in Proposition \ref{prop:NonzeroWindingAuxiliary}}
    \label{fig:three_graphs}
\end{figure}
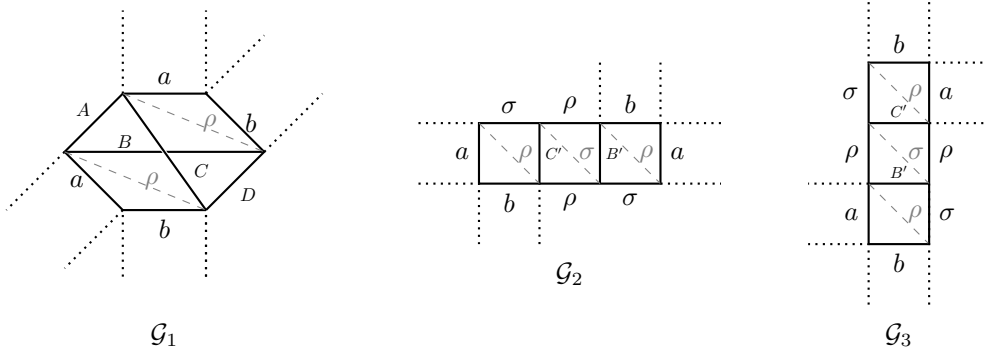

\begin{proof}

We will explicitly show that $\chi(\calG_1) = \chi(\calG_2)$. The proof that $\chi(\calG_1) = \chi(\calG_3)$ is analogous.

Using Lemma \ref{lem:HexagonMatch}, we can partition $\Match(\calG_1)$ into four disjoint subsets - $\mathcal{M}_A^1,\mathcal{M}_B^1,\mathcal{M}_C^1$ and $\mathcal{M}_D^1$ - where the subscript indicates which edge of the hexagonal tile is included in the matching. We partition $\Match(\calG_2)$ into 5 subsets - $\mathcal{M}_\sigma^2,\mathcal{M}_{B'}^2,\mathcal{M}_{C'}^2, \mathcal{M}_{B'C'}^2$ and $\mathcal{M}_{\rho}^2$ - based on which of the five possible configurations appears on the three tiles labelled by $\rho,\sigma,\rho$. 

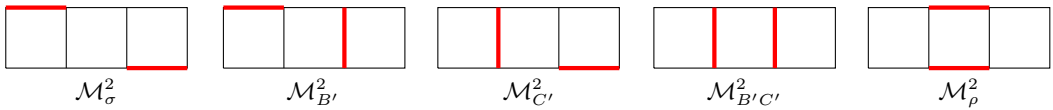
\begin{center}
\begin{tabular}{ccccc}
\begin{tikzpicture}[scale = 0.8]
\draw(0,0) -- (0,1) -- (3,1) -- (3,0) -- (0,0);
\draw(1,0) -- (1,1);
\draw (2,0) -- (2,1);
\draw[ultra thick,red] (0,1) -- (1,1);
\draw[ultra thick,red] (2,0) -- (3,0);
\end{tikzpicture}
&
\begin{tikzpicture}[scale = 0.8]
\draw(0,0) -- (0,1) -- (3,1) -- (3,0) -- (0,0);
\draw(1,0) -- (1,1);
\draw (2,0) -- (2,1);
\draw[ultra thick,red] (0,1) -- (1,1);
\draw[ultra thick,red] (2,0) -- (2,1);
\end{tikzpicture}
&
\begin{tikzpicture}[scale = 0.8]
\draw(0,0) -- (0,1) -- (3,1) -- (3,0) -- (0,0);
\draw(1,0) -- (1,1);
\draw (2,0) -- (2,1);
\draw[ultra thick,red] (1,0) -- (1,1);
\draw[ultra thick,red] (2,0) -- (3,0);
\end{tikzpicture}
&
\begin{tikzpicture}[scale = 0.8]
\draw(0,0) -- (0,1) -- (3,1) -- (3,0) -- (0,0);
\draw(1,0) -- (1,1);
\draw (2,0) -- (2,1);
\draw[ultra thick,red] (1,0) -- (1,1);
\draw[ultra thick,red] (2,0) -- (2,1);
\end{tikzpicture}
&
\begin{tikzpicture}[scale = 0.8]
\draw(0,0) -- (0,1) -- (3,1) -- (3,0) -- (0,0);
\draw(1,0) -- (1,1);
\draw (2,0) -- (2,1);
\draw[ultra thick,red] (1,0) -- (2,0);
\draw[ultra thick,red] (1,1) -- (2,1);
\end{tikzpicture}
\\ $\mathcal{M}_\sigma^2$ & $\mathcal{M}_{B'}^2$ & $\mathcal{M}_{C'}^2$ & $\mathcal{M}_{B'C'}^2$ & $\mathcal{M}_{\rho}^2$
\end{tabular}
\end{center}

From Lemma \ref{lem:HexagonMatch}, we know that the minimal matching $M_-^1$ of $\calG_1$ contains either $A$ or $D$. We will assume here that $A \in M_-^1$; the argument for the case with $D \in M_-^1$ follows the same logic. Since $A \in M_-^1$, and $\calG_1$ and $\calG_2$ are the same outside of these subgraphs, the minimal matching $M_-^2$ of $\calG_2$ must be in $\mathcal{M}_\sigma^2$. 

Let $\Psi^\sigma: \mathcal{M}_A^1 \to \mathcal{M}_\sigma^2$ be the map which sends $M \in \mathcal{M}_A^1$ to the unique matching in $\mathcal{M}_\sigma^2$ that can be decomposed as $\{\sigma,\sigma\}\sqcup M\vert_{H_1} \sqcup M\vert_{H_2}$. Because we assumed the weight of $A$ is the same as that of $\sigma$, it follows that $x(\Psi^\sigma(M)) = x_\sigma x(M)$.

Now, we compare $y(\Psi^\sigma(M))$ and $y(M)$ for $M \in \mathcal{M}_A^1$. Because we assumed $M_-^1 \in \mathcal{P}_1^A$, we know that for $M \in \mathcal{M}_A^1$, all cycles in $M \ominus M_-^1$ are completely contained in $H_1$ and $H_2$ and similarly for $\Psi^\sigma(M) \ominus P_-^2$. Because these cycles will contain the same labels, it immediately follows that $y(\Psi^\sigma(M)) = y(M)$. In particular, we have \[
\frac{1}{\cross(\calG_1)} \sum_{M \in \mathcal{M}_A^1} x(M) y(M) = \frac{1}{\cross(\calG_1)} \sum_{M \in \mathcal{M}_A^1} \frac{1}{x_\sigma}x(\Psi^A(M)) y(\Psi^A(M)) = \frac{1}{\cross(\calG_2)} \sum_{M \in \mathcal{M}_\sigma^2} x(M) y(M),
\]
where the final equality uses the identity $\cross(\calG_1) = x_\sigma  \cross(\calG_2)$.

Next, consider $M \in \mathcal{M}_B^1$. We can decompose $M$ as $\{B\} \sqcup M\vert_{H_1} \sqcup M\vert_{H_2}$. Depending on how $H_1,H_2$ are glued onto $G$, $M\vert_{H_i}$ may or may not be a good matching for $i = 1,2$. Let $\Psi^{B}: \mathcal{M}_B^1 \to \mathcal{M}_{B'}^2$ be the map which takes $M \in \mathcal{M}_B^{1}$ to the matching in $\mathcal{M}_{B'}^2$ which decomposes as $\{B',\sigma\} \sqcup M\vert_{H_1} \sqcup M\vert_{H_2}$. One can see that in each case for how the graphs $H_1,H_2$ are glued onto $\calG_2,\calG_2$,  $\Psi^{B}$ is a bijection. Moreover, since $\mathrm{wt}(B) = \mathrm{wt}(B')$, $x(\Psi^{B}(M)) = x_{\sigma} x(M)$.

Now we compare $y(\Psi^B(M))$ and $y(M)$.  Since $B$ is an interior edge, we know that for any $M \in \mathcal{M}_B^1$, $P \ominus M_-^1$ includes a cycle containing the diagonal $\rho$ above $B$ and not the diagonal below; all other cycles are completely contained in $H_1$ or $H_2$. We can also see that $\Psi^B(M) \ominus M_-^2$ contains the diagonal labelled $\rho$ to the right of the tile labelled $\sigma$; this cycle will contain the same labels as the cycle in $M \ominus M_-^1$ which contains the diagonal $\rho$, and the remainder of the cycles are contained in $H_1$ or $H_2$. Thus, $h(\Psi^B(M)) = h_\sigma h(M)$ so  $y(\Psi^B(M)) = y(M)$.

The same statements for $\mathcal{M}_C^1$ and $\mathcal{M}_{C'}^2$ as for $\mathcal{M}_B^1$ and $\mathcal{M}_{B'}^2$ follow by symmetry.

Finally, we construct two maps $\Psi^{BC}$ and $\Psi^\rho$ which send a good matching $M \in \mathcal{M}_D^1$ to $\mathcal{M}_{B'C'}^2$ and $\mathcal{M}_{\rho}^2$ respectively. We define $\Psi^{BC}(M)$ to be the matching in $\mathcal{M}_{B'C'}^2$ given by $M\vert_{H_1} \sqcup M\vert_{H_2} \sqcup \{B',C'\}$, and define $\Psi^{\rho}(M)$ to be the matching in $\mathcal{M}_{\rho}^2$ given by $M\vert_{H_1} \sqcup M\vert_{H_2} \sqcup \{\rho, \rho \}$. We can see that $x(\Psi^{BC}(M)) = \frac{\mathrm{wt}(B)\mathrm{wt}(C)}{\mathrm{wt}(A)} x(M)$ and $x(\Psi^{\rho}(M)) = \frac{x_\rho^2}{\mathrm{wt}(A)} x(M)$. Since we assumed $B'C' + x_{\rho}^2 = x_{\sigma}^2$ and $\mathrm{wt}(A) = x_\sigma$, we can see that for $M \in \mathcal{M}_A^1$, $x(\Psi^{BC}(M)) + x(\Psi^\rho(M)) = x_{\sigma}x(M)$. Since $M_-^2 \in \mathcal{M}_\sigma^2$, we can see that $h(\Psi^{BC}(M)) = h(M)$ and $h(\Psi^{\rho}(M)) = h_\sigma h(M)$ so that $y(\Psi^{BC}(M)) = y(\Psi^{\rho}(M)) = y(M)$. 

Combining these observations with the definition of $\chi(\mathcal{G}_1)$ and $\chi(\mathcal{G}_2)$ completes the proof.
\end{proof}

Via repeated applications of \Cref{prop:NonzeroWindingAuxiliary}, one can convert a loop or band graph containing an arbitrary number of hexagonal tiles into one containing only regular and auxiliary square tiles.

\begin{remark}
By comparing \Cref{lem:HexagonMatch}, \Cref{def:LaurentExpansionFromSnakeGraph}, and the weight assignments on the hexagonal tiles in \Cref{table:PuzzlePiecesWinding}, we observe that every good matching \( M \) of a loop or band graph containing a hexagonal tile associated with \( \rho \) carries a weight factor of \( x_\rho \). Consequently, since $ \text{cross}(\gamma,T)$ includes a factor of \( x_\rho^2 \) for this tile, one factor of \( x_\rho \) can be universally cancelled.
\end{remark}

In \Cref{subsec:aux}, we explained how we chose one specific loop or band graph with auxiliary tiles, $\mathcal{G}_\gamma$, to associate to each arc and closed curve on $\mathcal{O}$. Therefore, repeated use of \Cref{prop:NonzeroWindingAuxiliary} yields the following.

\begin{prop}\label{prop:CanChooseAuxiliaryTiles}
Let $\gamma$ be an arc or closed curve on an orbifold $\mathcal{O}$ which is not a notched corner arc. If $\mathcal{G}_{\gamma}$ is the associated loop or band graph with auxiliary tiles and $\mathcal{G}^{\text{gen}}_{\gamma}$ is the associated loop or band graph with hexagonal tiles, then \[
\chi(\mathcal{G}_{\gamma,T}) = \chi(\mathcal{G}^{\text{gen}}_{\gamma,T}).
\]
\end{prop}

An application of our analysis comparing graphs with hexagonal tiles and auxiliary tiles, along with Proposition \ref{prop:AuxAndPoset}, is the observation that the lattice of good matchings of a loop or band graph with hexagonal tiles is distributive. In Proposition \ref{prop:AuxAndPoset}, we show a bijection between good matchings of $\calG_{\gamma}$ and $\calP_{\gamma}$, in line with similar statements in earlier literature. A feature of this bijection is a 1-1 correspondence between tiles of $\calG_{\gamma}$ and elements of $\calP_{\gamma}$. Accordingly, let an \emph{auxiliary element} of $\calP_{\gamma}$ be one in correspondence with an auxiliary tile of $\calG_{\gamma}$.

\begin{theorem}\label{Thm:LatticeOfGenSGIsDistributive}
The lattice of good matchings of a loop or band graph with hexagonal tiles is distributive, and its poset of join-irreducibles is isomorphic to the induced subposet of $\calP_{\gamma}$ on the complement of the auxiliary elements. 
\end{theorem}

\begin{proof}
In the proof of Proposition \ref{prop:NonzeroWindingAuxiliary}, we give a surjection from the set of good matchings of a loop or band graph with hexagonal tiles $\calG_1$ to the set of good matchings of a loop or band graph $\calG_2$ with one less hexagonal tile and one auxiliary tile. Let the auxiliary tile be  $G_j$. 

Recall in the proof of Proposition \ref{prop:NonzeroWindingAuxiliary}, we assumed $A \in M_-$. We will keep this convention for now. This implies that the element associated to $G_j$ in $\calP_{\gamma}$ is maximal. 
Notice our map in this proof would be a bijection if we only consider the set of good matchings $M$ of $\calG_2$ such that $G_j$ is not contained in $M \ominus M^2_-$. This subset of good matchings is isomorphic to pairs of matchings of the two smaller graphs resulting from removing the two edges of $\calG_2$ labelled $\tau$. Note that the details of the proof of Proposition \ref{prop:NonzeroWindingAuxiliary} concerning the $h$-monomials and $y$-monomials guarantee we are considering order-preserving maps. Therefore, if there are no other hexagonal tiles in the loop or band graph, then from \cite[Theorem 5.2]{musiker2013bases} and \cite[Theorem 7.9]{wilson2020surface}, we know that the containment order on $M \ominus M^1_-$ endows the good matchings of $\calG_1$ with the structure of a distributive lattice.  If $\calG_2$ has other hexagonal tiles, we repeat this process until we have replaced all hexagonal tiles. 
\end{proof}

In our running example of a snake graph with hexagonal tiles $\calG^{\text{gen}}_{\gamma,T}$ (see item (II) in Table \ref{table:whole}), the corresponding poset of join irreducibles is the result of deleting element $w_2$ in the associated poset $\calP_{\gamma}$ (item (IV) in Table \ref{table:whole}).

The maximal connected subsets of an ordered set are referred as its \emph{connected components}~\cite{Schroder2003}. By Theorem~\ref{Thm:LatticeOfGenSGIsDistributive}, the lattice of good matchings for a loop or band graph with hexagonal tiles is isomorphic to the lattice of order ideals of the disjoint union of its connected components. Consequently, the rank generating polynomial of the set of good matchings is simply the product of the rank generating polynomials of its individual connected components.

We can compute these polynomials using our matrix formulas with each $w_i$ set equal to the formal variable $q$. If these  rank matrices associated with the components in the variable $q$ are denoted by $M_1, M_2, \dots, M_l$, then the desired rank generating polynomial is the upper-left entry of the matrix product:
\[ M_1 \begin{pmatrix} 1 & 0 \\ 0 & 0 \end{pmatrix} M_2 \begin{pmatrix} 1 & 0 \\ 0 & 0 \end{pmatrix} M_3 \dots \begin{pmatrix} 1 & 0 \\ 0 & 0 \end{pmatrix} M_l \]

For example, the rank generating function for the poset of perfect matchings of the snake graph with hexagonal tiles given in item (II) of Table \ref{table:whole} is the upper left entry of \[
\begin{pmatrix}
q+1 & -q \\ 1 & 0
\end{pmatrix}
\begin{pmatrix}
1 & 0 \\ 0 & 0
\end{pmatrix}
\begin{pmatrix}
q+1 & -q \\ 1 & 0
\end{pmatrix}
\begin{pmatrix}
q+1 & -q \\ 1 & 0
\end{pmatrix} = \begin{pmatrix} q^3 + 2q^2 + 2q + 1 & -q^3 - 2q^2 -q \\ q^2 + q + 1 & -q^2 -q \end{pmatrix}.
\]

\subsection{\texorpdfstring{Comparing $T$-walks and Loop Graphs with Hexagonal Tiles}{Comparing T-walks and Loop Graphs with Hexagonal Tiles}}\label{subsec:TWalkAndSnake}

The comparison between snake graphs and $T$-walks was first explored in \cite[Section 4.5]{musiker2010cluster}, where the authors established a direct bijection between the steps of a $T$-walk and the edges included in a perfect matching of the corresponding snake graph. The setting for this article was unpunctured surfaces, and the  correspondence was generalized to the case of punctured surfaces in \cite{ezgieminecluster2024}. In the following sections, we demonstrate that this bijection can be further, nontrivially extended to generalized cluster algebras arising from punctured orbifolds.

The lemma below follows directly from the definition of a $T$-walk and the method of assigning an $x$-weight and $y$-weight to each.

\begin{lemma}\label{lem:BreakUpTWalk}
Let $\gamma$ be a generalized arc or closed curve on $\mathcal{O}$ and let $\tau_{i_j} \in T$ be a pending arc crossed by $\gamma$ such that $\gamma$ winds $0<k<\mathbf{p}-2$ times around the enclosed orbifold point of order $\mathbf{p}$. Let \(\tau_{i_e},\tau_{i_{e+1}},\ldots,\tau_{i_{j-1}},\tau_{i_j}\), \(\tau_{i_{j+1}}\ldots,\tau_{i_{f-1}},\tau_{i_f}\) be the contiguous sublist of arcs crossed by $\gamma$ which share an endpoint with $\tau_{i_j}$ and lie counterclockwise from $\tau_{i_j}$.
\begin{enumerate}
    \item[(i)] Suppose that $j = 1$ and $\gamma$ crosses $\tau_{i_j}$ just once. The subset of $TW(\gamma)$ containing $T$-walks such that $\alpha_1$ is oriented minimally is in bijection with $TW(\gamma_1)$, where $\gamma_1$ is drawn on the left in Figure \ref{fig:ForBreakUpTWalkLemma}. Moreover, given such a $T$-walk, $T_{\vec{v}} \in TW(\gamma)$, which is sent to $T_{\vec{u_1}} \in TW(\gamma_1)$, we have $ x(T_{\vec{v}}) = U_{k-1}(\lambda_{\mathbf{p}}) x(T_{\vec{u_1}})$ and $y(T_{\vec{v}}) = y(T_{\vec{u_1}}).$ Similarly, the subset of $TW(\gamma)$ containing $T$-walks such that $\alpha_1$ is oriented maximally is in bijection with $TW(\gamma_2)$, where $\gamma_2$ is drawn on the left in Figure \ref{fig:ForBreakUpTWalkLemma}. Moreover, given such a $T$-walk, $T_{\vec{v}}$, which is sent to $T_{\vec{u_2}}$, we have $ x(T_{\vec{v}}) = U_{k}(\lambda_{\mathbf{p}})x(T_{\vec{u_2}})$ and $y(T_{\vec{v}}) = \Phi(h_{i_1}h_{i_2}\cdots h_{i_f})y(T_{\vec{u_2}}).$

\item[(ii)] Suppose $\gamma$ crosses $\tau_{i_j}$ twice consecutively, i.e., $\tau_{i_j} = \tau_{i_{j+1}}$. Define $\gamma_1,\gamma_2,\gamma_3,$ and $\gamma_4$ as on the right in Figure \ref{fig:ForBreakUpTWalkLemma}, where the solid black line is $\tau_{i_j}$ and the dashed black line is $\gamma$.
Partition $TW(\gamma)$ into sets $TW_{(a,b)} = \{T_{\vec{v}} \in TW(\gamma): v_{j} = a \text{ and } v_{j+1} = b\}$ for $a,b \in \{0,1\}$. We have bijections \begin{align*}
&TW_{(0,0)} \cong TW(\gamma_1) \times TW(\gamma_3) \qquad TW_{(1,0)} \cong TW(\gamma_2) \times TW(\gamma_3) \\
&TW_{(0,1)} \cong TW(\gamma_1) \times TW(\gamma_4) \qquad TW_{(1,1)} \cong TW(\gamma_2) \times TW(\gamma_4).
\end{align*}
Given $T_{\vec{v}} \in TW_{(a,b)}$, corresponding to the pair $(T_{\vec{v_1}},T_{\vec{v_2}})$, we have $x(T_{\vec{v}}) = \frac{C_j}{x_{i_j}}x(T_{\vec{v_1}})x(T_{\vec{v_2}})$ where $C_j$ is as in Table \ref{table:Twalkpending} and $y(T_{\vec{v}}) = y(a,b)y(T_{\vec{v_1}})y(T_{\vec{v_2}})$ where $y(a,b)$ is defined as follows, \[
y(0,0) = 1 \quad y(1,0) = \Phi(h_{i_e}h_{i_{e+1}}\cdots h_{i_j}) \quad y(0,1) = \Phi(h_{i_j}h_{i_{j+1}} \cdots h_{i_f}) \quad y(1,1) = \Phi(h_{i_e}h_{i_{e+1}}\cdots h_{i_f})
\]
\end{enumerate}

\end{lemma}

 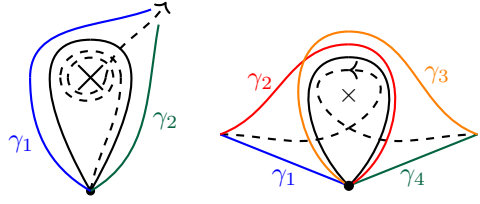
\begin{figure}[h]
     \centering
\begin{tabular}{cc}
\begin{tikzpicture}[scale=1, transform shape]
        \draw[out=60,in=0,looseness=1.2,thick] (0,-3) to (0,-1);
       \draw[out=120,in=180,looseness=1.2,thick] (0,-3) to (0,-1);
        \filldraw (0,-3) circle (1.5pt);
        \node[scale=2, thick] at (0,-1.5) {$\times$};
        \draw[out=80,in=-90,looseness=1, thick,dashed] (0,-3) to (0.4,-1.5);
        \draw[out=90,in=90,looseness=1.5, thick,dashed] (0.4,-1.5) to (-0.4,-1.5);
        \draw[out=-90,in=-90,looseness=1.5, thick,dashed] (-0.4,-1.5) to (0.3,-1.5);
        \draw[out=90,in=90,looseness=1.5, thick,dashed] (0.3,-1.5) to (-0.3,-1.5);
        \draw[out=-90,in=-100,looseness=1.5, thick,dashed] (-0.3,-1.5) to (0.2,-1.5);
        \draw[out=80,in=-120,looseness=1, thick,dashed, ->] (0.2,-1.5) to (1,-0.5);
        \draw[thick,Dgreen,out = 30, in = -100] (0,-3) to node[right]{$\gamma_2$} (0.9,-0.8);
       \draw[thick,blue, out = 150, in = 270] (0,-3) to node[left]{$\gamma_1$} (-0.8, -1.5);
       \draw[thick,blue] (-0.8,-1.5) to [out=90, in = 200] (0.8,-0.6);
    \end{tikzpicture} & 
\begin{tikzpicture}[scale=1.7]

\tikzset{->-/.style={decoration={
  markings,
  mark=at position #1 with {\arrow{>}}},postaction={decorate}}}

\draw[in=0,out=45,looseness =1.25,thick] (0,0.3) to (0,1.3);
\draw[in=180,out=135,looseness=1.25,thick] (0,0.3) to (0,1.3);

\draw[dashed,thick,out=0,in=210] (-1,0.7) to (0,0.75);
\draw[dashed,thick,out=30,in=0,looseness=2,->] (0,0.75) to (0,1.2);
\draw[dashed,thick,out=180,in=150,looseness=2] (0,1.2) to (0,0.75);
\draw[dashed,thick,out=-30,in=180] (0,0.75) to (1,0.7);

\draw[blue,thick] (-1,0.7) to node[below]{$\gamma_1$} (0,0.3);
\draw[Dgreen,thick](1,0.7) to node[below]{$\gamma_4$}(0,0.3);
\draw[red,thick, out = 20, in = 180](-1,0.7) to node[left]{$\gamma_2$} (0,1.4);
\draw[red,thick,out = 0, in = 30, looseness = 1.2] (0,1.4) to (0,0.3);
\draw[orange,thick,out=160, in = 0] (1,0.7) to node[right]{$\gamma_3$} (0,1.5);
\draw[orange, thick, out = 180, in = 150,looseness=1.2] (0,1.5) to (0,0.3);

\draw[fill=black] (0,0.3) circle [radius=1pt];
\draw[thick] (0,1) node {$\mathbf{\times}$};

\end{tikzpicture}

\end{tabular}
     \caption{Figures used in the proofs in Section \ref{subsec:TWalkAndSnake}. In each figure, the dashed arc is $\gamma$ and the solid, black arcs are in the triangulation (possibly on the boundary).}
     \label{fig:ForBreakUpTWalkLemma}
 \end{figure}

Next, we provide corresponding lemmas for good  matchings. Comparing these two lemmas will be the backbone of our proof of~\Cref{thm:SnakeTWalks}. The first is just a special case of snake graph calculus \cite{CS13}.

 \begin{lemma}\label{lem:BreakUpPMSOnFirstSquareTile}
Let $\gamma$ be an arc on an orbifold $\mathcal{O}$ with triangulation $T$ such that $s(\gamma)$ is plain. Suppose $\tau_{i_1}$ is a pending arc and $\gamma$ crosses $\tau_{i_1}$ only once.  Let $\tau_{i_1},\ldots,\tau_{i_f}$ be the contiguous sublist of arcs crossed by $\gamma$ which share an endpoint with $\tau_{i_1}$ and lie counterclockwise from $\tau_{i_1}$. Let $a$ and $b$ be the labels of the south and west edges of the first tile of $\mathcal{G}^{\text{gen}}_{\gamma,T}$ respectively. Let $\mathcal{M}_S$ be the set of good matchings of $\calG_{\gamma,T}$ which include the edge labelled $a$ and let $\mathcal{M}_W$ be the set of good matchings of $\calG_{\gamma,T}$ which include the edge labelled $b$. Let $\gamma_1$ and $\gamma_2$ be as on the left in Figure \ref{fig:ForBreakUpTWalkLemma}. Then, we have $\mathcal{M}_S \cong \mathrm{Match}(\calG_{\gamma_1,T})$ and $\mathcal{M}_W \cong \mathrm{Match}(\calG_{\gamma_2,T})$. Moreover, given $M \in \mathcal{M}_S$, $M' \in \mathcal{M}_W$, and the associated matchings of smaller graphs $M_1 \in \mathrm{Match}(\calG_{\gamma_1,T})$ and $M_2 \in \mathrm{Match}(\calG_{\gamma_2,T})$, we have \[
x(M) = x_a x(M_1) \qquad y(M) = y(M_1) \qquad x(M') = x_bx(M_2) \qquad y(M') = \Phi(h_{i_1}\cdots h_{i_f} )y(M_2).
\]
\end{lemma}

\begin{lemma}\label{lem:BreakUpPMsOnHexagon}
Let $\gamma$ be an arc or closed curve on $\mathcal{O}$ and let $\tau_{i_j} \in T$ be a pending arc crossed by $\gamma$ such that $\gamma$ winds $0 < k < \mathbf{p}-2$  times around the enclosed orbifold point of order $\mathbf{p}$ and $\gamma$ crosses $\tau_{i_j}$ twice, i.e., $\tau_{i_j} = \tau_{i_{j+1}}$. Let $\tau_{i_e},\tau_{i_{e+1}},\ldots,\tau_{i_{j-1}},\tau_{i_j},\tau_{i_{j+1}}\ldots,\tau_{i_{f-1}},\tau_{i_f}$ be the contiguous sublist of arcs crossed by $\gamma$ which share an endpoint with $\tau_{i_j}$ and lie counterclockwise from $\tau_{i_j}$. Let $A,B,C,D$ be the edges of the hexagonal tile in $\calG^{\text{gen}}_{\gamma,T}$ associated to $\tau_{i_j} = \tau_{i_{j+1}}$, as in Lemma \ref{lem:HexagonMatch}. Let $\mathcal{M}_A \subset \mathrm{Match}(\calG^{\text{gen}}_{\gamma,T})$ be the set of good matchings containing edge $A$ and define $\mathcal{M}_B,\mathcal{M}_C,$ and $\mathcal{M}_D$ similarly. Suppose $M_- \in \mathcal{M}_A$. We have the following 
\begin{itemize}
    \item $\mathcal{M}_A \cong \mathrm{Match}(\calG^{\text{gen}}_{\gamma_1,T}) \times \mathrm{Match}(\calG^{\text{gen}}_{\gamma_3,T}) $. 
    Given $M \in \mathcal{M}_A$ with associated $(M',M'') \in  \mathrm{Match}(\calG^{\text{gen}}_{\gamma_1,T}) \times \mathrm{Match}(\calG^{\text{gen}}_{\gamma_3,T})$, we have $x(M) = \frac{\mathrm{wt}(A)}{x_{i_j}^2}x(M')x(M'')$ and $y(M) = y(M')y(M'')$.
    \item  $\mathcal{M}_B \cong \mathrm{Match}(\calG^{\text{gen}}_{\gamma_2,T}) \times \mathrm{Match}(\calG^{\text{gen}}_{\gamma_3,T}) $.
    Given $M \in \mathcal{M}_B$ with associated $(M',M'') \in  \mathrm{Match}(\calG^{\text{gen}}_{\gamma_2,T}) \times \mathrm{Match}(\calG^{\text{gen}}_{\gamma_3,T})$, we have $x(M) = \frac{\mathrm{wt}(B)}{x_{i_j}^2}x(M')x(M'')$ and $y(M) = \Phi(h_{i_{j+1}}\cdots h_{i_f})y(M')y(M'')$.
    \item  $\mathcal{M}_C \cong \mathrm{Match}(\calG^{\text{gen}}_{\gamma_1,T}) \times \mathrm{Match}(\calG^{\text{gen}}_{\gamma_4,T}) $
    Given $M \in \mathcal{M}_C$ with associated $(M',M'') \in  \mathrm{Match}(\calG^{\text{gen}}_{\gamma_1,T}) \times \mathrm{Match}(\calG^{\text{gen}}_{\gamma_4,T})$, we have $x(M) = \frac{\mathrm{wt}(C)}{x_{i_j}^2}x(M')x(M'')$ and $y(M) = \Phi(h_{i_e}\cdots h_{i_j})y(M')y(M'')$.
    \item  $\mathcal{M}_D \cong \mathrm{Match}(\calG^{\text{gen}}_{\gamma_2,T}) \times \mathrm{Match}(\calG^{\text{gen}}_{\gamma_4,T}) $
    Given $M \in \mathcal{M}_D$ with associated $(M',M'') \in  \mathrm{Match}(\calG^{\text{gen}}_{\gamma_2,T}) \times \mathrm{Match}(\calG^{\text{gen}}_{\gamma_4,T})$, we have $x(M) = \frac{\mathrm{wt}(D)}{x_{i_j}^2}x(M')x(M'')$ and $y(M) = \Phi(h_{i_e}\cdots h_{i_j}h_{i_{j+1}}\cdots  h_{i_f})y(M')y(M'')$.
\end{itemize}
\end{lemma}

\begin{proof}
Let the tiles of $\mathcal{G}^{\text{gen}}_{\gamma,T}$ be  $(G_1,\ldots,G_j,G_{j+1},\ldots,G_d)$ where $G_j$ and $G_{j+1}$ comprise the compound hexagonal tile. We analyse the restriction of a matching in each of the sets to the subgraph on tiles $G_1,\ldots, G_{j-1}$ and on $G_{j+2},\ldots,G_d$. In some cases, this restriction will be forced to contain certain edges. By removing the tiles containing these forced edges, we recover the loop or band graphs for the appropriate arcs $\gamma_i$.
\end{proof}

With these ``reduction'' results established, we conclude with the correspondence between loop and band graphs with hexagonal tiles and $T$-walks.

\begin{theorem}\label{thm:SnakeTWalks}
   If $\gamma$ is an arc which is not a notched corner arc or a closed curve and $\mathcal{G}^{\text{gen}}_{\gamma,T}$ is its associated loop or band graph with hexagonal tiles, then there is a bijection between $\Match(\calG^{\text{gen}}_{\gamma,T})$ and $TW(\gamma,T)$, and moreover,
    \[
        \chi(\mathcal{G}^{\text{gen}}_{\gamma,T}) = \mathfrak{W}(\gamma,T).
    \]
\end{theorem}

\begin{proof}

We use induction on the number of pending arcs which $\gamma$ crosses and winds nontrivially within. If this number is zero, then the same proof as in \cite{ezgieminecluster2024} holds immediately. Note that if $\gamma$ crosses a pending arc but does not have a self-intersection inside, then the dynamics of both objects are analogous to crossing two distinct arcs. 

Now, at the inductive step, let $\tau_{i_j}$ be the first pending arc which $\gamma$ crosses and winds nontrivially within. Suppose first that $\gamma$ intersects $\tau_{i_j}$ only one time; this implies $j = 1$ or $j = d$. Without loss of generality, assume $j = 1$. By our assumption that $\gamma$ is not a notched corner arc, we know that $s(\gamma)$ is plain. 

By Lemma \ref{lem:BreakUpTWalk} part 1, we know that \[
TW(\gamma) \cong TW(\gamma_1) \sqcup TW(\gamma_2) \qquad \text{and} \qquad \mathfrak{W}(\gamma) = U_k(\lambda_{\mathbf{p}}) x_{i_1} \mathfrak{W}(\gamma_1) + U_{k-1}(\lambda_{\mathbf{p}}) x_{i_1}\Phi( h_{i_1}\cdots h_{i_f}) \mathfrak{W}(\gamma_2)
\]
where $\gamma_1$ and $\gamma_2$ are depicted on the left in Figure \ref{fig:ForBreakUpTWalkLemma} and $y_{i_1},\ldots,y_{i_f}$ are as described in the statement of Lemma \ref{lem:BreakUpTWalk}. Lemma \ref{lem:BreakUpPMSOnFirstSquareTile} gives us two analogous statements for $\Match(\calG^{\text{gen}}_{\gamma})$ and $\chi(\calG^{\text{gen}}_{\gamma})$. By induction, $TW(\gamma_i) \cong \Match(\calG^{\text{gen}}_{\gamma_i})$ and $\mathfrak{W}(\gamma_i) = \chi(\calG^{\text{gen}}_{\gamma_i})$ for each $i \in \{1,2\}$. Therefore, both statements hold for $\gamma$. 

The other case to consider is when $\gamma$ crosses $\tau_{i_j}$ twice. Here, we follow the same process but using part 2 of Lemma \ref{lem:BreakUpTWalk} and Lemma \ref{lem:BreakUpPMsOnHexagon}.

\end{proof}

\subsection{\texorpdfstring{Comparing Posets and $T$-walks for Notched Corner Arcs}{Comparing Posets and T-walks for Notched Corner Arcs}}\label{subsec:TWalkAndPoset}

Here, we directly compare $T$-walks and posets in the case for which we cannot utilize hexagonal tiles, namely, notched corner arcs. The proof method is analogous to those of~\Cref{prop:CanChooseAuxiliaryTiles} and~\Cref{thm:SnakeTWalks}. That is, we will find a way to partition each set of combinatorial objects and use induction to discuss these subsets. Like in Proposition \ref{prop:CanChooseAuxiliaryTiles}, this correspondence will not be bijective, and we will utilize formulas concerning Chebyshev polynomials.

\begin{theorem}\label{prop:TWalkPosetNotchedCorner}
Let $\mathcal{O}$ be an orbifold with triangulation $T$. If $\gamma^{(p)}$ is a corner arc which is notched at its starting point, then \[
\chi(\calP_{\gamma^{(p)}}^T) = \mathfrak{W}(\gamma^{(p)},T).
\]
\end{theorem}

\begin{proof}
 
Throughout, we will assume that $\gamma$ is not doubly-notched and not a corner arc at both ends; however, should these cases arise, one would simply have to repeat our arguments at the terminal end. Let $\rho$ be the pending arc which encloses the orbifold point that $\gamma$ winds around before its first intersection with $T$.

The poset for $\gamma^{(p)}$ has three elements labelled $\rho$ and two elements labelled with formal scalar multiples of $\rho$, as below. To distinguish between the elements sharing the label $\rho$, we use superscripts: $\rho$, $\rho^+$, and $\rho^-$. Here, each element in the poset is written as $(\text{label},\text{weight})$.

\begin{center}
\begin{tikzpicture}
\node(k-1) at (0,0.5) {$(U_{k}\rho,U_{k-1}U_{k+1})$};
\node(1) at (0,1.5) {$(\rho^-,\frac{U_{k}}{U_{k+1}}\hat{y}_{\rho})$};
\node[] at (0,2.1){$\vdots$};
\node(3) at (0,2.5){$(\rho^+,\frac{U_{k-2}}{U_{k-1}}\hat{y}_\rho)$};
\node(k) at (0,3.5){$(U_{k-1} \rho,\frac{1}{U_{k-2}U_{k}})$};
\node(2) at (2.5,2){$(\rho,\frac{U_{k}}{U_{k-1}}\hat{y}_\rho)$};
\draw(k-1) -- (1);
\draw(3) -- (k);
\draw(k-1) -- (2);
\draw(k) -- (2);
\node[] at (3.6,2){$\cdots$};
\end{tikzpicture}
\end{center}

We can partition order ideals of $\calP_{\gamma^{(p)}}$ into six sets based on which of the elements labelled $\rho$ are present: Let $A_0$ be the set of order ideals which do not contain any element labelled $\rho$. Let $A_1$ be the set of order ideals which contain $\rho^-$ but not $\rho$ nor $\rho^+$, and let $A_2$ be the set of order ideals which contain $\rho^+$ (implying they also contain $\rho^-$) but not $\rho$.   Let $A_3$ be the set of order ideals which contain $\rho$ but not $\rho^\pm$, and let $A_{4}$ be the  set of order ideals which contain $\rho$ and $\rho^-$ but not $\rho^+$. Finally, let $A_5$ be the set of order ideals which contain $\rho$ and $\rho^+$.

The set $A_0$ can be further partitioned into order ideals which contain the element labelled $U_{k}\rho$ and those which do not. Call these $A_0^1$ and $A_0^0$ respectively. Each subset $A_0^i$ is clearly in bijection with order ideals of $\calP_{\gamma^0}$ which do not include the first element, i.e., $\rho$. Following Proposition \ref{prop:AuxAndPoset} and Lemma \ref{lem:BreakUpPMSOnFirstSquareTile}, this subset of order ideals is in bijection with $\calP_{\gamma_1}$, where $\gamma_1$ is as on the left in Figure \ref{fig:ForBreakUpTWalkLemma}. By comparing the two posets, we see $x_{\min}^{\gamma^{(p)}} = \frac{1}{U_{k}(\lambda_{\mathbf{p}})x_\rho} x_{\min}^{\gamma_1}$. Therefore, we have\begin{align*}
x_{\min}^{\gamma^{(p)}} \sum_{I \in A_0} w(I) &= x_{\min}^{\gamma^{(p)}} \bigg(\sum_{I \in A_0^0} w(I) + \sum_{I \in A_0^1} w(I)\bigg) \\
&= \frac{1}{U_{k}(\lambda_{\mathbf{p}})x_\rho}  x_{\min}^{\gamma_1} \bigg(\sum_{I \in J(\calP_{\gamma_1})} w(I) 
+ U_{k-1}(\lambda_{\mathbf{p}})U_{k+1}(\lambda_{\mathbf{p}})\sum_{I \in J(\calP_{\gamma_1})} w(I)\bigg)\\
&= \frac{1 + U_{k-1}(\lambda_{\mathbf{p}})U_{k+1}(\lambda_{\mathbf{p}})}{U_{k}(\lambda_{\mathbf{p}})x_\rho} x_{\min}^{\gamma_1} \sum_{I \in J(\calP_{\gamma_1})} w(I) = \frac{U_{k}(\lambda_{\mathbf{p}}) }{x_\rho}\chi(\calP_{\gamma_1})
\end{align*}
where the final equality comes from Lemma \ref{lem:ChebyshevLogConcave}. One can perform a similar analysis for the other subsets of $J(\calP_{\gamma^{(p)}})$, using the following observations. 

Any other set which does not contain $\rho$  (i.e., $A_1$ and $A_2$) will be in correspondence  with $J(\calP_{\gamma_1,T})$ and any set which does (i.e., $A_3,A_4$ and $A_5$) will be in correspondence with $J(\calP_{\gamma_2,T})$, where $\gamma_1$ and $\gamma_2$ again come from the left of Figure \ref{fig:ForBreakUpTWalkLemma}. If an order ideal contains $\rho$, by definition, it must contain all elements less than $\rho$ in $\calP_\gamma^{(p)}$. This set consists of $U_k\rho$ and the elements corresponding to $\tau_{i_1} = \rho,\tau_{i_{2}},\ldots,\tau_{i_f}$ where this is the contiguous sublist of arcs crossed by $\gamma^0$ which share an endpoint with $\rho$ and lie counterclockwise from $\rho$.  Let $Y^\geq = \prod_{j=1}^f \Phi(h_{\tau_{i_j}})$ and $\hat{Y}^\geq = w_{U_k\rho}\prod_{j=1}^f \Phi(w_{\tau_{i_j}})$. As in previous proofs, one can check $\hat{Y}^\geq x_{\min}^{\gamma^{(p)}} = \frac{U_{k+1}(\lambda_{\mathbf{p}})}{x_\rho}Y^\geq x_{\min}^{\gamma_2} $.

 The sets which contain $\rho^-$ but not $\rho^+$ can be further partitioned based on how many of the elements between $\rho^-$ and $\rho^+$ are included. 
To this end, let $\sigma_1,\ldots,\sigma_m$ be the spokes at $s(\gamma^{(p)})$, ordered as in Section \ref{sec:Setting} but excluding $\rho$, and let $\hat{S}_p = \sum_{i=0}^m \prod_{j=1}^i \hat{y}_{\sigma_j}$. One can, for example, factor $\hat{S}_p$ out of $\sum_{I \in A_1} w(I)$. Set also $\hat{Y}_p:= \prod_{j=1}^m \hat{y}_{\sigma_j}$.

From this analysis, we have  \begin{align}
\chi(\calP_{\gamma^{(p)}}) &= \frac{1}{x_\rho}\bigg(\big(U_k(\lambda_{\mathbf{p}})+ U_{k-1}(\lambda_{\mathbf{p}}) \hat{y}_\rho \hat{S}_p +  U_{k-2}(\lambda_{\mathbf{p}})\hat{y}_{\rho}^2 \hat{Y}_p\big)\chi(\calP_{\gamma_1})\nonumber \\
&+ Y^\geq\big(U_{k+1}(\lambda_{\mathbf{p}}) + U_k(\lambda_{\mathbf{p}})\hat{y}_\rho \hat{S}_p + U_{k-1}(\lambda_{\mathbf{p}})\hat{y}_\rho^2\hat{Y}_p\big)\chi(\calP_{\gamma_2})\bigg) \label{eq:RewriteCornerArc}
\end{align}
where the righthand side corresponds to, in order, $A_0, A_1, A_2, A_3, A_4, A_5$. Moreover, by \cite[Corollary 1]{banaian2024skein} we have $\hat{y}_{\rho}^2 \hat{Y}_{p}= y_\rho^2\prod_{i=1}^m \Phi(h_{\sigma_i})$. 

Now, by our initial assumption, the arcs $\gamma_i$ are not a notched corner arcs, and so by combining~\Cref{prop:AuxAndPoset},~\Cref{prop:CanChooseAuxiliaryTiles}, and~\Cref{thm:SnakeTWalks}, $\chi(\calP_{\gamma_i}) = \mathfrak{W}(\gamma_i)$ for each $i$. Therefore, we will be done if we can show that $\mathfrak{W}(\gamma^{(p)})$ can be broken up in an analogous way as above. 

Indeed, our description of $T$-walks for notched corner arcs involves six cases (see Table \ref{tab:TwalkLoop}), and our claim is that these correspond exactly to the six subsets of order ideals of $\calP_{\gamma^{(p)}}$. Arrows 1,2, and 3 correspond to elements $\rho^+$,$\rho^-$, and $\rho$, respectively. If arrow 2 is in the maximal position and arrow 1 is in the minimal position, then it is possible that a contiguous interval of other spokes is also in the maximal position, analogous to how an order ideal can contain $\rho^-$ and some subset of the elements $\sigma_1,\ldots,\sigma_m$. One can also check that the weights agree in this correspondence. For example, we recovered that the weighted sum of order ideals in set $A_0$ has a factor of $U_k(\lambda_{\mathbf{p}})$, matching the weight assigned to the minimal configuration in Table \ref{tab:TwalkLoop}.
\end{proof}

\section{Verification of Formulas}\label{sec:main}

In this section, we show that all of our combinatorial expansion formulas give correct Laurent expressions for cluster variables in a punctured orbifold. We will treat plain, singly-notched, and doubly-notched arcs separately, in this order. The plain case is quite straightforward.

\begin{lemma}\label{lem:PlainOnPunctured}
If $\gamma$ is an ordinary plain arc on an orbifold $\Orb$ with triangulation $T$, then \[
x_\gamma  =\chi(\mathcal{P}_\gamma^T) = \chi(\mathcal{G}_{\gamma,T})=\mathfrak{W}(\gamma,T).
\]
\end{lemma}

\begin{proof}
For the first equality, the main result of \cite{banaian2020snake} is that this is true in the absence of punctures, and the proof strategy followed that of \cite{musiker2011positivity} for plain arcs. Since the latter article allowed for punctured surfaces, the proof methods of the former would still hold for a plain arc in a punctured orbifold. The other equalities follow from~\Cref{thm:MainExpansion}.
\end{proof}

\subsection{Singly-Notched Arcs}\label{subsec:singlynotched}

Let $\gamma^{(q)}$ be a singly-notched ordinary arc on $\mathcal{O}$. Let $\ell$ be the arc that cuts out a once-punctured monogon whose unique puncture is $q$. The three elements of the cluster algebra $x_\gamma,x_{\gamma^{(q)}},$ and $x_\ell$ satisfy $x_\ell = x_\gamma x_{\gamma^{(q)}}$. This identity comes from a statement on lambda lengths on surfaces \cite{fomin2018cluster}. By lifting such a configuration on an orbifold to a covering surface, we see that the identity still holds here. Therefore, if we can show one of our combinatorial expansion formulas satisfies an analogous relation, we will be done by \Cref{lem:PlainOnPunctured}.

\begin{prop}\label{prop:SinglyNotchedExpansion}
If $\gamma^{(q)}$ is a singly-notched ordinary arc on $\mathcal{O}$ with initial triangulation $T$, then
\[
x_{\gamma^{(q)}}  = \chi(\mathcal{P}_{\gamma^{(q)}}^T) = \chi(\mathcal{G}_{\gamma^{(q)},T})=\mathfrak{W}(\gamma^{(q)},T).
\]
\end{prop}

\begin{proof}
Let $\ell$ be the (necessarily ordinary) arc which forms a self-folded triangle with $\gamma$. We have the equality $x_\gamma x_{\gamma^{(q)}} = x_\ell$. From  Corollary 2 in \cite{banaian2024skein}, we have \[
\mathbf{x}^{\mathbf{g}_\gamma + \mathbf{g}_{\gamma^{(q)}}} \mathcal{W}(\mathcal{P}_\gamma) \mathcal{W}(\mathcal{P}_{\gamma^{(q)}}) = \mathbf{x}^{\mathbf{g}_\ell} \mathcal{W}(\mathcal{P}_\ell).
\]

Therefore, after rearranging, we have \[
\chi(\mathcal{P}_{\gamma^{(q)}}) = \mathbf{x}^{\mathbf{g}_{\gamma^{(q)}}} \mathcal{W}(\mathcal{P}_{\gamma^{(q)}})  = \mathbf{x}^{\mathbf{g}_\ell - \mathbf{g}_\gamma} \frac{\mathcal{W}(\mathcal{P}_\ell)}{\mathcal{W}(\mathcal{P}_\gamma)} = \frac{x_\ell}{x_\gamma} = x_{\gamma^{(q)}}
\] 
where the second-to-last equality follows from \Cref{lem:PlainOnPunctured}. 

The remaining equalities in the statement are a consequence of Theorem \ref{thm:MainExpansion}.
\end{proof}

Note that \Cref{prop:SinglyNotchedExpansion} is only valid for arcs with distinct endpoints because a singly-notched loop would not be an ordinary arc. In particular, this result does not apply to pending arcs. Given a singly-notched pending arc $\rho^{(q)}$ such that $\rho^0 \notin T$, one can use our previous constructions to assign a Laurent polynomial to $\rho^{(q)}$. Here, we complete the discussion of singly-notched arcs by introducing Laurent polynomials to each singly-notched pending arc whose underlying plain version lies in $T$. 

\begin{definition}\label{def:SinglyNotchedPendingArc}
Let $\rho \in T$ be a pending arc incident to a puncture $q$ and let $\ell'$ be a closed curve that intersects all spokes incident to $q$ (as on the top left of Table \ref{tab:SinglyNotchedPending}). Let these spokes, other than $\rho$, be $\sigma_1,\sigma_2,\ldots,\sigma_m$, and define $Y_q':= y_{\rho}\prod_{i=1}^m \Phi(h_{\sigma_i})$. Orient $\rho$ so that its path of travel has the orbifold point on its right. We set
    \begin{align*}
        x_{\rho^{(s(\rho))}} &:=x_{\ell'}+\lambda_\mathbf{p}Y'_{q},  \\
        x_{\rho^{(t(\rho))}} &:=y_{\rho}x_{\ell'}+\lambda_\mathbf{p}.
    \end{align*}
\end{definition}

\begin{table}
\begin{center}
\begin{tabular}{c|c|c}
    \begin{tikzpicture}[scale=0.7, transform shape]
      \draw[
        thick
      ] (0,-3) .. controls (-2.5,0) and (2.5,0) .. (0,-3);
    \filldraw (0,-3) circle (3pt);
    \draw[thick] (0,-3) to (1,-3);
    \draw[thick] (0,-3) to (-1,-3.5);
    \draw[thick] (0,-3) to (-0.5,-4);
    \draw[out=0,in=0,looseness=1.3,thick, densely dotted] (0,-3.3) to (0,-0.5);
    \draw[out=180,in=180,looseness=1.3,thick, densely dotted] (0,-3.3) to (0,-0.5);
    \node[scale=2, thick] at (0,-1.5) {$\times$};
    \node[scale=1.5, thick] at (0.7,-2.4) {$\rho$};
    \node[scale=1.5, thick] at (1.1,-1) {$\ell'$};
    \node[scale=1.5, thick] at (-0.4,-2.9) {$q$};
    \node[scale=1.3, thick] at (-1.2,-3.7) {$\sigma_1$};
    \node[scale=1.3, thick] at (-0.7,-4.2) {$\sigma_2$};
    \node[scale=1.3, thick] at (1.6,-3) {$\sigma_{m}$};
    \node[scale=1.3, thick] at (0.6,-3.5) {$\iddots$}; 
    \end{tikzpicture}
    &
    \begin{tikzpicture}[scale=0.7, transform shape]
      \draw[
        thick,
        postaction={
          decorate,
          decoration={
            markings,
            mark=at position 0.12 with {\node[ rotate=90] {$\bowtie$};}
          }
        }
      ] (0,-3) .. controls (-2.5,0) and (2.5,0) .. (0,-3);
    \filldraw (0,-3) circle (3pt);
    \node[scale=2, thick] at (0,-1.5) {$\times$};
    \node[scale=1.5, thick] at (0,-3.5) {$q$};

    \end{tikzpicture}
 &
    \begin{tikzpicture}[scale=0.7, transform shape]
      \draw[
        thick,
        postaction={
          decorate,
          decoration={
            markings,
            mark=at position 0.88 with {\node[ rotate=90] {$\bowtie$};}
          }
        }
      ] (0,-3) .. controls (-2.5,0) and (2.5,0) .. (0,-3);
    \filldraw (0,-3) circle (3pt);
    \node[scale=2, thick] at (0,-1.5) {$\times$};
    \node[scale=1.5, thick] at (0,-3.5) {$q$};
    \end{tikzpicture}
    \\[2mm]
    \highlight{$\rho$ (solid) and $\ell'$(dotted)} & \highlight{$\rho^{(s(\rho))}$} & \highlight{$\rho^{(t(\rho))}$} \\[2mm]
    \hline
    \begin{tikzpicture}
    \node[] (s1) at (4,-4){$(\sigma_1,\hat{y}_{\sigma_1})$};
    \node[] (s2) at (4,-3){$(\sigma_2,\hat{y}_{\sigma_2})$};
    \node[] (s3) at (4,-2){$(\sigma_3,\hat{y}_{\sigma_3})$};
    \node[] (dotsl) at (4,-1){$\vdots$};
    \node[] (ta) at (4,0){$(\sigma_m,\hat{y}_{\sigma_m})$};
    \draw (s1) -- (s2);
    \draw (s2) -- (s3);
    \draw(s3) -- (dotsl);
    \draw(dotsl) -- (ta);
    \end{tikzpicture}
    &
    \begin{tikzpicture}
    \node[] (s1) at (4,-3){$(\sigma_1,\hat{y}_{\sigma_1})$};
    \node[] (dotsl) at (4,-2){$\vdots$};
    \node[] (sm1) at (4,-1){$(\sigma_m,\hat{y}_{\sigma_m})$};
    \node[] (ta) at (4,0){$(\rho,\frac{U_2(\lambda_\mathbf{p})}{\lambda_\mathbf{p}}\hat{y}_\rho)$};
    \node[] (ta1) at (4,1){$(\lambda_{\mathbf{p}}\rho,\frac{1}{U_2(\lambda_\mathbf{p})})$};
    \draw(s1) -- (dotsl);
    \draw(dotsl) -- (sm1);
    \draw(sm1) -- (ta);
    \draw(ta) -- (ta1);
    \end{tikzpicture}
    &
    \begin{tikzpicture}
    \node[] (ta1) at (4,-3){$(\rho,\frac{1}{\lambda_\mathbf{p}}\hat{y}_{\rho})$};
    \node[] (s1) at (4,-2){$(\sigma_1,\hat{y}_{\sigma_1})$};
    \node[] (s2) at (4,-1){$(\sigma_2,\hat{y}_{\sigma_2})$};
    \node[] (dotsl) at (4,0){$\vdots$};
    \node[] (sm1) at (4,1){$(\sigma_{m},\hat{y}_{\sigma_m})$};
    \draw(ta1) -- (s1);
    \draw (s1) -- (s2);
    \draw(s2) -- (dotsl);
    \draw(dotsl) -- (sm1);
    \end{tikzpicture}
    \\
    \highlight{$\mathcal{P}(\ell')$} & \highlight{$\mathcal{P}(\rho^{(s(\rho))})$} & \highlight{$\mathcal{P}(\rho^{(t(\rho))})$} \\
\end{tabular}
\end{center}
\caption{Posets assigned to a singly-notched pending arc whose underlying plain arc is in the initial triangulation. Each element is denoted by (label,weight).}\label{tab:SinglyNotchedPending}
\end{table}

Following our conventions, we have $x_{\min}^{\rho^{(s(\rho))}} = \tfrac{x_{\sigma_m}}{x_{\sigma_1}}$ and $x_{\min}^{\rho^{(t(\rho))}}=\lambda_\mathbf{p}$. With these minimal terms, one can verify that~\Cref{def:SinglyNotchedPendingArc} is consistent with applying~\Cref{def:LaurentExpansionFromFencePoset} to the posets shown in Table \ref{tab:SinglyNotchedPending}.

\begin{remark}
   The inspiration for the definitions of  $\chi(\mathcal{P}_{\rho^{(s(\rho))}})$ and $\chi(\mathcal{P}_{\rho^{(t(\rho))}})$ come from the treatment of singly-notched loops in \cite[Definition 12.22]{musiker2011positivity}. The nature of the specific posets is discussed further in Section \ref{sec:Invariance}. This definition is also consistent with the skein relation given in~\cite[Proposition 3]{banaian2020snake}.
\end{remark} 

\subsection{Doubly-Notched Arcs}

In this section, we complete the proof of Theorem \ref{thm:Correctness} by considering doubly-notched arcs. Our proof method consists of showing that our combinatorial Laurent polynomials satisfy formulas analogous to those in the cluster algebra. We continue to find it most convenient to phrase the combinatorics in terms of posets; by~\Cref{thm:MainExpansion}, this is equivalent to working with loop graphs or with $T$-walks.

To prepare for the proof of \Cref{thm:DoublyNotched}, we present skein-like formulas for the polynomials $\chi(\calP_\gamma)$, complementing results of~\cite{banaian2025cluster,banaian2024skein}.  
We prove the following only in a certain setting  which appears when resolving intersections of ordinary arcs. A similar claim is true in wider generality, but by focusing on this case we are able to simplify notation.

\begin{lemma}\label{lem:SkeinRelationWindAtBeginning}
Let $\gamma$ be a corner arc with $p = s(\gamma)$ and winding number $k = 1$. Assume $\gamma$ has no other self-intersections.

Let $\rho \in T$ be the pending arc enclosing this orbifold point, and by abuse of notation, let $\rho$ also denote $\calP_\gamma(1)$. Let $\gamma_1,\gamma_2$ be as on the left of Figure \ref{fig:ForBreakUpTWalkLemma}. Let $\tau_{i_1} = \rho,\tau_{i_{2}},\ldots,\tau_{i_f}$ be the contiguous sublist of arcs crossed by $\gamma$ which share an endpoint with $\rho$ and lie counterclockwise from $\rho$, and let $\tau_{j_1} = \rho, \tau_{j_2}, \cdots, \tau_{j_e}$ be the contiguous sublist of arcs crossed by $\gamma$ which share an endpoint with $\rho$ and lie clockwise from $\rho$. Set \[
Y^\geq = \Phi(h_{j_1}h_{j_2}\cdots h_{j_e}) \qquad \qquad Y^\leq = \Phi(h_{i_1}h_{i_2} \cdots h_{i_f}).
\]

Then we have  \[
\chi(\mathcal{P}_\gamma^T) = \chi(\mathcal{P}_{\gamma_1}^T) + \lambda_{\mathbf{p}}Y^\geq \chi(\mathcal{P}_{\gamma_2}^T) \qquad \text{and} \qquad \chi(\mathcal{P}_{\gamma^{(p)}}^T) = \lambda_{\mathbf{p}}\chi(\mathcal{P}_{\gamma_2^{(p)}}^T) + Y^\leq \chi(\mathcal{P}_{\gamma_1^{(p)}}^T).
\]

Similarly, if $\gamma$ is a corner arc with winding number $k = \mathbf{p}-2$ , then \[
\chi(\mathcal{P}_\gamma^T) = \lambda_{\mathbf{p}}\chi(\mathcal{P}_{\gamma_1}^T) + Y^\geq \chi(\mathcal{P}_{\gamma_2}^T)\qquad \text{and} \qquad
\chi(\mathcal{P}_{\gamma^{(p)}}^T) = \chi(\mathcal{P}_{\gamma_2^{(p)}}^T) + \lambda_{\mathbf{p}}Y^\leq \chi(\mathcal{P}_{\gamma_1^{(p)}}^T).
\]
\end{lemma}

The notation $Y^\geq$ and $Y^\leq$ hints at how these sets appear in the poset $\calP_{\gamma}$.

\begin{proof}
The cases where $\gamma$ is plain come directly from combining Lemma \ref{lem:BreakUpTWalk} with the combinatorial translation between $T$-walks and order ideals of posets in Theorem \ref{thm:MainExpansion}. 

Now, consider a notched corner arc $\gamma^{(p)}$. Here, we focus on the $k = 1$ case as the $k = \mathbf{p}-2$ case is similar. Let the spokes incident to the puncture $p = s(\gamma)$ other than $\rho$ be $\sigma_1,\ldots,\sigma_m$, indexed in counterclockwise order. Let $\hat{S}_p = \sum_{i=0}^m \prod_{j=1}^i \hat{y}_{\sigma_j}$, $\hat{Y}_p = \prod_{i=1}^m\hat{y}_{\sigma_i}$ and $Y_p = \prod_{i=1}^m \Phi(h_{\sigma_i})$. Set also $\hat{Y}^\geq = \hat{y}_{j_1}\hat{y}_{j_2} \cdots \hat{y}_{j_e}$ where the arcs $\tau_{j_i}$ are as defined in the premable.

Recall in~\Cref{eq:RewriteCornerArc} in the proof of \Cref{prop:TWalkPosetNotchedCorner}, we gave a reduction formula for $\chi(\calP_{\gamma^{(p)}})$. 
Since $k = 1$ here, we have $U_{k-2}(\lambda_{\mathbf{p}}) = 0, U_{k-1}(\lambda_{\mathbf{p}}) = 1, U_{k-2}(\lambda_{\mathbf{p}}) = \lambda_{\mathbf{p}}$ and  $U_{k+1}(\lambda_{\mathbf{p}}) = \lambda_{\mathbf{p}}^2 - 1$. Therefore, we are done if we show \begin{align}
\lambda_{\mathbf{p}} x_\rho  \chi(\calP_{\gamma_2^{(p)}}) + x_\rho Y^{\leq} \chi(\calP_{\gamma_1^{(p)}}) &= (\lambda_{\mathbf{p}} + \hat{y}_\rho \hat{S}_p) \chi(\calP_{\gamma_1}) + Y^\geq (\lambda_{\mathbf{p}}^2 - 1 + \lambda_{\mathbf{p}}\hat{y}_\rho \hat{S}_p + \hat{y}_\rho^2\hat{Y}_p)\chi(\calP_{\gamma_2}) \label{eq:TwoWaysToSeeNotchedCorner}.
\end{align}

Notice that $Y^\geq = y_\rho$ or $Y^\leq = y_\rho$. This depends on whether $\gamma$ crosses $\sigma_1$ or $\sigma_m$ after crossing $\rho$; if $\gamma$ crosses neither, then $\gamma$ only crosses $\rho$ and the computation is small. For convenience, we focus on the case $Y^\leq = y_\rho$. The other case is similar. 

Notice that $\tau_{j_1}, \tau_{j_2},\ldots,\tau_{j_e}$ must necessarily consist of $\rho$ and some of the spokes $\sigma_i$. Let $l$ be such that the last arc in this list is $\sigma_l$ and let $\alpha$ be the first arc crossed by $\gamma_2$. Below, we draw $\calP_{\gamma_1^{(p)}}$ and $\calP_{\gamma_2^{(p)}}$ on the left and right, respectively. Posets $\calP_{\gamma_1}$ and $\calP_{\gamma_2}$  correspond to removing the chains at the end. 

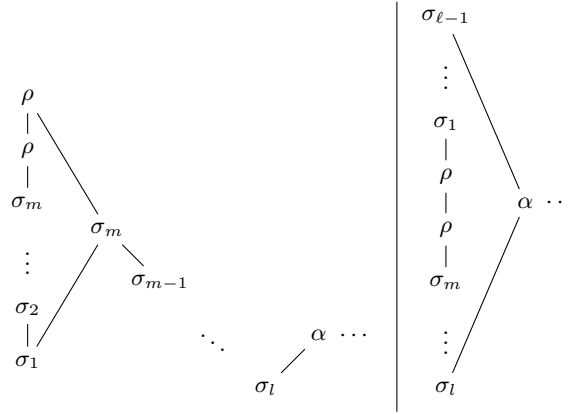
\begin{center}
\begin{tabular}{c|c}
\begin{tikzpicture}[scale=0.7]
\node(sm) at (0,0){$\sigma_m$};
\node(sm1) at (1,-1){$\sigma_{m-1}$};
\node at (2,-2){$\ddots$};
\node(sk) at (3,-3) {$\sigma_l$};
\node(a) at (4,-2){$\alpha$};
\node at (4.7,-2){$\cdots$};
\draw(sm) -- (sm1);
\draw(sk) -- (a);
\node(s1l) at (-1.5,-2.5){$\sigma_1$};
\node(s2l) at (-1.5,-1.5){$\sigma_2$};
\node at (-1.5,-0.5){$\vdots$};
\node(sml) at (-1.5,0.5){$\sigma_m$};
\node(rho1l) at (-1.5,1.5){$\rho$};
\node(rho2l) at (-1.5,2.5){$\rho$};
\draw(sm) -- (s1l);
\draw(sm) -- (rho2l);
\draw(s1l) -- (s2l);
\draw(sml) -- (rho1l);
\draw(rho1l) -- (rho2l);
\end{tikzpicture}
&
\begin{tikzpicture}[scale=0.7]
\node(a) at (0,0){$\alpha$};
\node at (0.7,0){$\cdots$};
\node(skl) at (-1.5,-3.5){$\sigma_l$};
\node at (-1.5,-2.5){$\vdots$};
\node(sml) at (-1.5,-1.5){$\sigma_m$};
\node(rho1l) at (-1.5,-0.5){$\rho$};
\node(rho2l) at (-1.5,0.5){$\rho$};
\node(s1l) at (-1.5,1.5){$\sigma_1$};
\node at (-1.5,2.5){$\vdots$};
\node(sk-1l) at (-1.5,3.5){$\sigma_{\ell-1}$};
\draw(a) -- (skl);
\draw(a) -- (sk-1l);
\draw(sml) -- (rho1l);
\draw(rho1l) -- (rho2l);
\draw(rho2l) -- (s1l);
\end{tikzpicture}
\end{tabular}
\end{center}

First we rewrite the weight polynomials of the posets $\calP_{\gamma_1^{(p)}}$ and $\calP_{\gamma_2^{(p)}}$. First, for $\calP_{\gamma_1^{(p)}}$, we partition order ideals based on how many elements in the first chain are present.
\[
\mathcal{W}(\calP_{\gamma_1^{(p)}}) = \sum_{\substack{I \in J(\calP_{\gamma_1})\\ \calP_{\gamma_1}(1) \notin I}} w(I) + (\hat{S}_p-1) \mathcal{W}(\calP_{\gamma_1}) + \lambda_{\mathbf{p}}\hat{y}_\rho \hat{Y}_p \mathcal{W}(\calP_{\gamma_1})+ \hat{y}_\rho \hat{Y}_p \hat{Y}^\geq \mathcal{W}(\calP_{\gamma_2})
\]

We proceed similarly  for  $\calP_{\gamma_2^{(p)}}$. We also recognize a subposet which is isomorphic to  $\calP_{\gamma_1}$.

\[
\mathcal{W}(\calP_{\gamma_2^{(p)}}) = \mathcal{W}(\calP_{\gamma_1}) + \hat{Y}^\geq  (\lambda_{\mathbf{p}} + \hat{y}_\rho  + \hat{y}_\rho \hat{y}_{\sigma_1} + \cdots +  \hat{y}_\rho \hat{y}_{\sigma_1}\cdots \hat{y}_{\sigma_{\ell-2}}) \mathcal{W}(\calP_{\gamma_2}) + \hat{y}_\rho^2 \hat{Y}_p \sum_{\substack{I \in J(\calP_{\gamma_2})\\ \alpha \in I}} w(I)
\]

Next, we compare the minimal terms of these four arcs. These have several nice relationships. 

\[
\hat{y}_\rho x_{\min}^{\gamma_2^{(p)}} = y_\rho x_{\min}^{\gamma_1^{(p)}} \qquad \hat{Y}^\geq x_{\min}^{\gamma_1} = Y^\geq x_{\min}^{\gamma_2} \qquad x_\rho x_{\min}^{\gamma_2^{(p)}} = x_{\min}^{\gamma_1}
\]

Now, we use these formulas to refine~\Cref{eq:TwoWaysToSeeNotchedCorner}. We can further sort each side based on the number of factors of $\lambda_{\mathbf{p}}$ and then show each of these smaller polynomials agrees. For instance, the terms on the left of~\Cref{eq:TwoWaysToSeeNotchedCorner} with coefficient $\lambda_{\mathbf{p}}$ are \begin{align*}
&x_{\min}^{\gamma_1}\big( \mathcal{W}(\calP_{\gamma_1}) + \hat{y}_{\rho} \hat{Y}^\geq  (1  + \hat{y}_{\sigma_1} + \cdots +  \hat{y}_{\sigma_1}\cdots \hat{y}_{\sigma_{\ell-2}}) \mathcal{W}(\calP_{\gamma_2}) + \hat{y}_\rho^2 \hat{Y}_p \sum_{\substack{I \in J(\calP_{\gamma_2}) \\ \alpha \in I}} w(I) + \hat{y}_\rho^2 \hat{Y}_p \mathcal{W}(\calP_{\gamma_1} )\big)
\end{align*}

whereas on the right we have \[
x_{\min}^{\gamma_1} \mathcal{W}(\calP_{\gamma_1}) + x_{\min}^{\gamma_2} Y^\geq \hat{y}_\rho \hat{S}_\rho \mathcal{W}(\calP_{\gamma_2}) = x_{\min}^{\gamma_1}(\mathcal{W}(\calP_{\gamma_1}) +  \hat{y}_\rho \hat{Y}^\geq (1 + \hat{y}_{\sigma_1} + \cdots + \hat{y}_{\sigma_1} \cdots \hat{y}_{\sigma_m}) \mathcal{W}(\calP_{\gamma_2}) )
\]
Subtracting these and factoring out $\hat{y}_{\rho}^2 \hat{Y}_p x_{\min}^{\gamma_1}$, we have \begin{align*}
&\mathcal{W}(\calP_{\gamma_1}) +\sum_{\substack{I \in J(\calP_{\gamma_2})\\ \alpha \in I}} w(I) - (1 + \hat{y}_{\sigma_l} + \cdots + \hat{y}_{\sigma_l} \cdots \hat{y}_{\sigma_m})\mathcal{W}(\calP_{\gamma_2})\\
&=\mathcal{W}(\calP_{\gamma_1}) +\sum_{\substack{I \in J(\calP_{\gamma_2})\\ \alpha \in I}} w(I) - \mathcal{W}(\calP_{\gamma_2}) - \bigg(\mathcal{W}(\calP_{\gamma_1}) - \sum_{\substack{I \in J(\calP_{\gamma_2})\\ \alpha \notin I}} w(I)\bigg) = 0
\end{align*}

where we can conclude this equals zero since every order ideal of $\calP_{\gamma_2}$ either contains $\alpha$ or not. The terms with zero factors of $\lambda_{\mathbf{p}}$ can be handled similarly, reducing each side and culminating with an observation about the relationship of sets of order ideals. The terms with two factors of $\lambda_{\mathbf{p}}$ immediately agree.

\end{proof}

Lemma \ref{lem:SkeinRelationWindAtBeginning} can be seen as a skein relation on arcs which have a self-intersection from winding around an orbifold point before crossing any arcs from $T$, as is depicted on the bottom row of Table \ref{table:PuzzlePiecesWinding}.

\begin{remark}
The formula for plain arcs in Lemma \ref{lem:SkeinRelationWindAtBeginning} bears a strong resemblance to the first, second, and fourth equations found in the proof of Proposition 9.4 in \cite{cerulli2015caldero}; see also \cite[Section 12]{labardini2019family}. This similarity is not surprising as the formula for the Caldero-Chapoton function of a module in a string algebra is very similar to our expression $\chi(\mathcal{P}_\gamma)$. For instance, see \cite[Remark 5]{banaian2024skein} which discusses the use of a $\gb$-vector in both formulas.
\end{remark}

\begin{theorem}\label{thm:DoublyNotched}
If $\gamma^{(p,q)}$ is a doubly-notched ordinary arc on an orbifold $\mathcal{O}$ with triangulation $T$, and $\mathcal{O}$ is not a twice-punctured closed orbifold, then
\[
x_{\gamma^{(p,q)}}  = \chi(\mathcal{P}_{\gamma^{(p,q)}}^T) = \chi(\mathcal{G}_{\gamma^{(p,q)},T})=\mathfrak{W}(\gamma^{(p,q)},T)
\]
\end{theorem}

\begin{proof}
We focus on the first equality since the latter two will follow from \Cref{thm:MainExpansion}. 

We first suppose that $\gamma$ is a standard arc and $p = s(\gamma) \neq t(\gamma) = q$. Theorem 12.9 in \cite{musiker2011positivity} gives a relation between the cluster variable $x_{\gamma^{(p,q)}}$ and the cluster variables resulting in leaving one or both endpoints plain, \begin{equation}
x_{\gamma^{(p,q)}}x_{\gamma} - x_{\gamma^{(p)}}x_{\gamma^{(q)}} = (1-Y_p)(1-Y_q)Y \label{eq:MSWThm12.9}
\end{equation}
where $Y_p,Y_q,$ and $Y$ are $y$-monomials determined by arcs in $T$ incident to $p$ and $q$ and those crossed by $\gamma$. Corollary 3 in \cite{banaian2024skein} shows that the Laurent monomials $\chi(\calP_{\gamma})$ satisfy the same formula. Since we know $x_\gamma = \chi(\calP_{\gamma})$ from \Cref{lem:PlainOnPunctured} and similarly for $\gamma^{(p)}$ and $\gamma^{(q)}$ from \Cref{prop:SinglyNotchedExpansion} , we conclude $x_{\gamma^{(p,q)}} = \chi(\calP_{\gamma^{(p,q)}})$. 

Next, let \( \gamma \) be a standard arc with \( s(\gamma) = t(\gamma) = p \). Remarks 12.11 and 12.12 in~\cite{musiker2011positivity} indicate that $x_{\gamma^{(p,p)}}$ satisfies an analogous equation to~\Cref{eq:MSWThm12.9} , where $x_{\gamma^{(p)}}$ is replaced with $\chi(\calG_{\gamma})$. Corollary 3 in \cite{banaian2024skein} does not depend on distinct endpoints, and therefore, our same proof method holds. 

For the remainder of the proof, we will assume  \( \gamma \) is a pending arc. Let $p =  s(\gamma) = t(\gamma)$. Let $\rho \in T$ be the pending arc incident to the same orbifold point as $\gamma$ and let $v = s(\rho) = t(\rho)$. Let $\alpha$ and $\beta$ be such that $\rho, \alpha,$ and $\beta$ form a triangle in $T$ where $\beta$ immediately follows $\alpha$ when traversing this triangle in the clockwise direction. We will divide our proof into several cases based on how $\gamma$ interacts with the triangle formed by $\rho,\alpha,$ and $\beta$. In each case, we will first describe a relation in $\mathcal{A}_{\mathcal{O}}$ which involves $x_{\gamma^{(p,p)}}$ and $x_\rho$ and then we will show that our poset Laurent polynomials satisfy the same relation. By using relations in which, for all arcs except $\gamma^{(p,p)}$, we know the poset Laurent polynomials match the corresponding cluster variables, we will be able to conclude the same for $\gamma^{(p,p)}$.

\textbf{Case 1) $p \neq v$.} While $\gamma$ could cross the arcs $\rho$ many times, since these are pending arcs incident to the same orbifold point, we know there is a central pair of crossings. Here, we subdivide this case further based on whether $\gamma$ crosses $\alpha$, $\beta$, or neither before and after this pair of crossings. 

\textbf{Subcase 1a) $\gamma$ crosses $\alpha$ immediately before and after central crossings.}

Let $T'$ be a tagged triangulation of $\mathcal{O}$ containing $\rho$ such that $(T' \backslash  \{\rho\}) \cup \{\gamma^{(p,p)}\}$ is another triangulation. Let $\mu^{(p)},\eta^{(p)}$ be the arcs forming a triangle with $\rho$, as depicted below. We draw the arcs in $T'$ in solid black and arcs which were in $T$ but not in $T'$ as dashed lines. 

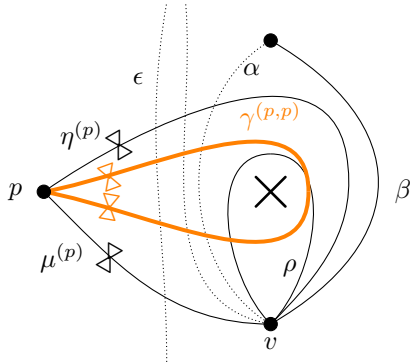
\begin{figure}[H]
    \centering
    \begin{tikzpicture}[scale=2.5, transform shape]
        
\draw[out=30,in=-30,looseness=1.5] (0,0.3) to (0,1.8);
\draw[out=150,in=-150,looseness=1, densely dotted] (0,0.3) to (0,1.8);
\draw[in=180,out=-45,looseness =1] (-1.2,1) to (0,0.3);
\draw[in=100,out=35,looseness =1] (-1.2,1) to (0.4,1.2);
\draw[in=45,out=-80,looseness=1] (0.4,1.2) to (0,0.3);
\draw[out=90,in=-100,looseness=1, densely dotted] (-.55,0.1) to (-.55,2);
\draw[out=170,in=-85,looseness=1, densely dotted] (0,0.3) to (-.45,2);

\draw[in=0,out=45,looseness =1] (0,0.3) to (0,1.2);
\draw[in=180,out=135,looseness=1] (0,0.3) to (0,1.2);

\draw[in=90,out=10,looseness=1, orange, line width = 1.5] (-1.2,1) to (0.2,1);
\draw[in=-90,out=-10,looseness=1, orange, line width = 1.5] (-1.2,1) to (0.2,1);

\draw[fill=black] (0,0.3) circle [radius=1pt];
\draw[fill=black] (0,1.8) circle [radius=1pt];
\draw[fill=black] (-1.2,1) circle [radius=1pt];

\draw[thick] (0,1) node {$\mathbf{\times}$};

\node[scale=0.4] at (0,0.2) {$v$};
\node[scale=0.4] at (-1.35,1) {$p$};
\node[scale=0.4] at (-0.1,1.65) {$\alpha$};
\node[scale=0.4] at (-0.7,1.6) {$\epsilon$};
\node[scale=0.4, orange] at (0,1.4) {$\gamma^{(p,p)}$};
\node[scale=0.4] at (-1,1.3) {$\eta^{(p)}$};
\node[scale=0.4] at (-1.1,0.65) {$\mu^{(p)}$};
\node[scale=0.4] at (0.7,1) {$\beta$};
\node[scale=0.4] at (0.1,0.6) {$\rho$};

\node[scale=0.5, rotate=120] at (-.8,1.25) {$\bowtie$};
\node[scale=0.5, rotate=60] at (-.85,.65) {$\bowtie$};

\node[scale=0.5, rotate=110, orange, line width = 1.5] at (-.85,1.08) {$\bowtie$};
\node[scale=0.5, rotate=80, orange, line width = 1.5] at (-.85,.92) {$\bowtie$};

\end{tikzpicture}
    \caption{Triangulation of subcase 1a).}
    \label{fig:Case1aT}
\end{figure}

Since flipping $\rho$ in $T'$ yields a triangulation containing $\gamma^{(p,p)}$, we know the cluster variables $x_\rho$ and $x_{\gamma^{(p,p)}}$ are related by an exchange relation. Using Lemma \ref{lem:FlipAndMutate} and Theorem \ref{thm:GeometryOfGenCA}, we have \[
x_\rho x_{\gamma^{(p,p)}} = x_{\mu^{(p)}}^2 + \lambda_{\mathbf{p}} y_\rho x_{\mu^{(p)}} x_{\eta^{(p)}} + y_\rho^2 x_{\eta^{(p)}}^2,
\]
where $\mathbf{p}$ is the order of the orbifold point enclosed by $\gamma$, and the $y$-monomials come from the observation that $b_\rho(L_\rho,T') = 1$ and for any other $\tau \in T \backslash \{\rho\}$, $b_\rho(L_\tau,T') = 0$.

From Theorem \ref{prop:SinglyNotchedExpansion}, we know $x_{\eta^{(p)}}^T = \chi(\calP_{\eta^{(p)}}^T)$ and similarly for $x_{\mu^{(p)}}^T$. Therefore, our goal is to show \begin{align}
x_\rho \chi(\calP_{\gamma^{(p,p)}}^T) &= \chi(\calP_{\mu^{(p)}}^T)^2 +\lambda_{\mathbf{p}} y_\rho \chi(\calP_{\mu^{(p)}}^T)\chi(\calP_{\eta^{(p)}}^T)+ y_\rho^2 \chi(\calP_{\eta^{(p)}}^T)^2. \label{eq:Case1a}
\end{align}

There are two copies of $\calP_{\mu^{(p)}}^T$ and two copies of $\calP_{\eta^{(p)}}^T$ embedded inside $\calP_{\gamma^{(p,p)}}^T$. We illustrate this below, where the overline denotes the poset being drawn in the opposite direction.  We denote the two central, consecutive elements labelled $\rho$ as $\rho^-$ and $\rho^+$ so that we can easily distinguish them. Recall $w_{\rho^-} = \lambda_{\mathbf{p}} \hat{y}_\rho$ and $w_{\rho^+} = \frac{1}{\lambda_{\mathbf{p}}} \hat{y}_\rho$.

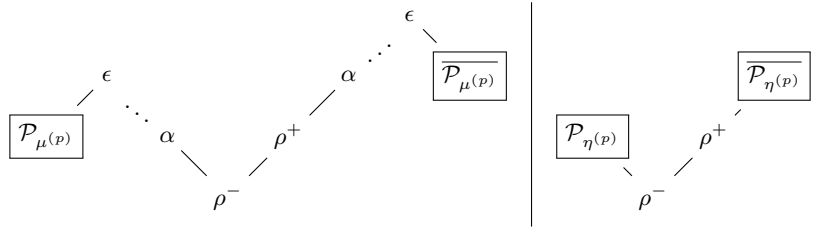
\begin{center}
\begin{tabular}{c|c}
\begin{tikzpicture}[scale=0.8]
\node(pm0) at (-4,0){$\boxed{\calP_{\mu^{(p)}}}$};
\node(e0) at (-3,1){$\epsilon$};
\node(dots) at (-2.5,0.5){$\ddots$};
\node(a0) at (-2,0){$\alpha$};
\node(t0) at (-1,-1){$\rho^-$};
\node(t1) at (0,0){$\rho^+$};
\node(a) at (1,1){$\alpha$};
\node(iddots) at (1.5,1.5){$\iddots$};
\node(e) at (2,2){$\epsilon$};
\node(pm) at (3,1){$\boxed{\overline{\calP_{\mu^{(p)}}}}$};
\draw(t1) -- (a);
\draw(e) -- (pm);
\draw(pm0) -- (e0);
\draw(a0) -- (t0);
\draw(t0) -- (t1);
\end{tikzpicture}&
\begin{tikzpicture}[scale=0.8]
\node(pm0) at (-2,0){$\boxed{\calP_{\eta^{(p)}}}$};
\node(t0) at (-1,-1){$\rho^-$};
\node(t1) at (0,0){$\rho^+$};
\node(pm1) at (1,1){$\boxed{\overline{\calP_{\eta^{(p)}}}}$};
\draw(pm1) -- (t1);
\draw(pm0) -- (t0);
\draw(t0) -- (t1);
\end{tikzpicture}
\end{tabular}
\end{center}

We begin by applying \cite[Proposition 12]{banaian2024skein} to the element $\rho^-$ in $\calP_{\gamma^{(p,p)},T}$. 
Since $\gamma^{(p,p)}$ is an ordinary arc, by Lemma \ref{lem:g}, we can replace $x_{\min}^{\gamma^{(p,p)}}$ with $\mathbf{x}^{\mathbf{g}_{\gamma^{(p,p)}}}$. This yields \begin{align}
x_\rho \chi(\calP_{\gamma^{(p,p)}}) = \chi(\calP_{\mu^{(p)}}^T)^2  + \lambda_{\mathbf{p}} y_\rho \chi(\calP_{\eta^{(p)}}^T)\chi(\mathcal{Q})\label{eq:Case1aStep1}
\end{align}

where $\mathcal{Q}$ is the poset below. To compute $\chi(\mathcal{Q})$, define the monomial $\mathbf{g}_{\mathcal{Q}}$ to be the vector resulting from applying the procedure from Lemma \ref{lem:g} to $\mathcal{Q}$. Notice that $\mathcal{Q}$ is the result of taking the poset of a corner arc based at vertex $v$ with winding number $k = 1$  and weighting the initial element with $\frac{1}{\lambda_{\mathbf{p}}} \hat{y}_\rho$ instead of $\lambda_{\mathbf{p}}\hat{y}_\rho$.  We remark that the presence of $\lambda_{\mathbf{p}}$ in~\Cref{eq:Case1aStep1} comes from the fact that the second term on the right-hand side corresponds to all order ideals  which contain $\rho^-$.

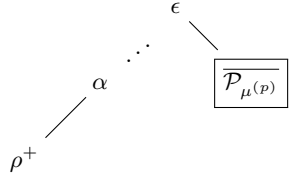
\begin{center}
\begin{tikzpicture}
\node(t1) at (0,0){$\rho^+$};
\node(a) at (1,1){$\alpha$};
\node(iddots) at (1.5,1.5){$\iddots$};
\node(e) at (2,2){$\epsilon$};
\node(pm) at (3,1){$\boxed{\overline{\calP_{\mu^{(p)}}}}$};
\draw(t1) -- (a);
\draw(e) -- (pm);
\end{tikzpicture}
\end{center}

We can apply the first part of Lemma \ref{lem:SkeinRelationWindAtBeginning} to rewrite $\chi(\mathcal{Q})$ as \begin{align}
\chi(\mathcal{Q}) =\chi(\calP_{\mu^{(p)}}^T) +  \frac{1}{\lambda_{\mathbf{p}}} y_\rho \chi(\calP_{\eta^{(p)}}^T).\label{eq:Case1aStep2}
\end{align}

Combining \Cref{eq:Case1aStep1} and \Cref{eq:Case1aStep2} finishes the proof of~\Cref{eq:Case1a}, which in turn completes the proof that $x_{\gamma^{(p,p)}}^T = \chi(\calP_{\gamma^{(p,p)}}^T)$.

\textbf{Subcase 1b) $\gamma$ crosses $\beta$ immediately before and after central crossings.}

 In this case, there will be more elementary laminations besides $L_\rho$ which contribute to the $y$-monomial of the exchange relation between $x_\rho$ and $x_{\gamma^{(p,p)}}$. This set of elementary laminations will come from arcs which $\gamma$ crosses before and after crossing $\rho$ which lie counterclockwise from $\rho$. However, this does not change the mechanics of the proof, and we omit further details.

\textbf{Subcase 1c) $\gamma$ only crosses $\rho$. }

This case is akin to case 1a. 

\textbf{Case 2) $p = v$}. Here, we must distinguish between cases when $\gamma^0$ coincides with $\rho$ and  when it is distinct. We begin with the former and then use this result  to address the latter.

\textbf{Subcase 2a) $\gamma^0 = \rho$}.

We begin by assuming that the bigon enclosing $\rho$ in $T$ has two distinct vertices, as below. This implies that $\alpha$ and $\beta$, the two arcs forming a triangle with $\rho$, are standard arcs.  Let $\delta$ be the arc which cuts out a monogon containing $p$ and the orbifold point and forms a triangle with $\alpha$ and $\beta$. Note that $\delta$ need not be in $T$, in which case it crosses all spokes incident to $p$ other than $\alpha, \beta,$ and $\rho$. Let $\tau$ be the pending arc incident to the same orbifold point as $\rho$ and enclosed in the bigon $\alpha,\beta$.

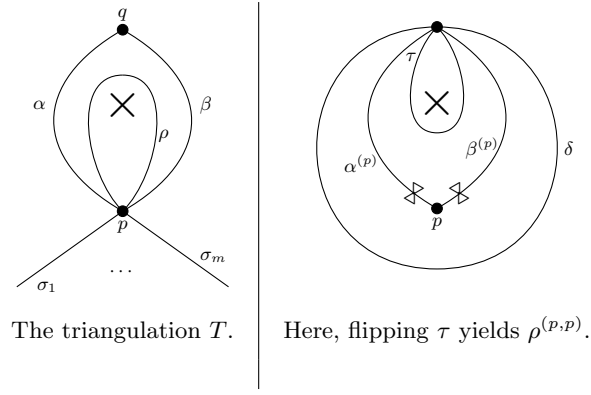
\begin{center}
    
\begin{tabular}{c|c}
    \centering
\begin{tikzpicture}[scale=2, transform shape]

\tikzset{->-/.style={decoration={
  markings,
  mark=at position #1 with {\arrow{>}}},postaction={decorate}}}


\draw[out=30,in=-30,looseness=1.5] (0,0.3) to (0,1.5);
\draw[out=150,in=-150,looseness=1.5] (0,0.3) to (0,1.5);

\draw[in=0,out=45,looseness =1] (0,0.3) to (0,1.2);
\draw[in=180,out=135,looseness=1] (0,0.3) to (0,1.2);

\draw[fill=black] (0,0.3) circle [radius=1pt];
\draw[fill=black] (0,1.5) circle [radius=1pt];
\draw[thick] (0,1) node {$\mathbf{\times}$};

\draw(0,0.3) to (-0.7,-0.2);
\draw(0,0.3) to (0.7,-0.2);

\node[scale=0.4] at (0,0.2) {$p$};
\node[scale=0.4] at (0,1.6) {$q$};
\node[scale=0.4] at (-0.55,1) {$\alpha$};
\node[scale=0.4] at (0.55,1) {$\beta$};
\node[scale=0.4] at (0.27,0.8) {$\rho$};
\node[scale=0.4] at (-0.5,-0.2) {$\sigma_1$};
\node[scale=0.4] at (0.6,0) {$\sigma_m$};
\node[scale=0.4] at (0,-0.1) {$\cdots$};

\node[scale=0.5] at (0,-.5){The triangulation $T$.};
\end{tikzpicture}
  &  
\begin{tikzpicture}[scale=2, transform shape]

\tikzset{->-/.style={decoration={
  markings,
  mark=at position #1 with {\arrow{>}}},postaction={decorate}}}


\draw[out=30,in=-30,looseness=1.5] (0,0.3) to (0,1.5);
\draw[out=150,in=-150,looseness=1.5] (0,0.3) to (0,1.5);

\draw[out=0,in=0,looseness=1.7] (0,1.5) to node[right,scale=0.4]{$\delta$}(0,-.1);
\draw[out=180,in=180,looseness=1.7] (0,1.5) to (0,-.1);

\draw[in=0,out=-45,looseness =1] (0,1.5) to (0,0.8);
\draw[in=-180,out=-135,looseness=1] (0,1.5) to (0,0.8);

\draw[fill=black] (0,0.3) circle [radius=1pt];
\draw[fill=black] (0,1.5) circle [radius=1pt];
\draw[thick] (0,1) node {$\mathbf{\times}$};

\node[scale=0.4] at (0,0.2) {$p$};
\node[scale=0.4] at (-0.5,0.6) {$\alpha^{(p)}$};
\node[scale=0.4] at (0.3,0.7) {$\beta^{(p)}$};
\node[scale=0.4] at (-0.17,1.3) {$\tau$};
\node[scale=0.5, rotate=120] at (0.15,0.4) {$\bowtie$};
\node[scale=0.5, rotate=60] at (-0.15,0.4) {$\bowtie$};
\node[scale=0.5] at (0,-.5){Here, flipping $\tau$ yields $\rho^{(p,p)}$.};
\end{tikzpicture}
\\
     & 
\end{tabular}
\end{center}

We have the following four exchange relations in $\mathcal{A}_{\mathcal{O}}$. 

\[
x_\rho x_\tau = x_\alpha^2 + \lambda_{\mathbf{p}} y_\rho x_\alpha x_\beta + y_\rho^2 x_\beta^2
\qquad \qquad 
x_{\rho^{(p,p)}} x_\tau = x_{\beta^{(p)}}^2 + \lambda_{\mathbf{p}} y_\alpha y_\rho x_{\alpha^{(p)}} x_{\beta^{(p)}} + y_\alpha^2 y_\rho^2 x_{\alpha^{(p)}}^2
\]

\[
x_{\alpha^{(p)}}x_\beta = x_\delta + y_{\sigma_1}y_{\sigma_2} \ldots y_{\sigma_m}y_\beta x_\tau
\qquad \qquad 
x_{\beta^{(p)}}x_\alpha = x_\tau + y_{\alpha}y_{\rho}^2 x_\delta
\]

Let $Y:= y_{\sigma_1}y_{\sigma_2} \ldots y_{\sigma_m}y_\beta$. 
We can combine these relations to give an expression for $x_{\rho^{(p,p)}}$ in terms of arcs in $T$. First, substituting expressions for $x_{\alpha^{(p)}}$ and $x_{\beta^{(p)}}$ into the exchange relation between $x_\tau$ and $x_{\rho^{(p,p)}}$ and sorting terms based on $x_\tau$ gives \begin{align*}
x_{\tau}x_{\rho^{(p,p)}} &= \frac{1}{x_\alpha^2x_\beta^2}\bigg(x_\tau^2\big(x_\beta^2 + \lambda_{\mathbf{p}} y_\alpha y_\rho Yx_\alpha x_\beta + y_\alpha^2 y_\rho^2 Y^2 x_\alpha^2 \big) \\ &+  y_\alpha y_\rho x_\delta x_\tau\big( 2y_\rho (x_\beta^2 + y_\alpha Y x_\alpha^2) + \lambda_{\mathbf{p}}  x_\alpha x_\beta  (1+ y_\alpha y_\rho^2Y) \big) + y_\alpha^2 y_\rho^2 x_\delta^2 \big(x_\alpha^2 + \lambda_{\mathbf{p}} y_\rho x_\alpha x_\beta + y_\rho^2 x_\beta^2 \big)\bigg)
\end{align*}

Now, using the exchange relation between $x_\rho$ and $x_\tau$,  we can divide the right-hand side by $x_\tau$. Then, we can expand $x_\delta$ in terms of $T$ and verify that the resulting Laurent polynomial is exactly $\chi(\calP_{\rho^{(p,p)}}^T)$.  This verification uses the fact that $\mathbf{x}^{\mathbf{g}_{\rho^{(p,p)}}} = \frac{1}{x_{\rho}}$ and the identity $U_2(\lambda_{\mathbf{p}}) + 1 = \lambda_{\mathbf{p}}^2$.

It is possible that one of $\alpha$ or $\beta$ is a pending arc or is in a self-folded triangle; suppose that this is $\alpha$. In such a case, we can follow the same strategy as before after first mutating $\beta$, thereby enclosing $\rho$ again in a genuine bigon. 

\textbf{Subcase 2b) $\gamma \neq \rho$ and $\gamma$ crosses $\alpha$ before and after central crossings.}

From the compatibility rules for tagged triangulations, there is no tagged triangulation $T$ containing $\rho$ such that $T- \{\rho\} \cup \{\gamma^{(p,p)}\}$ is another tagged triangulation. However, there do exist such triangulations for $\rho^{(p,p)}$. We draw a generic example in Figure \ref{fig:Case2bT}, using our assumption that $\gamma$ crosses $\alpha$ before and after its central crossings. As in Case 1, let $\mu^{(p,p)}$ and $\eta^{(p,p)}$ be arcs in $T$ which form a triangle with $\rho^{(p,p)}$.

\begin{figure}[h!]
    \centering
    \begin{tabular}{c | c}
        \begin{tikzpicture}[scale=2.5, transform shape]

            \tikzset{->-/.style={decoration={
                markings,
                mark=at position #1 with {\arrow{>}}},postaction={decorate}}}

            \draw[out=30,in=-30,looseness=1.5] (0,0.3) to (0,1.5);
            \draw[out=150,in=-150,looseness=1.5] (0,0.3) to (0,1.5);

            \draw[in=0,out=45,looseness =1] (0,0.3) to (0,1.2);
            \draw[in=180,out=135,looseness=1] (0,0.3) to (0,1.2);

            \draw[fill=black] (0,0.3) circle [radius=1pt];
            \draw[fill=black] (0,1.5) circle [radius=1pt];
            \draw[thick] (0,1) node {$\mathbf{\times}$};

            \draw(0,0.3) to (-0.7,-0.2);
            \draw(0,0.3) to (0.7,-0.2);

            \node[scale=0.4] at (0,0.2) {$p$};
            \node[scale=0.4] at (0,1.6) {$q$};
            \node[scale=0.4] at (-0.55,1) {$\alpha$};
            \node[scale=0.4] at (0.55,1) {$\beta$};
            \node[scale=0.4] at (0.27,0.8) {$\rho$};
            \node[scale=0.4] at (-0.5,-0.2) {$\sigma_1$};
            \node[scale=0.4] at (0.6,0) {$\sigma_m$};
            \node[scale=0.4] at (0,-0.1) {$\cdots$};
        \end{tikzpicture}
        &
        \begin{tikzpicture}[scale=2.5, transform shape]
            \draw[out=30,in=-30,looseness=1.5] (0,0.3) to (0,1.8);
            \draw[out=150,in=-150,looseness=1, densely dotted] (0,0.3) to (0,1.8);
            \draw[in=-135,out=-150,looseness =1] (0,0.3) to (-1.2,1);

            \draw[in=-135,out=195,looseness=1, orange, line width = 1.2] (0,0.3) to (-1,1);
            \draw[in=90,out=45,looseness=1, orange, line width = 1.2] (-1,1) to (0.1,1);
            \draw[in=-135,out=185,looseness=1, orange, line width = 1.2] (0,0.3) to (-0.9,1);
            \draw[in=-90,out=45,looseness=1, orange, line width = 1.2] (-0.9,1) to (0.1,1);

            \draw[in=100,out=35,looseness =1] (-1.2,1) to (0.4,1.2);
            \draw[in=45,out=-80,looseness=1] (0.4,1.2) to (0,0.3);

            \draw[out=170,in=-85,looseness=1, densely dotted] (0,0.3) to (-.45,2);
            \draw (0,0.3) to (-1.1,-0.4);

            \draw[in=45,out=180,looseness=1.5] (0,0.3) to (-0.85,0.9);
            \draw[in=-135,out=180,looseness=0.8] (0,0.3) to (-0.85,0.9);

            \draw[in=0,out=45,looseness =1] (0,0.3) to (0,1.2);
            \draw[in=180,out=135,looseness=1] (0,0.3) to (0,1.2);
            \node[scale=0.3, rotate=150] at (0.18,0.6) {$\bowtie$};
            \node[scale=0.3, rotate=30] at (-0.18,0.6) {$\bowtie$};
            \node[scale=0.3, rotate=150] at (0.37,0.8) {$\bowtie$};
            \node[scale=0.3, rotate=100] at (-.25,0.23) {$\bowtie$};
            \node[scale=0.3, rotate=130] at (.28,.5) {$\bowtie$};
            \node[scale=0.3, rotate=120] at (-.31,0.1) {$\bowtie$};
            \node[scale=0.2, rotate=60, orange, line width = 1.2] at (-.51,0.4) {$\bowtie$};
            \node[scale=0.25, rotate=60, orange, line width = 1.2] at (-.66,0.4) {$\bowtie$};
            \node[scale=0.2, rotate=60] at (-.4,0.4) {$\bowtie$};
            \node[scale=0.25, rotate=20] at (-.49,0.75) {$\bowtie$};

            \draw[fill=black] (0,0.3) circle [radius=1pt];
            \draw[fill=black] (0,1.8) circle [radius=1pt];

            \draw[thick] (0,1) node {$\mathbf{\times}$};

            \node[scale=0.3] at (0,0.2) {$p$};
            \node[scale=0.3] at (-0.1,1.65) {$\alpha$};
            \node[scale=0.3] at (-0.5,1.6) {$\sigma_l$};
            \node[scale=0.3] at (-1,1.3) {$\eta^{(p,p)}$};
            \node[scale=0.3, orange] at (0,1.35) {$\gamma^{(p,p)}$};
            \node[scale=0.25] at (-0.55,0.6) {$\mu^{(p,p)}$};
            \node[scale=0.3] at (-1,-0.15) {$\sigma_{h+1}^{(p)}$};
            \node[scale=0.3] at (0.7,1) {$\beta^{(p)}$};
            \node[scale=0.3] at (0.1,0.8) {$\rho^{(p,p)}$};

            \clip[draw](-0.7,0.8) circle(0.1);

            \begin{scope}
                \clip (-0.7,0.8) circle(0.1);
                \fill[pattern=north east lines](-0.7,0.8) circle(0.1);
            \end{scope}

        \end{tikzpicture}
    \end{tabular}
    \caption{The two triangulations, $T$(left) and $T'$(right), for Subcase 2b}
    \label{fig:Case2bT}
\end{figure}
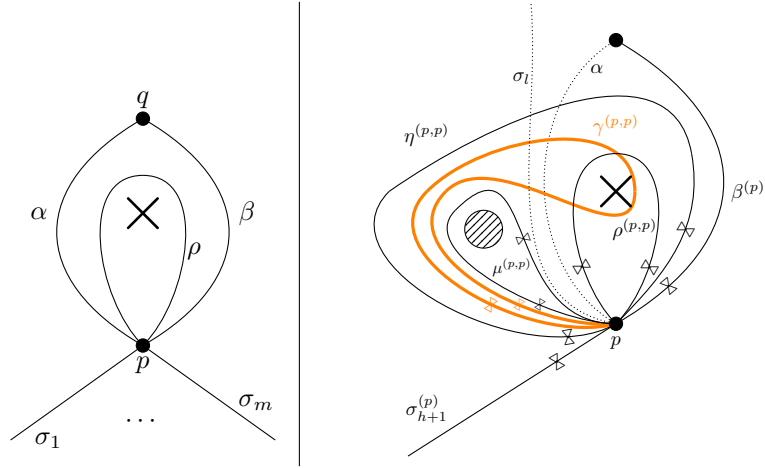

We elaborate on our indexing of arcs in $T$ and $T'$ as in Figure \ref{fig:Case2bT}. Suppose that the set of arcs incident to $p$ in $T$ are $\rho,\alpha = \sigma_1,\sigma_2,\ldots,\beta = \sigma_m$, where this list goes in counterclockwise order. Let $1 \leq l \leq m$ be such that $\gamma$ and $\eta$ cross $\sigma_1,\ldots,\sigma_l$. This implies that (up to choice of orientation) the last triangle which $\mu^{(p,p)}$ passes through is bordered by $\sigma_l$ and $\sigma_{l+1}$. Let $h \geq l$ be such that the first triangle which $\mu^{(p,p)}$ passes through is bordered by $\sigma_h$ and $\sigma_{h+1}$. For convenience, we also orient $\eta^{(p,p)}$ so that the first triangle it passes through is bordered by $\sigma_h$ and $\sigma_{h+1}$.

Next, we will calculate the column of $B_{T'}$ associated to $\rho^{(p,p)}$ where $B_{T'}$ is the $B$-matrix indexed by $T'$. From Lemma \ref{lem:FlipAndMutate}, we immediately have $b_{\eta^{(p,p)},\rho^{(p,p)}} = 1$ and $b_{\mu^{(p,p)},\rho^{(p,p)}} = -1$. Now, to compute the shear coordinates for $\rho^{(p,p)}$, we use \cite[Definition 13.1.i]{fomin2018cluster}. This shows that $b_{\rho^{(p,p)}}(L_{\rho},T') = -1$ and for all $1 \leq i \leq l$, $b_{\rho^{(p,p)}}(L_{\sigma_i},T') = -1$. Moreover, for all other elementary laminations $L$, we see $b_{\rho^{(p,p)}}(L,T') = 0$. Therefore, if $Y:= y_\rho y_\alpha y_{\sigma_2} \ldots y_{\sigma_l}$, we have \begin{align}
x_{\rho^{(p,p)}} x_{\gamma^{(p,p)}} = x_{\eta^{(p,p)}}^2 + \lambda_{\mathbf{p}} Y x_{\eta^{(p,p)}}x_{\mu^{(p,p)}} + Y^2 x_{\mu^{(p,p)}}^2.\label{eq:Case2bExchange}
\end{align}

By previous sections of this proof, we already know our combinatorial expansion formula to be valid for all arcs involved in this exchange relation except $\gamma^{(p,p)}$, and we will again be interested in showing that the poset Laurent polynomials satisfy the same equation. We first draw the poset $\calP_{\rho^{(p,p)}}^T$ using the notation here. As we have done before, each element is written as $(\text{label},\text{weight})$.

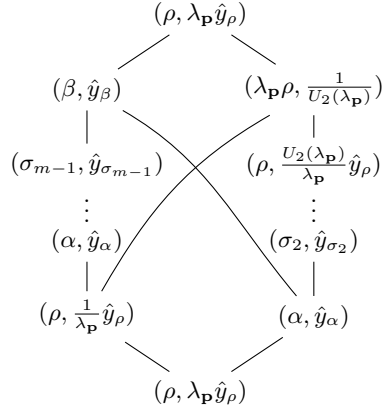
\begin{center}
\begin{tikzpicture}
\node(0) at (0,0){$(\rho,\lambda_{\mathbf{p}} \hat{y}_\rho)$};
\node(l1) at (-1.5,1){$(\rho,\frac{1}{\lambda_{\mathbf{p}}}\hat{y}_\rho)$};
\node(r1) at (1.5,1){$(\alpha,\hat{y}_\alpha)$};
\node(l2) at (-1.5,2){$(\alpha,\hat{y}_\alpha)$};
\node(r2) at (1.5,2){$(\sigma_2,\hat{y}_{\sigma_2})$};
\node(l3) at (-1.5,2.5){$\vdots$};
\node(r3) at (1.5,2.5){$\vdots$};
\node(l4) at (-1.5,3){$(\sigma_{m-1},\hat{y}_{\sigma_{m-1}})$};
\node(r4) at (1.5,3){$(\rho,\frac{U_2(\lambda_{\mathbf{p}})}{\lambda_{\mathbf{p}}}\hat{y}_\rho)$};
\node(l5) at (-1.5,4){$(\beta,\hat{y}_\beta)$};
\node(r5) at (1.5,4){$(\lambda_{\mathbf{p}} \rho, \frac{1}{U_2(\lambda_{\mathbf{p}})})$};
\node(1) at (0,5){$(\rho,\lambda_{\mathbf{p}} \hat{y}_{\rho})$};
\draw(0) -- (l1);
\draw(l1) -- (l2);
\draw(0) -- (r1);
\draw(r1) -- (r2);
\draw(l4) -- (l5);
\draw(r4) -- (r5);
\draw(l5) -- (1);
\draw(r5) -- (1);
\draw(l1) to[out = 60, in = 210] (r5);
\draw(r1) to[out = 130, in = -30] (l5);
\end{tikzpicture}   
\end{center}

As in Subcase 1a, the poset $\calP_{\gamma^{(p,p)}}^T$ contains two copies of $\calP_{\mu^{(s(\mu))}}^T$  and two copies of $\calP_{\eta^{(s(\eta))}}^T$, as is depicted to the left and right respectively below. Again, the weight of $\rho^-$ is $\lambda_{\mathbf{p}} \hat{y}_\rho$ and the weight of $\rho^+$ is $\frac{1}{\lambda_{\mathbf{p}}} \hat{y}_\rho$. 

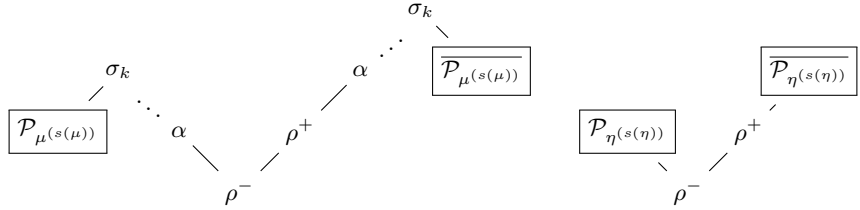
\begin{center}
\begin{tabular}{cc}
\begin{tikzpicture}[scale=0.8]
\node(pm0) at (-4,0){$\boxed{\calP_{\mu^{(s(\mu))}}}$};
\node(e0) at (-3,1){$\sigma_k$};
\node(dots) at (-2.5,0.5){$\ddots$};
\node(a0) at (-2,0){$\alpha$};
\node(t0) at (-1,-1){$\rho^-$};
\node(t1) at (0,0){$\rho^+$};
\node(a) at (1,1){$\alpha$};
\node(iddots) at (1.5,1.5){$\iddots$};
\node(e) at (2,2){$\sigma_k$};
\node(pm) at (3,1){$\boxed{\overline{\calP_{\mu^{(s(\mu))}}}}$};
\draw(t1) -- (a);
\draw(e) -- (pm);
\draw(pm0) -- (e0);
\draw(a0) -- (t0);
\draw(t0) -- (t1);
\end{tikzpicture}&
\begin{tikzpicture}[scale=0.8]
\node(pm0) at (-2,0){$\boxed{\calP_{\eta^{(s(\eta))}}}$};
\node(t0) at (-1,-1){$\rho^-$};
\node(t1) at (0,0){$\rho^+$};
\node(pm1) at (1,1){$\boxed{\overline{\calP_{\eta^{(s(\eta))}}}}$};
\draw(pm1) -- (t1);
\draw(pm0) -- (t0);
\draw(t0) -- (t1);
\end{tikzpicture}
\end{tabular}
\end{center}

We apply \cite[Proposition 14]{banaian2024skein} to the element $\rho^-$ in $\calP_{\gamma^{(p,p)}}^T$. This yields \begin{align}
x_{\rho^{(p,p)}}\chi (\calP_{\gamma^{(p,p)}}^T) = \chi(\calP_{\eta^{(p,p)}}^T) \chi(\mathcal{Q'}) + Y^2 \chi(\calP_{\mu^{(p,p)}}^T)^2, \label{eq:case2b1}
\end{align}

where $\mathcal{Q}'$ is the result of taking all elements strictly to the right of $\rho^-$ in $\calP_{\gamma^{(p,p)}}^T$ and including the following loop:  \[
(\rho^+,\frac{1}{\lambda_{\mathbf{p}} \hat{y}_\rho}) \succeq (\rho, \lambda_{\mathbf{p}}\hat{y}_\rho) \preceq (\alpha, \hat{y}_\alpha) \preceq (\sigma_2,\hat{y}_{\sigma_2})\preceq \cdots \preceq (\beta,\hat{y}_\beta) \preceq (\rho,\frac{U_{2}(\lambda_{\mathbf{p}})}{\lambda_{\mathbf{p}}}) \preceq (\lambda_{\mathbf{p}}\rho,\frac{1}{U_{2}(\lambda_{\mathbf{p}})}) \succeq (\rho^+,\frac{1}{\lambda_{\mathbf{p}} \hat{y}_\rho}).
\]

The fact that this is the loop we attach comes from the details of the combinatorial skein relation and the fact that this is a chain in the poset $\mathcal{P}_{\rho^{(p,p)}}^T$
Notice the poset $\mathcal{Q}'$ is also the poset for a notched corner arc based at $p$ with winding number $\mathbf{p}-2$. We can therefore apply Lemma \ref{lem:SkeinRelationWindAtBeginning} to $\chi(\mathcal{Q}')$, which yields the identity \begin{align}
\chi(\mathcal{Q'}) = \chi(\calP_{\eta^{(p,p)}}^T) + \lambda_{\mathbf{p}} Y \chi(\calP_{\mu^{(p,p)}}^T).\label{eq:Case2b2}
\end{align}

Now, combining~\Cref{eq:case2b1} and~\Cref{eq:Case2b2} shows that our poset Laurent polynomials satisfy the same relationship as~\Cref{eq:Case2bExchange}, and we conclude $x_{\gamma^{(p,p)}}^T = \chi(\calP_{\gamma^{(p,p)}}^T)$.

\textbf{Subcase 2c) $\gamma \neq \rho$ and $\gamma$ crosses $\beta$ before and after central crossings.}

This is similar to Subcase 2b, except here $\rho^{(p,p)}$ only has a nontrivial shear coordinate with $L_\rho$. We omit further details.

\textbf{Subcase 2d) $\gamma \neq \rho$ and $\gamma$ does not cross $\alpha$ or $\beta$.}

This is again similar to, and indeed simpler than, Subcase 2b.
\end{proof}

\begin{remark}
Our use of \cite[Theorem 12.9]{musiker2011positivity} forces us to exclude twice-punctured closed orbifolds (with an arbitrary number of orbifold points) from Theorem \ref{thm:DoublyNotched}. By using a geometric argument, Pilaud, Reading, and Schroll gave poset expansion formulas which include this case \cite{pilaud2023posets}. We expect that a suitable analogue of this argument would also extend our result to the case of twice-punctured closed orbifolds, and we will pursue this approach in the sequel \cite{BKK26}.
\end{remark}

\Cref{lem:PlainOnPunctured}, \Cref{prop:SinglyNotchedExpansion}, and \Cref{thm:DoublyNotched} along with \Cref{thm:MainExpansion}  establish \Cref{thm:Correctness}.
Another corollary of~\Cref{thm:DoublyNotched} is an orbifold analogue of Theorem~12.9 in~\cite{musiker2011positivity}. Define $\mathds{1}_{T}(\gamma)$ as the indicator function with $\mathds{1}_{T}(\gamma)= 1$ if $\gamma \in T$ and $\mathds{1}_{T}(\gamma)= 0$ otherwise. Recall $e(\gamma,\gamma')$ denotes the minimal number of crossings of arcs $\gamma$ and $\gamma'$ up to isotopy. Let $e_p(\gamma)$ be the number of ends of $\gamma$ incident to a given puncture $p$. 

\begin{cor}\label{cor:Thm12.9}
    Fix a tagged triangulation $T$ of a punctured orbifold $\mathcal{O}=(S,M,Q)$, and let $\mathcal{A}_{\mathcal{O}}$ be the corresponding generalized cluster algebra with principal coefficients. Let $p$ and $q$ be punctures in $\mathcal{O}$ and let $\gamma$ be a generalized arc between $p$ and $q$. Assume that no tagged arc in $T$ is notched at either $p$ or $q$. If $\gamma$ is a standard arc incident to $p$ and $q$ and $p \neq q$, then
\[x_{\gamma}x_{\gamma^{(p,q)}}-x_{\gamma^{(p)}}x_{\gamma^{(q)}}y_{\gamma}^{\mathds{1}_{T}(\gamma)}=\left(1-\prod_{\mu \in T}y_{\mu}^{e_{p}(\mu)}\right)\left(1-\prod_{\mu \in T}y_{\mu}^{e_{q}(\mu)}\right)\prod_{\mu\in T}y_{\mu}^{e(\mu,\gamma)}\]

and if $\gamma$ is an arc (standard or pending) such that $s(\gamma) = t(\gamma)$, then,  
\[x_{\gamma}x_{\gamma^{(p,q)}}-\chi(\mathcal{P}_{\gamma^{(s(\gamma))}}^T)\chi(\mathcal{P}_{\gamma^{(t(\gamma))}}^T)y_{\gamma}^{\mathds{1}_{T}(\gamma)}=\left(1-\prod_{\mu \in T}y_{\mu}^{e_{p}(\mu)}\right)\left(1-\prod_{\mu \in T}y_{\mu}^{e_{q}(\mu)}\right)\prod_{\mu\in T}y_{\mu}^{e(\mu,\gamma)}\]
\end{cor}
\begin{proof}
    Notice that, since $\gamma$ is standard, it does not wind nontrivially around any orbifold point. This implies that the poset $\calP_{\gamma}$ resembles the posets considered in \cite{banaian2024skein}. Therefore, we can combine Propositions 4 and 5 therein with \Cref{thm:DoublyNotched} to show the desired result.
\end{proof}

 ~\Cref{cor:Thm12.9} can also be proved using the method of~\cite{musiker2011positivity}, by deriving the identities from laminations and shear coordinates on the covering space.

\begin{remark}
Given a generalized arc or a closed curve, our expansion formulas \emph{define} a Laurent polynomial associated to this curve. Here, the way to check that our definition is reasonable is to check that these satisfy skein relations. Based on initial experiments and previous work (see also \cite[Theorem 8.17]{banaian2025cluster}), we expect to verify this in the sequel ~\cite{BKK26}. Indeed, many skein relations will follow from previous work~\cite{banaian2024skein,musiker2013bases}; it only remains to check new combinatorial configurations which arise from the presence of orbifold points. 
\end{remark}

\section{Concluding Discussion}\label{sec:conclusion}

In this section, we discuss our method of constructing the auxiliary tiles and elements. We show how this is a choice and hint at other choices which could be made. Crucially, in \Cref{prop:Invariance}, we show that these choices do not affect the final result. We will also discuss the future direction of this research, which involves exhibiting bases of these algebras, ideally recovering well-known bases from other settings.  

\subsection{Invariance of Formulas}\label{sec:Invariance}

Here, we explain the fundamental reasoning behind some of our construction. Given an orbifold $\mathcal{O}$ with triangulation $T$ and an arc $\gamma$ on $\mathcal{O}$, we construct a local lift of $(\mathcal{O},T,\gamma)$, which is a (possibly punctured) polygon with a dissection. We refer to this dissected polygon as the \emph{tile cover} of $\gamma$, inspired by a similar construction for triangulated surfaces in \cite{pilaud2023posets}, and denote it $(\widetilde{S},T_\gamma)$. Let $\widetilde{\gamma}$ be the lift of $\gamma$ to the tile cover. For this construction, look at only the solid arcs in Table \ref{tab:LiftsForConstruction}. To be specific, the polygons drawn here are $\mathbf{p}$-gons where $\mathbf{p}$ is the order of the orbifold point. The winding number of $\gamma$ informs us how many vertices of the polygon lie to the right of $\widetilde{\gamma}$.  The tile cover sits inside the universal cover of $(\mathcal{O},T)$; see, for example, figures in \cite{Chekhov-Shapiro}.

\begin{table}[H]
    \centering
    \begin{tabular}{c|c|c|c}
    \begin{tikzpicture}[scale=1, transform shape]
        \draw[out=60,in=0,looseness=1.2] (0,-3) to (0,-1);
        \draw[out=120,in=180,looseness=1.2] (0,-3) to (0,-1);
        \filldraw (0,-3) circle (3pt);
        \node[scale=2, thick] at (0,-1.5) {$\times$};
        \draw[out=-20,in=-100,looseness=1.5,line width = 1.2, orange] (-1,-1.5) to (0.4,-1.5);
        \draw[out=90,in=90,looseness=1.8,line width = 1.2,  orange] (0.4,-1.5) to (-0.4,-1.5);
        \draw[out=-90,in=-90,looseness=1.5,line width = 1.2,  orange] (-0.4,-1.5) to (0.3,-1.5);
        \draw[out=90,in=90,looseness=1.5,line width = 1.2,  orange] (0.3,-1.5) to (-0.3,-1.5);
        \draw[out=-90,in=180,looseness=1.5, line width = 1.2, orange] (-0.3,-1.5) to (0,-1.7);
        \draw[out=0,in=160,looseness=1,line width = 1.2,  orange, ->] (0,-1.7) to (1,-2);
        \node[scale=0.8] at (0.3,-0.8) {$\rho$};

    \end{tikzpicture}
    &
    \begin{tikzpicture}[scale=0.8, transform shape]
        \draw(0,-0.3) to (1,-1.3);
        \draw (0,-0.3) to (0,1);
        \draw (0,1) to (1,2);
        \draw (2,2) to (3,1);
        \draw(3,1) to (3,-0.3);
        \draw(3,-0.3) to (2,-1.3);
        \draw[dashed] (0,1) to (3,1);
        \draw[dashed](0,1) to (3,-0.3);
        \draw[dashed] (0,-0.3) to (3,-0.3);
        \node[] at (1.5,2){$\cdots$};
        \node[] at (1.5,-1.3){$\cdots$};

        \draw[line width = 1.2, orange, ->] (-0.5,0.3) to (3.5,0.3);
        \node at (-0.3,0.8) {$\rho$};
        \node at (3.3,0.8) {$\rho$};
        \node[scale=0.8] at (1.6,0.65) {$U_k\rho$};
        \node[scale=0.8] at (1.5,-0.5) {$U_{k-1}\rho$};
        \node[scale=0.8] at (1.5,1.3) {$U_{k+1}\rho$};
    \end{tikzpicture}
    &\begin{tikzpicture}[scale=1, transform shape]
        \draw[out=60,in=0,looseness=1.2] (0,-3) to (0,-1);
        \draw[out=120,in=180,looseness=1.2] (0,-3) to (0,-1);
        \filldraw (0,-3) circle (3pt);
        \node[scale=2, thick] at (0,-1.5) {$\times$};
        \draw[out=80,in=-90,looseness=1, line width = 1.2, orange] (0,-3) to (0.4,-1.5);
        \draw[out=90,in=90,looseness=1.5, line width = 1.2, orange] (0.4,-1.5) to (-0.4,-1.5);
        \draw[out=-90,in=-90,looseness=2, line width = 1.2,  orange] (-0.4,-1.5) to (0.3,-1.5);
        \draw[out=90,in=90,looseness=1.5, line width = 1.2, orange] (0.3,-1.5) to (-0.3,-1.5);
        \draw[out=-90,in=-100,looseness=1.5, line width = 1.2, orange] (-0.3,-1.5) to (0.2,-1.5);
        \draw[out=80,in=-140,looseness=1, line width = 1.2, orange, ->] (0.2,-1.5) to (1,-0.5);
        \node at (-0.8,-0.8) {$\rho=\tau_{i_{1}}$};

    \end{tikzpicture}
    &
    \begin{tikzpicture}[scale=0.8, transform shape]
        \draw(0,-0.3) to (1,-1.3);
        \draw(3,-0.3) to (2,-1.3);
        \draw (0,-0.3) to (0,1);
        \draw (0,1) to (1,2);
       \node[] at (1.5,2){$\cdots$};
       \node[] at (1.5,-1.3){$\cdots$};
        \draw (2,2) to (3,1);
        \draw (3,1) to (3,-0.3);
        \draw[dashed] (0,1) to (3,1);
        \draw[dashed] (0,1) to (3,-0.3);
        \draw[line width = 1.2, orange, ->] (0,1) to (3.5,0.3);
        \node at (-0.3,0.4) {$\rho$};
        \node at (3.3,0.8) {$\rho$};
        \node at (1.4,0) {$U_{k-1}\rho$};
        \node at (1.5,1.3) {$U_{k}\rho$};
        \node[orange] at (2.4,0.3) {$\gamma$};
    \end{tikzpicture}\\
     \begin{tikzpicture}[scale=1, transform shape]

    \draw (0,0) to (1,1);
    \draw (0,0) to (1,-1);
    \draw (1,1) to (1,-1);
    \draw[out=150,in=90,looseness=1.5] (0,0) to (-1,0);
    \draw[out=-150,in=-90,looseness=1.5] (0,0) to (-1,0);
    \draw[orange,line width = 1.2, ->] (0,0) to (2,0);
    \node at (0.4,0.8) {$\sigma_1$};
    \node at (0.4,-0.8) {$\sigma_m$};
    \node at (-0.7,0.5) {$\sigma_n$};
    \node at (1.3,0.3) {$\tau_{i_1}$};
    \node at (-0.7,0) {$\times$};
    \node at (-0.1,0.7) {$\iddots$};
    \node at (-0.3,-0.5) {$\ddots$};

    \node[orange] at (0.3,0) {\rotatebox[origin=c]{-90}{$\bowtie$}};
    \end{tikzpicture}
    &
    \begin{tikzpicture}[scale=0.8, transform shape]

    \draw (0,0) to (1,1);
    \draw (0,0) to (1,-1);
    \draw (1,1) to (1,-1);
    \draw (0,0) to (-0.5,-1);
    \draw (0,0) to (-0.5,1);
    \draw[dashed] (-0.5,-1) to (-0.5,1);
    \draw (-2,1) to (-0.5,1);
    \draw (-2,-1) to (-0.5,-1);
    \node[scale=1.4] at (-2,0){$\vdots$};

    \draw[line width = 1.2, orange,->] (0,0) to (2,0);
    \node at (0.4,0.7) {$\sigma_1$};
    \node at (0.4,-0.7) {$\sigma_m$};
    \node at (-0.2,-0.9) {$\sigma_n$};
    \node at (-0.2,0.9) {$\sigma_n$};
    \node at (-1,0) {$\lambda_{\mathbf{p}} \sigma_n$};
    \node at (1.3,0.3) {$\tau_{1}$};
    \node[orange] at (0.3,0) {\rotatebox[origin=c]{-90}{$\bowtie$}};

    \end{tikzpicture}
    &
 \begin{tikzpicture}[scale=1, transform shape]
        \draw[out=60,in=0,looseness=1.2] (0,-3) to (0,-1);
        \draw[out=50,in=0,line width = 1.2, looseness=1.3, orange] (0,-3) to (0,-0.8);
        \draw[out=120,in=180,looseness=1.2] (0,-3) to (0,-1);
        \draw[out=130,in=180,line width = 1.2, looseness=1.3, orange] (0,-3) to (0,-0.8);
        \filldraw (0,-3) circle (3pt);
        \node[scale=2, thick] at (0,-1.5) {$\times$};
        \node[orange, rotate=-25] at (0.29,-2.6){$\bowtie$};
        \node[orange, rotate=25] at (-0.29,-2.6){$\bowtie$};
        \node[scale=0.8] at (0.2,-2.3) {$\rho$};
    \end{tikzpicture}
         &  
 \begin{tikzpicture}[scale=0.8, transform shape]
        \draw(0,-0.5) to (1,-1.3);
        \draw (0,-0.5) to (0,1.2);
        \draw (0,1.2) to (1,2);
        \draw[dashed](0,-0.5) to node[right, yshift = -25pt, xshift = -10pt]{$\lambda_{\mathbf{p}} \rho$} (1,2);
        \draw[dashed] (1,-1.3) to node[right]{ $U_2 \rho$} (1,2);
        \draw(1,-1.3) -- (2,-1.3);
        \draw(1,2) -- (2,2);
        \draw (2,-1.3) -- (3,-0.5);
        \draw(2,2) -- (3,1.2);
        \node[] at (3,.35){$\vdots$};
        \draw[line width = 1.2, orange] (0,-0.5) to [out = 110, in = 250, looseness=1.2] (0,1.2);
        \node[left, orange, xshift = -5pt] at (0,.35){$\gamma$};
        \node[orange, rotate=25] at (-0.12,-0.25){$\bowtie$};
        \node[orange, rotate=-25] at (-0.12,0.9){$\bowtie$};
    \end{tikzpicture}
    \\
    \end{tabular}
    \caption{By gluing these pieces together along with the ordinary and self-folded triangles which $\gamma$ passes through triangles, one can construct a triangulated polygon $(\widetilde{S},\widetilde{T})$ with a diagonal $\widetilde{\gamma}$. In the case of a notched corner arc, we combine the top right and bottom left tiles. Our graphs with auxiliary tiles and posets come from the ordinary constructions with respect to $(\widetilde{S},\widetilde{T},\widetilde{\gamma})$. }
    \label{tab:LiftsForConstruction}
\end{table}

We regard every $\mathbf{p}$-gon with $\mathbf{p} > 3$ formed by this dissection as regular. Next, we complete this dissection to a triangulation by adding diagonals to these $\mathbf{p}$-gons. We call this a \emph{completion of the tile cover} and denote the additional arcs by $T^{\mathrm{aux}}$. In particular, $\widetilde{T}:= T_\gamma \cup T^{\mathrm{aux}}$ forms a triangulation of $\widetilde{S}$. Table \ref{tab:LiftsForConstruction} describes our conventional completion. This choice of completion reflects earlier conventions we have established (see Remark \ref{rmk:WhatCCWGivesUs}) and was partially established to minimize the number of auxiliary tiles, or equivalently, the number of additional arcs crossed by the lifted arc $\widetilde{\gamma}$. 

The following is evident from the construction. 

\begin{lemma}\label{lem:SnakeGraphFromLift}
The loop or band graph $\mathcal{G}_{\gamma,T}$ is equivalent as a weighted graph to the loop or band graph $\mathcal{G}_{\widetilde{\gamma},\widetilde{T}}$. Moreover, if $\widetilde{\chi}(\mathcal{G}_{\widetilde{\gamma},\widetilde{T}})$ is the result of setting all $y_\tau = 1$  in $\chi(\mathcal{G}_{\widetilde{\gamma},\widetilde{T}})$ for all $\tau \in T^{\mathrm{aux}}$, then $\chi(\mathcal{G}_{\gamma,T}) = \widetilde{\chi}(\mathcal{G}_{\widetilde{\gamma},\widetilde{T}})$.
\end{lemma}

From Lemma \ref{lem:SnakeGraphFromLift}, we can recover $\chi(\mathcal{G}_{\gamma,T})$ from $\mathcal{G}_{\widetilde{\gamma},\widetilde{T}}$ by appropriately adjusting the $y$-monomial; a priori, $\mathcal{G}_{\widetilde{\gamma},\widetilde{T}}$ has no auxiliary tiles because $\widetilde{\gamma}$ is simply a diagonal in a triangulated polygon.

While the tile cover is fixed by $(T,\gamma)$, the completion of the tile cover involves a set of choices. In this section, we highlight that our choices do not affect the resulting Laurent polynomials. We begin with an example.

\begin{example}\label{ex:tricover}
Let $\gamma$ be the arc in item (I) of Table \ref{table:whole} and let the order of the orbifold point be $\mathbf{p}=5$. Let $\phi = 2\cos(\pi/5)$ and recall $\phi^2 - 1 = \phi$. As there is only one orbifold point, we can easily draw the universal cover of the triangulated orbifold; this is given by the solid arcs in Figure \ref{fig:liftExamples}. The tile cover would consist of trimming the triangles which $\gamma$ does not pass through. Our choice of completion to a triangulation,  $\widetilde{T}$, is the leftmost triangulation (i.e., including dashed and solid arcs).  Two other ways to complete the dissection are also drawn in Figure \ref{fig:liftExamples} (call these $\widetilde{T}',\widetilde{T}''$), and the three corresponding snake graphs are provided in Figure \ref{fig:liftExamplesSnakeGraphs}. 

The fact that $\widetilde{\chi}(\mathcal{G}_{\gamma,\widetilde{T}}) = \widetilde{\chi}(\mathcal{G}_{\gamma,\widetilde{T}'})$ is a consequence of Proposition \ref{prop:NonzeroWindingAuxiliary}, where $\widetilde{\chi}$ is as in Lemma \ref{lem:SnakeGraphFromLift}.

Next, we compare $\widetilde{\chi}(\mathcal{G}_{\gamma,\widetilde{T}})$ and $\widetilde{\chi}(\mathcal{G}_{\gamma,\widetilde{T}''})$.
\begin{align*}
\widetilde{\chi}(\mathcal{G}_{\gamma, \widetilde{T}})& = \frac{1}{\phi x_ax_\rho^3}\bigg( \phi^2 x_\rho^3x_b x_c  +  \phi^2 y_a x_\rho^2x_b^2x_d + \phi y_\rho x_ax_\rho^3x_c+ (\phi^2+\phi) y_ay_\rho  x_ax_\rho^2x_bx_d \\
&+(\phi+1) y_ay_\rho^2 \phi x_a^2x_\rho^2x_d \bigg)
\end{align*}

\begin{align*}
\widetilde{\chi}(\mathcal{G}_{\gamma, \widetilde{T}''})& = \frac{1}{\phi^2 x_ax_\rho^4}\bigg((2\phi+1) x_\rho^4 x_bx_c+ (2\phi+1) y_a x_\rho^3x_b^2x_d + \phi^2 y_\rho  x_ax_\rho^4 x_c   \\
&+ (2\phi^2 + \phi) y_ay_\rho  x_ax_\rho^3 x_d+ \phi^3 y_ay_\rho^2 x_a^2 x_\rho^3 x_d\bigg)
\end{align*}

After repeatedly using the identity $\phi^2 = \phi + 1$, both expansions can be shown to reduce to the following
\[
\frac{1}{ x_ax_\rho}\bigg(\phi x_\rho x_bx_c+ \phi y_a x_b^2 x_d  + y_\rho x_ax_\rho x_c + \phi^2 y_ay_\rho  x_ax_bx_d  +\phi y_ay_\rho^2 x_a^2 x_d\bigg),
\]
which is also the Laurent expansion associated to the snake graph with hexagonal tile in Table \ref{table:whole}. 
\end{example}

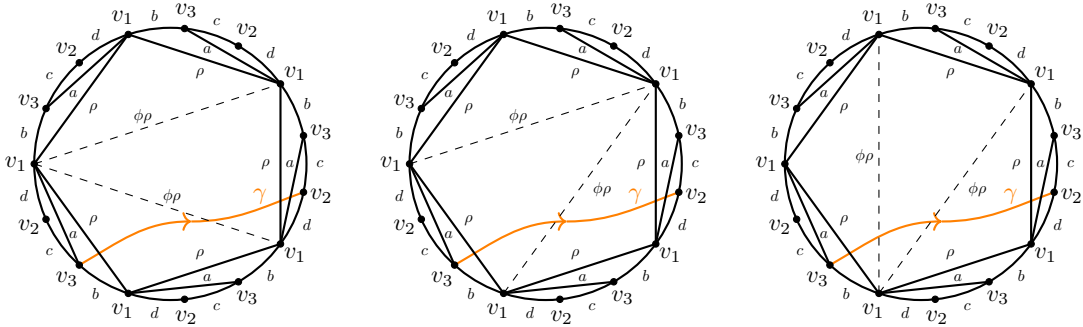
\begin{figure}[H]
    \centering
\begin{tabular}{ccc}
\begin{tikzpicture}[scale=0.6, transform shape]

    \def\radius{3}
    
    \draw[thick] (0,0) circle (\radius);

    \draw[in=180,out=30,looseness=1, orange, thick,  ->] (84+144:\radius) to (290:0.45*\radius); %
    \draw[in=200,out=0,looseness=1, orange, thick] (290:0.45*\radius) to (276+72:\radius);

    \node[orange, scale=1.5] at (340: 0.7*\radius) {$\gamma$};

    \foreach \i in {12,36,...,348} {
        \filldraw[black] (\i:\radius) circle (2pt);
    }

    \foreach \i in {12,84,84+72,84+144,84+144+72} {
        \node[scale=1.5] at (\i:\radius + 0.4) {$v_3$};
    }
    \foreach \i in {36,36+72,36+144,36+144+72, 36+288} {
        \node[scale=1.5] at (\i:\radius + 0.4) {$v_1$};
    }
    \foreach \i in {60,132,204,276,276+72} {
        \node[scale=1.5] at (\i:\radius + 0.4) {$v_2$};
    }

    \foreach \i in {0,72,144,216,288} {
        \node[scale=1] at (\i:\radius + 0.3) {$c$};
    }

    \foreach \i in {24,96,168,240,312} {
        \node[scale=1] at (\i:\radius + 0.3) {$b$};
    }
    \foreach \i in {48,120,192,264,264+72} {
        \node[scale=1] at (\i:\radius + 0.3) {$d$};
    }

    \foreach \i/\j in {36/108, 108/180, 180/252, 252/324, 324/36+360} {
        \draw[thick] (\i:\radius) -- (\j:\radius);
        \node at ({(\i+\j)/2}:\radius*0.7) {$\rho$};
    }

    \draw[thick] (36:\radius) -- (36+48:\radius);
    \node at (72:\radius*0.88) {$a$};
    
    \draw[thick] (108:\radius) -- (108+48:\radius);
    \node at (144:\radius*0.88) {$a$};

    \draw[thick] (180:\radius) -- (180+48:\radius);
    \node at (216:\radius*0.88) {$a$};

    \draw[thick] (252:\radius) -- (252+48:\radius);
    \node at (288:\radius*0.88) {$a$};

    \draw[thick] (324:\radius) -- (324+48:\radius);
    \node at (0:\radius*0.88) {$a$};

    \draw[dashed] (180:\radius) to node[left, yshift = 5pt]{$\phi \rho$} (180-144:\radius);
    \draw[dashed] (180:\radius) to node[right, yshift = 5pt]{$\phi \rho$} (180-144-72:\radius);
\end{tikzpicture}
&
\begin{tikzpicture}[scale=0.6, transform shape]

    \def\radius{3}
    
    \draw[thick] (0,0) circle (\radius);

    \draw[in=180,out=30,looseness=1, orange, thick,  ->] (84+144:\radius) to (290:0.45*\radius); %
    \draw[in=200,out=0,looseness=1, orange, thick] (290:0.45*\radius) to (276+72:\radius);

    \node[orange, scale=1.5] at (340: 0.7*\radius) {$\gamma$};

    \foreach \i in {12,36,...,348} {
        \filldraw[black] (\i:\radius) circle (2pt);
    }

    \foreach \i in {12,84,84+72,84+144,84+144+72} {
        \node[scale=1.5] at (\i:\radius + 0.4) {$v_3$};
    }
    \foreach \i in {36,36+72,36+144,36+144+72, 36+288} {
        \node[scale=1.5] at (\i:\radius + 0.4) {$v_1$};
    }
    \foreach \i in {60,132,204,276,276+72} {
        \node[scale=1.5] at (\i:\radius + 0.4) {$v_2$};
    }

    \foreach \i in {0,72,144,216,288} {
        \node[scale=1] at (\i:\radius + 0.3) {$c$};
    }

    \foreach \i in {24,96,168,240,312} {
        \node[scale=1] at (\i:\radius + 0.3) {$b$};
    }
    \foreach \i in {48,120,192,264,264+72} {
        \node[scale=1] at (\i:\radius + 0.3) {$d$};
    }

    \foreach \i/\j in {36/108, 108/180, 180/252, 252/324, 324/36+360} {
        \draw[thick] (\i:\radius) -- (\j:\radius);
        \node at ({(\i+\j)/2}:\radius*0.7) {$\rho$};
    }

    \draw[thick] (36:\radius) -- (36+48:\radius);
    \node at (72:\radius*0.88) {$a$};
    
    \draw[thick] (108:\radius) -- (108+48:\radius);
    \node at (144:\radius*0.88) {$a$};

    \draw[thick] (180:\radius) -- (180+48:\radius);
    \node at (216:\radius*0.88) {$a$};

    \draw[thick] (252:\radius) -- (252+48:\radius);
    \node at (288:\radius*0.88) {$a$};

    \draw[thick] (324:\radius) -- (324+48:\radius);
    \node at (0:\radius*0.88) {$a$};

    \draw[dashed] (180:\radius) to node[left, yshift = 5pt]{$\phi \rho$} (36:\radius);
    \draw[dashed] (252:\radius) to node[right, xshift = 5pt]{$\phi \rho$} (36:\radius);
\end{tikzpicture}
&
\begin{tikzpicture}[scale=0.6, transform shape]

    \def\radius{3}
    
    \draw[thick] (0,0) circle (\radius);

    \draw[in=180,out=30,looseness=1, orange, thick,  ->] (84+144:\radius) to (290:0.45*\radius); %
    \draw[in=200,out=0,looseness=1, orange, thick] (290:0.45*\radius) to (276+72:\radius);

    \node[orange, scale=1.5] at (340: 0.7*\radius) {$\gamma$};

    \foreach \i in {12,36,...,348} {
        \filldraw[black] (\i:\radius) circle (2pt);
    }

    \foreach \i in {12,84,84+72,84+144,84+144+72} {
        \node[scale=1.5] at (\i:\radius + 0.4) {$v_3$};
    }
    \foreach \i in {36,36+72,36+144,36+144+72, 36+288} {
        \node[scale=1.5] at (\i:\radius + 0.4) {$v_1$};
    }
    \foreach \i in {60,132,204,276,276+72} {
        \node[scale=1.5] at (\i:\radius + 0.4) {$v_2$};
    }

    \foreach \i in {0,72,144,216,288} {
        \node[scale=1] at (\i:\radius + 0.3) {$c$};
    }

    \foreach \i in {24,96,168,240,312} {
        \node[scale=1] at (\i:\radius + 0.3) {$b$};
    }
    \foreach \i in {48,120,192,264,264+72} {
        \node[scale=1] at (\i:\radius + 0.3) {$d$};
    }

    \foreach \i/\j in {36/108, 108/180, 180/252, 252/324, 324/36+360} {
        \draw[thick] (\i:\radius) -- (\j:\radius);
        \node at ({(\i+\j)/2}:\radius*0.7) {$\rho$};
    }

    \draw[thick] (36:\radius) -- (36+48:\radius);
    \node at (72:\radius*0.88) {$a$};
    
    \draw[thick] (108:\radius) -- (108+48:\radius);
    \node at (144:\radius*0.88) {$a$};

    \draw[thick] (180:\radius) -- (180+48:\radius);
    \node at (216:\radius*0.88) {$a$};

    \draw[thick] (252:\radius) -- (252+48:\radius);
    \node at (288:\radius*0.88) {$a$};

    \draw[thick] (324:\radius) -- (324+48:\radius);
    \node at (0:\radius*0.88) {$a$};

    \draw[dashed] (252:\radius) to node[left, yshift = 5pt]{$\phi \rho$} (252-144:\radius);
    \draw[dashed] (252:\radius) to node[right, xshift = 5pt]{$\phi \rho$} (252-144-72:\radius);
\end{tikzpicture}\end{tabular}
\caption{The triangulated polygon on the left is the tile cover of the arc $\gamma$ from Example \ref{ex:1poset} with orbifold point of order 5 using our conventions as in Table \ref{tab:LiftsForConstruction}. The other two triangulations consist of other choices of completion of the tile cover. Here, $\phi=\lambda_5 = 2\cos(\pi/5)$ denotes the golden ratio. Note $U_2(\phi) = \phi^2 - 1 = \phi$.}\label{fig:liftExamples}
\end{figure}

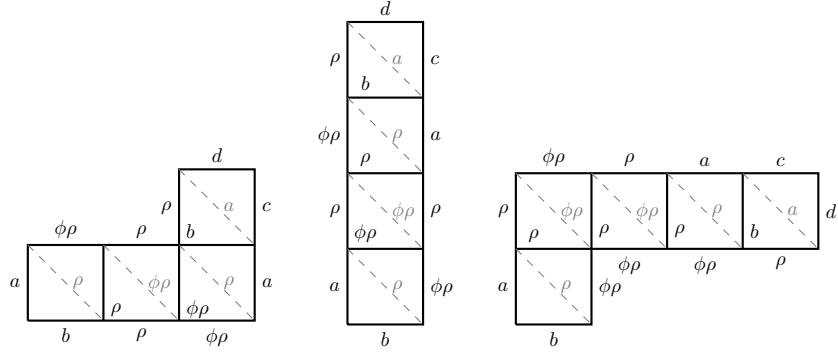
\begin{figure}[h]
    \centering
\begin{tabular}{ccc}
\begin{tikzpicture}[scale = 1]
\draw[thick] (0,0) to node[below, scale = 0.75]{$b$} (1,0) to node[below, scale = 0.75]{$\rho$} (2,0) to node[below, scale = 0.75]{$\phi \rho$} (3,0) to node[right, scale = 0.75]{$a$} (3,1) to node[right, scale = 0.75]{$c$} (3,2) to node[above, scale = 0.75]{$d$} (2,2) to node[left, scale = 0.75]{$\rho$} (2,1) to node[above, scale = 0.75]{$\rho$} (1,1) to node[above, scale = 0.75]{$\phi \rho$} (0,1) to node[left, scale = 0.75]{$a$} (0,0);
\draw[thick] (1,0) to node[right, yshift = -10pt, scale = 0.75]{$\rho$} (1,1);
\draw[thick] (2,0) to node[right, yshift = -10pt, scale = 0.75]{$\phi \rho$} (2,1) to node[above, xshift = -10pt, scale = 0.75]{$b$}(3,1);
\draw[gray,dashed] (0,1) to node[right, scale = 0.75]{$\rho$} (1,0);
\draw[gray,dashed] (1,1) to node[right, scale = 0.75]{$\phi \rho$} (2,0);
\draw[gray,dashed] (2,1) to node[right, scale = 0.75]{$\rho$} (3,0);
\draw[gray,dashed] (2,2) to node[right, scale = 0.75]{$a$} (3,1);
\end{tikzpicture}&
\begin{tikzpicture}[scale=1, every node/.style={scale=0.75}]
\draw[thick] (0,0) to node[below]{$b$} (1,0) to node[right]{$\phi \rho$} (1,1) to node[right] {$\rho$} (1,2) to node[right]{$a$} (1,3) to node[right]{$c$}(1,4) to node[above]{$d$} (0,4) to node[left]{$\rho$} (0,3) to node[left]{$\phi \rho$} (0,2) to node[left]{$\rho$} (0,1) to node[left]{$a$}(0,0);
\draw[thick] (0,1) to node[above, xshift = -10pt]{$\phi \rho$} (1,1);
\draw[thick] (0,2) to node[above, xshift = -10pt]{$\rho$} (1,2);
\draw[thick] (0,3) to node[above, xshift = -10pt]{$b$} (1,3);
\draw[gray,dashed] (0,1) to node[right]{$\rho$} (1,0);
\draw[gray,dashed] (0,2) to node[right]{$\phi \rho$} (1,1);
\draw[gray,dashed] (0,3) to node[right]{$\rho$} (1,2);
\draw[gray,dashed] (0,4) to node[right]{$a$} (1,3);
\end{tikzpicture}&
\begin{tikzpicture}[scale=1, every node/.style={scale=0.75}]
\draw[thick] (0,0) to node[below]{$b$} (1,0) to node[right]{$\phi \rho$} (1,1) to node[below]{$\phi \rho$} (2,1) to node[below]{$\phi \rho$} (3,1) to node[below]{$\rho$} (4,1) to node[right]{$d$} (4,2) to node[above]{$c$} (3,2) to node[above]{$a$} (2,2) to node[above]{$\rho$} (1,2) to node[above]{$\phi \rho$} (0,2) to node[left]{$\rho$} (0,1) to node[left]{$a$} (0,0);
\draw[thick] (0,1) to node[above, xshift = -10pt]{$\rho$}(1,1) to node[right, yshift = -10pt]{$\rho$} (1,2);
\draw[thick] (2,1) to node[right, yshift = -10pt]{$\rho$} (2,2);
\draw[thick] (3,1) to node[right, yshift = -10pt]{$b$} (3,2);
\draw[gray,dashed] (0,1) to node[right]{$\rho$} (1,0);
\draw[gray,dashed] (0,2) to node[right]{$\phi \rho$} (1,1);
\draw[gray,dashed] (1,2) to node[right]{$\phi \rho$} (2,1);
\draw[gray,dashed] (2,2) to node[right]{$\rho$} (3,1);
\draw[gray,dashed] (3,2) to node[right]{$a$} (4,1);
\end{tikzpicture}
\end{tabular}
    \caption{The three snake graphs associated to the arc $\gamma$ with respect to the three auxiliary triangulations in Figure \ref{fig:liftExamples}.}
    \label{fig:liftExamplesSnakeGraphs}
\end{figure}

\begin{prop}\label{prop:Invariance}
Let $\gamma$ be a (possibly generalized) arc or closed curve on an orbifold $\mathcal{O}$ with triangulation $T$. Let $\widetilde{T_1}$ and $\widetilde{T_2}$ be two completions of the tile cover of $(\mathcal{O},T,\gamma)$. Let $\widetilde{\chi}$ be as in Lemma \ref{lem:SnakeGraphFromLift}. If $\mathcal{G}_{\widetilde{\gamma},\widetilde{T_1}}$ and $\mathcal{G}_{\widetilde{\gamma},\widetilde{T_2}}$ are the two associated loop or band graphs, then $\widetilde{\chi}(\mathcal{G}_{\widetilde{\gamma},\widetilde{T_1}}) = \widetilde{\chi}(\mathcal{G}_{\widetilde{\gamma},\widetilde{T_2}})$.
\end{prop}

\begin{proof}

Thanks to Theorem \ref{thm:MainExpansion}, it is equivalent to show \begin{equation}
\sum_{\alpha \in TW(T_1,\gamma)} x(\alpha)\widetilde{y}(\alpha) = \sum_{\alpha \in TW(T_2,\gamma)} x(\alpha)\widetilde{y}(\alpha) \label{eq:SumOverTWalksEqual}
\end{equation}
where $\widetilde{y}(\alpha)$ is the result of setting to 1 all $y$-monomials associated to arcs in $\widetilde{T}_i \backslash T_\gamma$. 

It is sufficient to assume that $\widetilde{T_1}$ and $\widetilde{T_2}$ only differ in one subgon formed by the tile cover of $(\mathcal{O},T,\gamma)$. Assume first that this subgon is formed as in the top left of Table  \ref{tab:LiftsForConstruction}. That is, there is a pending arc $\rho$ which $\gamma$ crosses twice, contributing a $\mathbf{p}$-gon to the tile cover where $\mathbf{p}$ is the order of the enclosed orbifold point. For convenience, we label the two lifts of $\rho$ as $\rho_L$ and $\rho_R$ where $\gamma$ crosses $\rho_L$ before $\rho_R$.

Let us set certain subsets of $TW(\widetilde{T_1},\widetilde{\gamma})$ as follows:

\begin{itemize}
    \item[(i)] $A$ consists of all $T$-walks which travel along $\rho_L$ and $\rho_R$ in the positive direction;
    \item[(ii)] $B$ consists of all $T$-walks which travel along $\rho_L$ in the positive direction and $\rho_R$ in the negative direction;
    \item[(iii)] $C$ consists of all $T$-walks which travel along $\rho_L$ in the negative direction and $\rho_R$ in the positive direction;
    \item[(iv)] $D$ consists of all $T$-walks which travel along both $\rho_L$ and $\rho_R$ in the negative direction.
\end{itemize} 

See below for a diagram of a generic $T$-walk in $A$. Clearly, $A,B,C,$ and $D$ partition $TW(\widetilde{T_1},\widetilde{\gamma})$. Define $A',B',C',$ and $D'$ to be the parallel sets but with respect to $\widetilde{T_2}$.

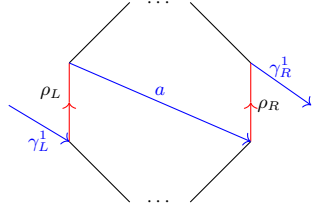
\begin{center}
\begin{tikzpicture}[scale=0.8, transform shape]
        \draw(0,-0.3) to (1,-1.3);
        \draw[->, red] (0,-0.3) to (0,0.35);
        \draw[red](0,0.35) -- (0,1);
        \draw (0,1) to (1,2);
        \draw (2,2) to (3,1);
        \draw[->, red] (3,-0.3) to (3,0.35);
        \draw[red](3,0.35) -- (3,1);
        \draw(3,-0.3) to (2,-1.3);
        \node[] at (1.5,2){$\cdots$};
        \node[] at (1.5,-1.3){$\cdots$};

        \draw[blue, ->] (-1,0.3) to node[below]{$\gamma_L^1$} (0,-0.3);
        \draw[blue,->] (3,1) to node[above]{$\gamma_R^1$} (4,0.3);
        \draw[blue,->] (0,1) to node[above]{$a$} (3,-0.3);
        \node at (-0.3,0.5) {$\rho_L$};
        \node at (3.3,0.3) {$\rho_R$};
    \end{tikzpicture}
\end{center}

Define arcs $\gamma_L^1, a, $ and $\gamma_R^1$ as in above figure.  Summing over $T$-walks in $A$ is equivalent to concatenating all possible $T$-walks for $\gamma_L^1$, $a$, and $\gamma_R^1$: \[
\sum_{\alpha \in A} x(\alpha)\widetilde{y}(\alpha) = \frac{1}{x_\rho^2}\bigg(\sum_{\alpha_1 \in TW(\widetilde{T_1},\gamma_L^1)} x(\alpha_1)\widetilde{y}(\alpha_1)\bigg) \bigg(\sum_{\alpha_2 \in TW(\widetilde{T_1},a)} x(\alpha_2)\widetilde{y}(\alpha_2)\bigg) \bigg(\sum_{\alpha_3 \in TW(\widetilde{T_1},\gamma_R^1)} x(\alpha_3)\widetilde{y}(\alpha_3)\bigg) 
\]

Note that $\widetilde{y}(\alpha_2) = 1$ for all $\alpha_2 \in TW(\widetilde{T}_1,a)$ since by construction there are no arcs from $T_\gamma$ inside this polygon.  We can view the expression $\sum_{\alpha_2 \in TW(\widetilde{T_1},a)} x(\alpha_2)$ as returning the result of taking the coefficient-free cluster variable associated to $a$ with respect to the cluster given by the restriction of $\widetilde{T_1}$ to the polygon, and specializing each initial cluster variable to a term $U_\ell(\lambda_\mathbf{p}) x_\rho$ for appropriate $\ell$ (as in Lemma \ref{lem:ChebyshevFact}). Such a specialization is a \emph{frieze}; that is, it respects all relations in the original cluster variable.  This was shown in \cite[Lemma 2.2]{holm2020p}. In particular, this shows $\sum_{\alpha_2 \in TW(\widetilde{T_1},a)} x(\alpha_2)\widetilde{y}(\alpha_2) = U_k(\lambda_\mathbf{p})x_\rho$ if $a$ is a $k$-diagonal. We conclude \[
\sum_{\alpha \in A} x(\alpha)\widetilde{y}(\alpha) = \frac{U_k(\lambda_\mathbf{p})}{x_\rho}\bigg(\sum_{\alpha_1 \in TW(\widetilde{T_1},\gamma_L^1)} x(\alpha_1)\widetilde{y}(\alpha_1)\bigg) \bigg(\sum_{\alpha_3 \in TW(\widetilde{T_1},\gamma_R^1)} x(\alpha_3)\widetilde{y}(\alpha_3)\bigg) .
\]

Notice that $\sum_{\alpha_1 \in TW(\widetilde{T_1},\gamma_L^1)} x(\alpha_1)\widetilde{y}(\alpha_1) = \sum_{\alpha_1 \in TW(\widetilde{T_2},\gamma_L^1)} x(\alpha_1)\widetilde{y}(\alpha_1)$ since $T_1$ and $T_2$ are identical outside of a single $\mathbf{p}$-gon and similarly for $\gamma_R^1$. Therefore, we find $\sum_{\alpha \in A} x(\alpha)\widetilde{y}(\alpha) = \sum_{\alpha' \in A'} x(\alpha')\widetilde{y}(\alpha')$. Moreover, we can repeat this idea with $B,C,$ and $D$.  We conclude that~\Cref{eq:SumOverTWalksEqual} is true in this case. The proof in other cases is similar.
\end{proof}

One can combine Propositions \ref{prop:AuxAndPoset} and \ref{prop:Invariance} to show a similar statement regarding posets from these completions of tile lifts.
However, such a result is not necessary for loop or band graphs with hexagonal tiles nor for $T$-walks as these do not depend on choices. Notice that the $T$-walks used in the proof of Proposition \ref{prop:Invariance} are slightly different from the $T$-walks we use for arcs on orbifolds. Indeed, on the tile cover, our $T$-walks would function as \emph{weak $T$-paths}, as in \cite{ccanakcci2021friezes}.

\subsection{Future Directions}\label{subsec:future}

The combinatorial framework developed in this paper provides tools for investigating the structural properties of orbifold-type generalized cluster algebras.

A natural first step, which we have already begun working on~\cite{BKK26}, is to expand the techniques developed in \cite{banaian2024skein} to derive explicit formulas for skein relations in these generalized cluster algebras using our combinatorial constructions. The presence of orbifold points allows several new types of scenarios to appear.

Our larger goal is to define analogues of the \emph{bangle, band, and bracelet} bases in the orbifold setting, as from \cite{musiker2013bases,Thurston14}. Skein relations will provide a critical step in the proof that these truly are bases. Natural follow-up questions include: Will the generalized bracelet basis coincide with the \emph{theta basis} of the generalized cluster algebra defined in~\cite{BLM25}? This was recently shown in the surface case by Mandel and Qin  ~\cite{MandelQin2023}. Similarly, will the generalized bangle basis coincide with the \emph{generic basis}~\cite{geiss2020generic, geiss2022schemes}? The proof of this fact in the surface case was given by Gei{\ss}, Labardini-Fragoso, and Wilson~\cite{geiss2023bangle}. Progress towards this claim in the unpunctured orbifold setting was given in~\cite{banaian2025snake}.

\printbibliography

\medskip

Institute for Mathematics, Paderborn University, Paderborn, Germany
\\
\textit{Email address: }\href{mailto:Esther.Banaian@math.uni-paderborn.de}{\texttt{Esther.Banaian@math.uni-paderborn.de}}\newline

International Center for Mathematical Sciences, Institute of Mathematics and Informatics, Bulgarian Academy of Sciences, Acad. G. Bonchev Str., Bl. 8, Sofia 1113, Bulgaria
\\
\textit{Email address: }\href{mailto:wonk@math.bas.bg}{\texttt{wonk@math.bas.bg}}\newline

David and Judi Proctor Department of Mathematics, University of Oklahoma, OK, USA
\\
\textit{Email address: }\href{mailto:elizabethkelley@ou.edu}{\texttt{elizabethkelley@ou.edu}}\newline

Department of Mathematics, Galatasaray University, Istanbul, Türkiye
\\
\textit{Email address: }\href{mailto:ezgikantarcioguz@gmail.com}{\texttt{ezgikantarcioguz@gmail.com}}\newline

International Center for Mathematical Sciences, Institute of Mathematics and Informatics, Bulgarian Academy of Sciences, Acad. G. Bonchev Str., Bl. 8, Sofia 1113, Bulgaria
\\
\textit{Email address: }\href{mailto:e.yildirim@math.bas.bg}{\texttt{e.yildirim@math.bas.bg}}

\end{document}